\documentclass[10pt]{amsart}
\usepackage{verbatim}

\usepackage{amssymb}
\usepackage{amsbsy}
\usepackage{amscd}
\usepackage{amsmath}
\usepackage{amsthm}
\usepackage{color}
\usepackage{version}

\newenvironment{NB}{
\color{red}{\bf NB}. \footnotesize 
}{}
\excludeversion{NB}
\newenvironment{NB2}{
\color{blue}{\bf NB}. \footnotesize
}{}

\theoremstyle{plain}
 \newtheorem{thm}{Theorem}[subsection]
 \newtheorem{lem}[thm]{Lemma}
 \newtheorem{prop}[thm]{Proposition}
 \newtheorem{cor}[thm]{Corollary}
 
\theoremstyle{definition}
 \newtheorem{defn}[thm]{Definition}
 \newtheorem{assume}[thm]{Assumption}
\theoremstyle{remark}
 \newtheorem{rem}[thm]{Remark}
 \newtheorem{ex}[thm]{Example}
 
\def\Bbb{\mathbb}
\def\frak{\mathfrak}
\def\cal{\mathcal}

\newcommand{ \Supp}{\operatorname{Supp}}
\newcommand{\Ext}{\operatorname{Ext}}
\newcommand{\Hom}{\operatorname{Hom}}
\newcommand{\Tor}{\operatorname{Tor}}

\newcommand{\im}{\operatorname{im}}
\newcommand{\Aut}{\operatorname{Aut}}
\newcommand{\rk}{\operatorname{rk}}
\newcommand{\gr}{\operatorname{gr}}

\newcommand{\NS}{\operatorname{NS}}
\newcommand{\coker}{\operatorname{coker}}
\newcommand{\Pic}{\operatorname{Pic}}
\newcommand{\ch}{\operatorname{ch}}
\newcommand{\td}{\operatorname{td}}

\newcommand{\Hilb}{\operatorname{Hilb}}
\newcommand{\Quot}{\operatorname{Quot}}

\newcommand{\Coh}{\operatorname{Coh}}
\newcommand{\Spec}{\operatorname{Spec}}

\newcommand{\WIT}{\operatorname{WIT}}

\newcommand{\Proj}{\operatorname{Proj}}
\newcommand{\Div}{\operatorname{Div}}

\newcommand{\Per}{\operatorname{Per}}
\newcommand{\tr}{\operatorname{tr}}
\newcommand{\alg}{\operatorname{alg}}
\newcommand{\id}{\operatorname{id}}
\newcommand{\GExt}{\operatorname{Ext}_G}
\newcommand{\GHom}{\operatorname{Hom}_G}
\newcommand{\GHilb}{G\operatorname{-Hilb}}
\newcommand{\Gchi}{\operatorname{\chi}_G}
\newcommand{\Stab}{\operatorname{Stab}}
\newcommand{\Def}{\operatorname{Def}}
\newcommand{\depth}{\operatorname{depth}}
\newcommand{\End}{\operatorname{End}}

\newcommand{\dslash}{/\!\!/} % for algebro-geometric quotient (double
\font\b=cmr10 scaled \magstep5
\def\bigzerou{\smash{\lower1.7ex\hbox{\b 0}}}

\numberwithin{equation}{section}
\begin{document}

\title
{Perverse coherent sheaves and Fourier-Mukai transforms
on surfaces}
\author{K\={o}ta Yoshioka
% Department of Mathematics\\
%Faculty of Science, Kobe University
}
\address{Department of Mathematics, Faculty of Science,
Kobe University,
Kobe, 657, Japan
}
\email{yoshioka@math.kobe-u.ac.jp}
\thanks{The author is supported by the Grant-in-aid for Scientific
Research (No.\ 18340010, No.\ 22340010), JSPS}
 \subjclass{14D20}

\begin{abstract}
We study perverse coherent sheaves on the resolution
of rational double points.
As examples, we consider rational double points on
2-dimensional moduli spaces of stable sheaves
on $K3$ and elliptic surfaces.
Then we show that perverse coherent sheaves appears
in the theory of Fourier-Mukai transforms.
As an application, we generalize the Fourier-Mukai
duality for $K3$ surfaces to our situation.
\end{abstract}
 
\maketitle

\tableofcontents

\section{Introduction}\label{sect:intro}
Let $\pi:X \to Y$ be a birational map such that
$\dim \pi^{-1}(y) \leq 1, y \in Y$. 
Then Bridgeland \cite{Br:4}
introduced the abelian category
${^{p}\Per}(X/Y) (\subset {\bf D}(X))$
of perverse coherent sheaves
in order to show that flops of smooth 3-folds
preserves the derived categories of coherent sheaves.
By using the moduli of perverse coherent sheaves on $X$,
Bridgeland constructed the flop $X' \to Y$ of $X \to Y$.
Then the Fourier-Mukai transform by the universal family induces
an equivalence ${\bf D}(X) \cong {\bf D}(X')$. 
In \cite{VB}, Van den Bergh showed that ${^{p}\Per}(X/Y)$
is Morita equivalent to the category $\Coh_{{\cal A}}(Y)$
of ${\cal A}$-modules
on $Y$ and gave a different proof of Bridgeland result, 
where ${\cal A}$ is a sheaf of (non-commutative) 
algebras over $Y$.
Although the main examples of the birational contraction
are small contraction of 3-folds, 
2-dimensional cases seem to be still interesting. 
In \cite{perv}, \cite{perv2}, Nakajima and the author 
studied perverse coherent sheaves 
for the blowing up $X \to Y$ of a smooth surface $Y$ at a point. 
In this case, by analysing wall-crossing phenomena, we 
related the moduli of stable perverse coherent sheaves
to the moduli of usual stable sheaves. 
%In this note, we shall study perverse coherent sheaves on the minimal
%resolution $X$ of rational double points on $Y$.
Next example is the minimal resolution of a rational
double point.
Let $G$ be a finite subgroup of $SU(2)$ acting on ${\Bbb C}^2$
and set $Y:={\Bbb C}^2/G$.
Let $\pi:X \to Y$ be the resolution of $Y$.
Then the relation between the perverse coherent sheaves and
the usual coherent sheaves on $X$ are discussed by Nakajima.
Their moduli spaces  
are constructed as Nakajima's quiver varieties \cite{Na:1994} and
their differences are described by the wall crossing phenomena
\cite{Na:ADHM}.
Toda \cite{T:1} also treated special cases.  
%
%For the topological invariants of
%the quiver varieties, 
%
In this paper, we are interested in the global case.
Thus we consider the minimal resolution $\pi:X \to Y$ 
of a normal projective surface $Y$ with 
rational double points as singuralities.
%
%By the construction of quiver varieties,
%moduli spaces are diffeomorphic. 
%Hence the wall crossing behaviors of topological invariants are
%trivial. 

As examples, we shall show that perverse coherent sheaves naturally appear
if we consider the Fourier-Mukai transforms on $K3$ and elliptic surfaces.
In our previous paper \cite{Y:Stability}, 
we studied Fourier-Mukai transforms defined by the 
moduli spaces of (semi)-stable sheaves $Y'$ on $X$.
Our assumption is the genericity of the polarization.
If the polarization is not general, then $Y'$ is singular at 
properly semi-stable sheaves. 
In this case, we still have the Fourier-Mukai transform
by using the resolution $X'$ of $Y'$.
Then the category of perverse coherent sheaves on $X'$
naturally appears.
In particular, we show that the universal family on
$X' \times X$ is the universal family of stable perverse coherent
sheaves on $X'$ (Theorem \ref{thm:duality}).
Thus we have a kind of duality between $X$ and $X'$,
which is a generalization of the relation between an abelian 
variety and its dual.
We call this kind of duality {\it Fourier-Mukai duality}. 
The Fourier-Mukai duality for a $K3$ surface
was studied by Bartocci, Bruzzo, Hern\'{a}ndez Ruip\'{e}rez
\cite{BBH:1}, Mukai \cite{Mu:5}, Orlov \cite{Or:1}, 
Bridgeland \cite{Br:2}, and
was first proved by Huybrechts in
\cite{H:category} under the genericity of the polarization.
He also proved that the Fourier-Mukai transform preserves 
nice abelian subcategories.
We also give a generalization of this result (Theorem \ref{thm:equiv-Phi}).  
Then we can generalize the result on the preservation of stability
by the Fourier-Mukai transform in \cite{Y:Stability}
to our situation.
%
%If we want to consider relative Fourier-Mukai transforms
%for families of surfaces, it is desireble to remove this assumption.

For the relative Fourier-Mukai transforms on elliptic surfaces,
we also get similar results.
Let $G$ be a finite group acting on a projective surface $X$.
Assume that $K_X$ is the pull-back of a line bundle on $Y:=X/G$.
Then the McKay correspondence \cite{VB} implies that
$\Coh_G(X)$ is equivalent to ${^{-1}\Per}(X'/Y)$, where
$X' \to Y$ is the minimal resolution of $Y$.
The equivalence is given by a Fourier-Mukai transform associated to
a moduli space of stable $G$-sheaves of dimension 0.
If $X$ is a $K3$ surface or an abelian surface, then
we have many 2-dimensional moduli spaces of stable $G$-sheaves. 
We also treat the Fourier-Mukai transform induced by the moduli
of $G$-sheaves.

In section \ref{sect:Morita}, 
we consider an abelian subcategory ${\cal C}$ of ${\bf D}(X)$
which is Morita equivalent to $\Coh_{{\cal A}}(Y)$, 
where $\pi:X \to Y$ be a birational contraction from a smooth variety $X$ 
and
${\cal A}$ is a sheaf of (non-commutative) 
algebras over $Y$.
We call an object of ${\cal C}$ a perverse coherent sheaf.
Since ${^{-1}\Per}(X/Y)$
is Morita equivalent to $\Coh_{{\cal A}}(Y)$
for an algebra ${\cal A}$ \cite{VB},
our definition is compatible with
Bridgeland's definition.
We also study irreducible objects 
and local projective generators of ${\cal C}$.
As examples, we shall give generalizations of $^p \Per(X/Y)$, $p=-1,0$. 
We next explain families of perverse coherent sheaves
and the relative version of Morita equivalence.
Then we can use Simpson's moduli spaces
of stable ${\cal A}$-modules \cite{S:1} to construct the moduli
spaces of stable perverse coherent sheaves. 
Since Simpson's stability is not good enough
for the 0-dimensionional objects, we also introduce a refinement
of the stability and construct the moduli space, 
which is close to King's stability \cite{King-mod}.

In section \ref{sect:RDP}, 
we study perverse coherent sheaves on the resolution
of rational double points.
We first introduce two kind of categories
$\Per(X/Y,{\bf b}_1,\dots,{\bf b}_n)$ and
$\Per(X/Y,{\bf b}_1,\dots,{\bf b}_n)^*$ 
associated to a sequence of line bundles on
the exceptional curves of the resolution of 
rational singularities 
and show that they are the category of
perverse coherent sheaves in the sense in section \ref{sect:Morita}.
They are generalizations 
of ${^{-1}\Per}(X/Y)$ and ${^{0}\Per}(X/Y)$ respectively.

We next study the moduli of 0-dimensional objects
on the resolution of rational double points.
We introduce the wall and the chamber structure and 
study the Fourier-Mukai transforms induced by the
moduli spaces.
Under a suitable stability condition for ${\Bbb C}_x$, $x \in X$,
we show that the category of perverse coherent sheaves
is equivalent to $^{-1} \Per(X/Y)$
(cf. Proposition \ref{prop:perverse=-1}).
We also construct local projective generators under suitable conditions.
 
In section \ref{sect:K3},
we consider the Fourier-Mukai transforms on $K3$ surfaces.
We generalize known facts on the 2-dimensional moduli spaces
of usual stable sheaves to those of stable objects.
Then we define similar categories ${\frak A}$ and ${\frak A}^{\mu}$
to those in \cite{Br:3}, and generalize results in
\cite{H:category}.
In particular, we study the relation of
Fourier-Mukai transforms and the categories ${\frak A},{\frak A}^\mu$
(Theorem \ref{thm:equiv-Phi}). 
This result will be used to study Bridgeland's stable objects
in \cite{MYY}.
We also prove the Fourier-Mukai duality (Theorem \ref{thm:duality}). 
Finally we give some conditions for the preservation of
Gieseker stability conditions.
Fourier-Mukai transforms on elliptic surfaces and    
Fourier-Mukai transforms by equivariant coherent sheaves 
are treated in sections \ref{sect:elliptic}
and \ref{sect:equivariant}.

{\bf Notation.}
\begin{enumerate}
\item
For a scheme $X$, $\Coh(X)$ denotes the category of coherent sheaves on 
$X$ and
${\bf D}(X)$ the bounded derived category
of $\Coh(X)$.
We denote the Grothendieck group of $X$ by $K(X)$.
\item  
Let ${\cal A}$ be a sheaf of ${\cal O}_X$-algebras on a scheme $X$
which is coherent as an ${\cal O}_X$-module.
Let $\Coh_{\cal A}(X)$ be the category of coherent ${\cal A}$-modules
on $X$ and ${\bf D}_{\cal A}(X)$ the bounded derived category
of $\Coh_{\cal A}(X)$.
\item
Assume that $X$ is a smooth projective variety.
Let $E$ be an object of ${\bf D}(X)$.
$E^{\vee}:={\bf R}{\cal H}om_{{\cal O}_X} (E,{\cal O}_X)$ denotes the
dual of $E$.
We denote the rank of $E$ by $\rk E$.
For a fixed nef divisor $H$ on $X$, 
$\deg(E)$ denotes the degree of $E$
with respect to $H$. 
For $G \in K(X)$, $\rk G>0$,
we also define the twisted rank and degree by
$\rk_G(E):=\rk (G^{\vee} \otimes E)$ and
$\deg_G(E):=\deg(G^{\vee} \otimes E)$ respectively.
We set $\mu_G(E):=\deg_G(E)/\rk_G(E)$, if $\rk E \ne 0$.
 \item
{\bf Integral functor.}
For two schemes $X$, $Y$ and
an object ${\cal E} \in {\bf D}(X \times Y)$,
$\Phi_{X \to Y}^{{\cal E}}:{\bf D}(X) \to {\bf D}(Y)$ 
is the integral functor
\begin{equation}
 \Phi_{X \to Y}^{{\cal E}}(E):=
{\bf R}p_{Y*}({\cal E} \overset{{\bf L}}{\otimes} p^*_X(E)),\;
E \in {\bf D}(X),
\end{equation}
where $p_X:X \times Y \to X$ and $p_Y:X \times Y \to Y$ are projections.
If $\Phi_{X \to Y}^{{\cal E}}$ is an equivalence,
it is said to be the {\it Fourier-Mukai transform}.
\item
${\bf D}(X)_{op}$ denotes the opposit category of
${\bf D}(X)$.
We have a functor
\begin{equation*}
\begin{matrix}
D_X:& {\bf D}(X) & \to & {\bf D}(X)_{op}\\
& E & \mapsto & E^{\vee}.
\end{matrix}
\end{equation*}
\item
Assume $X$ is a smooth projective surface.
\begin{enumerate}
\item
We set 
$H^{ev}(X,{\Bbb Z}):=\bigoplus_{i=0}^2 H^{2i}(X,{\Bbb Z})$.
In order to describe the element $x$ of 
$H^{ev}(X,{\Bbb Z})$, we use two kinds of expressions:
$x=(x_0,x_1,x_2)=x_0+x_1+x_2 \varrho_X$, where
$x_0 \in {\Bbb Z}, x_1 \in H^2(X,{\Bbb Z}), x_2 \in {\Bbb Z}$,
and $\int_X \varrho_X=1$.
For $x=(x_0,x_1,x_2)$, we set $\rk x:=x_0$ and $c_1(x)=x_1$.
\item
We define a homomorphism
\begin{equation}\label{eq:gamma}
\begin{matrix}
\gamma:& K(X) & \to & {\Bbb Z} \oplus \NS(X) \oplus {\Bbb Z}\\
& E & \mapsto & (\rk E,c_1(E),\chi(E))  
\end{matrix}
\end{equation}
and set $K(X)_{\mathrm{top}}:=K(X)/\ker \gamma$.
We denote $E \mod \ker \gamma$ by $\tau(E)$.
$K(X)_{\mathrm{top}}$ has a bilinear form $\chi(\;\;,\;\;)$.
\item
{\bf Mukai lattice.}
We define a lattice structure $\langle \quad,\quad \rangle$
on $H^{ev}(X,{\Bbb Z})$ by 
\begin{equation}
\begin{split}
\langle x,y \rangle:=&-\int_X x^{\vee} \cup y\\
=& (x_1,y_1)-(x_0 y_2+x_2 y_0),
\end{split} 
\end{equation}
where $x=(x_0,x_1,x_2)$ (resp. $y=(y_0,y_1,y_2)$)
and $x^{\vee}=(x_0,-x_1,x_2)$.
It is now called the {\it Mukai lattice}.
Mukai lattice has a weight-2 Hodge structure
such that the $(p,q)$-part is 
$\bigoplus_i H^{p+i,q+i}(X)$.
We set 
\begin{equation}
\begin{split}
H^{ev}(X,{\Bbb Z})_{\mathrm{alg}}=&
H^{1,1}(H^{ev}(X,{\Bbb {\Bbb C}})) \cap H^{ev}(X,{\Bbb Z})\\
\cong & {\Bbb Z} \oplus \NS(X) \oplus {\Bbb Z}.
\end{split}
\end{equation}
Let $E$ be an object of ${\bf D}(X)$. 
If $X$ is a $K3$ surface or $\rk E=0$, we define the
{\it Mukai vector} of $E$ as
\begin{equation}
\begin{split}
v(E):=&\rk(E)+c_1(E)+(\chi(E)-\rk(E))\varrho_X \in H^{ev}(X,{\Bbb Z}).
\end{split}
\end{equation}
Then for $E, F \in {\bf D}(X)$ 
such that the Mukai vectors are well-defined,
we have 
\begin{equation}
\chi(E,F)=-\langle v(E),v(F) \rangle.
\end{equation}
%and $\varrho_X$ is the fundamental class of $X$.
\item
Since $\deg_G(E)$ is determined by the Chern character
$\ch(E)$, we can also define $\deg_G(v)$, 
$v \in H^{ev}(X,{\Bbb Z})_{\mathrm{alg}}$ by
using $E \in {\bf D}(X)$ with $v(E)=v$.
\end{enumerate}
\end{enumerate}

\section{Perverse coherent sheaves and their moduli spaces.}
\label{sect:Morita}

\subsection{Tilting and Morita equivalence.}\label{subsect:Morita}
Let $X$ be a smooth projective variety and $\pi:X \to Y$ a birational map.
Let ${\cal O}_Y(1)$ be an ample line bundle on $Y$
and ${\cal O}_X(1):=\pi^*({\cal O}_Y(1))$.
We are interested in a subcategory ${\cal C}$ of ${\bf D}(X)$ such that
\begin{enumerate}
\item
${\cal C}$ is the heart of a bounded $t$-structure of ${\bf D}(X)$.
\item
There is a local projective generator $G$ of ${\cal C}$ \cite{VB}:
\begin{enumerate}
\item
${\bf R}\pi_* {\bf R}{\cal H}om(G,E) \in \Coh(Y)$ for all $E \in {\cal C}$
and 
\item
${\bf R}\pi_* {\bf R}{\cal H}om(G,E)=0$, $E \in {\cal C}$ 
if and only if $E=0$.
\end{enumerate}
\end{enumerate}
By these properties, we get
\begin{equation}
{\cal C}=\{ E \in {\bf D}(X) | 
{\bf R}\pi_* {\bf R}{\cal H}om(G,E) 
\in \Coh(Y) \}.
\end{equation}

\begin{defn}\label{defn:perverse}
\begin{enumerate}
\item[(1)]
A {\it perverse coherent sheaf} $E$ is an object of ${\cal C}$.
${\cal C}$ is the {\it category of perverse coherent sheaves}.
\item[(2)]
For $E \in {\bf D}(X)$,
${^p H}^i(E) \in {\cal C}$ denotes the $i$-th cohomology
object of $E$ with respect to the $t$-structure. 
\end{enumerate}
\end{defn}

\begin{NB}
For $E \in {\cal C}^{\leq n}$,
we have an exact triangle
\begin{equation}
E' \to E \to A_n[n] \to E'[1]
\end{equation}
such that $E' \in {\cal C}^{\leq n-1}$ and $A_n \in {\cal C}$.
Then $H^i({\bf R}\pi_* {\bf R}{\cal H}om(G, \otimes E')) \to
H^i({\bf R}\pi_* {\bf R}{\cal H}om(G, \otimes E))$ is surjective for $i>n$.
Hence we see that $H^i({\bf R}\pi_* {\bf R}{\cal H}om(G, \otimes E))=0$
for $i>n$.
Then we also see that 
$H^n({\bf R}\pi_* {\bf R}{\cal H}om(G, \otimes E)) \cong 
H^0({\bf R}\pi_* {\bf R}{\cal H}om(G, \otimes A_n))$. 
If $H^0({\bf R}\pi_* {\bf R}{\cal H}om(G, \otimes A_n))=0$, then
by the properties of $G$, we have $A_n=0$, which implies that
$E \in {\cal C}^{\leq n-1}$.
In particular if ${\bf R}\pi_* {\bf R}{\cal H}om(G, \otimes E)=0$, then
$E \in \cap_n {\cal C}^{\leq n}=0$.
If ${\bf R}\pi_* {\bf R}{\cal H}om(G, \otimes E) \in \Coh(Y)$, then
$E \in {\cal C}^{\leq 0}$ and
applying the above argument for $n=0$,
we get ${\bf R}\pi_* {\bf R}{\cal H}om(G, \otimes E')=0$, which implies that
$E'=0$. Therefore $E \in {\cal C}$.
\end{NB}

The following is an easy consequence of the properties (a), (b) of
$G$. For a convenience sake, we give a proof. 
\begin{lem}\label{lem:projective}
Let $G$ be a local projective generator of ${\cal C}$.
\begin{enumerate}
\item[(1)]
For $E \in {\cal C}$, there is a locally free sheaf $V$ on $Y$
and a surjective morphism 
\begin{equation}\label{eq:projective} 
\phi:\pi^*(V) \otimes G \to E 
\end{equation} 
in ${\cal C}$.
In particular, we have a resolution
\begin{equation}
\cdots \to \pi^*(V_{-1}) \otimes G
\to \pi^*(V_{0}) \otimes G \to E \to 0
\end{equation}
of $E$ such that $V_i$, $i \leq 0$ are locally free sheaves on $Y$.  
\item[(2)]
Let $G' \in {\cal C}$ be a local projective object of ${\cal C}$:
${\bf R}\pi_* {\bf R}{\cal H}om(G',E) \in \Coh(Y)$
for all $E \in {\cal C}$.
If $G$ is a locally free sheaf, then  $G'$ is also a locally free sheaf
\end{enumerate}
\end{lem}

\begin{proof}
(1)
By the property (a) of $G$,
we can take a morphism 
$\varphi:V \to {\bf R}\pi_*{\bf R}{\cal H}om(G,E)$
in ${\bf D}(Y)$ such that 
$V \to H^0({\bf R}\pi_*{\bf R}{\cal H}om(G,E))$
is surjective in $\Coh(Y)$.
Since
\begin{equation}
\begin{split}
\Hom({\bf L}\pi^*({\bf R}\pi_*{\bf R}{\cal H}om(G,E)) \otimes G,E)
=& \Hom({\bf L}\pi^*({\bf R}\pi_*{\bf R}{\cal H}om(G,E)),
{\bf R}{\cal H}om(G,E))\\
=& 
\Hom({\bf R}\pi_*{\bf R}{\cal H}om(G,E),{\bf R}\pi_*{\bf R}{\cal H}om(G,E)),
\end{split}
\end{equation}
we have a morphism
$\phi:\pi^*(V) \otimes G \to E$
such that the induced morphism 
$V \to 
{\bf R}\pi_*{\bf R}{\cal H}om(G,\pi^*(V) \otimes G) \to
{\bf R}\pi_*{\bf R}{\cal H}om(G,E)$ is $\varphi$.
Then $\coker \phi \in {\cal C}$ satisfies
${\bf R}\pi_*{\bf R}{\cal H}om(G,\coker \phi)=0$.
By our assumption on $G$,
$\coker \phi=0$.
Thus $\phi$ is surjective in ${\cal C}$.

(2)
We take a surjective homomorphism \eqref{eq:projective}
for $G'$.
Let $U$ be an affine open subset of $Y$.
We note that
\begin{equation}
\Hom(G'_{|\pi^{-1}(U)},\ker \phi_{|\pi^{-1}(U)}[1]) 
=H^1(U,{\bf R}\pi_* 
{\bf R}{\cal H}om(G'_{|\pi^{-1}(U)},\ker \phi_{|\pi^{-1}(U)}))=0.
\end{equation}
Hence 
\begin{equation}
\Hom(G'_{|\pi^{-1}(U)},\pi^*(V) \otimes G_{|\pi^{-1}(U)})
\to  \Hom(G'_{|\pi^{-1}(U)},G'_{|\pi^{-1}(U)}) 
\end{equation}
is surjective.
Therefore $G'_{|\pi^{-1}(U)}$ is a direct summand of
$\pi^*(V) \otimes G_{|\pi^{-1}(U)}$. 
\end{proof}
From now on, we assume the following:
\begin{itemize}
\item
 $G$ is a local projective generator
of ${\cal C}$ which is a locally free sheaf. 
\end{itemize}

\begin{prop}(\cite[Lem. 3.2, Cor. 3.2.8]{VB})\label{prop:Morita}
We set ${\cal A}:=\pi_*(G^{\vee} \otimes G)$.
Then we have an equivalence
\begin{equation}
\begin{matrix}
{\cal C} & \to & \Coh_{{\cal A}}(Y)\\
E & \mapsto & {\bf R}\pi_*(G^{\vee} \otimes E)
\end{matrix}
\end{equation} 
whose inverse is 
$F \mapsto \pi^{-1}(F) \overset{{\bf L}}{\otimes}_{\pi^{-1}({\cal A})} G$. 
Moreover this equivalence induces an equivalence
${\bf D}(X) \to {\bf D}_{{\cal A}}(Y)$.
\end{prop}

Let ${\cal O}_Y(1)$ be an ample line bundle on $Y$.
For $F \in \Coh_{{\cal A}}(Y)$, 
we have a surjective morphism 
$H^0(Y,F(n)) \otimes {\cal A}(-n) \to F$, $n \gg 0$.
Hence we have a resolution 
$V^{\bullet} \to F$
by locally free ${\cal A}$-modules $V^i$.
If $V^i_{|U} \cong {\cal A}_U^{\oplus n}$ on an open subset of $Y$,
then $(\pi^{-1}(V^i) \otimes_{\pi^{-1}({\cal A})} G)_{|\pi^{-1}(U)}
\cong G^{\oplus n}_{|\pi^{-1}(U)}$.
Thus 
$\pi^{-1}(F) \overset{{\bf L}}{\otimes}_{\pi^{-1}({\cal A})} G$
is isomorphic to
$\pi^{-1}(V^{\bullet}) \otimes_{\pi^{-1}({\cal A})} G$. 

\begin{NB}
\begin{enumerate}
\item[(1)]
For a morphism $V \overset{\psi}{\to} W$ 
of locally free ${\cal A}$-modules on $Y$,
we have a morphism 
$V \otimes_{\cal A} G \overset{\psi'}{\to} 
 W \otimes_{\cal A} G$.
Then ${\bf R}\pi_*(G^{\vee} \otimes \ker \psi')=\ker \psi$,
${\bf R}\pi_*(G^{\vee} \otimes \im \psi')=\im \psi$.
\item[(2)]
Let $U \overset{\phi}{\to} V \overset{\psi}{\to} W$ be an exact sequence
of locally free ${\cal A}$-modules on $Y$.
Then $U \otimes_{\cal A} G \overset{\phi'}{\to} 
V \otimes_{\cal A} G \overset{\psi'}{\to} 
 W \otimes_{\cal A} G$ is exact.
\end{enumerate}

\begin{proof}
(1)
We have exact sequences in $\Per(X/Y)$
\begin{equation}
\begin{split}
& 0 \to \im \psi' \to W\otimes_{\cal A} G \to \coker \psi' \to 0,\\
& 0 \to \ker \psi' \to V \otimes_{\cal A} G
\to \im \psi' \to 0.\\
%& 0 \to \ker \phi' \to U \otimes_{\cal A} G
%\to \im \phi' \to 0,\\
%& 0 \to \im \phi' \to \ker \psi' \to \ker \psi'/\im \phi' \to 0. 
\end{split}
\end{equation}
Since ${\bf R}\pi_* :\Per(X/Y) \to \Coh({\cal A})$ is an exact functor,
we have exact sequences 
\begin{equation}
\begin{split}
& 0 \to {\bf R}\pi_*(G^{\vee} \otimes \im \psi') \to 
W \to {\bf R}\pi_*(G^{\vee} \otimes \coker \psi') \to 0,\\
& 0 \to {\bf R}\pi_*(G^{\vee} \otimes \ker \psi') \to 
V 
\to {\bf R}\pi_*(G^{\vee} \otimes \im \psi') \to 0.\\
%& 0 \to {\bf R}\pi_*(G^{\vee} \otimes \ker \phi')
% \to U 
%\to {\bf R}\pi_*(G^{\vee} \otimes \im \phi') \to 0.\\
%&
%0 \to {\bf R}\pi_*(G^{\vee} \otimes  \im \phi') \to 
%{\bf R}\pi_*(G^{\vee} \otimes \ker \psi') \to 
%{\bf R}\pi_*(G^{\vee} \otimes \ker \psi'/\im \phi') \to 0. 
\end{split}
\end{equation}
Thus claim (1) holds. 

We have an exact sequence
\begin{equation}
0 \to {\bf R}\pi_*(G^{\vee} \otimes  \im \phi') \to 
{\bf R}\pi_*(G^{\vee} \otimes \ker \psi') \to 
{\bf R}\pi_*(G^{\vee} \otimes \ker \psi'/\im \phi') \to 0. 
\end{equation}
Hence 
(2) follows from (1) and the property (b) of $G$.
\end{proof}
\end{NB}

\begin{NB}
\begin{rem}
For $E^{\bullet} \in {\bf D}(X)$, there is a bounded complex $E_1^{\bullet}$
such that $E^{\bullet} \cong E_1^{\bullet}$ in ${\bf D}(X)$ and 
$E_1^i \in \Coh(X) \cap {\cal C}$.
By using Lemma \ref{lem:projective},
we have a complex $E_2^{\bullet}$ such that 
$E_2^{\bullet}$ is bounded above and
a morphism $f:E_2^{\bullet} \to E_1^{\bullet}$ 
such that $\mathrm{Cone}(f)$ gives an exact 
sequence in ${\cal C}$. Then 
it is exact in $\Coh(X)$, which implies that
$E^{\bullet} \cong E_1^{\bullet} \cong E_2^{\bullet}$ in ${\bf D}(X)$.
\end{rem}
\end{NB}

\begin{assume}\label{ass:1}
From now on, we assume that 
$\dim \pi^{-1}(y) \leq 1$ for all $y \in Y$
and set 
\begin{equation}
Y_{\pi}:=\{ y \in Y| \dim \pi^{-1}(y)=1 \}.
\end{equation}
\end{assume}

\begin{lem}\label{lem:tilting}
Assume that $\dim \pi^{-1}(y) \leq 1$ for all $y \in Y$.
Let $G$ be a locally free sheaf on $X$ and set
\begin{equation}
\begin{split}
T:= & \{ E \in \Coh(X) | R^1 \pi_*(G^{\vee} \otimes E)=0 \},\\
S:= &\{ E \in \Coh(X) | \pi_*(G^{\vee} \otimes E)=0 \}.
\end{split}
\end{equation}
\begin{enumerate}
\item[(1)]
$(T,S)$ is a torsion pair of $\Coh(X)$ such that
$G \in T$ if and only if 
$R^1 \pi_*(G^{\vee} \otimes G)=0$ and $S \cap T=0$.
\item[(2)]
If $(T,S)$ is a torsion pair such that
$G \in T$, then $G$ is 
a local projective generator of the tilted category
\begin{equation}
{\cal C}_G:=
\{E \in {\bf D}(X)|H^{-1}(E) \in S,H^0(E) \in T,\;H^i(E)=0,\;i \ne -1,0 \}.
\end{equation}
\item[(3)]
Assume that $(T,S)$ is a torsion pair such that
$G \in T$.
If $(T',S')$ is a torsion pair of $\Coh(X)$ such that
$G \in T'$ and 
$S \cap T'=0$, then $(T',S')=(T,S)$.
\end{enumerate}
\end{lem}

\begin{proof}
(1)
We shall prove that $(T,S)$ is a torsion pair under
$R^1 \pi_*(G^{\vee} \otimes G)=0$ and $S \cap T=0$.
For $E \in \Coh(X)$, let
$\phi:\pi^*(\pi_*(G^{\vee} \otimes E)) \otimes G \to E$
be the evaluation map.
Then we see that
$\pi_*(G^{\vee} \otimes \coker \phi)=0$,
$R^1 \pi_*(G^{\vee} \otimes \im \phi)=0$ and
$R^1 \pi_*(G^{\vee} \otimes E) \cong 
R^1 \pi_*(G^{\vee} \otimes \coker \phi)$.
Hence we have a desired decomposition
\begin{equation}
0 \to E_1 \to E \to E_2 \to 0
\end{equation}
where $E_1:=\im \phi \in T$ and $E_2:=\coker \phi \in S$. 

(2)
If $(T,S)$ is a torsion pair, then
for $E \in {\cal C}_G$,
we have an exact sequence
\begin{equation}
0 \to R^1\pi_*(G^{\vee} \otimes H^{-1}(E))
\to {\bf R}\pi_*(G^{\vee} \otimes E) \to
 \pi_*(G^{\vee} \otimes H^0(E)) 
 \to 0.
\end{equation}
Hence $ {\bf R}\pi_*(G^{\vee} \otimes E) \in \Coh(Y)$
and $ {\bf R}\pi_*(G^{\vee} \otimes E)=0$ if and only if
$R^1\pi_*(G^{\vee} \otimes H^{-1}(E))=
\pi_*(G^{\vee} \otimes E)=0$, which is equivalent to
$H^{-1}(E), H^0(E) \in S \cap T=0$.
Therefore $G$ is a local projective generator of ${\cal C}_G$.

(3)
We first prove that $T \subset T'$.
For an object $E \in T$,
(2) implies that there is a surjective morphism
$\phi:\pi^*(V) \otimes G \to E$ in ${\cal C}_G$,
where $V$ is a locally free sheaf on $Y$.
Since $\phi$ is surjective in $\Coh(X)$ and $G \in T'$,
$E \in T'$.  
Since $S \cap T'=0$, we get $S \subset S'$.
Therefore $(T',S')=(T,S)$.
\end{proof}

By the proof of Lemma \ref{lem:tilting}, we get the following. 
\begin{cor}
Let $G$ be a locally free sheaf on $X$ which gives a local projective
generator of ${\cal C}_G$ in Lemma \ref{lem:tilting}. 
Let $E$ be a coherent sheaf on $X$
and $\phi:\pi^*(\pi_*(G^{\vee} \otimes E))\otimes G \to E$ 
the evaluation map. Then $E_1:=\im \phi \in T$ and
$E_2:=\coker \phi \in S$. Thus 
we have a decomposition of $E$ 
\begin{equation}
0 \to \im \phi \to E \to \coker \phi \to 0 
\end{equation}
with respect to the torsion pair $(T,S)$.
\end{cor}

\begin{NB}
Let $(S,T)$ be a torsion pair.
For an exact sequence
$$
0 \to E_1 \to E \to E_2 \to 0,
$$
if $E \in T$, then $E_2 \in T$, and if $E \in S$, then
$E_1 \in S$:
We take a decomposition
$$
0 \to E_{i,T} \to E_i \to E_{i,S} \to 0,\;
E_{i,T} \in T,E_{i,S} \in S.
$$
Assume that $E \in T$.
Since $\Hom(E,E_{2,S})=0$ and $E \to E_2 \to E_{2,S}$
is surjective,
$E_{2,S}=0$.
If $E \in S$, then
$\Hom(E_{1,T},E)=0$ and the injectivity of 
$E_{1,T} \to E_1 \to E$ implies that $E_{1,T}=0$. 
\end{NB}

\begin{lem}\label{lem:tilting:C}

Assume that the local projective generator
$G \in {\cal C}$ is a locally free sheaf.
We set
\begin{equation}
\begin{split}
T:= & \{ E \in \Coh(X) | R^1 \pi_*(G^{\vee} \otimes E)=0 \},\\
S:= &\{ E \in \Coh(X) | \pi_*(G^{\vee} \otimes E)=0 \}.
\end{split}
\end{equation}
Then $(T,S)$ is a torsion pair of $\Coh(X)$ whose tilting is 
${\cal C}$.
\end{lem}

\begin{NB}
\begin{equation}
{\cal C}=\left\{E \in {\bf D}(X) \left|
\begin{aligned}
H^q(E)=0, q \ne -1,0,\\
\pi_*(G^{\vee} \otimes H^{-1}(E))=0\\
R^1 \pi_*(G^{\vee} \otimes H^{0}(E))=0
\end{aligned}
\right. \right\}.
\end{equation}
\end{NB}

\begin{proof}
Since $G \in {\cal C}$, we have  
${\bf R}\pi_*(G^{\vee} \otimes G) \in \Coh(Y)$.
By the definition of a local projective generator, we have
$S \cap T=0$.
\begin{NB}
If $E \in S \cap T \subset T \subset {\cal C}$, then
by the property (1) (b) of $G$,
$E=0$. Thus $S \cap T=0$.
\end{NB}
By Lemma \ref{lem:tilting},
$(T,S)$ is a torsion pair. Let ${\cal C}_G$ be the tilted category.
Since $S[1], T \subset {\cal C}$, we get
${\cal C}_G \subset {\cal C}$.
Conversely for $E \in {\cal C}$, we have a spectral sequence
\begin{equation}
E_2^{p,q}=
R^p \pi_*(G^{\vee} \otimes H^q(E)) \Longrightarrow 
E_{\infty}^{p+q}=R^{p+q}\pi_*(G^{\vee} \otimes E). 
\end{equation}
Since $\pi^{-1}(y) \leq 1$ for all $y \in Y$,
this spectral sequence degenerates.
Hence we have
${\bf R}\pi_*(G^{\vee} \otimes H^q(E))=0$ for $q \ne -1,0$,
$\pi_*(G^{\vee} \otimes H^{-1}(E))=0$ and 
$R^1\pi_*(G^{\vee} \otimes H^0(E))=0$.
Therefore $E \in {\cal C}_G$.
\end{proof}

\begin{NB}
\begin{proof}
Since $G \in {\cal C}$, we have  
${\bf R}\pi_*(G^{\vee} \otimes G) \in \Coh(Y)$.
By the definition of a local projective generator, we have
$S \cap T=0$.
By Lemma \ref{lem:tilting},
$(T,S)$ is a torsion pair. Let ${\cal C}_G$ be the tilted category.
Since $S[1], T \subset {\cal C}$, we get
${\cal C}_G \subset {\cal C}$, which implies that 
${\cal C}_G={\cal C}$.
\end{proof}
\end{NB}

\begin{lem}\label{lem:tilting:dual}

For the locally free sheaf $G$ on $X$ and the tilted category
${\cal C}_G$ in Lemma \ref{lem:tilting}, we set
\begin{equation}
\begin{split}
T^D:=&\{E \in \Coh(X)|R^1 \pi_*(G \otimes E)=0 \},\\
S^D:=&\{E \in \Coh(X)| \pi_*(G \otimes E)=0 \}.
\end{split}
\end{equation}
Then $(T^D,S^D)$ is a torsion pair and $G^{\vee}$ is a local projective 
generator of the tilted category. 
We denote the tilted category
by ${\cal C}_G^D$. 
\end{lem}

\begin{proof}
Since $R^1 \pi_*(G^{\vee} \otimes G)=0$,
$G^{\vee} \in T^D$.
We show that $T^D \cap S^D=0$.
Assume that ${\bf R}\pi_*(G \otimes E)=0$ for a coherent sheaf $E$ on $X$.
%Then $E$ does not contain a 0-dimensional subsheaf.
%Since $\pi_*(G \otimes E)=0$, $\Supp(E)$ is contained in the exceptional
%locus. Therefore $E$ is purely 1-dimensional.
%Then $D_X(E)={\cal E}xt^{n-1}_{{\cal O}_X}(E,{\cal O}_X)[-n+1]$,
%where $n=\dim X$.
Since 
\begin{equation}
\begin{split}
H^{i}(Y,{\bf R}\pi_*(G \otimes E)(-k))=& H^{i}(X,G \otimes E(-k))\\
=& H^{n-i}(X,G^{\vee} \otimes D_X(E)(K_X) \otimes {\cal O}_X(k))^{\vee}\\
=& H^{n-i}(Y,{\bf R}\pi_*(G^{\vee}\otimes D_X(E)(K_X))(k))^{\vee}
\end{split}
\end{equation}
for all $k \in {\Bbb Z}$
and $H^j(Y,H^{n-i}({\bf R}\pi_*(G^{\vee} \otimes D_X(E)(K_X)))(k))=0$
for $k \gg 0$ and $j \ne 0$,
we get $H^{n-i}(Y,{\bf R}\pi_*(G^{\vee} \otimes D_X(E)(K_X))(k))
=H^0(Y,H^{n-i}({\bf R}\pi_*(G^{\vee} \otimes D_X(E)(K_X)))(k))=0$
for $k \gg 0$.
Therefore ${\bf R}\pi_*(G^{\vee}\otimes D_X(E)(K_X))=0$.
Since $\dim \pi^{-1}(y) \leq 1$ for all $y \in Y$,
we see that ${\bf R}\pi_*(G^{\vee}\otimes H^i(D_X(E)(K_X)))
={\bf R}\pi_*(H^i(G^{\vee}\otimes D(E)(K_X)))
=0$ (see the proof of Lemma \ref{lem:tilting:C}).
Since $G$ is a local projective generator of ${\cal C}_G$,
$H^i(D_X(E)(K_X))=0$ for all $i$.
Therefore $D_X(E)(K_X)=0$, which implies that $E=0$.
\end{proof}

\begin{rem}\label{rem:tilting:dual}
If $E$ is a local projective object of ${\cal C}_G$, that is,
$R^1 \pi_*(E^{\vee} \otimes F)=0$ for all $F \in {\cal C}_G$, 
then $E^{\vee} \in {\cal C}_G^D$.
Indeed by $G \in {\cal C}_G$, we have
$R^1 \pi_*(E^{\vee} \otimes G)=0$, which implies that
$E^{\vee} \in T^D$.
Moreover since $G^{\vee}$ is a local projective generator of
${\cal C}^D$ and $R^1\pi_*(E \otimes G^{\vee})=0$,
$E^{\vee}$ is a local projective object of ${\cal C}^D$. 
\begin{NB}
$E^{\vee} \in {\cal C}_G^D$ implies there is a surjection
$G^{\vee} \otimes \pi^*(W) \to E^{\vee}$, where $W$ is locally free.
Hence there is an inclusion 
$E \hookrightarrow G \otimes \pi^*(W^{\vee})$.
Hence $\pi_*(E \otimes F)=0$ for $F \in \Coh(X)$
with $\pi_*(G \otimes F)=0$.
Since there is a surjection $G \otimes \pi^*(V) \to E$,
$R^1 \pi_*(E \otimes F)=0$ for
$F \in \Coh(X)$ with $R^1 \pi_*(G \otimes F)=0$.    
\end{NB}
\end{rem}

\subsubsection{Irreducible objects of ${\cal C}$.}
\begin{lem}\label{lem:tilting:restriction}

Let $G$ be a locally free sheaf on $X$ such that
${\bf R}\pi_*(G^{\vee} \otimes F) \ne 0$ for
all non-zero coherent sheaf $F$ on a fiber of $\pi$.
Then for a coherent sheaf $E$ on $X$, $\pi_*(G^{\vee} \otimes E)=0$
implies $R^1 \pi_*(G^{\vee} \otimes E_{|\pi^{-1}(y)}) \ne 0$
for all $y \in \pi(\Supp(E))$.
\end{lem}

\begin{proof}
Assume that $R^1 \pi_*(G^{\vee} \otimes E_{|\pi^{-1}(y)})=0$.
By Lemma \ref{lem:tilting:TFF} below,
$R^1 \pi_*(G^{\vee} \otimes E)=0$ in a neighborhood of $y$.
Thus ${\bf R}\pi_*(G^{\vee} \otimes E)=0$ in a neighborhood of $y$.
Then 
${\bf R}\pi_*(G^{\vee} \otimes E \overset{{\bf L}}{\otimes} 
{\bf L}\pi^*({\Bbb C}_y))=
{\bf R}\pi_*(G^{\vee} \otimes E)
\overset{{\bf L}}{\otimes} {\Bbb C}_y=
0$.
Since the spectral sequence 
\begin{equation}
E^{pq}_2=R^p \pi_*(H^q(G^{\vee} \otimes E \overset{{\bf L}}{\otimes} 
{\bf L}\pi^*({\Bbb C}_y))) \Longrightarrow 
E^{p+q}_\infty=
H^{p+q}({\bf R}\pi_*(G^{\vee} \otimes E \overset{{\bf L}}{\otimes} 
{\bf L}\pi^*({\Bbb C}_y)))
\end{equation}
degenerates,
$R^p \pi_*(G^{\vee} \otimes E \otimes \pi^*({\Bbb C}_y))=0$.
By our assumption on $G$,
we have $E_{|\pi^{-1}(y)}=0$, which is a contradiction.  
\end{proof}

\begin{defn}\label{defn:0-dim}
\begin{enumerate}
\item[(1)]
An object $E \in {\cal C}$ is {\it $0$-dimensional}, if
${\bf R}\pi_*(G^{\vee} \otimes E)$ is 0-dimensional as an object of
$\Coh(Y)$.
\item[(2)]
An object $E \in {\cal C}$ is {\it irreducible}, if
$E$ does not have a proper subobject except 0.
\item[(3)] 
For a 0-dimensional object $E \in {\cal C}$,
we take a filtration
\begin{equation}
0 \subset F_1 \subset F_2 \subset \cdots \subset F_s=E
\end{equation}
such that $F_i/F_{i-1}$ are irreducible objects of ${\cal C}$.
Then $\oplus_i F_i/F_{i-1}$ is 
the {\it Jordan-H\"{o}lder decomposition} of $E$.
\end{enumerate}
\end{defn}

\begin{rem}
In section \ref{subsect:stability}, we shall define the dimension
of $E$ generally. According to the definition of the stability
in Definition \ref{defn:Simpson-stability},
we also have the following.
\begin{enumerate}
\item[(1)]
A 0-dimensional object $E$ is $G$-twisted semi-stable and
a $G$-twisted stable object corresponds to an irreducible object. 
\item[(2)]
The Jordan-H\"{o}lder decomposition of $E$ is nothing but
the standard representative of the $S$-equivalence class of
$E$. 
\end{enumerate}
\end{rem}

\begin{lem}\label{lem:tilting:irreducible}
Let $G$ be a locally free sheaf on $X$ and 
${\cal C}_G$ the tilted category in Lemma \ref{lem:tilting}.
\begin{enumerate}
\item[(1)]
${\Bbb C}_x \in {\cal C}_G$ for all $x \in X$.
\item[(2)]
For ${\Bbb C}_x, x \in \pi^{-1}(y)$,
the Jordan-H\"{o}lder decomposition
does not depend on the choice of
$x \in \pi^{-1}(y)$. 
\item[(3)]
Let $\oplus_{j=0}^{s_y}E_{yj}^{\oplus a_{yj}}$ be the Jordan-H\"{o}lder 
decomposition of ${\Bbb C}_x$, $x \in \pi^{-1}(y)$.
Then the irreducible objects of ${\cal C}_G$ are
\begin{equation}\label{eq:tilting:generater}
{\Bbb C}_x, (x \in X \setminus \pi^{-1}(Y_\pi)),\;\;
E_{yj}, (y \in Y_\pi,\; 0 \leq j \leq s_y).
\end{equation}
In particular, if ${\bf R}\pi_*(G^{\vee} \otimes E)$ is a 0-dimensional
${\cal A}$-module, then $E$ is generated by
\eqref{eq:tilting:generater}.
\end{enumerate}
\end{lem}

\begin{NB}
Since $\dim X$ may not be 2, we cannot use the theory of 
Fourier-Mukai functor. For the 2-dimensional case, see
Lemma \ref{lem:simple-generator}.
\end{NB}

\begin{proof}
(1)
We note that ${\bf R}\pi_*(G^{\vee} \otimes {\Bbb C}_x)=
 \pi_*(G^{\vee} \otimes {\Bbb C}_x)$. Hence ${\Bbb C}_x \in {\cal C}_G$.
(2)
Since the trace map $G^{\vee} \otimes G \to {\cal O}_X$ is surjective,
we have a surjective map
\begin{equation}
R^1 \pi_*(G^{\vee} \otimes G) \to R^1 \pi_*({\cal O}_X)
\to R^1 \pi_*({\cal O}_{\pi^{-1}(y)_{\mathrm{red}}}),
\end{equation}
where $\pi^{-1}(y)_{\mathrm{red}}$ is the reduced subscheme of
$\pi^{-1}(y)$. 
Since $R^1 \pi_*(G^{\vee} \otimes G)=0$,
we get 
$$
H^1(\pi^{-1}(y)_{\mathrm{red}},
{\cal O}_{\pi^{-1}(y)_{\mathrm{red}}})=
H^0(Y,R^1 \pi_*({\cal O}_{\pi^{-1}(y)_{\mathrm{red}}}))=0.
$$
Then we see that $\pi^{-1}(y)_{\mathrm{red}}$
is a tree of smooth rational curves. 
Let $C_{yj}$, $j=0,...,s_y$ be the irreducible component of
$\pi^{-1}(y)_{\mathrm{red}}$.
Since the restriction map
$R^1 \pi_*(G^{\vee} \otimes G) \to
R^1 \pi_*(G^{\vee} \otimes G_{|C_{yj}})$ is surjective,
$R^1 \pi_*(G^{\vee} \otimes G_{|C_{yj}})=0$.
Thus we can write 
$G_{|C_{yj}} \cong {\cal O}_{C_{yj}}(d_{yj})^{\oplus r_{yj}}
\oplus {\cal O}_{C_{yj}}(d_{yj}+1)^{\oplus r_{yj}'}$.
Since $R^1 \pi_*(G^{\vee} \otimes {\cal O}_{C_{yj}}(d_{yj}))=0$
and $\pi_*(G^{\vee} \otimes {\cal O}_{C_{yj}}(d_{yj}-1))=0$,
${\cal O}_{C_{yj}}(d_{yj}),{\cal O}_{C_{yj}}(d_{yj}-1)[1] \in
{\cal C}_G$.
For $x \in C_{yj}$, we have an exact sequence in ${\cal C}_G$
\begin{equation}
0 \to {\cal O}_{C_{yj}}(d_{yj}) \to {\Bbb C}_x \to
{\cal O}_{C_{yj}}(d_{yj}-1)[1] \to 0.
\end{equation} 
Hence the Jordan-H\"{o}lder decomposition of ${\Bbb C}_x$ is 
constant on $C_{yj}$.
%${\cal O}_{C_{yj}}(d_{yj}) \oplus {\cal O}_{C_{yj}}(d_{yj}-1)[1]$,
Since $\pi^{-1}(y)$ is connected, 
the Jordan-H\"{o}lder decomposition of ${\Bbb C}_x$ is determined by 
$y$.

(3)
Let $E$ be an irreducible object of ${\cal C}_G$.
Then we have (i) $E=F[1], F \in \Coh(X)$ or (ii) $E \in \Coh(X)$.
In the first case, since $F \in S$, we have 
$\pi_*(G^{\vee} \otimes F)=0$.
By Lemma \ref{lem:tilting:restriction}, we have
$R^1 \pi_*(G^{\vee} \otimes F_{|\pi^{-1}(y)}) \ne 0$
 for $y \in \pi(\Supp(F))$, which implies that
there is a quotient $F_{|\pi^{-1}(y)} \to F'$ such that 
$0 \ne F' \in S$ for $y \in \pi(\Supp(F))$. Then
we have a non-trivial morphism 
$F[1] \to F'[1]$, which should be injective in ${\cal C}_G$.
Therefore $\pi(\Supp(F))$ is a point.
In the second case,
we also see that $\pi(\Supp(E))$ is a point.
Therefore  
${\bf R}\pi_*(G^{\vee} \otimes E)$ is a 0-dimensional sheaf. 
(i) If $E=F[1]$, then
since $\pi_*(G^{\vee} \otimes F)=0$, $F$ is purely 1-dimensional.
Then $\Hom({\Bbb C}_x,F[1])=\Hom(D(F)[n-1],D({\Bbb C}_x)[n]) \ne 0$
for $x \in \Supp(F)$, where $n=\dim X$.
Hence we have a non-trivial
morphism $E_{yj} \to E$, 
$y \in \pi(\Supp(F)) \cap Y_\pi$,
which is an isomorphism.  
(ii) If $E \in \Coh(X)$, then
$\Hom(E,{\Bbb C}_x) \ne 0$ for $x \in \Supp(E)$,
which also implies that $E \cong E_{yj}$ for
$\Supp(E) \subset \pi^{-1}(y)$ or 
$E \cong {\Bbb C}_x$ for $\Supp(E) \subset X \setminus Y_\pi$.
\end{proof}

\begin{rem}
Since $\pi_*(G^{\vee} \otimes {\Bbb C}_x)$ is a coherent sheaf
on the reduced point $\{y \}$,
the multiplication $\pi^*(t):E_{yj} \to E_{yj}$,
$t \in I_{y}$ is zero.
Thus 
$H^i(E_{yj})$ are coherent sheaves on the scheme $\pi^{-1}(y)$.
\end{rem}

\begin{lem}\label{lem:tilting:generate}
Let $E$ be a coherent sheaf such that $\pi(\Supp(E))=\{y \}$.
\begin{enumerate}
\item[(1)]
For $E \in T$ with $\Supp(E) \subset \pi^{-1}(y)$, there is a filtration
\begin{equation}
0 \subset F_1 \subset F_2 \subset \cdots \subset F_s=E
\end{equation} 
such that for every $F_k/F_{k-1}$, 
there is $E_{yj} \in T$ and a surjective homomorphism
$E_{yj} \to F_k/F_{k-1}$ in $\Coh(X)$.
\item[(2)]
For $E \in S$, there is a filtration
\begin{equation}
0 \subset F_1 \subset F_2 \subset \cdots \subset F_s=E
\end{equation} 
such that for every $F_k/F_{k-1}$, 
there is $E_{yj}[-1] \in S$ and an injective homomorphism
$F_k/F_{k-1} \to E_{yj}[-1]$ in $\Coh(X)$.
\end{enumerate}
\end{lem}

\begin{proof}
(1)
Since $E \in T$, $E$ contains $E_{yj}$ in ${\cal C}$.
Let $F$ be the quotient in ${\cal C}$. Then  we have an exact sequence
\begin{equation}
0 \to H^{-1}(E_{yj}) \to 0 \to H^{-1}(F) 
\to H^0(E_{yj}) \to E \to H^0(F) \to 0.
\end{equation}
Hence $E_{yj} \in T$ and $H^0(F) \in T$.
We set $F_1:=\im(E_{yj} \to E)$ in $\Coh(X)$.
Since $E/F_1 \in T$ and $\Supp(E/F_1) \subset \pi^{-1}(y)$,
by the induction on the support of $E$,
we get the claim.

(2)
Since $E \in S$, there is a quotient $E[1] \to E_{yj}$ in ${\cal C}$.
Let $F$ be the kernel in ${\cal C}$. Then we have an exact sequence
\begin{equation}
0 \to H^{-1}(F) \to E \to H^{-1}(E_{yj}) 
\to H^0(F) \to 0 \to H^0(E_{yj}) \to 0.
\end{equation}
Hence $E_{yj}[-1] \in S$ and $H^{-1}(F) \in S$.
\begin{NB}
$\Hom(F_1,F_2)=0, F_1 \in T, F_2 \in S$ implies that
$E \subset F$, $F \in S$ implies that $E \in S$ and
$E \subset F$, $F \in T$ implies that $F/E \in T$.
\end{NB}
We set $E':=\im(E \to H^{-1}(E_{yj}))$ in $\Coh(X)$.
Then $E'$ is a subsheaf of $E_{yj}[-1]$ and
$E$ is an extension of $E'$ by $H^{-1}(F) \in S$. 
Since $\Supp(H^{-1}(F)) \subset \pi^{-1}(y)$,
by the induction on the support of $E$,
we get the claim.
\end{proof}

\begin{lem}\label{lem:tilting:TFF}
\begin{enumerate}
\item[(1)]
$\pi^*(\pi_*(I_{\pi^{-1}(y)})) \to I_{\pi^{-1}(y)}$ is surjective.
In particular, $\Hom(I_{\pi^{-1}(y)},{\cal O}_{C_{yj}}(-1))=0$
for all $j$. 
\item[(2)]
$\Ext^1({\cal O}_{\pi^{-1}(y)},{\cal O}_{C_{yj}}(-1))=0$ for all $j$.
In particular,
$$
H^1(X,{\cal H}om({\cal O}_{\pi^{-1}(y)},{\cal O}_{C_{yj}}(-1)))=
H^0(X,{\cal E}xt^1({\cal O}_{\pi^{-1}(y)},{\cal O}_{C_{yj}}(-1)))=0.
$$
\item[(3)]
For a coherent sheaf $E$ on $X$,
$R^1 \pi_*(E)=0$ at $y$ if and only if $R^1 \pi_*(E_{|\pi^{-1}(y)})=0$.
\end{enumerate}
\end{lem}

\begin{proof}
Since $I_{\pi^{-1}(y)}=\im(\pi^*(I_{y}) \to {\cal O}_X)$, (1) holds.
(2)
Since $\Hom({\cal O}_X,{\cal O}_{C_{yj}}(-1)[k])=0$
for all $j$ and $k$,
the first claim follows from the exact sequence 
\begin{equation}
0 \to I_{\pi^{-1}(y)} \to {\cal O}_X \to {\cal O}_{\pi^{-1}(y)}
\to 0.
\end{equation}
Since $H^2(X,{\cal H}om({\cal O}_{\pi^{-1}(y)},{\cal O}_{C_{yj}}(-1)))=0$,
the second claim follows from the local-global spectral sequence.

(3)
The proof is similar to \cite{I:1}.
Assume that $R^1 \pi_*(E_{|\pi^{-1}(y)})=0$.
We take a locally free sheaf $V$ on $Y$ such that
$V \to I_{y}$ is surjective.
Then (1) implies that 
$\pi^*(V) \to I_{\pi^{-1}(y)}$ is surjective.
Hence we 
have a surjective homomorphism
$\pi^*(V^{\otimes n}) \otimes {\cal O}_{\pi^{-1}(y)}
\to  I_{\pi^{-1}(y)}^n/I_{\pi^{-1}(y)}^{n+1}$.
Then we see that 
$R^1 \pi_*(E \otimes {\cal O}_X/I_{\pi^{-1}(y)}^n)=0$.
By the theorem of formal functions, we get the claim.
\end{proof}

\begin{NB}
Old version.
\begin{lem}
\begin{enumerate}
\item[(1)]
$\Ext^1({\cal O}_{\pi^{-1}(p_i)},{\cal O}_{C_{ij}}(-1))=0$ for all $j$.
\item[(2)]
$\Hom(I_{\pi^{-1}(p_i)},{\cal O}_{C_{ij}}(-1))=0$
for all $j$. In particular,
$\pi^*(\pi_*(I_{\pi^{-1}(p_i)})) \to I_{\pi^{-1}(p_i)}$ is surjective.
\item[(3)]
For a locally free sheaf $E$ on $X$,
$R^1 \pi_*(E)=0$ at $p_i$ if and only if $R^1 \pi_*(E_{|\pi^{p_i}})=0$.
\end{enumerate}
\end{lem}

\begin{proof}
(1)
If $\Ext^1({\cal O}_{\pi^{-1}(p_i)},{\cal O}_{C_{ij}}(-1)) \ne 0$, then
we take a non-trivial extension
\begin{equation}\label{eq:tilting:extension}
0 \to {\cal O}_{C_{ij}}(-1) \to F \to {\cal O}_{\pi^{-1}(p_i)} \to 0.
\end{equation}
Since $\Hom({\cal O}_{\pi^{-1}(p_i)},{\cal O}_{C_{ij}}(-1))=0$
for all $j$ and \eqref{eq:tilting:extension} does not split,
$\Hom(F,{\cal O}_{C_{ij}}(-1))=0$ for all $j$.
Hence $\pi^*(\pi_*(F)) \to F$ is surjective.
By \eqref{eq:tilting:extension}, $\pi_*(F) \to 
\pi_*({\cal O}_{\pi^{-1}(p_i)})$ is an isomorphism.
Hence we have a homomorphism ${\cal O}_{\pi^{-1}(p_i)}\cong
\pi^*(\pi_*(F)) \to F$, which gives a splitting of 
\eqref{eq:tilting:extension}. Therefore
$\Ext^1({\cal O}_{\pi^{-1}(p_i)},{\cal O}_{C_{ij}}(-1))=0$
for all $j$.

(2)
We consider the exact sequence
\begin{equation}
0 \to I_{\pi^{-1}(p_i)} \to {\cal O}_X \to {\cal O}_{\pi^{-1}(p_i)}
\to 0.
\end{equation}
Since $\Hom({\cal O}_X,{\cal O}_{C_{ij}}(-1))=0$
for all $j$, (1) implies the claim.

(3)
The proof is similar to \cite{I:1}.
Assume that $R^1 \pi_*(E_{|\pi^{p_i}})=0$.
We take a locally free sheaf $V$ on $Y$ such that
$V \to \pi_*(I_{\pi^{-1}(p_i)})$ is surjective.
Then (2) implies that 
$\pi^*(V) \to I_{\pi^{-1}(p_i)}$ is surjective.
Hence we 
have a surjective homomorphism
$\pi^*(V^{\otimes n}) \otimes {\cal O}_{\pi^{-1}(p_i)}
\to  I_{\pi^{-1}(p_i)}^n/I_{\pi^{-1}(p_i)}^{n+1}$.
Hence we see that 
$R^1 \pi_*(E \otimes {\cal O}_X/I_{\pi^{-1}(p_i)}^n)=0$.
By the theorem of formal functions, we get the claim.
\end{proof}
\end{NB}

\begin{lem}\label{lem:tilting:T}
Let $E_{yj}$, $y \in Y_\pi$ be the irreducible objects of 
${\cal C}_G$.
Let $E$ be a coherent sheaf on $X$.
If $\Hom(E,E_{yj}[-1])=0$ for all $E_{yj}[-1] \in S$, then
$E \in T$.
\end{lem}

\begin{proof}
We note that $\Hom(E_{|\pi^{-1}(y)},E_{yj}[-1])=0$ for all
$E_{yj}[-1] \in S$. By Lemma \ref{lem:tilting:generate} (2),
$E_{|\pi^{-1}(y)} \in T$. 
Then $R^1 \pi_*(G^{\vee} \otimes E_{|\pi^{-1}(y)})=0$.
By Lemma \ref{lem:tilting:TFF}, $R^1 \pi_*(G^{\vee} \otimes E)=0$
in a neighborhood of $y$.
Since $y$ is any point of $Y_\pi$,
$R^1 \pi_*(G^{\vee} \otimes E)=0$, which implies that 
$E \in T$.
\end{proof}

For a subcategory ${\cal C}$ of ${\bf D}(X)$, we set
\begin{equation}
{\cal C}_y:=\{E \in {\cal C}| \pi(\Supp(H^i(E)))=\{ y\}, i \in {\Bbb Z} \}.
\end{equation}

\begin{NB}
In the following lemma, we don't assume 
the existence of a local projective generator.
\end{NB}
\begin{lem}\label{lem:tilting:without}
Let $(S,T)$ be a torsion pair of $\Coh(X)$ and ${\cal C}$ the tilted category. 
Assume that 
\begin{enumerate}
\item
$\# Y_\pi <\infty$ and 
every object of ${\cal C}_y$, $y \in Y$ is of finite length. 
\item
${\Bbb C}_x \in {\cal C}$ for all $x \in X$.
\item
$\pi(\Supp(E)) \subset Y_\pi$ for $E \in S$.
\end{enumerate}
Then the claims of 
Lemma \ref{lem:tilting:generate} and Lemma \ref{lem:tilting:T}
hold.
\end{lem}

\begin{proof}
By (i) and (iii), irreducible objects are 
$E={\Bbb C}_x, x \in X \setminus \pi^{-1}(Y_\pi)$
or irreducible objects of ${\cal C}_y, y \in Y_\pi$.
By (ii), we get Lemma \ref{lem:tilting:irreducible} (3).
The other claims of Lemma \ref{lem:tilting:irreducible}
and Lemma \ref{lem:tilting:generate} are obvious.
For $0 \ne E \in S$,
 (i) and Lemma \ref{lem:tilting:generate} imply that
there is a coherent sheaf $E_{yj}[-1] \in S$
such that $\Hom(E,E_{yj}[-1]) \ne 0$.
 Hence Lemma \ref{lem:tilting:T} also holds.
\end{proof}

\begin{prop}\label{prop:tilting:S-T}
Assume that
$Y_\pi=\{p_1,...,p_m \}$.
Let $G$ be a locally free sheaf on $X$ and 
${\cal C}_G$ the tilted category in Lemma \ref{lem:tilting}.
For ${\Bbb C}_x, x \in \pi^{-1}(p_i)$,
let $\oplus_{j=0}^{s_i}E_{ij}^{\oplus a_{ij}}$ 
be the Jordan-H\"{o}lder decomposition of ${\Bbb C}_x$, where 
$E_{ij}$ are irreducible objects.
\begin{enumerate}
\item[(1)]
We set 
\begin{equation}
\begin{split}
\Sigma:=&\{E_{ij}[-1]|i,j\} \cap \Coh(X)\\
{\cal T}:=&\{E \in \Coh(X)|\Hom(E,c)=0, c \in \Sigma \}\\
{\cal S}:=&\{ E \in \Coh(X)| \text{ $E$ is  a successive extension of 
subsheaves of $c \in \Sigma$ }\}.
\end{split} 
\end{equation}
Then $({\cal T},{\cal S})$ is a torsion pair of $\Coh(X)$
whose tilting is ${\cal C}_G$.
In particular, ${\cal C}_G$ is characterized by $\Sigma$. 
\item[(2)]
${\cal C}_G^D$ is characterized by 
\begin{equation}
\Sigma:=\{(D_X(E_{ij})\otimes K_X[n]) [-1]|i,j\} \cap \Coh(X)=
D_X(\{E_{ij}|i,j\} \cap \Coh(X))\otimes K_X[n-1],
\end{equation}
where $n=\dim X$.
\end{enumerate}
\end{prop}

\begin{proof}
(1)
For $E \in \Coh(X)$, we consider
$\phi:G \otimes \pi^*(\pi_*(G^{\vee} \otimes E)) \to E$.
We set $E_1:=\im \phi$ and $E_2:=\coker \phi$.
Since $\Hom(G,E_{ij}[-1])=0$ for all $E_{ij}$,
$G \in {\cal T}$.
Hence $E_1 \in {\cal T}$.
We shall show that
$E_2 \in {\cal S}$.
We note that ${\bf R}\pi_*(G^{\vee} \otimes E_1)=
\pi_*(G^{\vee} \otimes E_1)$ and
${\bf R}\pi_*(G^{\vee} \otimes E_2)=
R^1 \pi_*(G^{\vee} \otimes E)[-1]$. 
Then $E_1, E_2[1] \in {\cal C}_G$.
Since $\Supp(E_2) \subset \cup_{i=1}^n \pi^{-1}(p_i)$,
Lemma \ref{lem:tilting:irreducible} (3) implies that
$E_2[1]$ is generated by $E_{ij}$.
Hence if $E_2 \ne 0$, then $\Hom(E_2[1],c[1])\ne 0$ for an object
$c \in \Sigma$.
Let $E'_2$ be the kernel of $E_2 \to c$ in $\Coh(X)$.
Then $E_2'[1] \in {\cal C}_G$.
Hence by the induction on the support of $E_2$,
we see that $E_2 \in {\cal S}$.
Therefore $({\cal T},{\cal S})$ is a torsion pair
of $\Coh(X)$.
We also see that 
\begin{equation}
\begin{split}
{\cal T}&=\{E \in \Coh(X)|
R^1 \pi_*(G^{\vee} \otimes E)=0 \},\\
{\cal S}&=\{E \in \Coh(X)|
\pi_*(G^{\vee} \otimes E)=0 \}
\end{split}
\end{equation} 
and ${\cal C}_G$ is the tilting of 
$\Coh(X)$.

(2) We note that ${\Bbb C}_x, x \in \pi^{-1}(p_i)$
is $S$-equivalent to
$\oplus_{j=0}^{s_i}D_X(E_{ij})\otimes K_X[n]^{\oplus a_{ij}}$, where
$D_X(E_{ij})\otimes K_X[n] \in {\cal C}_G^D$.
Hence the claim follows from (1).
\end{proof}

\subsubsection{Local projective generators of ${\cal C}$.}

Let $(S,T)$ be a torsion pair of $\Coh(X)$ such that
the tilted category ${\cal C}$ satisfies one of the following
conditions.
\begin{enumerate}
\item[(1)]
There is a local projective generator $G \in T$ of ${\cal C}$, that is,
${\cal C}$ is the category of perverse coherent sheaves or
\item[(2)]
${\cal C}$ satisfies the following conditions:
\begin{enumerate}
\item
$\# Y_\pi <\infty$ and 
every object of ${\cal C}_y$, $y \in Y$ is of finite length. 
\item
$\pi(\Supp(E)) \subset Y_\pi$ for $E \in S$.
\end{enumerate}
\end{enumerate}

We shall give a criterion for a two term complex
to be a local projective generator
of ${\cal C}$.
Let $E_{yj}, j \in J_y$ be the irreducible objects of ${\cal C}_y$. 

\begin{lem}\label{lem:tilting:freeness}
Let $E$ be an object of ${\bf D}(X)$ such that $H^i(E)=0$ for $i \ne -1,0$.
If $\Ext^1(E,{\Bbb C}_x)=0$, then $E$ is a free sheaf in a neighborhood
of $x$.
\end{lem}

\begin{proof}
Since $E$ fits in the exact triangle
\begin{equation}
0 \to H^{-1}(E)[1] \to E \to H^0(E) \to H^{-1}(E)[2],
\end{equation}
we have an exact sequence
\begin{equation}
0 \to {\cal E}xt^1_{{\cal O}_X}(H^0(E),{\Bbb C}_x) \to
{\cal E}xt^1_{{\cal O}_X}(E,{\Bbb C}_x) \to 
{\cal H}om_{{\cal O}_X}(H^{-1}(E),{\Bbb C}_x) \to
{\cal E}xt^2_{{\cal O}_X}(H^0(E),{\Bbb C}_x).
\end{equation}
Since $\Ext^1(E,{\Bbb C}_x)=H^0(X,{\cal E}xt^1_{{\cal O}_X}(E,{\Bbb C}_x))$,
${\cal E}xt^1_{{\cal O}_X}(E,{\Bbb C}_x)=0$.
Then ${\cal E}xt^1_{{\cal O}_X}(H^0(E),{\Bbb C}_x)=0$, which implies that
$H^0(E)$ is a free sheaf in a neighborhood of $x$.
Then ${\cal E}xt^i_{{\cal O}_X}(H^0(E),{\Bbb C}_x)=0$ for $i>0$.
Hence ${\cal H}om_{{\cal O}_X}(H^{-1}(E),{\Bbb C}_x)=0$.
Therefore $H^{-1}(E)=0$ in a neighborhood of $x$. 
\begin{NB}
We set $n:=\dim X$.
We note that $\Ext^1(E,{\Bbb C}_x) \cong \Ext^{n-1}({\Bbb C}_x,E)^{\vee}$
and 
$\Ext^{n-1}({\Bbb C}_x,E)=0$ if and only if
${\cal E}xt^{n-1}_{{\cal O}_X}({\Bbb C}_x,E)=0$.
Since ${\bf R}{\cal H}om_{{\cal O}_X}({\Bbb C}_x,{\cal O}_X)
={\Bbb C}_x[-n]$, we get
${\bf R}{\cal H}om_{{\cal O}_X}({\Bbb C}_x,E)={\Bbb C}_x 
\overset{{\bf L}}{\otimes} E[-n]$.
Hence
\begin{equation}
\begin{split}
0={\cal E}xt^{n-1}_{{\cal O}_X}({\Bbb C}_x,E)= 
&H^{n-1}({\bf R}{\cal H}om_{{\cal O}_X}({\Bbb C}_x,E))\\
=&H^{-1}({\Bbb C}_x \overset{{\bf L}}{\otimes} E)=
{\cal T}or_1^{{\cal O}_X}({\Bbb C}_x,E),
\end{split}
\end{equation}
 which implies that
$E$ is free in a neighbourhood of $x$.
\end{NB}
\end{proof}

\begin{lem}\label{lem:tilting:chi(G,E)}
Let $E_{yj}$, $y \in Y$ be the irreducible objects of ${\cal C}$
in Lemma \ref{lem:tilting:irreducible}.
Let $G_1$ be a locally free sheaf of rank $r$ on $X$ such that
\begin{equation}\label{eq:tilting:chi(G,E)}
(a)\;\Hom(G_1,E_{yj}[p])=0, p \ne 0\;\;
(b)\;\chi(G_1,E_{yj})>0
\end{equation}
for all $y,j$.
\begin{enumerate}
\item[(1)]
$G_1$ is a locally free sheaf.  
If $0 \ne E \in S$, then
$\pi_*(G_1^{\vee} \otimes E)=0$ and $R^1 \pi_*(G_1^{\vee} \otimes E) \ne 0$.
\item[(2)]
If $R^1 \pi_*(G_1^{\vee} \otimes E)= 0$, then $E \in T$. 
\item[(3)]
If $0 \ne E \in T$ and $\Supp(E) \subset \pi^{-1}(y)$, then 
$\pi_*(G_1^{\vee} \otimes E) \ne 0$ and 
$R^1 \pi_*(G_1^{\vee} \otimes E)= 0$.
In particular, $\chi(G_1,E)> 0$.
\end{enumerate}
\end{lem}

\begin{proof}

(1)
%Since ${\Bbb C}_x$ is generated by $E_{yj}$, Lemma \ref{lem:freeness}
%implies $G_1$ is a locally free sheaf.
We note that
$G_1 \in T$ by Lemma \ref{lem:tilting:T}.
We first treat the case where ${\cal C}$ is the category of 
perverse coherent sheaves. 
We consider the homomorphism
$\pi^*(\pi_*(G_1^{\vee} \otimes E)) \otimes G_1 \to E$.
Then $\im \phi \in T \cap S=0$. 
Since $\pi_*(G_1^{\vee} \otimes \im \phi)=\pi_*(G_1^{\vee} \otimes E)$,
we get $\pi_*(G_1^{\vee} \otimes E)=0$.
Let $F \ne 0$ be a coherent sheaf on a fiber
and take the decomposition
\begin{equation}
0 \to F_1 \to F \to F_2 \to 0
\end{equation}
with
$F_1 \in T, F_2 \in S$.
Since $F_1, F_2[1] \in {\cal C}$,
the condition $\chi(G_1,E_{yj})>0$ implies that
$\chi(G_1,F_1)>0$ or $\chi(G_1,F_2)<0$, which imply
that 
$\pi_*(G_1^{\vee} \otimes F_1) \ne 0$ or
$R^1 \pi_*(G_1^{\vee} \otimes F_2) \ne 0$.
Since $\pi_*(G_1^{\vee} \otimes F_1)$ is a subsheaf of
$\pi_*(G_1^{\vee} \otimes F)$ and
$R^1 \pi_*(G_1^{\vee} \otimes F_2)$ is a quotient of
$R^1 \pi_*(G_1^{\vee} \otimes F)$,
we get ${\bf R}\pi_*(G_1^{\vee} \otimes F) \ne 0$.
Then we can apply Lemma \ref{lem:tilting:restriction} to $E$ and get
$R^1 \pi_*(G_1^{\vee} \otimes E_{|\pi^{-1}(y)})\ne 0$ 
for $y \in \pi(\Supp(E))$.
Since $R^1 \pi_*(G_1^{\vee} \otimes E) \to 
R^1 \pi_*(G_1^{\vee} \otimes E_{|\pi^{-1}(y)})\ne 0$ is surjective,
we get the claim.  

We next assume that $\# Y_\pi<\infty$.
Then $E[1]$ is generated by $E_{yj}$. Hence
\eqref{eq:tilting:chi(G,E)} imply that
$\chi(G_1,E[1])>0$ and ${\bf R}\pi_*(G_1^{\vee} \otimes E[1]) \in \Coh(Y)$.
Hence $R^1\pi_*(G_1^{\vee} \otimes E) \ne 0$ and
$\pi_*(G_1^{\vee} \otimes E)=0$.

(2) For $E \in \Coh(X)$, we take a decomposition
\begin{equation}
0 \to E_1 \to E \to E_2 \to 0
\end{equation} 
such that $E_1 \in T$ and $E_2 \in S$.
If $R^1 \pi_*(G_1^{\vee} \otimes E)=0$, then
(1) implies that $E_2=0$. 

(3)
By Lemma \ref{lem:tilting:generate},  
we may assume that $E$ is a quotient of 
$E_{yj}$, $E_{yj} \in T$ in $\Coh(X)$.
Since $E_{yj}$ is irreducible,
$\phi:E_{yj} \to E$ is injective in ${\cal C}$.
We set $F:=\ker(E_{yj} \to E)$ in $\Coh(X)$.
Then $F \in S$ and $F[1]$ is the cokernel of 
$\phi$ in ${\cal C}$.
Hence $\pi_*(G_1^{\vee} \otimes F)=0$ by (1).
%Since $E_{yj}$ is irreducible, we see that
%$\Hom(E_{yj},F)=0$.
%By Lemma \ref{lem:tilting:generate},
%$\Supp(F) \subset \pi^{-1}(y)$ implies that
%$F \in S$,
%which implies that $\pi_*(G_1^{\vee} \otimes F)=0$.
By our assumption,
$\pi_*(G_1^{\vee} \otimes E_{yj}) \ne 0$, $E_{yj} \in T$ and
$R^1 \pi_*(G_1^{\vee} \otimes E_{yj})=0$.
Therefore our claim holds.
\end{proof}

\begin{prop}\label{prop:tilting:generator}
Let $G_1$ be an object of ${\bf D}(X)$ such that $H^i(E)=0$ for $i \ne -1,0$
and satisfies
\begin{equation}\label{eq:tilting:generator}
(a)\; \Hom(G_1,E_{yj}[p])=0, p \ne 0\;\;
(b)\; \chi(G_1,E_{yj})>0.
\end{equation}
\begin{enumerate}
\item[(1)]
$G_1$ is a locally free sheaf on $X$.
\item[(2)]
$R^1 \pi_*(G_1^{\vee} \otimes G_1)=0$.
\item[(3)]
For $E \in \Coh(X)$,
$E \in T$ if and only if
$R^1 \pi_*(G_1^{\vee} \otimes E)=0$, and 
$E \in S$ if and only if
$\pi_*(G_1^{\vee} \otimes E)=0$.
\item[(4)]
$G_1$ is a local projective generator of 
${\cal C}_G$.
\end{enumerate}
\end{prop}

\begin{proof}
(1) 
The claim follows from Lemma \ref{lem:tilting:freeness} and (a).
(2)
It is sufficient to prove that 
$R^1 \pi_*(G_1^{\vee} \otimes G_{1|\pi^{-1}(y)})=0$ for all $y \in Y_\pi$.
By Lemma \ref{lem:tilting:T},
$G_1 \in T$.
Since $\Supp(G_{1|\pi^{-1}(y)})=\pi^{-1}(y)$ and 
$G_{1|\pi^{-1}(y)} \in T$,
Lemma \ref{lem:tilting:generate} (1) implies that
$G_{1|\pi^{-1}(y)} \in T$ is a successive extension of quotients 
of $E_{yj} \in T$.
Hence it is sufficient to prove
$R^1 \pi_*(G_1^{\vee} \otimes Q)=0$ for all quotients $Q$ of 
$E_{yj} \in T$.
By our assumption on $G_1$,
we have $R^1 \pi_*(G_1^{\vee} \otimes E_{yj})=0$ for $E_{yj} \in T$.
Therefore the claim holds.

(3)
We set 
\begin{equation}
\begin{split}
T_1:=& \{E \in \Coh(X)| R^1 \pi_*(G_1^{\vee} \otimes E)=0 \},\\
S_1:=& \{E \in \Coh(X)|\pi_*(G_1^{\vee} \otimes E)=0 \}.
\end{split}
\end{equation}
By Lemma \ref{lem:tilting:chi(G,E)} (2), we get
\begin{equation}
%\{ E \in \Coh(X)| {\bf R}\pi_*(G_1^{\vee} \otimes E)=0 \}
T_1 \cap S_1 
\subset T \cap S_1=\{ E \in T|\pi_*(G_1^{\vee} \otimes E)=0 \}.
\end{equation}
If $T \cap S_1=0$,
then Lemma \ref{lem:tilting} (1) implies that $G_1$ is a local
projective generator of ${\cal C}_{G_1}$.
Since $G_1 \in T$ by (2), 
Lemma \ref{lem:tilting} (3) also implies that
${\cal C}={\cal C}_{G_1}$.
Therefore we shall prove that 
$T \cap S_1=0$.
%By Lemma \ref{lem:tilting} (3), we prove that 
%$E=0$ if $E \in T$ and $\pi_*(G_1^{\vee} \otimes E)=0$.
Assume that $E \in T$ satisfies $\pi_*(G_1^{\vee} \otimes E)=0$.
We first prove that $R^1 \pi_*(G_1^{\vee} \otimes E)=0$.
By Lemma \ref{lem:tilting:TFF}, it is sufficient to prove
$R^1 \pi_*(G_1^{\vee} \otimes E_{|\pi^{-1}(y)})=0$ for
all $y \in Y$. This follows from Lemma \ref{lem:tilting:chi(G,E)} (3).
Hence ${\bf R}\pi_*(G_1^{\vee} \otimes E)=0$.
Then we see that ${\bf R}\pi_*(G_1^{\vee} \otimes E_{|\pi^{-1}(y)} )=0$
for all $y \in Y$ by the proof of Lemma \ref{lem:tilting:restriction}.
Since $E_{|\pi^{-1}(y)} \in T$,
Lemma \ref{lem:tilting:chi(G,E)} (3) implies that 
$E_{|\pi^{-1}(y)}=0$ for all $y \in Y$.
Therefore $E=0$.

(4)
This is a consequence of (3) and Lemma \ref{lem:tilting} (2).
\end{proof}

\begin{rem}\label{rem:tilting:generator}
If $G_1$ in Proposition \ref{prop:tilting:generator}
satisfies
\eqref{eq:tilting:generator} (a) only, then
the proofs of Lemma \ref{lem:tilting:chi(G,E)} and
Proposition \ref{prop:tilting:generator} imply that 
$G_1$ is a locally free sheaf such that
$R^1 \pi_*(G_1^{\vee} \otimes G_1)=0$
and ${\bf R}\pi_*(G_1^{\vee} \otimes F) \in \Coh(Y)$
for $F \in {\cal C}_G$. 
\end{rem}

\begin{lem}\label{lem:tilting:R1=0}
Let $(S,T)$ be a torsion pair of $\Coh(X)$
and ${\cal C}$ its tilting.
Assume that one of the following holds.
\begin{enumerate}
\item
${\cal C}$ is the category of perverse coherent sheaves.
\item
$\# Y_\pi<\infty$, ${\cal C}_y$ is Artinian and
$\pi(\Supp(E)) \subset Y_\pi$ for $E \in S$.
\end{enumerate}
Let $G_1$ be a locally free sheaf of rank $r$ on $X$ such that
\begin{equation}
\chi(G_1,E_{yj})>0.
\end{equation}
Then $\Hom(G_1,E_{yj}[k])=0, k \ne 0$
if and only if
$R^1 \pi_*(G_1^{\vee} \otimes G_1)=0$.
\end{lem}

\begin{proof}
Assume that $R^1 \pi_*(G_1^{\vee} \otimes G_1)=0$.
We first prove that $G_1 \in T$.
Assume that $G_1 \not \in T$. Then there is a surjective homomorphism
$G_1 \to E$ in $\Coh(X)$ such that $E \in S$.
If ${\cal C}$ has a local projective generator
$G$, then $\pi_*(G^{\vee} \otimes E)=0$.
By Lemma \ref{lem:tilting:restriction},
we have $R^1 \pi_*(G^{\vee} \otimes E_{|\pi^{-1}(y)}) \ne 0$
for a point $y \in Y$.
Hence we may assume that $\Supp(E) \subset \pi^{-1}(y)$. 
In the second case, since $\# Y_\pi <\infty$,
we may also assume that $\Supp(E) \subset \pi^{-1}(y)$.
Then $E[1]$ is generated by $E_{yj}$, $0 \leq j \leq s_y$.
By our assumption, $\chi(G_1,E[1])>0$.
Hence $\Ext^1(G_1,E) \ne 0$, which implies that
$R^1 \pi_*(G_1^{\vee} \otimes G_1) \ne 0$.
Therefore $G_1 \in T$.
For $E_{yj} \in T$,
we consider the homomorphism
$\phi:\pi^*(\pi_*(G_1^{\vee} \otimes E_{yj})) \otimes G_1 \to E_{yj}$.
Since $E_{yj}$ is an irreducible object,
$\phi$ is surjective in ${\cal C}_G$,
which implies that $\phi$ is surjective in $\Coh(X)$.
Hence $\Ext^1(G_1, E_{yj})=0$.
For $E_{yj} \in S[1]$,
$\dim \pi^{-1}(y) \leq 1$ and the locally freeness of $G_1$ imply
that $\Ext^1(G_1, E_{yj})=0$.
Since $G_1 \in T$, we also get $\Hom(G_1,E_{yj}[-1])=0$ for all
irreducible objects of ${\cal C}$.
\end{proof}

\begin{lem}\label{lem:characterize}
Let $G$ be a locally free sheaf on $X$ such that
$R^1 \pi_*(G^{\vee} \otimes G)=0$.
Let $E$ be a 1-dimensional sheaf on a fiber of $\pi$
such that $\chi(G,E)=0$.
Then
${\bf R}\pi_*(G^{\vee} \otimes E)=0$ if and only if
$E$ is a $G$-twisted semi-stable sheaf 
with respect to an ample divisor $L$ on $X$.  
\end{lem}

\begin{proof}
By the proof of Lemma \ref{lem:tilting} (1),
we can take a decomposition
\begin{equation}
0 \to E_1 \to E \to E_2 \to 0
\end{equation}
such that ${\bf R}\pi_*(G^{\vee} \otimes E_1)=\pi_*(G^{\vee} \otimes E)$
and ${\bf R} \pi_*(G^{\vee} \otimes E_2)=
R^1 \pi_*(G^{\vee} \otimes E)[-1]$.
Then $\chi(G,E_1) \geq 0 \geq \chi(G,E_2)$.
Hence if $E$ is $G$-twisted semi-stable, then
$\pi_*(G^{\vee} \otimes E_1)=\pi_*(G^{\vee} \otimes E)=0$,
which also implies that $R^1 \pi_*(G^{\vee} \otimes E)=0$.
Conversely if
$\pi_*(G^{\vee} \otimes E)=R^1 \pi_*(G^{\vee} \otimes E)=0$,
then $\pi_*(G^{\vee} \otimes E')=0$ for any subsheaf
$E'$ of $E$.
Hence $E$ is $G$-twisted semi-stable.
\end{proof}

\begin{cor}\label{cor:characterize}
Assume that $\pi:X \to Y$ is the minimal resolution of a
rational double point. Let $H$ be the pull-back of an ample
divisor on $Y$.
Then 
%we have the following.
%\begin{enumerate}
%\item[(1)]
a locally free sheaf $G$ on $X$ is a tilting generator
of the categery ${\cal C}_G$ in Lemma \ref{lem:tilting}
if and only if
\begin{enumerate}
\item
[(i)] $R^1 \pi_*(G^{\vee} \otimes G)=0$ and
\item
[(ii)] there is no $G$-twisted stable sheaf $E$ such that
$\rk E=0$,
$\chi(G^{\vee} \otimes E)=0$, $(c_1(E),H)=0$ and $(c_1(E)^2)=-2$.
\end{enumerate}
Moreover (ii) is equivalent to 
$\rk G \not |(c_1(G),D)$ for $D$ with $(D,H)=0$ and $(D^2)=-2$.  
%\item[(2)]
%A locally free sheaf $G$ on $X$ is a tilting generator
% of the categery in Lemma \ref{lem:tilting}
%if and only if $G^{\vee}$ is a tilting generator of the corresponding
%categoty.
%\end{enumerate}
\end{cor}

\begin{proof}
%(1)
Let $E$ be a 1-dimensional $G$-twisted stable sheaf on $X$.
Then $E$ is a sheaf on the exceptional locus if and only if
$(c_1(E),H)=0$. 
Under this assumption, we have
$\chi(E,E)=-(c_1(E)^2)>0$.
Hence $(c_1(E)^2)=-2$.
By Lemma \ref{lem:characterize},
we get the first part of our claim.
Since $\chi(G,E)=-(c_1(G),c_1(E))+\rk G \chi(E)$, 
we also get the second claim by \cite[Prop. 4.6]{Y:action}.
%(2) is a consequence of (1).
\end{proof}

\begin{NB}

\begin{rem}
Assume that ${\bf R}\pi_*(G^{\vee} \otimes E)=\pi_*(G^{\vee} \otimes E)$.
Let $V$ ve a locally free sheaf on $Y$ with a surjection
$V \to \pi_*(G^{\vee} \otimes E)$.
Then we have a morphism $V \to {\bf R}\pi_*(G^{\vee} \otimes E)$
in ${\bf D}(Y)$ which induces the surjective homomorphism
$V \to H^0({\bf R}\pi_*(G^{\vee} \otimes E))=\pi_*(G^{\vee} \otimes E)$.
Then we have a morphism
$\pi^*(V) \otimes G \to 
{\bf L}\pi^*({\bf R}\pi_*(G^{\vee} \otimes E)) \otimes G
\overset{e}{\to} E$ such that
$(e \otimes 1_{G^{\vee}}) \circ \phi$ in the following diagram 
induces an isomorphsim ${\bf R}\pi_*((e \otimes 1_{G^{\vee}}) \circ \phi)$:
\begin{equation}
\begin{CD}
\pi^*(V) @>>>
{\bf L}\pi^*({\bf R}\pi_*(G^{\vee} \otimes E)) 
@. \\
@VVV @VV{\phi}V @.\\
\pi^*(V) \otimes G \otimes G^{\vee} @>>>
{\bf L}\pi^*({\bf R}\pi_*(G^{\vee} \otimes E)) \otimes G \otimes G^{\vee}
@>{e \otimes 1_{G^{\vee}}}>> E \otimes G^{\vee}.
\end{CD}
\end{equation} 
Hence $F:=\mathrm{Cone}(\pi^*(V) \otimes G \to E)[1]$ satisfies
${\bf R}\pi_*(G^{\vee} \otimes F) \in \Coh(Y)$.
\end{rem}
\end{NB}

\subsection{Examples of perverse coherent sheaves}
\label{subsect:perverse-examples}

Let $\pi:X \to Y$ be a birational map in subsection \ref{subsect:Morita}
with Assumption \ref{ass:1}.
Let $G$ be a locally free sheaf on $X$ such that 
$R^1 \pi_*(G^{\vee} \otimes G)=0$.
We set ${\cal A}:=\pi_*(G^{\vee} \otimes G)$ as before.
\begin{NB}
For an ${\cal A}$-module $E$ on $Y$,
we have a surjective ${\cal A}$-homomorphism
$H^0(Y,E(n))\otimes {\cal A}(-n) \to E$, $n \gg 0$.
Hence we have a locally free resolution $V^{\bullet}$ of $E$ in
$\Coh_{{\cal A}}(Y)$.
\end{NB}
Let $F$ be a coherent ${\cal A}$-module on $Y$.
Then 
${\bf R}\pi_*((\pi^{-1}(F) 
\overset{{\bf L}}{\otimes}_{\pi^{-1}({\cal A})} G)\otimes G^{\vee})
\cong F$ as an ${\cal A}$-module.
\begin{NB}
$\pi_*((\pi^{-1}(V^{\bullet}) 
\otimes_{\pi^{-1}({\cal A})} G)\otimes G^{\vee})
\cong V^{\bullet}$.
\end{NB}
By using the spectral sequence, we see that
\begin{equation}\label{eq:tilting:F}
R^p \pi_*(G^{\vee} \otimes H^q(\pi^{-1}(F) 
\overset{{\bf L}}{\otimes}_{\pi^{-1}({\cal A})} G))=0,\;p+q \ne 0
\end{equation}
and we have an exact sequence
\begin{equation}
0 \to R^1\pi_*(G^{\vee} \otimes H^{-1}(\pi^{-1}(F) 
\overset{{\bf L}}{\otimes}_{\pi^{-1}({\cal A})} G))
\to F \overset{\lambda}{\to}
\pi_*(G^{\vee} \otimes H^0(\pi^{-1}(F) 
\overset{{\bf L}}{\otimes}_{\pi^{-1}({\cal A})} G))
\to 0.
\end{equation}
We set 
\begin{equation}
\pi^{-1}(F) \otimes_{\pi^{-1}({\cal A})} G
:=H^0(\pi^{-1}(E) \overset{{\bf L}}{\otimes}_{\pi^{-1}({\cal A})} G)
\in \Coh(X).
\end{equation}
We set 
\begin{equation}
\begin{split}
S_0:=& \{E \in \Coh(X)|{\bf R}\pi_*(G^{\vee} \otimes E)=0 \},\\
S:=& \{E \in \Coh(X)| \pi_*(G^{\vee} \otimes E)=0 \},\\
T:=& \{E \in \Coh(X)| R^1\pi_*(G^{\vee} \otimes E)=0,\;
\Hom(E,c)=0, c \in S_0\}.
\end{split}
\end{equation}

\begin{lem}\label{lem:tilting:C(G)}
For $E \in \Coh(X)$,
let
$\phi:\pi^{-1}(\pi_*(G^{\vee} \otimes E)) \otimes_{\pi^{-1}({\cal A})} G
\to E$ be the evaluation map.
\begin{enumerate}
\item[(1)]
${\bf R}\pi_*(G^{\vee} \otimes \ker \phi)=0$,
$\pi_*(G^{\vee} \otimes \coker \phi)=0$ and
$R^1 \pi_*(G^{\vee} \otimes E) \cong 
R^1\pi_*(G^{\vee} \otimes \coker \phi)$.
\item[(2)]
$(S,T)$ is a torsion pair of $\Coh(X)$ and
the decomposition of $E$ is given by
\begin{equation}
0 \to \im \phi \to E \to \coker \phi \to 0,
\end{equation}
$\im \phi \in T$, $\coker \phi \in S$.
\end{enumerate}
\end{lem}

\begin{NB}
The action of $\lambda \in \pi_*(G^{\vee} \otimes G)$ on
$\pi_*(G^{\vee} \otimes E)$:
$\lambda(f):=f \circ \lambda, f:G \to E$.
Then $\lambda \cdot f \otimes x-f \otimes \lambda \cdot x
=f \circ\lambda \otimes x-f \otimes \lambda \circ x$.
Therefore
the evaluation map
\begin{equation}
\begin{matrix}
\phi:&
\pi^{-1}(\pi_*(G^{\vee} \otimes E)) \otimes_{\pi^{-1}({\cal A})}G & \to& E\\
& f \otimes x & \mapsto & f(x)
\end{matrix}
\end{equation}
is well-defined.

By Lemma \ref{lem:tilting:A-module},
$\phi$ induces 
\begin{equation}
\pi^*(\pi_*(G^{\vee} \otimes E)) \to G^{\vee} \otimes E.
\end{equation}
\end{NB}

\begin{proof}
(1)
We have a homomorphism
\begin{equation}
\pi_*(G^{\vee} \otimes E) \overset{\lambda}{\longrightarrow} 
\pi_*(G^{\vee} \otimes 
\pi^{-1}(\pi_*(G^{\vee} \otimes E)) \otimes_{\pi^{-1}({\cal A})} G)
\overset{\pi_*(1_{G^{\vee}} \otimes \phi)}{\longrightarrow} 
\pi_*(G^{\vee} \otimes E)
\end{equation}
which is the identity.
Then $\lambda$ and $\pi_*(1_{G^{\vee}} \otimes \phi)$ are isomorphic. 
Hence we get
$\im \pi_*(1_{G^{\vee}} \otimes \phi)=\pi_*(G^{\vee} \otimes \im \phi)=
\pi_*(G^{\vee} \otimes E)$.
Since $R^1 \pi_*(G^{\vee} \otimes 
\pi^{-1}(\pi_*(G^{\vee} \otimes E)) \otimes_{\pi^{-1}({\cal A})} G)=0$,
we get ${\bf R}\pi_*(G^{\vee} \otimes \ker \phi)=0$.
Since $R^1 \pi_*(G^{\vee} \otimes \im \phi)=0$,
we also get the remaining claims.

(2) 
We shall prove that $\im \phi \in T$.
If $\im \phi \not \in T$, then
there is a homomorphism $\psi:\im \phi \to F$
such that $F \in S$. Replacing $F$ by $\im \psi$, we may assume that
$\psi$ is surjective. Since $\psi \circ \phi$ is surjective,
$\Hom(G,F) \ne 0$, which is a contradiction.
Therefore $\im \phi \in T$.
Obviously we have $S \cap T=\{ 0 \}$. Therefore $(S,T)$ is a torsion pair.
\end{proof}

Let ${\cal C}(G)$ be the tilting of $\Coh(X)$.
Then ${\cal C}(G)$ is the category of perverse coherent sheaves in the
sense of Definition \ref{defn:perverse}.
Indeed we have the following.

\begin{lem}\label{lem:tilting:VB}
(cf. \cite[Prop. 3.2.5]{VB})
Let $G$ be a locally free sheaf on $X$ such that
$R^1 \pi_*(G^{\vee} \otimes G)=0$.
Let ${\cal C}(G)$ be the associated category.
Then there is a local projective generator of ${\cal C}(G)$.
\end{lem}

\begin{proof}
Let $L$ be a line bundle on $X$ such that
$G^{\vee} \otimes L$ is generated by global sections and
$\det(G^{\vee} \otimes L)$ is ample.
%is surjective and $\rk G c_1(L)-c_1(G)$ is ample.
We take a locally free resolution
$0 \to L_{-1} \to L_0 \to L \to 0$ such that
$R^1 \pi_*(L_0^{\vee} \otimes G)=0$.
Then 
\begin{equation}
{\bf R}\pi_*(L^{\vee} \otimes G)[1]=
\mathrm{Cone}
(\pi_*(L_0^{\vee} \otimes G) \to \pi_*(L_{-1}^{\vee} \otimes G)).
\end{equation}
We take a surjective homomorphism
$V \to \pi_*(L_{-1}^{\vee} \otimes G)$
from a locally free sheaf $V$ on $Y$.
Then we have a morphism
$\pi^*(V) \otimes L \to 
{\bf L}\pi^*({\bf R} \pi_*(L^{\vee} \otimes G))[1] \otimes L
\to G[1]$, which induces a surjective homomorphism
$V \to R^1 \pi_*(L^{\vee} \otimes G)$.
Hence we have a morphism
\begin{equation}
L \to G[1] \otimes \pi^*(V)^{\vee}
\end{equation}
such that the induced homomorphism
\begin{equation}\label{eq:tilting:E}
V \to \pi_*({\cal H}om(G[1],G[1])) \otimes V \to
R^1 \pi_*(L^{\vee} \otimes G)
\end{equation}
 is surjective.
We set $E:=\mathrm{Cone}(L \to G[1] \otimes \pi^*(V)^{\vee})[-1]$.
Then $E$ is a locally free sheaf on $X$
and $\phi:\pi^*(\pi_*(G^{\vee} \otimes E)) \otimes G \to E$
is surjective by our choice of $L$.
By \eqref{eq:tilting:E} and our assumption, we have
$R^1 \pi_*(E^{\vee} \otimes G)=0$.
For $F \in T$, 
we consider the evaluation map
$\varphi:\pi^*(\pi_*(G^{\vee} \otimes F)) \otimes G \to F$.
The proof of Lemma \ref{lem:tilting} (1) implies that
$\coker \varphi \in S_0$.
By the definition of $T$, $\coker \varphi =0$.
Thus $\varphi$ is surjective.
Hence $R^1 \pi_*(E^{\vee} \otimes F)=0$ for $F \in T$.

For $F \in S$, the surjectivity of $\phi$ implies that 
$\pi_*(E^{\vee} \otimes F)=0$.   
If $F \not \in S_0$, then
$R^1 \pi_*(G^{\vee} \otimes F) \ne 0$,
which implies that $R^1 \pi_*(E^{\vee} \otimes F) \ne 0$.
Assume that $F \in S_0$. Then
since ${\bf R} \pi_*(G^{\vee} \otimes F)=0$ for $F \in S_0$,
we have $R^1 \pi_*(E^{\vee} \otimes F) \cong 
R^1 \pi_*(L^{\vee} \otimes F)$. % for $F \in S_0$.
Assume that $R^1 \pi_*(L^{\vee} \otimes F)=0$ and $F \ne 0$.
Let $W$ be an irreducible component of $\Supp(F)$.
Then $F$ contains a subsheaf $F'$ whose support is contained in $W$.
If $W \to Y$ is generically finite, then
$\pi_*(G^{\vee} \otimes F')\ne 0$, which is a contradiction.
Therefore $\dim F'=\dim \pi(F')+1$.
For a point $y \in \pi_*(F')$,
we can take a homomorphism 
$\psi:{\cal O}_X^{\oplus (\rk G)-1} \to G^{\vee}\otimes L$ such that
$\psi_{|\pi^{-1}(y)}$ is injective for any point of
$\pi^{-1}(y)$.
Then $\coker \psi$ is a line bundle in a neighborhood of 
$\pi^{-1}(y)$.
Since $\pi$ is proper, there is an open neighborhood $U$ of $y$
such that $\coker \psi_{\pi^{-1}(U)}$ is a line bundle.
Hence we have an exact sequence on $\pi^{-1}(U)$:
\begin{equation}
0 \to {\cal O}_{\pi^{-1}(U)}^{\oplus (\rk G)} \to 
(G^{\vee}\otimes L)_{|\pi^{-1}(U)}
\to C  \to 0,
\end{equation}
where $C:=\coker \psi_{\pi^{-1}(U)}/{\cal O}_{\pi^{-1}(U)}$.
We may assume that $\Supp(C)_{|\pi^{-1}(y)}$ is a finite set.
Then $\Supp(F' \otimes C) \to Y$ is generically finite.
Hence $\pi_*(F' \otimes C \otimes L^{\vee}) \ne 0$, which implies that
$\pi_*(F \otimes C \otimes L^{\vee}) \ne 0$.
On the other hand, our assumptions impliy that
${\bf R} \pi_*(F \overset{{\bf L}}{\otimes} C \otimes L^{\vee})=0$.  
Since the spectral spectral sequence
\begin{equation}
E^{pq}_2=R^p \pi_*(H^q(F \overset{{\bf L}}{\otimes} C \otimes L^{\vee}))
\Rightarrow E^{p+q}_{\infty}=
H^{p+q}({\bf R} \pi_*(F \overset{{\bf L}}{\otimes} C \otimes L^{\vee}))
\end{equation}
degenerates, we have 
$\pi_*(F \otimes C \otimes L^{\vee})=0$, which is a contradiction.
Hence $R^1 \pi_*(L^{\vee} \otimes F) \ne 0$ for
all non-zero $F \in S_0$.
Then $G_1:=G \oplus E$ satisfies
\begin{equation}
\begin{split}
\pi_*(G_1^{\vee} \otimes F) \ne 0,\;\; & 
R^1 \pi_*(G_1^{\vee} \otimes F)=0, \;\;
0 \ne F \in T\\
\pi_*(G_1^{\vee} \otimes F) = 0,\;\; & R^1 \pi_*(G_1^{\vee} \otimes F) \ne0, 
\;\; 0 \ne F \in S. 
\end{split}
\end{equation}
Therefore $G_1$ is a local projective generator of ${\cal C}(G)$.
  \end{proof}

We set 
\begin{equation}
\begin{split}
S^*:=& \{E \in \Coh(X)| \pi_*(G^{\vee} \otimes E)=0,\;
\Hom(c,E)=0, c \in S_0 \},\\
T^*:=& \{E \in \Coh(X)| R^1\pi_*(G^{\vee} \otimes E)=0\}.
\end{split}
\end{equation}

\begin{lem}
$(S^*,T^*)$ is a torsion pair of $\Coh(X)$ and 
the tilted category ${\cal C}(G)^*$ has a local projective generator. 
\end{lem}

\begin{proof}
We set 
\begin{equation}
\begin{split}
S_0':=& \{E \in \Coh(X)|{\bf R}\pi_*(G \otimes E)=0 \},\\
S_1:=& \{E \in \Coh(X)| \pi_*(G \otimes E)=0 \},\\
T_1:=& \{E \in \Coh(X)| R^1\pi_*(G \otimes E)=0,\;
\Hom(E,c)=0, c \in S_0' \}.
\end{split}
\end{equation}
Then $(S_1,T_1)$ is a torsion pair of $\Coh(X)$ and Lemma \ref{lem:tilting:VB}
implies that the tilted category
${\cal C}(G^{\vee})$ has a local projective generator $G^{\vee} \oplus E_1$,
where $E_1$ is a locally free sheaf on $X$
such that $\phi:\pi^*(\pi_*(G \otimes E_1)) \otimes G^{\vee} \to E_1$ 
is surjective and
$R^1 \pi_*(G^{\vee} \otimes E_1^{\vee})=0$.
By Lemma \ref{lem:tilting:dual},
$(S_1^D,T_1^D)$ is a torsion pair of $\Coh(X)$.
We prove that ${\cal C}(G)^*={\cal C}(G^{\vee})^D$ by showing
$(S_1^D,T_1^D)=(S^*,T^*)$.
By the surjectivity of $\phi$, we have
\begin{equation}
T_1^D=\{E \in \Coh(X)| R^1 \pi_*(G^{\vee} \otimes E)=R^1 \pi_*(E_1 \otimes E)
=0 \}=T^*.
\end{equation}
For a coherent sheaf $E$ with $\pi_*(G^{\vee} \otimes E)=0$,
we consider $\psi:\pi^*(\pi_*(E_1 \otimes E)) \otimes E_1^{\vee}
\to E$.
Then $\im \psi \in T_1^D=T^*$ and $\coker \psi \in S_1^D$.
\begin{NB}
We set 
\begin{equation}
(S_0')^D:=\{ E \in \Coh(X)|{\bf R}\pi_*(G^{\vee} \otimes E)=0 \}=S_0.
\end{equation}
\end{NB}
Since $\pi_*(G^{\vee} \otimes \im \psi)=0$,
$\im \psi \in S_0$. Therefore if $E \in S^*$, then
$\im \psi=0$, which means that $E \in S_1^D$.
Conversely if $E \in S_1^D$, then
$S_0 \subset T_1^D$ implies that $E \in S^*$.
Therefore $(S_1^D,T_1^D)=(S^*,T^*)$.
\end{proof}

Let $E_{yj}$, $y \in Y_\pi$ be the irreducible objects of ${\cal C}$
in Lemma \ref{lem:tilting:irreducible} (3).

\begin{lem}
We set $S_{0y}:=\{E \in S_0| \pi(\Supp(E))=\{ y\} \}$.
Then
$S_{0y}[1]$ is generated by
$\{E_{yj}| E_{yj} \in S_0[1] \}$.
\end{lem}

\begin{proof}
For an exact sequence 
\begin{equation}
0 \to E_1 \to E \to E_2 \to 0
\end{equation}
in ${\cal C}$, we have an exact sequence
\begin{equation}
0 \to {\bf R}\pi_*(G^{\vee} \otimes E_1) \to
{\bf R}\pi_*(G^{\vee} \otimes E) \to
{\bf R}\pi_*(G^{\vee} \otimes E_2) \to 0
\end{equation}
in $\Coh(Y)$.
If $E \in S_0[1]$, then ${\bf R}\pi_*(G^{\vee} \otimes E_1)=
{\bf R}\pi_*(G^{\vee} \otimes E_2)=0$.
Then ${\bf R} \pi_*(G^{\vee} \otimes H^{-1}(E_1))=
{\bf R} \pi_*(G^{\vee} \otimes H^{-1}(E_2))=0$
and ${\bf R} \pi_*(G^{\vee} \otimes H^0(E_1))=
{\bf R} \pi_*(G^{\vee} \otimes H^0(E_2))=0$.
By the definition of $T$, $H^0(E_1)=H^0(E_2)=0$.
Hence $E_1, E_2 \in S_0[1]$.
Therefore the claim holds.
\end{proof}

By the construction of ${\cal C}(G)$ and
${\cal C}(G)^*$, we have the following.
\begin{prop}\label{prop:tilting:contraction}
We set ${\cal A}_0:=\pi_*(G^{\vee} \otimes G)$.
Then we have morphisms
\begin{equation}
\begin{matrix}
{\cal C}(G) & \to & \Coh_{{\cal A}_0}(Y)\\
E & \mapsto & {\bf R} \pi_*(G^{\vee} \otimes E)
\end{matrix}
\end{equation} 
and
\begin{equation}
\begin{matrix}
{\cal C}(G)^* & \to & \Coh_{{\cal A}_0}(Y)\\
E & \mapsto & {\bf R} \pi_*(G^{\vee} \otimes E).
\end{matrix}
\end{equation} 
\end{prop}
Let $\tau^{\geq -1}:{\bf D}(X) \to {\bf D}(X)$
be the trancation morphism such that
$H^p(\tau^{\geq -1}(E))=0$ for $p<-1$ and
$H^p(\tau^{\geq -1}(E))=H^p(E)$ for $p \geq -1$.
By \eqref{eq:tilting:F}, we have
\begin{equation}
\begin{split}
H^q(\pi^{-1}(F) 
\overset{{\bf L}}{\otimes}_{\pi^{-1}({\cal A})} G) & \in S_0,
\;q \ne -1,0,\\
\Sigma(F):=\tau^{\geq -1}(\pi^{-1}(F) 
\overset{{\bf L}}{\otimes}_{\pi^{-1}({\cal A})} G) & \in
{\cal C}(G).
\end{split}
\end{equation} 
Thus we have a morphism $\Sigma:\Coh_{{\cal A}_0}(Y) \to {\cal C}(G)$
such that ${\bf R}\pi_*(G^{\vee} \otimes \Sigma(F))=F$ for
$F \in \Coh_{{\cal A}_0}(Y)$.

\subsubsection{$^p\Per(X/Y)$, $p=-1,0$ and their generalizations.}
If $S_0=\{0 \}$, then $G$ is a local projective generator
of ${\cal C}(G)$.
We give examples such that $S_0 \ne \{ 0 \}$.
For $y \in Y_{\pi}$, we set
$Z_y:=\pi^{-1}(y)$ and
$C_{yj}$, $j=1,...,s_y$ the irreducible components of $Z_y$.
Assume that ${\bf R}\pi_*({\cal O}_X)={\cal O}_Y$.
Then $S_0$ for ${\cal O}_X$ contains ${\cal O}_{C_{yj}}(-1)$, $y \in Y_\pi$.
%We set
%\begin{equation}
%\begin{split}
%S_0:=&\{ E \in \Coh(X)| {\bf R}\pi_*(E)=0 \},\\
%S:=&\{E \in \Coh(X)|\pi_*(E)=0 \},\\
%T:=&\{ E \in \Coh(X)|R^1 \pi_*(E)=0, \Hom(E,c)=0, c \in S_0 \}.
%\end{split}
%\end{equation}
%Then $(S,T)$ is a torsion pair of $\Coh(X)$ and the tilting 
In this case,
${\cal C}({\cal O}_X)$ is nothing but the category $^{-1}\Per(X/Y)$ 
defined by Bridgeland.
We also have ${\cal C}({\cal O}_X)^*={\cal C}({\cal O}_X^{\vee})^D=
^{0}\Per(X/Y)$. 
We shall study $S_0$ containing line bundles on $C_{yj}$, $y \in Y_\pi$. 
For this purpose, we first prepare some properties of
$S_0$ for ${\cal O}_X$.

\begin{lem}\label{lem:tilting:chi=1}
\begin{enumerate}
\item
[(1)]
Let $E$ be a stable 1-dimensional sheaf such that
$\Supp(E) \subset Z_y$ and $\chi(E)=1$.
Then there is a curve $D \subset Z_y$ and $E \cong {\cal O}_D$.
Conversely if 
${\cal O}_D$ is purely 1-dimensional,
$\chi({\cal O}_D)=1$ and $\pi(D)=\{y \}$, 
then ${\cal O}_D$ is stable. In particular,
$D$ is a subscheme of $Z_y$.
\item[(2)]
${\cal O}_{Z_y}$ is stable.
\end{enumerate}
\end{lem}

\begin{proof}
(1)
Since $\chi(E)=1$, 
$\pi_*(E) \ne 0$.
Since $\pi_*(E)$ is 0-dimensional,
we have a homomorphism
${\Bbb C}_y \to \pi_*(E)$.
Then we have a homomorphism
$\phi:{\cal O}_{Z_y}=\pi^*({\Bbb C}_y) \to E$.
We denote the image by ${\cal O}_D$.
Since $R^1 \pi_*({\cal O}_X)=0$, we have 
$H^1(X,{\cal O}_D)=0$.
Hence $\chi({\cal O}_D) \geq 1$.
Since $E$ is stable,
$\phi$ must be surjective.

Conversely we assume that ${\cal O}_D$ satisfies
$\chi({\cal O}_D)=1$. 
For a quotient ${\cal O}_D \to {\cal O}_C$,
$H^1(X,{\cal O}_C)=0$ implies that
$\chi({\cal O}_C) \geq 1$, which implies that ${\cal O}_D$ is
stable.

(2)
By
${\cal O}_{Z_y}=\pi^*({\Bbb C}_y)$ and the surjectivity of
${\Bbb C}_y \to \pi_* (\pi^*({\Bbb C}_y))$,
we get $\chi({\cal O}_{Z_y})=1$. 
Hence ${\cal O}_{Z_y}$ is stable.
\end{proof}

\begin{lem}\label{lem:tilting:Rpi*E=0}
\begin{enumerate}
\item[(1)]
Let $E$ be a stable purely 1-dimension sheaf such that 
$\pi(\Supp(E)) =\{y \}$ and $\chi(E)=0$.
Then $E \cong {\cal O}_{C_{yj}}(-1)$.
\item[(2)]
Let $E$ be a 1-dimensional sheaf such that
${\bf R}\pi_*(E)=0$. Then $E$ is a semi-stable 1-dimensional sheaf
with $\chi(E)=0$. In particular, $E$ 
is a successive extension of
${\cal O}_{C_{yj}}(-1)$, $y \in Y$, $1 \leq j \leq s_y$.
\end{enumerate}
\end{lem}

\begin{proof}
(1) 
We set $n:=\dim X$. We take a point $x \in \Supp(E)$.
Then
${\cal E}xt^1_{{\cal O}_X}({\Bbb C}_x,E)
={\Bbb C}_x \overset{{\bf L}}{\otimes} E[-n+1]$.
Since $E$ is purely 1-dimensional,
$\depth_{{\cal O}_{X,x}}E_x=1$. Hence
the projective dimension of $E$ at $x$ is $n-1$.
Then ${\cal T}or_{n-1}^{{\cal O}_X}({\Bbb C}_x,E)=
H^0({\Bbb C}_x \overset{{\bf L}}{\otimes} E[-n+1]) \ne 0$.
Since $\Ext^1({\Bbb C}_x,E)=
H^0(X,{\cal E}xt^1_{{\cal O}_X}({\Bbb C}_x,E)) \ne 0$,
we can take a non-trivial extension
\begin{equation}
0 \to E \to F \to {\Bbb C}_x \to 0.
\end{equation}
If $F$ is not semi-stable,
then since $\chi(F)=1$, there is a quotient $F \to F'$ of $F$ such that
$F'$ is a stable sheaf with $\chi(F') \leq 0$.
Then $E \to F'$ is an isomorphism, which is a contradiction.
By Lemma \ref{lem:tilting:chi=1}, $F={\cal O}_D$.
We take an integral curve $C \subset D$ containing $x$.
Since ${\cal O}_D \to {\Bbb C}_x$ factor through ${\cal O}_C$,
we have a surjective homomorphism
$E \to {\cal O}_C(-1)$. By the stability of $E$,
$E \cong {\cal O}_C(-1)$.  

(2)
Let $F$ be a subsheaf of $E$. Then we have $\pi_*(F)=0$, which implies that
$\chi(F) \leq 0$. Therefore $E$ is semi-stable.
\end{proof}

\begin{NB}
\begin{lem}\label{lem:tilting:-1Per-S}
Let $E$ be a 1-dimensional sheaf such that $\pi_*(E)=0$.
Then there is a homomorphism $E \to {\cal O}_{C_{yj}}(-1)$.
In particular, $E$ is generated by subsheaves of ${\cal O}_{C_{yj}}(-1)$,
$y \in Y$, $1 \leq j \leq s_y$.
\end{lem}

\begin{proof}
We note that $\chi(E) \leq 0$.
If $\chi(E)=0$, then $\chi(R^1 \pi_*(E))=0$.
Since $\dim E=1$ and $\pi_*(E)=0$,
$\Supp(E) \to Y$ is not finite.
Hence $\dim \pi(\Supp(E))=0$.
Then we have 
$R^1 \pi_*(E)=0$. Hence the claim holds.
We assume that $\chi(E)<0$. 
Since $\pi_*(E)=0$ for
any stable 1-dimensional sheaf $E$ with $\chi(E) \leq 0$,
replacing $E$ by a stable quotiet sheaf $E \to E'$ of $E$,
we may assume that $E$ is stable.
We take a non-trivial extension
\begin{equation}
0 \to E \to F \to {\Bbb C}_x \to 0.
\end{equation}
Then $F$ is purely 1-dimensional.
Assume that there is a quotient $F \to F'$ of $F$ such that
$F'$ is a stable sheaf with
$\chi(F')/(c_1(F'),L)<\chi(F)/(c_1(F),L)$.
Then $\phi:E \to F'$ is non-zero.
Hence $\chi(F')/(c_1(F'),L) \geq 
\chi(\im \phi)/(c_1(\im \phi),L) \geq \chi(E)/(c_1(E),L)$, 
which implies that
$\chi(\im \phi) \geq \chi(E)$.
If $\chi(\im \phi)=\chi(E)$, then $\phi$ is injective.
Since $(c_1(F),L)=(c_1(E),L)$,
we get $F=F'$, which is a contradiction.
Therefore $F$ is stable or $\chi(F')>\chi(E)$.
Thus we get a homomorphism $\psi:E \to E'$
such that $E'$ is a stable sheaf with $\chi(E)<\chi(E')<0$
and $\psi$ is surjective in codimension $n-1$. 
By the induction on $\chi(E)$, we get the claim.
\end{proof}

\begin{lem}\label{lem:-1per:characterize}
For a point $y \in Y_\pi$,
let $E$ be a 1-dimensional sheaf on $X$ satisfying the following two
conditions:
\begin{enumerate}
\item
 $\Hom(E,{\cal O}_{C_{yj}}(-1))=\Ext^1(E,{\cal O}_{C_{yj}}(-1))=0$
for all $j$.
\item
There is an exact sequence
\begin{equation}
0 \to F \to E \to {\Bbb C}_x \to 0
\end{equation} 
such that $F$ is a semi-stable 1-dimensional sheaf with
$\pi(\Supp(F))=\{y \}$, $\chi(F)=0$ and $x \in Z_y$.
\end{enumerate}
Then $E \cong {\cal O}_{Z_y}$. 
\end{lem}

\begin{proof}
We first prove that ${\cal O}_{Z_y}$ satisfies
(i) and (ii).
For the exact sequence
\begin{equation}
0 \to F' \to \pi^*(\pi_*({\Bbb C}_x)) \to {\Bbb C}_x \to 0, 
\end{equation}
we have ${\bf R}\pi_*(F')=0$.
Hence (ii) holds by Lemma \ref{lem:tilting:Rpi*E=0}.
(i) follows from Lemma \ref{lem:tilting:TFF}.
Conversely we assume that $E$ satisfies (i) and (ii).
By (ii), $\pi_*(E) \cong \pi_*({\Bbb C}_x) \cong {\Bbb C}$.
By (i), $\pi^*(\pi_*(E)) \to E$ is surjective.
Hence we have an exact sequence
\begin{equation}
0 \to F' \to {\cal O}_{Z_y} \to E \to 0,
\end{equation}
where $F'$ is a semi-stable 1-dimensional sheaf with 
$\chi(F')=0$.
Since $\Ext^1(E,{\cal O}_{C_{yj}}(-1))=0$
for all $j$,
${\cal O}_{Z_y} \cong E \oplus F'$, which implies that 
${\cal O}_{Z_y} \cong E$.
\end{proof}

We set
\begin{equation}
E_{yj}:=
\begin{cases}
{\cal O}_{Z_y},&j=0,\\
{\cal O}_{C_{yj}}(-1)[1],& j>0.\\
\end{cases}
\end{equation}
\begin{prop}(\cite{VB})\label{prop:tilting:-1Per-irred}
\begin{enumerate}
\item[(1)]
$E_{yj}$, $j=0,...,s_y$ are irreducible objects of $^{-1}\Per(X/Y)$. 
\item[(2)]
${\Bbb C}_x$, $x \in \pi^{-1}(y)$ is generated by
$E_{yj}$. In particular, irreducible objects of $^{-1}\Per(X/Y)$
are 
\begin{equation}
{\Bbb C}_x, ( x \in X \setminus \pi^{-1}(Y_\pi)),\;\; 
E_{yj}, (y \in Y_\pi,j=0,1,...,s_y).
\end{equation}  
\end{enumerate}
\end{prop}

\begin{proof}
(1)
Assume that there is an exact sequence in $^{-1}\Per(X/Y)$:
\begin{equation}
0 \to E_1 \to {\cal O}_{Z_y} \to E_2 \to 0.
\end{equation}
Since $H^{-1}(E_1)=0$,
$E_1 \in T$ and $\pi(E_1) \cong \pi_*( {\cal O}_{Z_y})={\Bbb C}_y$.
Hence we have a non-zero morphism ${\cal O}_{Z_y} \to E_1$.
Since $\Hom({\cal O}_{Z_y},{\cal O}_{Z_y}) \cong {\Bbb C}$,
$E_1 \cong {\cal O}_{Z_y}$ and $E_2=0$.
For ${\cal O}_{C_{yj}}(-1)[1]$, assume that
there is an exact sequence in $^{-1}\Per(X/Y)$:
\begin{equation}
0 \to E_1 \to {\cal O}_{C_{yj}}(-1)[1] \to E_2 \to 0.
\end{equation}
Since $H^0(E_2)=0$, we have $E_2[-1] \in S$. Then
Lemma \ref{lem:tilting:-1Per-S} implies that
we have a non-zero morphism $E_2 \to {\cal O}_{C_{yj}}(-1)[1]$.
Since $\Hom({\cal O}_{C_{yj}}(-1)[1],{\cal O}_{C_{yj}}(-1)[1])
={\Bbb C}$, we get $E_1=0$.
Therefore ${\cal O}_{C_{yj}}(-1)[1]$ is irreducible.
\end{proof}

We set
\begin{equation}
E_{yj}^*:=E_{yj}^{\vee}[\dim X]=
\begin{cases}
\omega_{Z_y}[1],&j=0,\\
{\cal O}_{C_{yj}}(-1),& j>0.\\
\end{cases}
\end{equation}
Then we also have the following.
\begin{prop}\cite{VB}\label{prop:tilting:0Per-irred}
\begin{enumerate}
\item[(1)]
$E_{yj}^*$, $j=0,...,s_y$ are irreducible objects of $^{0}\Per(X/Y)$. 
\item[(2)]
${\Bbb C}_x$, $x \in \pi^{-1}(y)$ is generated by
$E_{yj}$. In particular, irreducible objects of $^{0}\Per(X/Y)$
are 
\begin{equation}
{\Bbb C}_x, ( x \in X \setminus \pi^{-1}(Y_\pi)),\;\; 
E_{yj}^*, (y \in Y_\pi,j=0,1,...,s_y).
\end{equation}  
\end{enumerate}
\end{prop}
\end{NB}

We shall slightly generalize $^{-1}\Per(X/Y)$.
Let $G$ be a locally free sheaf on $X$.
\begin{assume}\label{ass:2}
Assume that
$R^1 \pi_*(G^{\vee} \otimes G)=0$
and
there are line bundles ${\cal O}_{C_{yj}}(b_{yj})$ on $C_{yj}$
such that
${\bf R}\pi_*(G^{\vee} \otimes {\cal O}_{C_{yj}}(b_{yj}))=0$.
\end{assume}

\begin{lem}\label{lem:tilting:pull-back}
\begin{enumerate}
\item[(1)]
Let $E$ be a locally free sheaf of rank $r$ on $X$ such that
$E_{|C_{yj}} \cong {\cal O}_{C_{yj}}^{\oplus r}$.
Then $E$ is the pull-back of a locally free sheaf on $Y$.
\item[(2)] 
$G^{\vee} \otimes G \cong \pi^*(\pi_*(G^{\vee} \otimes G))$. 
\end{enumerate}
\end{lem}

\begin{proof}
(1)
We consider the map
$\phi:H^0(E_{|Z_y}) \otimes {\cal O}_{Z_y} \to E_{|Z_y}$.
For any point $x \in Z_y$,
we have an exact sequence 
\begin{equation}
0 \to F_x \to {\cal O}_{Z_y} \to {\Bbb C}_x \to 0
\end{equation}
such that ${\bf R}\pi_*(F_x)=0$.
By Lemma \ref{lem:tilting:Rpi*E=0} (2) and
our assumption,
 we have ${\bf R}\pi_*(E \otimes F_x)=0$.
Hence $H^0(E_{|Z_y}) \to H^0(E_{|\{x \}})$ is
isomorphic and $H^1(E_{|Z_y})=0$.
Therefore $\phi$ is a surjective homomorphism
of locally free sheaves of the same rank, which implies that
$\phi$ is an isomorphism.
By $R^1 \pi_*(E)=0$ (Lem. \ref{lem:tilting:TFF} (3))
and the surjectivity of 
$\pi^*(\pi_*(I_{Z_y})) \to I_{Z_y}$,
$R^1 \pi_*(E \otimes I_{Z_y})=0$. Hence
$\pi_*(E) \to \pi_*(E_{|Z_y})$ is surjective. 
Then we can take a homomorphism
${\cal O}_U^{\oplus r} \to \pi_*(E)_{|U}$ in a neighborhood of
$y$ such that ${\cal O}_U^{\oplus r} \to \pi_*(E_{|Z_y})$ is surjective.
Then we have a homomorphism 
$\pi^*({\cal O}_U^{\oplus r}) \to E_{|\pi^{-1}(U)}$
which is surjective on $Z_y$.
Since $\pi$ is proper, replacing $U$ by a small neighborhood of $y$,
we have an isomorphism $\pi^*({\cal O}_U^{\oplus r}) \to E_{|\pi^{-1}(U)}$.
Therefore $E$ is the pull-back of a locally free sheaf on $Y$.
    
(2)
Since $G^{\vee} \otimes  {\cal O}_{C_{yj}}(b_{yj}) $
is a locally free sheaf on $C_{yj}$ with
${\bf R}\pi_*(G^{\vee} \otimes  {\cal O}_{C_{yj}}(b_{yj}))=0$,
we have
$G^{\vee} \otimes  {\cal O}_{C_{yj}}(b_{yj}) 
\cong {\cal O}_{C_{yj}}(-1)^{\oplus \rk G}$.
Hence $G_{|C_{yj}} \cong  
{\cal O}_{C_{yj}}(1)^{\oplus \rk G} \otimes {\cal O}_{C_{yj}}(b_{yj})$. 
Hence $G^{\vee} \otimes G_{|C_{yj}} \cong 
{\cal O}_{C_{yj}}^{\oplus (\rk G)^2}$.
By (1), we get the claim.

\end{proof}

\begin{lem}\label{lem:tilting:A-module}
For $E \in \Coh(X)$, we have
\begin{equation}
\pi^{-1}(\pi_*(G^{\vee} \otimes E)) 
\otimes_{\pi^{-1}({\cal A})}G \otimes_{{\cal O}_X} G^{\vee}
\cong \pi^* \pi_*(G^{\vee} \otimes E). 
\end{equation}
\end{lem}

\begin{proof}
By Lemma \ref{lem:tilting:pull-back}, we get
\begin{equation}
\begin{split}
\pi^{-1}(\pi_*(G^{\vee} \otimes E)) 
\otimes_{\pi^{-1}({\cal A})}G \otimes_{{\cal O}_X} G^{\vee} 
\cong &
\pi^{-1}(\pi_*(G^{\vee} \otimes E)) 
\otimes_{\pi^{-1}({\cal A})}
\pi^{-1}(\pi_*(G \otimes_{{\cal O}_X} G^{\vee}))
\otimes_{\pi^{-1}({\cal O}_Y)}  {\cal O}_X \\
\cong & \pi^{-1}(\pi_*(G^{\vee} \otimes E)) 
\otimes_{\pi^{-1}({\cal O}_Y)}  {\cal O}_X \\
=&\pi^*(\pi_*(G^{\vee} \otimes E)).
\end{split}
\end{equation}
Therefore the claims hold.
\end{proof}

\begin{lem}\label{lem:tilting:A_y}
${\cal A}$-module 
$\pi_*(G^{\vee} \otimes {\Bbb C}_x)$ does not depend on the 
choice of $x \in \pi^{-1}(y)$.
We set
\begin{equation}
A_y:=\pi^{-1}(\pi_*(G^{\vee} \otimes {\Bbb C}_x)) 
\otimes_{\pi^{-1}({\cal A})}G,\; x \in Z_y.
\end{equation}
\end{lem}

\begin{proof}
For the exact sequence
\begin{equation}
0 \to {\cal O}_{C_{yj}}(b_{yj}) \to {\cal O}_{C_{yj}}(b_{yj}+1)
\to {\Bbb C}_x \to 0,
\end{equation}
we have $\pi_*(G^{\vee} \otimes {\cal O}_{C_{yj}}(b_{yj}+1))
\cong \pi_*(G^{\vee} \otimes {\Bbb C}_x)$.
Hence $\pi_*(G^{\vee} \otimes {\Bbb C}_x)$ does not depens on the 
choice of $x \in Z_y$.
\end{proof}

\begin{lem}
\begin{enumerate}
\item[(1)]
$A_y$ is a unique line bundle on $Z_y$ such that
$A_{y|C_{yj}} \cong {\cal O}_{C_{yj}}(b_{yj}+1)$.
\item[(2)]
$G^{\vee} \otimes A_y \cong {\cal O}_{Z_y}^{\oplus \rk G}$.
\end{enumerate}
\end{lem}

\begin{proof}
By Lemma \ref{lem:tilting:A-module},
$G^{\vee} \otimes A_y \cong \pi^*(\pi_*(G^{\vee} \otimes {\Bbb C}_x))
\cong
{\cal O}_{Z_y}^{\oplus \rk G}$.
Thus (2) holds.
Since $G_{|Z_y}$ is a locally free sheaf on $Z_y$, 
$A_y$ is a line bundle on $Z_y$.
Then $A_y^{\otimes \rk G} \cong \det G_{|Z_y}$.
Since the restriction map
$\Pic(Z_y) \to \prod_j \Pic(C_{yj})$ is bijective and
$\Pic(C_{yj}) \cong {\Bbb Z}$,
$G_{|C_{yj}} \cong {\cal O}_{C_{yj}}(b_{yj}+1)^{\oplus \rk G}$
imply the claim (1). 
\end{proof}

\begin{NB}
There is a line bundle $L$ on $Z_y$ such that
$L_{|C_{yj}}={\cal O}_{C_{yj}}(b_{yj}+1)$.
Then $L \otimes {\cal O}_{C_{yj}}(-1)={\cal O}_{C_{yj}}(b_{yj})$.
We shall prove that
$G_{|Z_y} \cong V \otimes L$, where $V$ is a trivial bundle
of rank $\rk G$ on $Z_y$.
Since $\Hom(G_{|Z_y} \otimes L^{\vee},{\cal O}_{C_{yj}}(-1))=0$
for all $j$, 
$\phi:H^0(Z_y,G_{|Z_y} \otimes L^{\vee}) \otimes {\cal O}_{Z_y}
\to G_{|Z_y} \otimes L^{\vee}$ is surjective.
Since $c_1(G_{|Z_y} \otimes L^{\vee})_{|C_{yj}}=0$,
$\chi(Z_y,G_{|Z_y} \otimes L^{\vee})=\rk G$ 
Hence $\phi$ is isomorphic.
Thus
\begin{equation}
A_y=\pi^{-1}(W^{\vee} \otimes{\Bbb C}_y) 
\otimes_{\pi^{-1}(\End(W \otimes {\Bbb C}_y))}
\pi^{-1}(W \otimes {\Bbb C}_y) \otimes L 
\cong \pi^{-1}({\Bbb C}_y) \otimes L.
\end{equation} 
\end{NB}

\begin{lem}\label{lem:rk(G)|chi(G,E)}
For a coherent sheaf $E$ with $\Supp(E) \subset Z_y$,
$\chi(G,E) \in {\Bbb Z}\rk G $.
\end{lem}

\begin{proof}
We note that $K(Z_y)$ is generated by
${\cal O}_{C_{yj}}(b_{yj})$ and ${\Bbb C}_x$.
For $E$ with $\Supp(E) \subset Z_y$,
we have a filtration
$0 \subset F_1 \subset F_2 \subset \cdots \subset F_n=E$
such that $F_i/F_{i-1} \in \Coh(Z_y)$.
Hence the claim follows from
$\chi(G,{\cal O}_{C_{yj}}(b_{yj}))=0$ and $\chi(G,{\Bbb C}_x)=\rk G$. 
\end{proof}

\begin{lem}\label{lem:tilting:chi=rkG}
\begin{enumerate}
\item
[(1)]
Let $E$ be a $G$-twisted stable 1-dimensional sheaf such that
$\Supp(E) \subset Z_y$ and $\chi(G,E)=\rk G$.
Then there is a subscheme $C$ of $Z_y$ such that $\chi({\cal O}_C)=1$
and
$E \cong A_y \otimes {\cal O}_C$.
Conversely for a subscheme $C$ of $Z_y$ such that
${\cal O}_C$ is 1-dimensional,
$\chi({\cal O}_C)=1$,
$E=A_y \otimes {\cal O}_C$ is a $G$-twisted stable
sheaf with $\chi(G,E)=\rk G$ and $\pi(\Supp(E))=\{y \}$.
\begin{NB}
Let $T$ be the 0-dimensional submodule of ${\cal O}_C$.
Then $\chi({\cal O}_C)>\chi({\cal O}_C/T)>0$.
\end{NB}
\item[(2)]
$A_y$ is $G$-twisted stable.
\end{enumerate}
\end{lem}

\begin{proof}
(1)
We choose an exact sequence
\begin{equation}
0 \to K \to E \to {\Bbb C}_x \to 0.
\end{equation} 
Since $E$ is a $G$-twisted stable 1-dimensional sheaf with
$\chi(G,E)=\rk G$, $K$ is a $G$-twisted semi-stable sheaf with
$\chi(G,K)=0$.
If $\pi_*(G^{\vee} \otimes K) \ne 0$, then
we have a non-zero homomorphism
$\phi:\pi^{-1}(\pi_*(G^{\vee} \otimes K)) \otimes_{\pi^{-1}({\cal A})}G 
\to K$ such that $\pi_*(G^{\vee} \otimes \im \phi)
=\pi_*(G^{\vee} \otimes K)$.
Since $R^1 \pi_*(G^{\vee} \otimes \im \phi)=0$,
$\chi(G,\im \phi)>0$, which is a contradiction.
Therefore $\pi_*(G^{\vee} \otimes K)=0$.
Hence 
$\xi:\pi_*(G^{\vee} \otimes E) \to \pi_*(G^{\vee} \otimes {\Bbb C}_x)$
is injective.
Since $\dim H^0(Y,\pi_*(G^{\vee} \otimes E)) \geq \chi(G,E)=\rk G$,
$\xi$ is an isomorphism.  
Then we have a homomorphism $\psi:A_y \to E$.
Since $\pi_*(G^{\vee} \otimes \im \psi)=\pi_*(G^{\vee} \otimes E)$
and $R^1 \pi_*(G^{\vee} \otimes \im \psi)=0$,
we get $\im \psi=E$. 
Since $E \otimes A_y^{D}$,
$A_y^{D}:={\cal H}om(A_y,{\cal O}_{Z_y})$ is a quotient of
${\cal O}_{Z_y}$, there is a subscheme $C$ of $Z_y$
such that $E \otimes A_y^D \cong {\cal O}_C$.
Since $\chi(G,E)=\chi(G,A_y \otimes {\cal O}_C)=
\chi({\cal O}_C^{\oplus \rk G})$,
we have $\chi({\cal O}_C)=1$.

Conversely 
for $E \otimes A_y^{\vee} \cong {\cal O}_C$ such that
${\cal O}_C$ is 1-dimensional,
$C \subset Z_y$ and $\chi({\cal O}_C)=1$,
we consider a quotient $E \to F$.
Then $F=A_y \otimes {\cal O}_D$, $D \subset C$.
Since $R^1 \pi_*(G^{\vee} \otimes F)=0$ and
$G^{\vee} \otimes A_y \otimes {\cal O}_D \cong
{\cal O}_D^{\oplus \rk G}$,
we get $\chi(G,F) \geq 
\rk G$.
From this fact, we first see that
$E$ is purely 1-dimensional, and then we see
that $G$-twisted stable.
\begin{NB}
$E$ is generated by $G^n$, which implies that
$\Hom(G,E) \ne 0$.
\end{NB}

(2) follows from (1) and $\chi({\cal O}_{Z_y})=1$.
%Since 
%${\bf R}\pi_*(G^{\vee} \otimes A_y)=\pi_*(G^{\vee} \otimes {\Bbb C}_x)$,
%we get $\chi(G,A_y)=\rk G$. 
%Hence $A_y$ is stable.
\end{proof}

\begin{lem}\label{lem:tilting:G-stable}
Let $E$ be a $G$-twisted stable purely 1-dimension sheaf such that 
$\pi(\Supp(E)) =\{y \}$ and $\chi(G,E)=0$.
Then $E \cong A_y \otimes {\cal O}_{C_{yj}}(-1)
\cong {\cal O}_{C_{yj}}(b_{yj})$.
\end{lem}

\begin{proof}
We set $n:=\dim X$. We take a point $x \in \Supp(E)$.
Then
${\cal E}xt^1_{{\cal O}_X}({\Bbb C}_x,E)
={\Bbb C}_x \overset{{\bf L}}{\otimes} E[-n+1]$.
Since $E$ is purely 1-dimensional,
$\depth_{{\cal O}_{X,x}}E_x=1$. Hence
the projective dimension of $E$ at $x$ is $n-1$.
Then ${\cal T}or_{n-1}^{{\cal O}_X}({\Bbb C}_x,E)=
H^0({\Bbb C}_x \overset{{\bf L}}{\otimes} E[-n+1]) \ne 0$.
Since $\Ext^1({\Bbb C}_x,E)=
H^0(X,{\cal E}xt^1_{{\cal O}_X}({\Bbb C}_x,E)) \ne 0$,
we can take a non-trivial extension
\begin{equation}
0 \to E \to F \to {\Bbb C}_x \to 0.
\end{equation}
If $F$ is not $G$-twisted semi-stable,
then since $\chi(G,F)=\rk G$, 
there is a quotient $F \to F'$ of $F$ such that
$F'$ is a $G$-twisted stable sheaf with $\chi(G,F') \leq 0$.
Then $E \to F'$ is an isomorphism, which is a contradiction.
By Lemma \ref{lem:tilting:chi=rkG}, $F$ is a quotient of
$A_y$. Thus we may write
$F=A_y \otimes {\cal O}_D$, where $D$ is a subscheme of $Z_y$.
We take an integral curve $C \subset D$ containing $x$.
Since ${\cal O}_D \to {\Bbb C}_x$ factor through ${\cal O}_C$,
we have a surjective homomorphism
$E \to A_y \otimes {\cal O}_C(-1)$. By the stability of $E$,
$E \cong A_y \otimes {\cal O}_C(-1)$.  
\end{proof}

\begin{lem}\label{lem:tilting:G-1Rpi*E=0}
Let $E$ be a 1-dimensional sheaf such that
$\chi(G,E)=0$ and $\pi(\Supp(E))=\{ y\}$.
Then the following conditions are equivalent.
\begin{enumerate}
\item[(1)]
${\bf R}\pi_*(G^{\vee} \otimes E)=0$. 
\item[(2)]
$E$ is a $G$-twisted semi-stable 1-dimensional sheaf
with $\pi(\Supp(E))=\{ y\}$.
\item[(3)]
$E$ is a successive extension of
$A_y \otimes {\cal O}_{C_{yj}}(-1)$, $1 \leq j \leq s_y$.
\end{enumerate}
\end{lem}

\begin{proof}
Lemma \ref{lem:characterize} gives the equivalence of (1) and (2).
The equivalence of (2) and (3) follows from 
Lemma \ref{lem:tilting:G-stable}. 
\end{proof}

\begin{lem}\label{lem:tilting:G-1Per-S}
Let $E$ be a 1-dimensional sheaf such that $\pi_*(G,E)=0$.
Then there is a homomorphism $E \to A_y \otimes {\cal O}_{C_{yj}}(-1)$.
In particular, $E$ is generated by subsheaves of 
$A_y \otimes {\cal O}_{C_{yj}}(-1)$,
$y \in Y$, $1 \leq j \leq s_y$.
\end{lem}

\begin{proof}
Since $\pi(\Supp(E))$ is 0-dimensional,
we have a decomposition
$E=\oplus_i E_i$, $\Supp(E_i) \cap \Supp(E_j)=\emptyset$,
$i \ne j$.
So we may assume that
$\pi(\Supp(E))$ is a point.
We note that $\chi(G,E) \leq 0$.
If $\chi(G,E)=0$, then $\chi(R^1 \pi_*(G^{\vee} \otimes E))=0$.
Since $\dim E=1$ and $\pi_*(G^{\vee} \otimes E)=0$,
we get $\dim \pi(\Supp(E))=0$.
Then we have 
$R^1 \pi_*(G^{\vee} \otimes E)=0$. Hence the claim follows from
Lemma \ref{lem:tilting:G-1Rpi*E=0}.
We assume that $\chi(G,E)<0$. 
Let 
\begin{equation}
0 \subset F_1 \subset F_2 \subset \cdots \subset F_s=E
\end{equation}
be a filtration such that
$E_i:=F_i/F_{i-1}$, $1 \leq i \leq s$ are $G$-twisted stable
and $\chi(G,E_i)/(c_1(E_i),L) \leq \chi(G,E_{i-1})/(c_1(E_{i-1}),L)$,
where $L$ is an ample divisor on $X$.
Since $\pi_*(G^{\vee} \otimes E)=0$ for
any $G$-twisted stable 1-dimensional sheaf $E$ 
on a fiber with $\chi(G,E) \leq 0$,
\begin{NB}
If $\pi_*(G^{\vee} \otimes F) \ne 0$,
then $E:=\im(\pi^{-1}(\pi_*(G^{\vee} \otimes F)) 
\otimes_{\pi^{-1}({\cal A})}G \to F)$ satisfies $\Hom(G,E) \ne 0$
and $R^1 \pi_*(G^{\vee} \otimes E)=0$.
\end{NB}
replacing $E$ by a $G$-twisted stable sheaf $E_s$,
we may assume that $E$ is $G$-twisted stable.
We take a non-trivial extension
\begin{equation}
0 \to E \to F \to {\Bbb C}_x \to 0.
\end{equation}
Then $F$ is purely 1-dimensional and
$\chi(G,F)=\chi(G,E)+\rk G \leq 0$ by
Lemma \ref{lem:rk(G)|chi(G,E)}.
Assume that there is a quotient $F \to F'$ of $F$ such that
$F'$ is a $G$-twisted stable sheaf with
$\chi(G,F')/(c_1(F'),L)<\chi(G,F)/(c_1(F),L) \leq 0$.
Then $\phi:E \to F'$ is surjective over $X \setminus \{ x \}$.
Hence $\chi(G,F')/(c_1(F'),L) \geq 
\chi(G,\im \phi)/(c_1(\im \phi),L) \geq \chi(G,E)/(c_1(E),L)$.
Since $(c_1(F'),L) \leq (c_1(F),L)=(c_1(E),L)$, 
we get $\chi(G,F') \geq \chi(G,E)(c_1(F'),L)/(c_1(E),L) \geq
\chi(G,E)$.
If $\chi(G,F')=\chi(G,E)$, then $\phi$ is an isomorphism.
Since the extension is non-trivial, this is a contradiction.
Therefore $F$ is $G$-twisted semi-stable or $\chi(G,F')>\chi(G,E)$.
Thus we get a homomorphism $\psi:E \to E'$
such that $E'$ is a stable sheaf with $\chi(G,E)<\chi(G,E')<0$
and $\psi$ is surjective in codimension $n-1$. 
By the induction on $\chi(G,E)$, we get the claim.
\end{proof}

\begin{lem}\label{lem:G-1per:characterize}
For a point $y \in Y_\pi$,
let $E$ be a 1-dimensional sheaf on $X$ satisfying the following two
conditions:
\begin{enumerate}
\item
 $\Hom(E,A_y \otimes {\cal O}_{C_{yj}}(-1))=
\Ext^1(E,A_y \otimes {\cal O}_{C_{yj}}(-1))=0$
for all $j$.
\item
There is an exact sequence
\begin{equation}
0 \to F \to E \to {\Bbb C}_x \to 0
\end{equation} 
such that $F$ is a $G$-twisted semi-stable 1-dimensional sheaf with
$\pi(\Supp(F))=\{y \}$, $\chi(G,F)=0$ and $x \in Z_y$.
\end{enumerate}
Then $E \cong A_y$. 
Conversely, $E:=A_y$ satisfies (i) and (ii).
\end{lem}

\begin{proof}
We first prove that $A_y$ satisfies
(i) and (ii).
For the exact sequence
\begin{equation}
0 \to F' \to A_y \to {\Bbb C}_x \to 0, 
\end{equation}
we have ${\bf R}\pi_*(G,F')=0$.
Hence (ii) holds by Lemma \ref{lem:tilting:G-1Rpi*E=0}.
(i) follows from Lemma \ref{lem:tilting:TFF}.
Conversely we assume that $E$ satisfies (i) and (ii).
By (ii), $\pi_*(G^{\vee} \otimes E) \cong 
\pi_*(G^{\vee} \otimes {\Bbb C}_x)$
and $R^1 \pi_*(G^{\vee} \otimes E)=0$.
By (i), Lemma \ref{lem:tilting:C(G)} and
Lemma \ref{lem:tilting:G-1Rpi*E=0},
$\pi^{-1}(\pi_*(G^{\vee} \otimes E)) 
\otimes_{\pi^{-1}({\cal O}_Y)} G \to E$ is surjective.
Hence we have an exact sequence
\begin{equation}
0 \to F' \to A_y \to E \to 0,
\end{equation}
where $F'$ is a $G$-twisted semi-stable 1-dimensional sheaf with 
$\chi(G,F')=0$.
Since $\Ext^1(E,A_y \otimes {\cal O}_{C_{yj}}(-1))=0$
for all $j$,
$A_y  \cong E \oplus F'$, which implies that 
$A_y \cong E$.
\end{proof}

We set
\begin{equation}
E_{yj}:=
\begin{cases}
A_y,&j=0,\\
A_y \otimes {\cal O}_{C_{yj}}(-1)[1],& j>0.\\
\end{cases}
\end{equation}
\begin{prop}(\cite{VB})\label{prop:tilting:G-1Per-irred}
\begin{enumerate}
\item[(1)]
$E_{yj}$, $j=0,...,s_y$ are irreducible objects of ${\cal C}(G)$. 
\item[(2)]
${\Bbb C}_x$, $x \in \pi^{-1}(y)$ is generated by
$E_{yj}$. In particular, irreducible objects of ${\cal C}(G)$
are 
\begin{equation}
{\Bbb C}_x, ( x \in X \setminus \pi^{-1}(Y_\pi)),\;\; 
E_{yj}, (y \in Y_\pi,j=0,1,...,s_y).
\end{equation}  
\end{enumerate}
\end{prop}

\begin{proof}
(1)
Assume that there is an exact sequence in ${\cal C}(G)$:
\begin{equation}
0 \to E_1 \to A_y \to E_2 \to 0.
\end{equation}
Since $H^{-1}(E_1)=0$,
$E_1 \in T$ and $\pi_*(G^{\vee} \otimes E_1) 
\cong \pi_*(G^{\vee} \otimes A_y)={\Bbb C}_y^{\oplus \rk G}$.
\begin{NB}
$\rk G |\chi(G,E_1)$.
\end{NB}
Hence we have a non-zero morphism $A_y \to E_1$.
Since $\Hom(A_y,A_y) \cong {\Bbb C}$,
$E_1 \cong A_y$ and $E_2=0$.
For $A_y \otimes {\cal O}_{C_{yj}}(-1)[1]$, assume that
there is an exact sequence in ${\cal C}(G)$:
\begin{equation}
0 \to E_1 \to A_y \otimes {\cal O}_{C_{yj}}(-1)[1] \to E_2 \to 0.
\end{equation}
Since $H^0(E_2)=0$, we have $E_2[-1] \in S$. Then
Lemma \ref{lem:tilting:G-1Per-S} implies that
we have a non-zero morphism $E_2 \to A_y \otimes {\cal O}_{C_{yj}}(-1)[1]$.
Since $\Hom(A_y \otimes {\cal O}_{C_{yj}}(-1)[1],
A_y \otimes {\cal O}_{C_{yj}}(-1)[1])
={\Bbb C}$, we get $E_1=0$.
Therefore $A_y \otimes {\cal O}_{C_{yj}}(-1)[1]$ is irreducible.
\end{proof}

We give a characterization of $T$.  
\begin{prop}\label{prop:tilting:Per-equiv}
\begin{enumerate}
\item[(1)]
For $E \in \Coh(X)$, the following are equivalent.
\begin{enumerate}
\item
$E \in T$.
\item
$\Hom(E,A_y \otimes {\cal O}_{C_{yj}}(-1))=0$ 
for all $y,j$.
\item
$\phi:\pi^{-1}(\pi_*(G^{\vee} \otimes E)) \otimes_{\pi^{-1}({\cal A})}
G \to E$ is surjective.
\end{enumerate}
\item[(2)]
If (c) holds, then
$\ker \phi \in S_0$.
\end{enumerate}
\end{prop}

\begin{proof}
(1) is a consequence of 
Lemma \ref{lem:tilting:C(G)} and Lemma \ref{lem:tilting:T}.

(2) 
The claim follows from Lemma \ref{lem:tilting:C(G)}.
\begin{NB}
Assume that $\# Y_\pi <\infty$.
By Lemma \ref{lem:tilting:A-module},
we have an exact sequence
\begin{equation}
0 \to G^{\vee} \otimes \ker \phi \to \pi^*(\pi_*(G^{\vee} \otimes E))
\to G^{\vee} \otimes \im \phi \to 0.
\end{equation}
Since $\pi_*(\pi^*(\pi_*(G^{\vee} \otimes E))) \to
\pi_*(G^{\vee} \otimes \im \phi)$ is isomorphic
and $R^1 \pi_*(\pi^*(\pi_*(G^{\vee} \otimes E)))=0$,
${\bf R}\pi_*(G^{\vee} \otimes \ker \phi)=0$.
Therefore the claim follows from Lemma \ref{lem:tilting:G-1Rpi*E=0}.
\end{NB}
\end{proof}

We note that $G \otimes {\cal H}om_{{\cal O}_{Z_y}}(A_y,{\cal O}_{Z_y})
\cong {\cal O}_{Z_y}^{\oplus \rk G}$.
Then we have
${\cal H}om_{{\cal O}_{Z_y}}(A_y,{\cal O}_{Z_y})
\cong \pi^{-1}(\pi_*(G \otimes {\Bbb C}_x)) \otimes_{\pi^{-1}({\cal A})}
G^{\vee}$.
We set
\begin{equation}
E_{yj}^* :=
\begin{cases}
A_y \otimes \omega_{Z_y}[1],&j=0,\\
A_y \otimes {\cal O}_{C_{yj}}(-1),& j>0.\\
\end{cases}
\end{equation}
\begin{NB}
We set
\begin{equation}
E_{yj}' :=
\begin{cases}
{\cal H}om_{{\cal O}_{Z_y}}(A_y,{\cal O}_{Z_y}),&j=0,\\
{\cal H}om_{{\cal O}_{Z_y}}(A_y,{\cal O}_{Z_y}) \otimes 
{\cal O}_{C_{yj}}(-1)[1],& j>0
\end{cases}
\end{equation}
are irreducible objects of ${\cal C}(G^{\vee})$ and
$E_{yj}^*={\bf R}{\cal O}om_{{\cal O}_X}(E_{yj}',K_X)[n]$,
where $n=\dim X$.
\end{NB}
Then we also have the following.
\begin{prop}\cite{VB}\label{prop:tilting:G0-Per-irred}
\begin{enumerate}
\item[(1)]
$E_{yj}^*$, $j=0,...,s_y$ are irreducible objects of ${\cal C}(G)^*$. 
\item[(2)]
${\Bbb C}_x$, $x \in \pi^{-1}(y)$ is generated by
$E_{yj}^*$. In particular, irreducible objects of ${\cal C}(G)^*$
are 
\begin{equation}
{\Bbb C}_x, ( x \in X \setminus \pi^{-1}(Y_\pi)),\;\; 
E_{yj}^*, (y \in Y_\pi,j=0,1,...,s_y).
\end{equation}  
\end{enumerate}
\end{prop}

\begin{lem}\label{lem:G-1per:characterize*}
For a point $y \in Y_\pi$,
let $E$ be a 1-dimensional sheaf on $X$ satisfying the following two
conditions:
\begin{enumerate}
\item
 $\Hom(A_y \otimes {\cal O}_{C_{yj}}(-1),E)=
\Ext^1(A_y \otimes {\cal O}_{C_{yj}}(-1),E)=0$
for all $j$.
\item
There is an exact sequence
\begin{equation}
0 \to E \to F \to {\Bbb C}_x \to 0
\end{equation} 
such that $F$ is a $G$-twisted semi-stable 1-dimensional sheaf with
$\pi(\Supp(F))=\{y \}$, $\chi(G,F)=0$ and $x \in Z_y$.
\end{enumerate}
Then $E \cong A_y \otimes \omega_{Z_y}$. 
\end{lem}

\begin{proof}
We set $n:=\dim X$.
For a purely 1-dimensional sheaf $E$ on $X$,
${\bf R}{\cal H}om_{{\cal O}_X}(E,K_X[n-1]) \in \Coh(X)$ and
${\bf R}{\cal H}om_{{\cal O}_X}(E,K_X[n-1])
={\cal H}om_{{\cal O}_C}(E,\omega_C)$
if $E$ is a locally free sheaf on a curve without embedded primes.
Hence the claim follows from Lemma \ref{lem:G-1per:characterize}. 
\end{proof}

\begin{NB}

\begin{lem}
Assume that $X$ is a smooth surface.
Let $D$ be an effective divisor such that $\pi(D)$ is a point.
Then 
$\Ext^1({\cal O}_D,{\cal O}_D) \cong H^0({\cal O}_X(D))=0$ and
$\dim \Ext^2({\cal O}_D,{\cal O}_D)=-(D,D-K_X)/2$.
\end{lem}

\begin{proof}
By the exact sequence
\begin{equation}
0 \to {\cal O}_X(-D) \to {\cal O}_X \to {\cal O}_D \to 0
\end{equation}
and $H^1({\cal O}_D)=0$,
we get $H^0({\cal O}_D(D)) \cong \Ext^1({\cal O}_D,{\cal O}_D)$.
By the exact sequence
\begin{equation}
0 \to {\cal O}_X \to {\cal O}_X(D) \to {\cal O}_D(D) \to 0
\end{equation}
 and $R^1 \pi_*({\cal O}_X)=0$,
we have an exact sequence
\begin{equation}
0 \to {\cal O}_Y \to \pi_*({\cal O}_X(D)) \to \pi_*({\cal O}_D(D)) \to 0.
\end{equation}
Let $j:Y \setminus \pi_*(D) \to Y$ be the 
inclusion. 
Since ${\cal O}_Y \cong j_*(j^*({\cal O}_Y))$ and 
$\pi_*({\cal O}_X(D))$ is torsion free,
${\cal O}_Y \to \pi_*({\cal O}_X(D))$ is an isomorphism.
Hence $\pi_*({\cal O}_D(D))=0$.
Therefore $H^0({\cal O}_D(D))=0$.

By the Serre duality,
$\Ext^2({\cal O}_D,{\cal O}_D) \cong H^0(X,{\cal O}_D(K_X))^{\vee}$.
Since $K_{X|D}=K_D(-D)$, 
$H^0(X,{\cal O}_D(K_X))=H^1(D,{\cal O}_D(D))$.
Then $\dim H^1({\cal O}_D(D))=-\chi({\cal O}_D(D))=-(D,D-K_X)/2$,
which implies the claim.
\end{proof}

Since $H^1({\cal O}_D)=0$,
$1 \leq \chi({\cal O}_D)=-(D,D+K_X)/2$.

\end{NB}

\subsection{Families of perverse coherent sheaves. }
\label{subsect:family}

We shall explain families of complexes which correspond to
families of ${\cal A}$-modules via Morita equivalence.
Let $f:X \to S$ and $g:Y \to S$ be flat families of projective varieties 
parametrized by a scheme $S$
and $\pi:X \to Y$ an $S$-morphism.
Let ${\cal O}_Y(1)$ be a relatively ample line bundle over
$Y \to S$.
We assume that 
\begin{enumerate}
\item
$X \to S$ is a smooth family, 
\item
there is a locally free sheaf $G$ on $X$ such that
$G_s:=G_{|f^{-1}(s)}$, $s \in S$ are local projective generators of 
a family of abelian categories ${\cal C}_s \subset {\bf D}(X_s)$ and
\item
$\dim \pi^{-1}(y) \leq 1$ for all $y \in Y$,
i.e.,
$\pi$ satisfies Assumption \ref{ass:1}.
\end{enumerate}
Then ${\cal C}_s$ is a tilting of $\Coh(X_s)$.
\begin{rem}
(i), (ii) and (iii) imply that
\begin{enumerate}
\item[(iv)]
$R^1 \pi_*(G^{\vee} \otimes G)=0$.
\item[(v)]
\begin{equation}
\{E \in \Coh(X)|{\bf R}\pi_*(G^{\vee} \otimes E)=0 \}=0.
\end{equation}
Thus $G$ defines a tilting ${\cal C}$ of $\Coh(X)$.
\end{enumerate}
Indeed if $E \in \Coh(X)$ satisfies
${\bf R}\pi_*(G^{\vee} \otimes E)=0$, then
the projection formula implies that
${\bf R}\pi_*(G^{\vee} \otimes E 
\overset{{\bf L}}{\otimes}{\bf L}f^*({\Bbb C}_s))
={\bf R}\pi_*(G^{\vee} \otimes E) 
\overset{{\bf L}}{\otimes}{\bf L}g^*({\Bbb C}_s)=0$
for all $s \in S$.
Then ${\bf R}\pi_*(G^{\vee} \otimes 
H^p(E\overset{{\bf L}}{\otimes}{\bf L}f^*({\Bbb C}_s)))=0$
for all $p$ and $s \in S$.
%\begin{NB}
%$H^0(E\overset{{\bf L}}{\otimes}{\bf L}f^*({\Bbb C}_s))=0$ implies that
%$E=0$ in a neighborhood of $\pi^{-1}(s)$.
%\end{NB}
By (ii), $H^p(E\overset{{\bf L}}{\otimes} {\bf L}f^*({\Bbb C}_s))=0$ 
for all $p$ and $s \in S$.
\begin{NB}
In particular $E \otimes {\Bbb C}_s=
H^0(E\overset{{\bf L}}{\otimes} {\bf L}f^*({\Bbb C}_s))=0$.
By Nakayama's lemma, $E=0$ in a neighborhood of $\pi^{-1}(s)$. 
\end{NB}
Therefore (v) holds. (iv) is obvious. Conversely if (i), (iii), 
(iv) and (v) hold, then (ii) holds. So we may replace (ii) by (iv) and (v).
\end{rem}

For a morphism $T \to S$, 
we set $X_T:=X \times_S T$, $Y_T:=Y \times_S T$ and 
$\pi_T:=\pi \times \id_T$. 
\begin{NB}
\begin{lem}
Let $\Coh(X_T)'$ be the subcategory of $\Coh(X_T)$ consisting of
objects $E$ of finite torsion dimension over $T$.
For $G_T:=G \otimes_{{\cal O}_S}{\cal O}_T$, we set 
\begin{equation}
\begin{split}
{\cal T}_T:=&\{E \in \Coh(X_T)'| 
R^1 \pi_{T*}(G_T^{\vee} \otimes E)=0 \},\\
{\cal S}_T:=&\{E \in \Coh(X_T)'| 
\pi_{T*}(G_T^{\vee} \otimes E)=0 \}.
\end{split}
\end{equation}
Then ${\cal T}_T \cap {\cal S}_T=0$.
%We denote the tilting of $\Coh(X_T)$
%by ${\cal C}_T$.
\end{lem}

\begin{proof} 
Since $R^1 \pi_{t*}((G^{\vee} \otimes G)_t)=0$ for all $t \in T$,
$R^1 \pi_{T*}(G_T^{\vee} \otimes G_T)=0$.
Assume that ${\bf R}\pi_{T*}( G_T^{\vee} \otimes E)=0$
for $E \in {\bf D}(X_T)$.
Then ${\bf R}\pi_{ t*}( (G_T^{\vee} \otimes E)_t)=0$ for all $t$.
By our assumption,
we get $E_t =0$ for all $t \in T$, which implies that
$E= 0$.
\end{proof}
If $T$ is smooth, then $\Coh(X_T)'=\Coh(X_T)$ and
$({\cal T}_T,{\cal S}_T)$ is a torsion pair 
of $\Coh(X_T)$.
Let ${\cal C}_T$ be the tilting of 
$\Coh(X_T)$.
We have
\begin{equation}
{\cal C}_T=\{E \in {\bf D}(X_T)|
 {\bf R}\pi_{T*}(G_T^{\vee} \otimes E) \in \Coh(Y_T) \}.
\end{equation}
%For simplicity, we denote ${\cal C}_S$ by ${\cal C}$.
%For a complex
%$E$ on $X$,
%$S^0:=\{s \in S| E_s \in {\cal C}_s \}$ is an open subscheme
%of $S$ and $E \in {\cal C}$ if and only if $S^0=S$. 

For a point $y \in Y$, there is an affine open neighborhood
$U$ of $Y$ such that 
$\pi^{-1}(U)=U_1 \cup U_2$,
where $U_1, U_2$ are affine open subsets of $\pi^{-1}(U)$:
   
We may assume that $\pi^{-1}(U)$ is a subscheme of
$U \times {\Bbb P}^n$.
There are hypersurfaces $f_1,f_2$ such that
$\{x \in U| \pi(x)=y, f_1(x)=f_2(x)=0 \}=\emptyset$.
Then $\pi(\{x \in U|f_1(x)=f_2(x)=0 \})$ is a closed subset of $U$
which does not contain $y$.
Replacing $U$ by a small neighborhood, we may assume that
$\{x \in U|f_1(x)=f_2(x)=0 \}=\emptyset$.
We set $U_1=\{x \in U|f_1(x) \ne 0 \}$ and
$U_2=\{x \in U|f_2(x) \ne 0 \}$. Then $U_1,U_2$ satisfy the desired 
properties.

We can compute ${\bf R}\pi_{T*}(G_T^{\vee} \otimes E)_{|U}$
by using the resolution 
\begin{equation}
0 \to E \to E_{|U_1} \oplus E_{|U_2} \to E_{|U_1 \cap U_2} \to 0.
\end{equation}
of $E$.

We assume that ${\bf R}\pi_{T*}(G_T^{\vee} \otimes E)=0$.
By taking a locally free resolution of $E$,
we see that 
${\bf R}\pi_{t*}(G_t^{\vee} \otimes E
\overset{{\bf L}}{\otimes}{\Bbb C}(t) )_{|U_t}
\cong {\bf R}\pi_{T*}(G_T^{\vee} \otimes E)_{|U} 
\overset{{\bf L}}{\otimes} {\Bbb C}(t)=0$ for all $t \in g(U)$.
By the spectral sequence, we get
${\bf R}\pi_{t*}(G_T^{\vee} \otimes H^i(E
\overset{{\bf L}}{\otimes}{\Bbb C}(t)) )=0$
for all $i$ and $t \in T$.
Hence $H^i(E
\overset{{\bf L}}{\otimes}{\Bbb C}(t))=0$ for all $i$ and $t \in T$.
Therefore $E \overset{{\bf L}}{\otimes}{\Bbb C}(t)=0$,
which implies that $E=0$.

As a representative of $E \in {\cal C}_S$,
we use the following type of complexes.
\end{NB}

\begin{defn}
\begin{enumerate}
\item[(1)]
A family of objects in ${\cal C}_s, s \in S$ means 
a bounded complex
$F^{\bullet}$ of coherent sheaves on $X$ such that 
$F^i$ are flat over $S$ and $F_s^{\bullet} \in {\cal C}_s$ for all $s \in S$.
\item[(2)]
A family of local projective generators is a locally free sheaf $G$ on
$X$ such that 
$G_s:=G_{|f^{-1}(s)}$, $s \in S$ are local projective generators of 
a family of abelian categories ${\cal C}_s$. 
\end{enumerate}
\end{defn}

\begin{rem}
If $F^{\bullet}_s \in \Coh(X_s)$
for all $s \in S$, then $F^{\bullet}$ is isomorphic
to a coherent sheaf on $X$ which is flat over $S$. 
\end{rem}

\begin{lem}\label{lem:good-family}
For a family $F^{\bullet}$ of objects in ${\cal C}_s$, $s \in S$,
there is a complex $\widetilde{F}^{\bullet}$
such that 
(i) $\widetilde{F}^{i}_s \in {\cal C}_s$, $s \in S$,
(ii) $\widetilde{F}^i$ are flat over $S$, and
(iii)$F^{\bullet} \cong \widetilde{F}^{\bullet}$.  
\end{lem}

\begin{proof}
We set $d:=\dim X_s, s \in S$.
For the bounded complex $F^{\bullet}$, 
we take a locally free resolution of ${\cal O}_X$
\begin{equation}
0 \to V_{-d} \to \cdots \to  V_{-1} \to V_0 \to {\cal O}_X \to 0
\end{equation}
such that 
$R^k \pi_*((G^{\vee} \otimes V_i^{\vee} \otimes F^{j})_s)=0$, $k>0$
%$\Ext^k((G(-n),V_i^{\vee} \otimes F^{\bullet})_s)=0$, $k>0$
for $0 \leq i \leq d-1$ and all $j$.
Since $X \to Y$ is projective, we can take such a resolution. 
Then
$R^k \pi_*((G^{\vee} \otimes V_{-d}^{\vee} \otimes F^{j})_s)=0$, $k>0$
for all $j$.
Therefore we have an isomorphism
$F^{\bullet} \cong V_{\bullet}^{\vee} \otimes F^{\bullet}$ such that  
$(V_{\bullet}^{\vee} \otimes F^{\bullet})^i$ are $S$-flat
and $(V_{\bullet}^{\vee} \otimes F^{\bullet})_s^i=
\oplus_{p+q=i} V_{-p}^{\vee} \otimes F_s^q \in {\cal C}_s$
for all $s \in S$.
\end{proof}

%We set 
%$W^{\bullet}:= 
%\Hom_{p}(G(-n),
%V_{\bullet}^{\vee} \otimes F^{\bullet})$, where 
%$p:X \to S$ be the projection.
%Then we have a morphism 
%$G \otimes p^*(W^{\bullet}) \to 
%(V_{\bullet}^{\vee} \otimes F^{\bullet})$.
%Moreover if we assume that 
%$\Hom(G_s(-n),F^{\bullet}_s[i])=0$ for $i \ne 0$ and all $s \in S$.
%Then the base change theorem implies that
%$U:=\Hom_{p}(G(-n),F^{\bullet})$ is a locally free 
%sheaf on $S$ and  
%$\Hom_{p}(G(-n),F^{\bullet})_s \cong
%\Hom(G(-n)_s,F^{\bullet}_s)$.
%Hence $G(-n) \otimes p^*(W^{\bullet}) \cong 
%G(-n) \otimes p^*(U)$. 
%In particular, if $F^{\bullet}_s \in {\cal C}_s$ for all
%$s \in S$,
%then we have a family of exact sequences
%\begin{equation}
%0 \to E^{\bullet} \to G(-n) \otimes p^*(U) \to F^{\bullet} \to 0  
%\end{equation}
%in ${\cal C}_s$, $s \in S$.
%Assume that ${\cal C}_s$ is a tilting of $\Coh(X_s)$
%by a torsion pair.
%Since $G \in \Coh(X)$,
%we have $E^{\bullet} \in \Coh(X)$ which is
%flat over $S$.

\begin{prop}\label{prop:Morita-family}
\begin{enumerate}
\item[(1)]
Let $F^{\bullet}$ be a family of objects in ${\cal C}_s$, $s \in S$.
Then we get
\begin{equation}
F^{\bullet} \cong \mathrm{Cone}(E_1 \to E_2),
\end{equation}
where $E_i \in \Coh(X)$ are flat over $S$
and $(E_i)_s \in {\cal C}_s$, $s \in S$.
\item[(2)]
Let $F^{\bullet}$ be a family of objects in ${\cal C}_s$, $s \in S$.
Then we have a complex
\begin{equation}
G(-n_1) \otimes f^*(U_1) \to G(-n_2) \otimes f^*(U_2)
\to F^{\bullet} \to 0
\end{equation}
whose restriction to $s \in S$ is exact
in ${\cal C}_s$, where $U_1,U_2$ are locally free sheaves on $S$.
\item[(3)]
Let $F$ be an ${\cal A}$-module flat over $S$.
Then we can attach a family $E$
of objects in ${\cal C}_s$, $s \in S$ 
 such that ${\bf R}\pi_*(G^{\vee} \otimes E)=F$.
The correspondence is functorial and $E$ is unique in ${\bf D}(X)$.
We denote $E$ by $\pi^{-1}(F) \otimes_{\pi^{-1}({\cal A})} G$.
%$E_s=\pi^{-1}(F_s) \otimes_{\pi^{-1}({\cal A}_s)} G_s$.
\end{enumerate}
\end{prop}

\begin{proof}
(1)
We may assume that (i), (ii), (iii) in Lemma \ref{lem:good-family}
hold for $F^{\bullet}$.  
We take a sufficiently large $n$ with
$\Hom_{f}(G(-n),F^j[i])=0$, $i>0$ for all $j$.
Then $W^j:=\Hom_{f}(G(-n),F^j)$ are locally free sheaves.
Let
$W^{\bullet}:= 
{\bf R}\Hom_{f}(G(-n),F^{\bullet})$ be the complex defined by 
$W^j, j \in {\Bbb Z}$.
Then we have a morphism 
$G(-n) \otimes f^*(W^{\bullet}) \to {F}^{\bullet}$.
Since $F^{\bullet}_s \in {\cal C}_s$, $s \in S$,
$\Hom(G_s(-n),F^{\bullet}_s[i])=0$ for $i \ne 0$ and all $s \in S$.
Then the base change theorem implies that
$U:=\Hom_{f}(G(-n),F^{\bullet})$ is a locally free 
sheaf on $S$ and  
$\Hom_{f}(G(-n),F^{\bullet})_s \cong
\Hom(G(-n)_s,F^{\bullet}_s)$.
Hence $G(-n) \otimes f^*(W^{\bullet}) \cong 
G(-n) \otimes f^*(U)$, which defines
a family of morphisms 
\begin{equation}
G(-n) \otimes f^*(U) \to F^{\bullet}. 
\end{equation}
Since $F^{\bullet}_s \in {\cal C}_s$ for all
$s \in S$, 
${\bf R}\pi_*(G^{\vee} \otimes F^{\bullet})$ 
is a coherent sheaf on $Y$ which is flat over $S$, and
$g^* g_*(\pi_*(G^{\vee} \otimes F^{\bullet})(n)) \to
 \pi_*(G^{\vee} \otimes F^{\bullet})(n)$ is surjective in $\Coh(Y)$
for $n \gg 0$.
Since $W^{\bullet} \cong g_*(\pi_*(G^{\vee} \otimes F^{\bullet})(n))$,
%$\pi_*(G^{\vee} \otimes F^{\bullet})_s \to 
%\pi_*((G^{\vee} \otimes F^{\bullet})_s)$ is isomorphic,
the homomorphism
\begin{equation}
\pi_*(G^{\vee} \otimes G)(-n) \otimes g^*(U) \to 
\pi_*(G^{\vee} \otimes F^{\bullet}) 
\end{equation}
in $\Coh(Y)$ is surjective for $n \gg 0$.
Thus we have a family of exact sequences
\begin{equation}
0 \to E^{\bullet} \to G(-n) \otimes f^*(U) \to F^{\bullet} \to 0  
\end{equation}
in ${\cal C}_s$, $s \in S$.
%Assume that ${\cal C}_s$ is a tilting of $\Coh(X_s)$
%by a torsion pair.
Since $G \in \Coh(X)$,
we have $E^{\bullet} \in \Coh(X)$ which is
flat over $S$.
(2) is a consequence of (1).

(3)
We take a resolution of $F$ 
\begin{equation}\label{eq:F-resolution}
\cdots \overset{d^{-3}}{\to} 
g^*(U_{-2}) \otimes {\cal A}(-n_2)
\overset{d^{-2}}{\to} 
g^*(U_{-1}) \otimes {\cal A}(-n_1)
\overset{d^{-1}}{\to} 
g^*(U_{0}) \otimes {\cal A}(-n_0)
\to F \to 0,
\end{equation}
where $U_i$ are locally free sheaves on $S$.
Then we have a complex 
\begin{equation}\label{eq:E-resolution}
\cdots \overset{\tilde{d}^{-3}}{\to}
f^*(U_{-2}) \otimes G(-n_2)
\overset{\tilde{d}^{-2}}{\to} 
f^*(U_{-1}) \otimes G(-n_1)
\overset{\tilde{d}^{-1}}{\to} 
f^*(U_{0}) \otimes G(-n_0).
\end{equation}
%\begin{equation}\label{eq:F-resolution}
%\cdots \overset{d^{-3}}{\to} 
%V_{-2} \otimes_{{\cal O}_Y} {\cal A}
%\overset{d^{-2}}{\to}  V_{-1} \otimes_{{\cal O}_Y} {\cal A}
%\overset{d^{-1}}{\to} V_{0} \otimes_{{\cal O}_Y} {\cal A}
%\to F \to 0,
%\end{equation}
%where $V_i$ are locally free sheaves on $Y$.
%Then we have a complex 
%\begin{equation}\label{eq:E-resolution}
%\cdots \overset{\tilde{d}^{-3}}{\to}\pi^*(V_{-2}) \otimes_{{\cal O}_X} G
%\overset{\tilde{d}^{-2}}{\to} \pi^*(V_{-1}) \otimes_{{\cal O}_X} G
%\overset{\tilde{d}^{-1}}{\to} \pi^*(V_{0}) \otimes_{{\cal O}_X} G.
%\end{equation}
%
By the Morita equivalence (Proposition \ref{prop:Morita}),
we have $\im \tilde{d}_s^{-i}=\ker \tilde{d}_s^{-i+1}$
in ${\cal C}_s$ for all $s \in S$.
Let $\coker \tilde{d}^{-2}$ 
be the cokernel of $\tilde{d}^{-2}$ in $\Coh(X)$.
Then by Lemma \ref{lem:flatness} below,
$\coker \tilde{d}^{-2}$ is flat over $S$, 
$(\coker \tilde{d}^{-2})_s=\coker (\tilde{d}^{-2}_s) \in {\cal C}_s$
and 
\begin{equation}
E:=\mathrm{Cone}(\coker \tilde{d}^{-2}
 \to f^*(U_{0}) \otimes G(-n_0))
\end{equation}
is a family of objects in ${\cal C}_s$.
By the construction, we have
$E_s =\pi^{-1}(F_s) \otimes_{\pi^{-1}({\cal A}_s)} G_s$.
It is easy to see the class of $E$ in ${\bf D}(X)$ does not depend
on the choice of the resolution \eqref{eq:F-resolution}
(cf. \cite[Lem. 14]{B-S:1}).
\begin{NB}
Let $V^{\bullet}$ be a bounded complex of locally free sheaves
on $Y$.
If $H^i(V^{\bullet}) =0$ for $i \ne 0$, then
replacing $V^{\bullet}$ by another complex, we may assume that
$V^i=0$ for $i>0$.
Moreover if $H^0(V^{\bullet})$ is flat over $S$,
then $H^i(V^{\bullet})_s=H^i(V^{\bullet}_s)$. 

If $X$ is smooth over ${\Bbb C}$,
we can represent $E$ as a finite complex of locally free sheaves on $X$.
Then we can represent ${\bf R}\pi_*(G^{\vee} \otimes E)$ as 
a bounded complex of locally free sheaves.
Since ${\bf R}\pi_*(G^{\vee} \otimes E)=F$ is flat over
$S$, ${\bf R}\pi_{s*}(G_s^{\vee} \otimes E_s)=F_s$ for all $s \in S$.
Thus $E$ is a family of objects in ${\cal C}_s$.
\end{NB}
\end{proof}

\begin{lem}\label{lem:flatness}
Let $E^i$, $0 \leq i \leq 3$ be coherent sheaves on $X$ which
are flat over $S$. 
Let 
\begin{equation}
E^0 \overset{d^0}{\to} E^1 \overset{d^1}{\to} E^2 \overset{d^2}{\to} E^3
\end{equation}
be a complex in $\Coh(X)$.
\begin{enumerate}
\item[(1)]
If $\ker d^1_s=\im d^0_s$ in $\Coh(X_s)$, 
then $(\im d^1)_s \to E^2_s$ is injective.
In particular if $\ker d^1_s=\im d^0_s$ in $\Coh(X_s)$
for all $s \in S$, then 
$\coker d^1, \im d^1, \ker d^1$ in $\Coh(X)$ are flat over $S$ and
$\im d^0=\ker d^1$.
\item[(2)] 
Assume that $E^i_s \in {\cal C}_s$ for all $s \in S$.
We denote the kernel, cokernel and the image of $d^i_s$ 
in ${\cal C}_s$ by
$\ker_{{\cal C}_s} d^i_s, \coker_{{\cal C}_s} d^i_s$ and 
$\im_{{\cal C}_s} d^i_s$ respectively.
If $E^i_s \in {\cal C}_s$ and
$\ker_{{\cal C}_s} d_s^i=\im_{{\cal C}_s} d^{i-1}_s$, 
$i=1,2$ in ${\cal C}_s$ for all $s$,
then $\im_{{\cal C}_s} d^{i-1}_s$ coincide with the image of $d^{i-1}_s$
in $\Coh(X_s)$ for $i=1,2$ and
$\ker_{{\cal C}_s} d_s^1$ coincides with the kernel of $d_s^1$ in $\Coh(X_s)$.
In particular, 
$\overline{E}^{\bullet}:
E^2/d^1(E^1) \to E^3$ is a family of objects in ${\cal C}_s$
and we get an exact triangle: 
\begin{equation}
\ker d^0 \to E^{\bullet} \to \overline{E}^{\bullet} \to \ker d^0[1]
\end{equation}
where $\ker d^0$ is the kernel of $d^0$ 
in $\Coh(X)$, which is flat over $S$.
\end{enumerate}
\end{lem}

\begin{proof}
(1)
Let $K$ be the kernel of $\xi:(\im d^1)_s \to E^2_s$.
Then we have an exact sequence
\begin{equation}
(\ker d^1)_s \to \ker(d_s^1) \to K \to 0.
\end{equation} 
Since the image of 
$E^0_s \to (\ker d^1)_s \to E^1_s$ is $d_s^0(E^0_s)=\ker(d_s^1)$,
$K=0$.  
The other claims are easily follows from this.

(2)
By our assumption,
$\im_{{\cal C}_s} d^i_s=\coker_{{\cal C}_s} d^{i-1}_s$ for $i=1,2$.
Since $\im_{{\cal C}_s} d^i_s$ is a subobject of $E^{i+1}_s$ for $i=0,1,2$,
$\im_{{\cal C}_s} d^i_s \in \Coh(X_s)$ for $i=0,1,2$ and 
$H^{-1}(\coker_{{\cal C}_s} d^{i-1}_s)=
H^{-1}(\im_{{\cal C}_s} d^i_s)=0$ for $i=1,2$.
Then $H^0(\im_{{\cal C}_s} d^{i-1}_s) \to H^0(E^i_s)$ is injective
for $i=1,2$, 
which implies that
$\im_{{\cal C}_s} d^{i-1}_s$ is the image of $d^{i-1}_s$ in $\Coh(X_s)$
for $i=1,2$.
\begin{NB}
$H^0(E') \to H^0(\im_{{\cal C}_s} d^1) \to 0$ and
$ H^0(\im_{{\cal C}_s} d^1)  \hookrightarrow H^0(E^2)$.
\end{NB}
By the exact sequence
\begin{equation}
0 \to H^0(\ker_{{\cal C}_s} d_s^1) \to H^0(E_s^1) 
\to H^0(\im_{{\cal C}_s} d_s^1) \to 0
\end{equation}
and the injectivity of $H^0(\im_{{\cal C}_s} d_s^1) \to H^0(E^2_s)$,
$\ker_{{\cal C}_s} d_s^1$ is the kernel of $d_s^1$ in $\Coh(X_s)$.
Then the other claims follow from (1). 
\end{proof}

\subsubsection{Quot-schemes}

%Let $X \to S$ and $Y \to S$ be flat family of projective varieties 
%parametrized by a scheme $S$
%and $\pi:X \to Y$ a $S$-morphism.
\begin{lem}
Let ${\cal A}$ be an ${\cal O}_{Y}$-algebras 
on $Y$ which is flat over $S$. 
Let $B$ be a coherent ${\cal A}$-module on $Y$
which is flat over $S$.
There is a closed subscheme $\Quot_{B/Y/S}^{{\cal A},P}$
of $Q:=\Quot_{B/Y/S}^P$ 
parametrizing all quotient
${\cal A}_s$-modules $F$ of $B_s$ with
$\chi(F(n))=P(n)$.
\end{lem}

\begin{proof}
Let ${\cal Q}$ and ${\cal K}$ be the universal quotient and 
the universal subsheaf of $B \otimes_{{\cal O}_S} {\cal O}_Q$:
\begin{equation}
0 \to {\cal K} \to B \otimes_{{\cal O}_S} {\cal O}_Q 
\to {\cal Q} \to 0.
\end{equation}
Then we have a homomorphism
\begin{equation}
{\cal K} \otimes_{{\cal O}_S} {\cal A} \to 
B \otimes_{{\cal O}_S} {\cal O}_Q \otimes_{{\cal O}_S} {\cal A}
\to B \otimes_{{\cal O}_S} {\cal O}_Q \to {\cal Q}
\end{equation}
induced by the multiplication map
$B \otimes_{{\cal O}_S} {\cal O}_Q \otimes_{{\cal O}_S} {\cal A}
\to B \otimes_{{\cal O}_S} {\cal O}_Q$.
Let $Z=\Quot_{B/Y/S}^{{\cal A},P}$ be the zero locus of this homomorphism.
Then for an $S$-morphism $T \to Q$,
${\cal K} \otimes_{{\cal O}_S}{\cal O}_T$ is an 
${\cal A} \otimes_{{\cal O}_S} {\cal O}_T$-submodule
of $B \otimes_{{\cal O}_S}{\cal O}_T$ 
if and only if $T \to Q$ factors through $Z$.
\end{proof}

\begin{cor}
Let $G'$ be a family of objects in ${\cal C}_s$, $s \in S$. 
Then there is a quot-scheme $\Quot_{G'/X/S}^{{\cal C},P}$
parametrizing all quotients $G_s' \to E$ in ${\cal C}_s$, 
where $P$ is the 
$G_s$-twisted Hilbert-polynomial
of the quotient $G_s \to E, s \in S$.
\end{cor}

\begin{proof}
We set ${\cal A}:=\pi_*(G^{\vee} \otimes_{{\cal O}_{X}} G)$.
Then ${\cal A}$ is a flat family of ${\cal O}_{Y}$-algebras on $Y$ and 
we have an equivalence
between the category of ${\cal A}_T$-modules $F$ flat
over $T$ and the categoty of
families $E$ of objects in ${\cal C}_t, t \in T$
by $F \mapsto \pi_T^{-1}(F) \otimes_{\pi^{-1}({\cal A}_T)} G_T$.
So the claim holds. 
\end{proof}

\subsection{Stability for perverse coherent sheaves.}\label{subsect:stability}

For a non-zero object $E \in {\cal C}_s$, 
$\chi(G_s,E(n))=\chi({\bf R}\pi_*(G_s^{\vee} \otimes E)(n))>0$
for 
$n \gg 0$ and
there are integers $a_i(E)$
such that 
\begin{equation}
\chi(G_s,E(n))=\sum_i a_i(E) \binom{n+i}{i}.
\end{equation}

\begin{defn}[Simpson]\label{defn:Simpson-stability}
Assume that ${\cal C}_s$ is a tilting of $\Coh(X_s)$ for all $s \in S$.
\begin{enumerate}
\item[(1)]
An object $E \in {\cal C}_s$ is $d$-dimensional, if
$a_d(E)>0$ and $a_i(E)=0$, $i>d$. 
\item[(2)]
An object $E \in {\cal C}_s$ of dimension $d$ is $G_s$-twisted semi-stable
if 
\begin{equation}
\chi(G_s,F(n))  \leq \frac{a_d(F)}{a_d(E)}\chi(G_s,E(n)), n \gg 0
\end{equation}
for all proper subobject $F$ of $E$. 
\end{enumerate}
\end{defn}

\begin{rem}
\begin{enumerate}
\item[(1)]
If $\dim E> \dim \pi(Z_s)$ and $E$ is $G_s$-twisted semi-stable,
then $H^{-1}(E)=0$.
Indeed $H^{-1}(E)[1]$ is a subobject of $E$ with
\begin{equation}
\deg\chi(G_s,H^{-1}(E)(n)) \leq \dim \pi(Z_s)<
\deg \chi(G_s,E(n)).
\end{equation}
\item[(2)]
Assume that $E \in \Coh(X_s) \cap {\cal C}_s$.
For an exact sequence
\begin{equation}
0 \to F \to E \to F' \to 0
\end{equation}
in ${\cal C}_s$,
we have an exact sequence in $\Coh(X_s)$
\begin{equation}
H^{-1}(F') \overset{\varphi}{\to} H^0(F) \to H^0(E) \to H^0(F') \to 0.
\end{equation}
Since $\chi(G_s,H^0(F)(n)) \leq \chi(G_s,(\coker \varphi)(n))$,
in order to check the semi-stability of $E$,
we may assume that $H^{-1}(F')=0$.
\end{enumerate}
\end{rem}

\begin{prop}
There is a coarse moduli scheme $\overline{M}_{X/S}^{{\cal C},P} \to S$
of $G_s$-twisted semi-stable objects
$E \in {\cal C}_s$ with the $G_s$-twisted Hilbert polynomial $P$.
$\overline{M}_{X/S}^{{\cal C},P}$ is a projective scheme over $S$.
\end{prop}

\begin{proof}
The claim is due to Simpson \cite[Thm. 4.7]{S:1}.
We set
${\cal A}:=\pi_*(G^{\vee} \otimes G)$.
If we set $\Lambda_0={\cal O}_Y$
and $\Lambda_k={\cal A}$ for $k \geq 1$,
then a sheaf of ${\cal A}$-module is an example of
$\Lambda$-modules
in \cite{S:1}.
Let $Q^{ss}$ be an open subscheme
of $\Quot_{{\cal A}(-n) \otimes V/Y/S}^{{\cal A},P}$
consisting of semi-stable ${\cal A}_s$-modules on $Y_s$, $s \in S$. 
Then we have the moduli space 
$\overline{M}_{Y/S}^{{\cal A},P} \to S$
of semi-stable ${\cal A}_s$-modules on $Y_s$
as a GIT-quotient $Q^{ss} \dslash GL(V)$,
where we use a natural polarization on the embedding of
the quot-scheme into the Grassmannian.
By a standard argument due to Langton,
we see that $\overline{M}_{Y/S}^{{\cal A},P}$ is 
projective over $S$. 
Since
the semi-stable
${\cal A}_s$-modules correspond to
$G_s$-twisted semi-stable objects via the Morita
equivalence (Proposition \ref{prop:Morita-family}),
we get the moduli space $\overline{M}_{X/S}^{{\cal C},P} \to S$,
which is
projective over $S$.
\end{proof}

\begin{NB}
Let $G$ be a locally free sheaf on $X$ satisfying Assumption \ref{ass:2}.
Let ${\cal C}(G)$ and ${\cal C}(G)^*$ be the tiltings in subsection
\ref{subsect:perverse-examples}.
Let $G^+$ and $G^-$ be local projective generators
of ${\cal C}(G)$ and ${\cal C}(G)^*$
such that $c_1(G^{\pm})/\rk G^{\pm}$ are sufficiently
close to $c_1(G)/\rk G$.
Then Proposition \ref{prop:tilting:contraction}
implies the morphims
${\cal M}_{X/S}^{{\cal C}(G),P} \to {\cal M}_{Y/S}^{{\cal A}_0}$
and
${\cal M}_{X/S}^{{\cal C}(G)^*,P} \to {\cal M}_{Y/S}^{{\cal A}_0}$,
where ${\cal A}_0:=\pi_*(G^{\vee} \otimes G)$.
\end{NB}

We consider a natural relative polarization on 
$\overline{M}_{X/S}^{{\cal C},P}$.
Let $Q^{ss}$ be the open subscheme of
$\Quot_{G(-n) \otimes V/X/S}^{{\cal C},P}
\cong \Quot_{{\cal A}(-n) \otimes V/Y/S}^{{\cal A},P}$ such that
$\overline{M}_{X/S}^{{\cal C},P}=Q^{ss} \dslash GL(V)$,
where $V$ is a vector space of dimension $P(n)$.
Let ${\cal Q}$ be the universal quotient on $Q^{ss} \times X$. Then
${\cal Q}_{|\{q \} \times X}$ is $G$-twisted semi-stable
for all $q \in Q^{ss}$.
By the construction of the moduli space,
we have a $GL(V)$-equivariant isomorphism
$V \to p_{Q^{ss}}(G^{\vee} \otimes {\cal Q}(n))$.
We set
\begin{equation}
\begin{split}
{\cal L}_{m,n}:=&
\det p_{Q^{ss} !}(G^{\vee} \otimes {\cal Q}(n+m))^{\otimes P(n)} \otimes
\det p_{Q^{ss} !}(G^{\vee} \otimes {\cal Q}(n))^{\otimes (-P(m+n))}\\
=& \det p_{Q^{ss} !}(G^{\vee} \otimes {\cal Q}(n+m))^{\otimes P(n)} \otimes
\det V^{\otimes (-P(m+n))}.
\end{split}
\end{equation}
We note that ${\bf R}\pi_*(G^{\vee} \otimes {\cal Q})$ 
gives the
universal quotient ${\cal A}$-module on $Y \times 
\Quot_{{\cal A}(-n) \otimes V/Y/S}^{{\cal A},P}$.
By the construction of the moduli space, we get the following. 
\begin{lem}\label{lem:polarization}
${\cal L}_{m,n}$, $m \gg n \gg 0$ is the pull-back of a relatively
ample line bundle on $\overline{M}_{X/S}^{{\cal C},P}$.
\end{lem}
Assume that $S=\Spec({\Bbb C})$ and $\dim X=2$.
We set ${\cal O}_X(1)={\cal O}_X(H)$.
\begin{defn}
\begin{enumerate}
\item[(1)]
For ${\bf e} \in K(X)_{\mathrm{top}}$,
$\overline{M}_H^G({\bf e})$ is the moduli space
of $G$-twisted semi-stable objects $E$ of ${\cal C}$ with
$\tau(E)={\bf e}$ and
${M}_H^G({\bf e})$ the open subscheme consisting of
$G$-twisted stable objects.
\item[(2)]
Let ${\cal M}_H({\bf e})^{\mu\text{-ss}}$ 
(resp. ${\cal M}_H^G({\bf e})^{ss}, {\cal M}_H^G({\bf e})^{s}$)
be the moduli stack of
$\mu$-semi-stable (resp. $G$-twisted semi-stable, $G$-twisted stable)
objects $E$ of ${\cal C}$ with
$\tau(E)={\bf e}$. 
\end{enumerate}
\end{defn}
We set $r_0:=\rk {\bf e}$ and $\xi_0 :=c_1({\bf e})$. 
Then we see that  
\begin{equation}\label{eq:det-bdle}
\begin{split}
&\ch(P(n)G^{\vee}((n+m)H)-P(n+m)G^{\vee}(nH))\\
=& m \left[\frac{(\rk G) r_0}{2}(H^2)\left\{
(m-2n)\ch G^{\vee}-n(n+m) ((\rk G) H-(c_1(G),H)\varrho_X) \right\} \right.\\
& \left. +(H,(\rk G) \xi_0-r_0 c_1(G)-\frac{(\rk G) r_0}{2}K_X)
\left(-\ch G^{\vee}+\frac{n(n+m)}{2}(H^2)(\rk G) \varrho_X \right)\right].
\end{split}
\end{equation}

\begin{NB}
\begin{equation}
\begin{split}
&\chi(G^{\vee} \otimes E(nH))\\
=&
\int_X (\rk G-c_1(G)+\ch_2(G))(\rk E+c_1(E)+\ch_2(E))
(1+nH+\frac{n^2}{2}(H^2)\varrho_X)
(1-\frac{1}{2}K_X+\chi({\cal O}_X)\varrho_X)\\
=& \chi(G,E)+\int_X (\rk G-c_1(G))(\rk E+c_1(E))
(nH+\frac{n^2}{2}(H^2)\varrho_X)(1-\frac{1}{2}K_X)\\
=& n(H,\rk G c_1(E)-\rk E c_1(G)-\frac{\rk G \rk E}{2}K_X)
+\frac{n^2 \rk G \rk E}{2}(H^2)
\end{split}
\end{equation}

\begin{equation}
\begin{split}
\ch G^{\vee}(nH)=\ch G^{\vee}+nH \rk G -n(c_1(G),H)+\frac{n^2}{2}(H^2)\rk G 
\end{split}
\end{equation}
\end{NB}

\begin{lem}\label{lem:det-bdle}
We take $\zeta \in K(X)$ with $\ch(\zeta)=r_0 H+(\xi_0,H)\varrho_X$.
Assume that $\tau(G) \in {\Bbb Z}{\bf e}$.
If $\chi({\bf e},{\bf e})=0$ and $E \cong E \otimes K_X$
for all $E \in {\cal M}_H^G({\bf e})^{ss}$,
 then 
$\det p_{Q^{ss} !}({\cal Q}\otimes \zeta^{\vee}) \cong
\det p_{Q^{ss} !}({\cal Q}^{\vee}\otimes \zeta)^{\vee}$ 
is the pull-back of an ample line bundle ${\cal L}(\zeta)$
on $\overline{M}_H^G({\bf e})$.
\end{lem}

\begin{proof}
We first note that 
$\det p_{Q^{ss} !}({\cal Q}\otimes E^{\vee}) \cong
{\cal O}_{Q^{ss}}$ for $E \in {\cal M}_H^G({\bf e})^{ss}$. 
We set $\tau(G)=\lambda {\bf e}$, $\lambda \in {\Bbb Z}_{>0}$.
Then $P(n)G^{\vee}((n+m)H)-P(n+m)G^{\vee}(nH)
\equiv mn(n+m)\lambda \zeta^{\vee} \mod {\Bbb Z}{\bf e}^{\vee}$.
By Lemma \ref{lem:polarization}, we get our claim.
\end{proof}

\begin{NB}
$\det p_{Q^{ss} !}({\cal Q}(-\tau))$ belongs to the closure of the ample cone.
If $\chi(G,G)=0$, then
it is ample.  
\end{NB}

\begin{defn}
\begin{enumerate}
\item[(1)]
$P({\bf e})$ is the set of  
subobject $E'$ of $E \in {\cal M}_H({\bf e})^{\mu\text{-ss}}$ such that
\begin{equation}
\frac{(c_1(G^{\vee} \otimes E),H)}{\rk E}=
\frac{(c_1(G^{\vee} \otimes E'),H)}{\rk E'}.
\end{equation}
\item[(2)]
For $E' \in P({\bf e})$,
we define a wall $W_{E'} \subset \NS(X) \otimes {\Bbb R}$ 
as the set of $\alpha \in \NS(X)\otimes {\Bbb R}$ satisfying
\begin{equation}
\left(\alpha,\frac{c_1(G^{\vee} \otimes E)}{\rk E}-
\frac{c_1(G^{\vee} \otimes E')}{\rk E'} \right)+
\left(\frac{\chi(G^{\vee} \otimes E)}{\rk E}-
\frac{\chi(G^{\vee} \otimes E')}{\rk E'} \right)=0.
\end{equation}
\end{enumerate}
\end{defn}
Since $\tau(E')$ is finite,
$\cup_{E'} W_{E'}$ is locally finite.
If $\alpha \in \NS(X) \otimes {\Bbb Q}$ does not lie
on any $W_{E'}$, we say that $\alpha$ is general.
If a local projective generator $G'$
satisfies $\alpha:=c_1(G')/\rk G'-c_1(G)/\rk G \not \in \cup_{E'} W_{E'}$,
then we also call $G'$ is general. 
\begin{lem}
If $G$ is general, i.e., $0 \not \in \cup_{E'} W_{E'}$,
then 
for $E' \in P({\bf e})$,
\begin{equation}
\frac{\chi(G,{\bf e})}{\rk {\bf e}}=\frac{\chi(G,E')}{\rk E'}
\Longleftrightarrow
\frac{{\bf e}}{\rk {\bf e}}=\frac{\tau(E')}{\rk E'} 
\in K(X)_{\mathrm{top}} \otimes {\Bbb Q}. 
\end{equation}
In particular, if ${\bf e}$ is primitive, then
$\overline{M}_H^G({\bf e})={M}_H^G({\bf e})$
for a general $G$. 
\end{lem}

\subsection{A generalization of stability for 0-dimensional objects.}

It is easy to see that every 0-dimensional object is
$G_s$-twisted semi-stable.
Our definition is not
sufficient in order to get a good moduli space.
So we introduce a refined version of twisted stability. 

\begin{defn}
Let $G,G'$ be families of local projective generators of ${\cal C}_s$.
A 0-dimensional object $E$ is $(G_s,G_s')$-twisted semi-stable, if
\begin{equation}
\frac{\chi(G_s',E_1)}{\chi(G_s,E_1)}  \leq 
\frac{\chi(G_s',E)}{\chi(G_s,E)}
\end{equation}
for all proper subobject $E_1$ of $E$. 
\end{defn}
By a modification of Simpson's construction of moduli spaces,
we can construct the coarse moduli scheme of $(G_s,G_s')$-twisted
semi-stable objects.
From now on, we assume that $S=\Spec({\Bbb C})$ for simplicity.

\begin{lem}\label{lem:decomposition}
Let $G$ be a locally free sheaf on $X$ 
which is a local projective generator of ${\cal C}$.
%Assume that 
%(i)
%${\cal C}$ is a tilting of $\Coh(X)$ by a torsion
%pair or (ii) $\pi^*(K_Y)=K_X$.
\begin{enumerate}
\item[(1)]
Assume that there is an exact sequence in ${\cal C}$
\begin{equation}
0 \to E' \to V_0 \to 
V_1 \to \cdots 
\to V_r \to E \to 0
\end{equation}
such that $V_i$ are local projective objects of ${\cal C}$.
If $r \geq \dim X$, then
$E'$ is a local projective object of ${\cal C}$.
\item[(2)]
For $E \in K(Y)$,
there is a local projective generator $G'$ of
${\cal C}$ such that $E=G'-N G(-n)$, where $N$ and $n$ 
are sufficiently large integers.
\end{enumerate}
\end{lem}

\begin{proof}
(1)
We first prove that
$H^i({\bf R}\pi_* {\bf R}{\cal H}om(E,F))=0$, $i>\dim X+1$
for all $F \in {\cal C}$.
%If (i) holds, then 
Since ${\cal C}$ is a tilting of $\Coh(X)$ (Lemma \ref{lem:tilting:C}),
$H^i(E)=H^i(F)=0$ for $i \ne -1,0$.
By using a spectral sequence, we get
\begin{equation}
H^i({\bf R}\pi_* {\bf R}{\cal H}om(H^{-p}(E)[p],H^{-q}(F)[q]))=0
\end{equation}
for $i>\dim X +1$.
Hence we get 
$H^i({\bf R}\pi_* {\bf R}{\cal H}om(E,F))=0$, $i>\dim X+1$.
Then we see that
\begin{equation}
H^i({\bf R}\pi_* {\bf R}{\cal H}om(E',F)) \cong
H^{i+r+1}({\bf R}\pi_* {\bf R}{\cal H}om(E,F))=0
\end{equation}
for all integer with $i>\max\{\dim X-r,0\}=0$.
%If $\pi^*(K_Y)=K_X$, then $E(K_X) \in {\cal C}$ and 
%by using the Serre duality, we see that
%$H^i({\bf R}\pi_* {\bf R}{\cal H}om(E',F))=0$ for
%all positive integer with $i>\max\{\dim X-r,0\}=0$.
\begin{NB}
Serre duality implies that
$\Hom(E,F(n)[i+r+1]) \cong \Hom(F(n),E(K_X)[\dim X-i-r-1])^{\vee}=0$.
\end{NB}
Therefore $E'$ is a local projective object.

(2)
We first prove that there are local projective generators
$G_1, G_2$ such that $E=G_1-G_2$.
We may assume that
$E \in {\cal C}$.
We take a resolution of $E$
\begin{equation}
0 \to E' \to G(-n_r)^{\oplus N_r} \overset{\phi}{\to} 
G(-n_{r-1})^{\oplus N_{r-1}} \to \cdots 
\to G(-n_0)^{\oplus N_0} \to E \to 0.
\end{equation}
If $r \geq \dim X $,
then (1) implies that
$E'$ is a local projective object.
We set $r:=2j_0+1$.
We set $G_1:=E' \oplus \oplus_{j=0}^{j_0} G(-n_{2j})^{\oplus N_{2j}}$ and
$G_2:=\oplus_{j=0}^{j_0} G(-n_{2j+1})^{\oplus N_{2j+1}}$
Then $G_1$ and $G_2$ are local projective generators and 
$E=G_1-G_2$.
We take a resolution
\begin{equation}
0 \to G_2' \to G(-n)^{\oplus N} \to G_2 \to 0
\end{equation}
such that $G_2' \in {\cal C}$.
Then we see that ${\bf R}\pi_* {\bf R}{\cal H}om(G_2',F)
\in \Coh(Y)$ for any $F \in {\cal C}$.
Since $E=(G_1 \oplus G_2')-G(-n)^{\oplus N}$ and
$G_1 \oplus G_2'$ is a local projective generator,
we get our claim.
\end{proof}

%\begin{lem}
%For $E, F \in {\cal C}$,
%$H^i({\bf R}\pi_*({\bf R}{\cal H}om(E,F)))=0$ for $i> \dim X$.
%\end{lem}
%
%$
%H^0(Y,H^i({\bf R}\pi_*({\bf R}{\cal H}om(E,F)))(n))
%=\Hom(E,F(n)[i])=\Hom(F(n),E(K_X)[\dim X -i])^{\vee}=0$
%if $K_X=\pi^*(K_Y)$.

\begin{defn}
Let $A$ be an element of $K(Y) \otimes {\Bbb Q}$ and 
$G$ a local projective generator.
A 0-dimensional object $E$ is $(G,A)$-twisted semi-stable,
if
\begin{equation}
\frac{\chi(A,F)}{\chi(G,F)}  \leq \frac{\chi(A,E)}{\chi(G,E)}
\end{equation}
for all proper subobject $F$ of $E$.  
\end{defn}

By Lemma \ref{lem:decomposition}, we write
$N' A=G'-NG(-n) \in K(X)$, where $G'$ is a local projective generator
and $n,N,N'>0$.
Then
\begin{equation}
\frac{\chi(G',E)}{\chi(G,E)}=N'\frac{\chi(A,E)}{\chi(G,E)}+N.
\end{equation}
Hence $E$ is $(G,G')$-twisted semi-stable if and only if
$E$ is $(G,A)$-twisted semi-stable.
Thus we get the following proposition.

\begin{prop}
Let $A$ be an element of $K(Y) \otimes {\Bbb Q}$ and 
$G$ a local projective generator.
Let $v$ be a Mukai vector of a 0-dimensional object.
\begin{enumerate}
\item[(1)]
There is a coarse moduli scheme $\overline{M}_{{\cal O}_X(1)}^{G,A}(v)$
of $(G,A)$-twisted semi-stable objects of ${\cal C}$.
\item[(2)]
If $v$ is primitive and $A$ is general in $K(Y) \otimes {\Bbb Q}$,
then $\overline{M}_{{\cal O}_X(1)}^{G,A}(v)$ consists
of $(G,A)$-twisted stable objects. 
Moreover 
$\overline{M}_{{\cal O}_X(1)}^{G,A}(\varrho_X)$ is a fine moduli space. 
\end{enumerate}
\end{prop}

\begin{rem}\label{rem:G-indep}
If $v(E)=\varrho_X$ and $\rk A=0$, then
$E$ is $(G,A)$-twisted semi-stable if and only if
$\chi(A,E') \leq 0$ for all subobject $E'$ of $E$ in ${\cal C}$. 
Thus the semi-stability does not depend on the choice of $G$.
\end{rem}

\begin{rem}
In subsection \ref{subsect:twisted},
we deal with the twisted sheaves.
In this case, we still have the moduli spaces
of 0-dimensional stable objects, but 
$\overline{M}_{{\cal O}_X(1)}^{G,A}(\varrho_X)$ does not
 have a universal family. 
\end{rem}

\begin{NB}
For $E \in \overline{M}_{{\cal O}_X(1)}^{G,A}(\varrho_X)$,
we have $\chi(E)=1$. Hence there is a universal family.
\end{NB}

\subsection{Construction of the moduli spaces of ${\cal A}$-modules
of dimension 0.}\label{subsect:A-module}
By Proposition \ref{prop:Morita},
we have an equivalence
${\cal C} \to \Coh_{{\cal A}}(Y)$.
We set ${\cal B}:=\pi_*(G^{\vee} \otimes G')$.
Then ${\cal B}$ is a local projective generator of
$\Coh_{{\cal A}}(Y)$:
For all $F \in \Coh_{{\cal A}}(Y)$, 
${\bf R}{\cal H}om_{{\cal A}}({\cal B},F)=
{\cal H}om_{{\cal A}}({\cal B},F)$ and
${\bf R}{\cal H}om_{{\cal A}}({\cal B},F)=0$ if and only if $F=0$.
In particular, we have a surjective morphism
\begin{equation}
\phi:{\cal H}om_{{\cal A}}({\cal B},F) \otimes_{{\cal A}} {\cal B} \to F.
\end{equation}
For $F \in \Coh_{{\cal A}}(Y)$, 
we set
\begin{equation}
\chi_{{\cal A}}({\cal B},F):=
\chi({\bf R}{\cal H}om_{{\cal A}}({\cal B},F)).
\end{equation}
For $F \in \Coh_{{\cal A}}(Y)$,
$\pi^{-1}(F) \otimes_{\pi^{-1}({\cal A})} G$ is 
$(G,G')$-twisted semi-stable, if 
\begin{equation}
\frac{\chi_{{\cal A}}({\cal B},F_1)}{\chi(F_1)}  \leq 
\frac{\chi_{{\cal A}}({\cal B},F)}{\chi(F)}
\end{equation}
for all proper sub ${\cal A}$-module $F_1$ of $F$. 
We define the $({\cal A},{\cal B})$-twisted semi-stability
by this inequality.
\begin{prop}\label{prop:A-module}
There is a coarse moduli scheme of
$({\cal A},{\cal B})$-twisted semi-stable ${\cal A}$-modules of 
dimension 0.
\end{prop}
{\it Proof of Proposition \ref{prop:A-module}.}
Let $F$ be an ${\cal A}$-module of dimension 0.
Then $\Hom_{{\cal A}}({\cal B},F) \otimes {\cal B} \to F$
is surjective.
Hence all 0-dimensional objects $F$ are parametrized by
a quot-scheme 
$Q:=\Quot_{V \otimes {\cal B}/Y/{\Bbb C}}^{{\cal A},m}$,
where $m=\chi(F)$ and $\dim V=\chi_{{\cal A}}({\cal B},F)$.  
Let $V \otimes {\cal O}_Q \otimes {\cal B} \to {\cal F}$
be the universal quotient.
For simplicity, we set
${\cal F}_q:={\cal F}_{|\{q \} \times Y}$, $q \in Q$.
For a sufficiently large integer $n$,
we have a quotient
$V \otimes H^0(Y,{\cal B}(n)) \to 
H^0(Y,F(n))$. 
We set $W:=H^0(Y,{\cal B}(n))$. Then we have an embedding
\begin{equation}
\Quot_{V \otimes {\cal B}/Y/{\Bbb C}}^{{\cal A},m} 
\hookrightarrow Gr(V \otimes W,m).
\end{equation}
This embedding is equivariant with respect to the natural action 
of $PGL(V)$. The following is well-known.

\begin{lem}\label{lem:Hilbert-Mumford}
Let $\alpha:V \otimes W \to U$ be a point of ${\frak G}:=Gr(V \otimes W,m)$.
Then $\alpha$ belongs to the set ${\frak G}^{ss}$ of semi-stable points
if and only if
\begin{equation}
\frac{\dim U}{\dim V} \leq \frac{\dim \alpha(V_1 \otimes W)}{\dim V_1}
\end{equation}
for all proper subspace $V_1 \ne 0$ of $V$.
If the inequality is strict for all $V_1$, then $\alpha$ is
stable.
\end{lem}
We set
\begin{equation}
Q^{ss}:=\{ q \in Q | \text{
${\cal F}_q$ is $({\cal A},{\cal B})$-twisted semi-stable }\}.
\end{equation}
For $q \in Q^{ss}$, 
$V \to \Hom_{\cal A}({\cal B},F)$ is an isomorphism.
We only prove that $Q^{ss}={\frak G}^{ss} \cap Q$.
Then Proposition \ref{prop:A-module} easily follows.

For an ${\cal A}$-submodule $F_1$ of $F$,
we set $V_1:=\Hom_{\cal A}({\cal B},F_1)$.
Then we have a surjective homomorphism 
$V_1 \otimes {\cal B} \to F_1$.
Conversely for a subspace $V_1$ of $V$,
we set $F_1:=\im(V_1 \otimes {\cal B} \to F)$.
Then $V_1 \to \Hom_{\cal A}({\cal B},F_1)$ is
injective.

We set
\begin{equation}
{\frak F}:=\{\im(V_1 \otimes {\cal B} \to {\cal F}_q )|
q \in Q,\; V_1 \subset V \}.
\end{equation}
Since ${\frak F}$ is bounded,
we can take an integer $n$ in the definition of $W$ such that
$V_1 \otimes W \to H^0(Y,F_1)$ is surjective
for all $F_1 \in {\frak F}$. 
Assume that ${\cal F}_q$ is $({\cal A},{\cal B})$-twisted semi-stable.
For any $V_1 \subset V$,
we set $F_1:=\im(V_1 \otimes {\cal B} \to {\cal F}_q)$.
Then $\alpha(V_1 \otimes W)=H^0(Y,F_1)$.
Hence 
\begin{equation}
\frac{\dim \alpha(V_1 \otimes W)}{\dim V_1}
\geq \frac{\chi(F_1)}{\dim \Hom_{{\cal A}}({\cal B},F_1)}=
\frac{\chi(F_1)}{\chi_{{\cal A}}({\cal B},F_1)} \geq
\frac{\chi({\cal F}_q)}{\chi_{{\cal A}}({\cal B},{\cal F}_q)}
=\frac{\dim \alpha(V \otimes W)}{\dim V}.
\end{equation}    
Thus $q \in {\frak G}^{ss}$.

We take a point $q \in {\frak G}^{ss} \cap Q$.
We first prove that 
$\psi:V \to \Hom_{{\cal A}}({\cal B},{\cal F}_q)$ is 
an isomorphism.
We set $V_1:=\ker \psi$.
Since $V_1 \otimes {\cal B} \to {\cal F}_q$ is 0,
we get $\alpha(V_1 \otimes W)=0$.
Then 
\begin{equation}
\frac{\dim U}{\dim V} \leq \frac{\dim \alpha(V_1 \otimes W)}{\dim V_1}=0,
\end{equation}
which is a contradiction.
Therefore $\psi$ is injective.
Since $\dim V=\dim \Hom_{{\cal A}}({\cal B},{\cal F}_q)$,
$\psi $ is an isomorphism.
Let $F_1 \ne 0$ be a proper ${\cal A}$-submodule of ${\cal F}_q$.
We set $V_1:=\Hom_{{\cal A}}({\cal B},F_1)$.
Then 
\begin{equation}
\frac{\chi(F_1)}{\dim \Hom_{{\cal A}}({\cal B},F_1)} \geq 
\frac{\dim \alpha(V_1 \otimes W)}{\dim V_1}\geq
\frac{\dim \alpha(V \otimes W)}{\dim V}=
\frac{\chi({\cal F}_q)}{\chi_{{\cal A}}({\cal B},{\cal F}_q)}.
\end{equation}    
\begin{NB}
We set $F':=\im(V_1 \otimes {\cal B} \to {\cal F}_q)
(\subset F_1)$.
Then $\alpha(V_1 \otimes W)=H^0(Y,F_1')$.
Hence $\dim \alpha(V_1 \otimes W)=
\chi(F_1') \leq \chi(F_1)$. 
\end{NB}
Hence ${\cal F}_q$ is $({\cal A},{\cal B})$-twisted semi-stable.
If $q$ is a stable point, then we also see that ${\cal F}_q$ is
$({\cal A},{\cal B})$-twisted stable.

\begin{NB}

$\Hom_{{\cal A}'}(\pi_*(E \otimes_{{\cal A}}G \otimes {G'}^{\vee}),F)
=\Hom(E \otimes_{{\cal A}}G,F \otimes_{{\cal A}'}G')=
\Hom_{{\cal A}}( E,\pi_*(F \otimes_{{\cal A}'}G' \otimes G^{\vee}))$.

%For a quotient $V \otimes {\cal A} \to E$, we have
%$V \otimes \pi_*(G \otimes  {G'}^{\vee}) \to 
%\pi_*(E \otimes_{\cal A} G \otimes {G'}^{\vee})$. 

\begin{lem}
By the equivalence $\Coh_{{\cal A}'}(Y) \cong {\cal C} \cong
\Coh_{\cal A}(Y)$,
we have an isomorphism
\begin{equation}
\Quot_{V \otimes {\cal A}'/Y/{\Bbb C}}^{{\cal A}'} \to
\Quot_{V \otimes G'/X/{\Bbb C}}^{{\cal A}'}
\to \Quot_{V \otimes {\cal B}/Y/{\Bbb C}}^{{\cal A}}. 
\end{equation}
\end{lem}

For a quotient $V \otimes {\cal A}' \to E$, we have
a quotient of 
${\cal A}$-module
\begin{equation}
V \otimes {\cal B} \to 
\pi_*(E \otimes_{{\cal A}'} G^{\vee}\otimes {G'}).
\end{equation}

$H^0(Y,E \otimes_{{\cal A}'} G^{\vee}\otimes {G'}(n))=
\Hom_{\cal A}({\cal A},\pi_*(E \otimes_{{\cal A}'} G^{\vee}\otimes {G'}))
=\Hom(G,E \otimes_{{\cal A}'} {G'})=
\Hom_{{\cal A}'}(\pi_*(G \otimes {G'}^{\vee}),E)$

$\Hom_{\cal A}(\pi_*(G^{\vee} \otimes  {G'}),
\pi_*(E \otimes_{{\cal A}'} G^{\vee}\otimes {G'}))
=\Hom(G',E \otimes_{{\cal A}'} {G'})=
\Hom_{{\cal A}'}({\cal A}',E)$

\begin{lem}
Let $E, E'$ be an object of ${\cal C}$.
Then
$\Hom(E,E'[i])=0$ for $i<0$.
\end{lem}
\end{NB}

\subsection{Twisted case.}\label{subsect:twisted}
\subsubsection{Definition.}
Let $X=\cup_i X_i$ be an analytic open covering of $X$ and
$\beta=\{\beta_{ijk} \in 
H^0(X_i \cap X_j \cap X_k,{\cal O}_X^{\times}) \}$
a Cech 2-cocycle of ${\cal O}_X^{\times}$.
We assume that $\beta$ defines a torsion element $[\beta]$ of
$H^2(X,{\cal O}_X^{\times})$.
Let $E=(\{E_i \}, \{\varphi_{ij}\})$ 
be a coherent $\beta$-twisted sheaf:
\begin{enumerate}
\item
$E_i$ is a coherent sheaf on $X_i$.
\item
$\varphi_{ij}:E_{i|X_i \cap X_j} \to E_{j|X_i \cap X_j}$ 
is an isomorphism.
\item $\varphi_{ji}=\varphi_{ij}^{-1}$.
\item $\varphi_{ki} \circ \varphi_{jk} \circ \varphi_{ij}=
\beta_{ijk}\id_{X_i \cap X_j \cap X_k}$.  
\end{enumerate}
Let $G$ be a locally free $\beta$-twisted sheaf and
$P:={\Bbb P}(G^{\vee})$ the associated projective bundle over $X$ 
(cf. \cite[sect. 1.1]{Y:twisted}).
Let $w(P) \in H^2(X,{\Bbb Z}/r{\Bbb Z})$ be the characteristic class 
of $P$ (\cite[Defn. 1.2]{Y:twisted}). Then $[\beta]$ is trivial if and only if
$w(P) \in \im(\NS(X) \to H^2(X,{\Bbb Z}/r{\Bbb Z}))$
(\cite[Lem. 1.4]{Y:twisted}).

Let $\Coh^{\beta}(X)$ be the category
of coherent $\beta$-twisted sheaves on $X$ and
 ${\bf D}^{\beta}(X)$ the bounded derived category
of $\Coh^{\beta}(X)$. 
Let $K^{\beta}(X)$ be the Grothendieck group of $\Coh^{\beta}(X)$.
Then similar statements in Lemma \ref{lem:tilting} hold 
for $\Coh^{\beta}(X)$. 
%We also denote the tilted category
%by ${\cal C}_G$.    
Then all results in sections \ref{subsect:family}
and \ref{subsect:stability} hold.
In particular,
if a locally free $\beta$-twisted sheaf $G$ defines a
torsion pair, then we have the moduli of $G$-twisted semi-stable objects.  
Replacing $\zeta \in K(X)$ by
$\zeta \in K^{\beta}(X)$ with $c_1(\zeta)=r_0 H$ and
$\chi(G \otimes \zeta^{\vee})=0$,
Lemma \ref{lem:det-bdle} also holds.

\subsubsection{Chern character }
We have a homomorphism
\begin{equation}
\begin{matrix}
\ch_G:& {\bf D}^{\beta}(X)& \to & H^{ev}(X,{\Bbb Q})\\
& E & \mapsto & 
\frac{\ch(G^{\vee} \otimes E)}{\sqrt{\ch(G^{\vee} \otimes G)}}.
\end{matrix}
\end{equation} 
Obviously $\ch_G(E)$ depends only on the class in $K^{\beta}(X)$. 
Since 
\begin{equation}
\ch_G(E)^{\vee}\ch_G(F)=\frac{\ch((G^{\vee} \otimes E)^{\vee} \otimes 
(G^{\vee} \otimes F))}{\ch(G^{\vee} \otimes G)}=
\ch(E^{\vee} \otimes F),
\end{equation}
we have the following Riemann-Roch formula.
\begin{equation}\label{eq:RR}
\chi(E,F)=\int_X \ch_G(E)^{\vee} \ch_G(F)\td_X.
\end{equation}
Assume that $X$ is a surface.
For a torsion $G$-twisted sheaf $E$, we can attach the codimension 1 part of
the scheme-theoretic
support $\Div(E)$ as in the usual sheaves.
Then we see that
\begin{equation}
\ch_G(E)=(0,[\Div(E)],a), a \in {\Bbb Q},
\end{equation}
where $[\Div(E)]$ denotes the 
homology class of the divisor $\Div(E)$ and
we regard it as an element of $H^2(X,{\Bbb Z})$
by the Poincar\'{e} duality.
More generally,
if $E \in {\bf D}^{\beta}(X)$
satisfies $\rk H^i(E)=0$ for all $i$, 
then 
\begin{equation}
\ch_G(E)=(0,\sum_i (-1)^i [\Div(H^i(E))],a), a \in {\Bbb Q}.
\end{equation}
We set $c_1(E):=\sum_i (-1)^i [\Div(H^i(E))]$.

\begin{rem}\label{rem:integral}
If $H^3(X,{\Bbb Z})$ is torsion free, then
we have an automorphism $\eta$ of $H^*(X,{\Bbb Q})$
such that the image of $\eta \circ \ch_G$ is contained in
$\ch(K(X)) \subset {\Bbb Z} \oplus H^2(X,{\Bbb Z}) \oplus
H^4(X,\frac{1}{2}{\Bbb Z})$ and \eqref{eq:RR} holds if we replace 
$\ch_G$ by $\eta \circ \ch_G$ (cf. \cite{Y:twisted}):
We first note that
\begin{equation}
\ch(K(X))=\{(r,D,a)|r \in {\Bbb Z}, D \in H^2(X,{\Bbb Z}), 
a-(D,K_X)/2 \in {\Bbb Z} \}. 
\end{equation}
Replacing the statement of 
\cite[Lem. 3.1]{Y:twisted} by
\begin{equation}
\begin{split}
& c_2(E^{\vee} \otimes E) +r(r-1)(w(E),K_X)\\
\equiv 
& -(r-1)((w(E)^2)-r(w(E),K_X)) \mod 2r,
\end{split}
\end{equation}
 we can prove
a similar claim to \cite[Lem. 3.3]{Y:twisted}.
\end{rem}

\begin{lem}\label{lem:twisted:0-dim}
Let $E$ be a $\beta$-twisted sheaf of $\rk E=0$.
Then 
\begin{equation}
[\chi(G,E) \mod r{\Bbb Z}] \equiv -w(P) \cap [\Div(E)],
\end{equation}
where we identified $H_0(X,{\Bbb Z}/r{\Bbb Z})$ with 
${\Bbb Z}/r{\Bbb Z}$.
\end{lem}

\begin{proof}
Since $\chi(G,E)$ and $[\Div(E))]$ are additive,
it is sufficient to prove the claim for pure sheaves.
If $\dim E=0$ as an object of $\Coh^{\beta}(X)$,
then $r|\chi(G,E)$ and $\Div(E)=0$. Hence the claim holds.
We assume that $E$ is purely 1-dimensional.
Then $E$ is a twisted sheaf on $C:=\Div(E)$.
Since $C$ is a curve, there is a $\beta$-twisted line bundle
$L$ on $C$ and we have an equivalence
\begin{equation}\label{eq:twisted-equiv}
\begin{matrix}
\varphi:& \Coh^{\beta}(C)& \to &\Coh(C)\\
& E & \mapsto & E \otimes L^{\vee}.
\end{matrix}
\end{equation}
Then we can take a filtration $0 \subset F_1 \subset F_2 \subset 
\cdots \subset
F_n=E$ of $E$ such that
$\Div(F_i/F_{i-1})$ are reduced and irreducible curve and 
$F_i/F_{i-1}$ are torsion free $\beta$-twisted sheaves of rank 1
on $\Div(F_i/F_{i-1})$.
Replacing $E$ by $F_i/F_{i-1}$,
we may assume that $E$ is a twisted sheaf of rank 1
on an irreducible and reduced curve $C=\Div(E)$.
Then $\chi(G,E)=\chi(\varphi(G_{|C})^{\vee} \otimes \varphi(E))
=\int_C c_1(\varphi(G_{|C})^{\vee})+r\chi(\varphi(E))$.
Since $w(P)_{|C}=w(P_{|C})=c_1(\varphi(G_{|C})) \mod r{\Bbb Z}$,
$[\chi(G,E)\mod r{\Bbb Z}] \equiv -w(P) \cap [C]$. 
\end{proof}

\begin{NB}
The universal coefficient theorem says that
$$
0 \to H_i(X,{\Bbb Z}) \otimes R \to H_i(X,R) \to
\Tor(H_{i-1}(X,{\Bbb Z}),R) \to 0. 
$$
Hence $H_0(X,R) \cong H_0(X,{\Bbb Z}) \otimes R \cong R$
for any $R$.
\end{NB}

\begin{cor}\label{cor:twisted:0-dim}
For an object $E$ of ${\bf D}^{\beta}(X)$,
assume that $\rk H^i(E)=0$ for all $i$.
Then 
\begin{equation}
[\chi(G,E)\mod r{\Bbb Z}] \equiv -w(P) \cap [\Div(E)].
\end{equation}
Moreover if $c_1(E)=0$, then
$\ch_G(E) \in {\Bbb Z}\varrho_X$.
\end{cor}

\begin{proof}
The second claim follows from 
 $\int_X \ch_G(E)=\chi(G,E)/r=(\chi(G,E)/r) \int_X \varrho_X$. 
\end{proof}

\section{Perverse coherent sheaves for the resolution of
rational double points.}\label{sect:RDP}

\subsection{Perverse coherent sheaves on the resolution
of rational singularities.}
\label{subsect:Per(X/Y)}
Let $Y$ be a projective normal surface with at worst rational 
singularities and $\pi:X \to Y$ the minimal resolution.
Let $p_i$, $i=1,2,...,n$ be the singular points of $Y$ and
$Z_i:=\pi^{-1}(p_i)=\sum_{j=1}^{s_i} a_{ij} C_{ij}$ their fundamental cycles. 
Let $\beta$ be a 2-cocycle of ${\cal O}_X^{\times}$
whose image in $H^2(X,{\cal O}_X^{\times})$ is a torsion
element.
For $\beta$-twisted line bundles 
$L_{ij}$ on $C_{ij}$,
we shall define abelian categories 
$\Per(X/Y,\{L_{ij}\})$ and
$\Per(X/Y,\{ L_{ij} \})^*$.

\begin{prop}\label{prop:C(G)}
\begin{enumerate}
\item[(1)]
There is a locally free sheaf $E$
such that $\chi(E,L_{ij})=0$
for all $i,j$ and 
$R^1 \pi_*( E^{\vee} \otimes E)=0$.
\item[(2)]
${\cal C}(E)$
is the tilting of $\Coh^{\beta}(X)$ with respect to the torsion 
pair $(S,T)$ such that
\begin{equation}
\begin{split}
S:=&\{E \in \Coh^{\beta}(X)| 
\text{ $E$ is generated by subsheaves of $L_{ij}$ } \},\\
T:=& \{E \in \Coh^{\beta}(X)| \Hom(E,L_{ij})=0 \}.
\end{split}
\end{equation}
\item[(3)]
${\cal C}(E)^*$
is the tilting of $\Coh^{\beta}(X)$ with respect to the torsion 
pair $(S^*,T^*)$ such that
\begin{equation}
\begin{split}
S^*:=&\{E \in \Coh^{\beta}(X)| 
\text{ $E$ is generated by subsheaves of 
$A_{p_i} \otimes \omega_{Z_i}$ } \},\\
T^*:=& \{E \in \Coh^{\beta}(X)| \Hom(E,A_{p_i} \otimes \omega_{Z_i})=0 \}.
\end{split}
\end{equation}
\end{enumerate}
\end{prop}

\begin{NB}
$A_{p_i}$ is a unique line bundle on $Z_i$
such that $A_{p_i|C_{ij}}=L_{ij}(1)$.
\end{NB}

For the proof of (1), 
we shall use the deformation theory of a coherent twisted sheaf. 
\begin{defn}
For a coherent $\beta$-twisted sheaf $E$ on a scheme $W$,
$\Def(W,E)$ denotes the local deformation space of $E$
fixing $\det E$. 
\end{defn}
For a complex $E \in {\bf D}^{\beta}(X)$, 
let 
\begin{equation}
\Ext^i(E,E)_0:=\ker(\Ext^i(E,E) \overset{\tr}{\to} H^i(X,{\cal O}_X))
\end{equation}
be the kernel of the trace map. 
If $\Ext^2(E,E)_0=0$, then $\Def(W,E)$ is smooth and
the Zariski tangent space at $E$ is 
$\Ext^1(E,E)_0$. 
The following is well-known.
\begin{lem}\label{lem:generic-vanishing} 
Let $D$ be a divisor on $X$.
For $E \in \Coh^{\beta}(X)$ with $\rk E>0$,
we have a torsion free $\beta$-twisted sheaf $E'$ 
such that $\tau(E')=\tau(E)-n\tau({\Bbb C}_x)$ and
$\Ext^2(E',E'(D))_0=0$.  
\end{lem}

\begin{proof}
For a locally free $\beta$-twisted sheaf $E$,
we consider a general surjective homomorphism
$\phi:E \to \oplus_{i=1}^n {\Bbb C}_{x_i}$, $x_i \in X$.
If $n$ is sufficiently large, then $E':=\ker \phi$
satisfies the claim. 
\end{proof}

\begin{lem}\label{lem:RDP:irred-C}
Let $C$ be an effective divisor on $X$.
For $(r,\cal L) \in {\Bbb Z}_{>0} \times \Pic(C)$, 
the moduli stack of locally free sheaves $E$ on $C$ such that
$(\rk E,\det E)=(r,\cal L)$ is irreducible.
\end{lem}

\begin{proof}
For a locally free sheaf $E$ on $C$
we consider $\phi:H^0(X,E(n)) \otimes {\cal O}_C(-n) \to E$.
Assume that
$\phi$ is surjective. Then
there is a subvector space $V \subset H^0(X,E(n))$
of $\dim V=r-1$
such that $\psi:V \otimes {\cal O}_C(-n) \to E$ is injective
for any point of $C$.
Then $\coker \psi$ is a line bundle which is isomorphic to
$\det(E) \otimes {\cal O}_C((r-1)n)$.  
Hence $E$ is parametrized an affine space
$\Ext^1_{{\cal O}_C}({\cal L}\otimes {\cal O}_C((r-1)n),
{\cal O}_C(-n) \otimes V)
=H^1(C,{\cal L}^{\vee}(-rn) \otimes V)$.
Since the surjectivity of $\phi$ is an open condition and
$\phi$ is surjective for $n \gg 0$,
we get our claim. 
\end{proof}

{\it Proof of Proposition \ref{prop:C(G)}.} 
(1)
For a locally free $\beta$-twisted sheaf $G$ on $X$,
we set $g_{ij}:=\chi(G,L_{ij})$.
Let $\alpha \in \oplus_{i=1}^n \oplus_{j=1}^{s_i} 
{\Bbb Q} [C_{ij}]$ be a ${\Bbb Q}$-divisor
such that $\rk G(\alpha,C_{ij})=g_{ij}$.
We take a locally free sheaf $A \in \Coh(X)$ such that
$c_1(A)/\rk A=\alpha$.
Then $\chi(G \otimes A,L_{ij})=
\rk A(g_{ij}-\rk G (\alpha,C_{ij}))=0$ for all $i,j$.
By Lemma \ref{lem:generic-vanishing}, 
there is a torsion free $\beta$-twisted sheaf
$E$ on $X$ such that
$\tau(E)=\tau(G \otimes A)-n\tau({\Bbb C}_x)$
and $\Hom(E,E(K_X+C_{ij}))_0 =0$ for all $i,j$.
We consider the restriction morphism
\begin{equation}
\phi_{ij}:\Def(X,E) \to \Def(C_{ij},E_{|C_{ij}}).
\end{equation}
Since $\Ext^2(E,E(-C_{ij}))_0=0$, we get $\Ext^2(E,E)_0=0$.
Thus $\Def(X,E)$ is smooth. We also have
the smoothness of $\Def(C_{ij},E_{|C_{ij}})$, by the locally
freeness of $E_{|C_{ij}}$.
We consider the homomorphism of the tangent spaces
\begin{equation}
\Ext^1_{{\cal O}_X}(E,E)_0 \to 
\Ext^1_{{\cal O}_{C_{ij}}}(E_{|C_{ij}},E_{|C_{ij}})_0.
\end{equation}   
Then it is surjective by $\Ext^2(E,E(-C_{ij}))_0=0$.
Therefore $\phi$ is submersive.
By the equivalence
$\varphi:\Coh^{\beta}(C_{ij}) \to \Coh(C_{ij})$ 
in \eqref{eq:twisted-equiv},
we have an isomorphism 
$\Def(C_{ij},E_{|C_{ij}}) \to \Def(C_{ij},\varphi(E_{|C_{ij}}))$.
Since $\chi(E,L_{ij})=0$, $\det (E_{|C_{ij}} \otimes L_{ij}^{\vee})
={\cal O}_{C_{ij}}(\rk E)$.
\begin{NB}
$E_{|C_{ij}} \otimes L_{ij}^{\vee} \cong
\varphi(E_{|C_{ij}}) \otimes \varphi(L_{ij})^{\vee}$.
\end{NB}
Then Lemma \ref{lem:RDP:irred-C} implies that 
$E$ deforms to a $\beta$-twisted
sheaf such that
$E_{|C_{ij}} \cong L_{ij}(1)^{\oplus \rk E}$.
Since these conditions are open condition,
there is a locally free $\beta$-twisted sheaf
$E$ such that $E_{|C_{ij}} \cong L_{ij}(1)^{\oplus \rk E}$
for all $i,j$.
By taking the double dual of $E$ and using 
Lemma \ref{lem:tilting:pull-back}, we get (1).

(2) 
Note that $L_{ij}=A_{p_i} \otimes {\cal O}_{C_{ij}}(-1)$.
By Proposition \ref{prop:tilting:G-1Per-irred} and
Proposition \ref{prop:tilting:S-T}, 
we get the claim.
For (3), we use Proposition \ref{prop:tilting:G0-Per-irred}
and Proposition \ref{prop:tilting:S-T}.
\qed

\begin{NB}

\begin{defn}
For ${\bf b}=(b_1,b_2,\dots,b_s)$,
let ${\cal S}({\bf b})$ be the full subcategory of $\Coh(X)$ 
generated by subsheaves of ${\cal O}_{C_i}(b_i)$,
$1 \leq i \leq s$.
Let ${\cal T}({\bf b})$ be the full subcategory
of $\Coh(X)$ such that
\begin{equation}
{\cal T}({\bf b}):=\{E \in \Coh(X)| \Hom(E,{\cal O}_{C_i}(b_i))=0, 
1 \leq i \leq s \}.
\end{equation}
\end{defn}

\begin{prop}\label{prop:torsion-pair}
$T({\bf b}):=({\cal T}({\bf b}),{\cal S}({\bf b}))$ 
is a torsion pair of $\Coh(X)$. That is,
for $E \in \Coh(Y)$, there is an exact sequence
\begin{equation}
0 \to E' \to E \to E'' \to 0
\end{equation}
such that $E' \in {\cal T}({\bf b})$ and $E'' \in {\cal S}({\bf b})$.
Since $\Hom(E',E'')=0$, this decomposition is unique.
\end{prop}

\begin{proof}
We first assume that $\Supp(E) \subset Z$.
If $\Hom(E,{\cal O}_{C_i}(b_i)) \ne 0$,
then we have an exact sequence
\begin{equation}
0 \to E_1 \to E \to F_1 \to 0,
\end{equation}
where $F$ is a non-zero subsheaf of ${\cal O}_{C_i}(b_i)$.
Since $(c_1(E),L)>(c_1(E_1),L)$, by the induction on 
$(c_1(E),L)$, we get a desired decomposition.
We treat the general case.
We fix a locally free sheaf $G$ on $X$ and set
$G_n:=G \otimes L^{\otimes n}$.
We consider 
$\phi:\pi^*(\pi_*(G_n^{\vee} \otimes E)) \otimes G_n
\to E$.
If $n$ is sufficiently large,
then $\Hom(G_n,{\cal O}_{C_i}(b_i))=0$ for all $i$.
Hence $\Hom(\im \phi,{\cal O}_{C_i}(b_i))=0$ for all $i$.
Since $\Supp(\coker \phi) \subset Z$, we have a decomposition
\begin{equation}
0 \to E_1' \to \coker \phi \to E_2' \to 0
\end{equation}
such that $E_1' \in {\cal T}({\bf b})$ and
$E_2' \in {\cal S}({\bf b})$.
Since ${\cal T}({\bf b})$ is closed under extensions,
we get a desired extension.
\end{proof}

\end{NB}

\begin{defn}
\begin{enumerate}
\item[(1)]
We set 
$\Per(X/Y,\{ L_{ij} \}):={\cal C}(E)$ and
$\Per(X/Y,\{ L_{ij} \})^*:={\cal C}(E)^*$.
\item[(2)]
If $\beta$ is trivial, then we can write
$L_{ij}={\cal O}_{C_{ij}}(b_{ij})$.
In this case, we set
$\Per(X/Y,{\bf b}_1,\dots,{\bf b}_n):=\Per(X/Y,\{L_{ij}\})$ and
$\Per(X/Y,{\bf b}_1,\dots,{\bf b}_n)^*:=\Per(X/Y,\{ L_{ij} \})^*$,
where
${\bf b}_i:=(b_{i1},b_{i2},\dots,b_{i s_i})$.
\end{enumerate}
\end{defn}

\begin{rem}
If ${\bf b}_i(0)=(-1,-1,\dots,-1)$, 
then $\Per(X/Y,{\bf b}_1(0),...,{\bf b}_n(0))={^{-1}\Per(X/Y)}$.
\end{rem}

\begin{defn}
We set
\begin{equation}
\begin{split}
A_0({\bf b}_i):=& A_{p_i},\\
A_0({\bf b}_i)^* :=& A_{p_i} \otimes \omega_{Z_i}.
\end{split}
\end{equation}
\end{defn}

We collect easy facts on $A_0({\bf b}_i)$ and
$A_0({\bf b}_i)^*$ which follow from 
Lemma \ref{lem:G-1per:characterize}
and Lemma \ref{lem:G-1per:characterize*}.
\begin{lem}\label{lem:A_0}
\begin{enumerate}
\item[(1)]
\begin{enumerate}
\item
For $E=A_0({\bf b}_i)$, we have
\begin{equation}
\begin{split}
\Hom(E,{\cal O}_{C_{ij}}(b_{ij}))=
\Ext^1(E,{\cal O}_{C_{ij}}(b_{ij}))&=0,\; 1 \leq j \leq s_i
\end{split}
\end{equation}
and there is an exact sequence
\begin{equation}
\begin{CD}
0 @>>> F @>>> E @>>> {\Bbb C}_x @>>> 0\\
\end{CD}
\end{equation}
such that $F$ is a successive extension of
${\cal O}_{C_{ij}}(b_{ij})$ and $x \in Z_i$.
\item
Conversely if $E$ satisfies these conditions, then
$E \cong A_0({\bf b}_i)$. 
\end{enumerate}
\item[(2)]
\begin{enumerate}
\item
For $E=A_0({\bf b}_i)^*$, we have
\begin{equation}
\begin{split}
\Hom({\cal O}_{C_{ij}}(b_{ij}),E)=
\Ext^1({\cal O}_{C_{ij}}(b_{ij}),E)&=0,\; 1 \leq j \leq s_i
\end{split}
\end{equation}
and there is an exact sequence
\begin{equation}
\begin{CD}
0 @>>> E @>>> F @>>> {\Bbb C}_x @>>> 0\\
\end{CD}
\end{equation}
such that $F$ is a successive extension of
${\cal O}_{C_{ij}}(b_{ij})$ and $x \in Z_i$.
\item
Conversely if $E$ satisfies these conditions, then
$E \cong A_0({\bf b}_i)^*$. 
\end{enumerate}
\end{enumerate}
\end{lem}

\begin{NB}

Let $\alpha \in \oplus_{i=1}^s {\Bbb Q} [C_i]$ be a ${\Bbb Q}$-divisor
such that $(\alpha,C_i)=b_i+1$.
We take a locally free sheaf $E_{{\bf b}}$ such that
$c_1(E_{{\bf b}})/\rk E_{{\bf b}}=\alpha$.
Then $\chi(E_{{\bf b}},{\cal O}_{C_i}(b_i))=
\rk(E_{{\bf b}})(b_i+1-(\alpha,C_i))=0$ for all $i$ and
$\chi(E_{{\bf b}},{\Bbb C}_x)=\rk(E_{{\bf b}})$.
For a 1-dimensional sheaf $E$ with
$\Supp(E) \subset Z$, we have 
$\rk(E_{{\bf b}})|\chi(E_{{\bf b}},E)$.

\begin{lem}\label{lem:B-stability-deg0}
Let $E$ be an $E_{{\bf b}}$-twisted stable sheaf
such that $\Supp(E) \subset Z$ and $\chi(E_{{\bf b}},E)=0$.
Then $E \cong {\cal O}_{C_i}(b_i)$.
\end{lem}

\begin{proof}
We note that $\chi({\cal O}_{C_i}(b_i),E)=-(C_i,c_1(E))$.
Hence we can find ${\cal O}_{C_i}(b_i)$ such that
$\chi({\cal O}_{C_i}(b_i),E)>0$.
Then we have a homomorphism
$\phi:{\cal O}_{C_i}(b_i) \to E$ or
$\psi:E \to {\cal O}_{C_i}(b_i)$.
Since $E$ and ${\cal O}_{C_i}(b_i)$
are $E_{{\bf b}}$-twisted stable sheaves with
$\chi(E_{{\bf b}},E)=\chi(E_{{\bf b}},{\cal O}_{C_i}(b_i))=0$,
we see that
$E \cong {\cal O}_{C_i}(b_i)$. 
\end{proof}
We next treat $E_{{\bf b}}$-twisted stable
sheaves with $\chi(E_{{\bf b}},E)=\pm \rk E_{{\bf b}}$.
By \cite[sect. 4.1]{Y:action}, we get the following two lemmas.

\begin{lem}(cf. \cite[Prop. 4.6]{Y:action}.)\label{lem:B-stability}
Let $v \in \oplus_{i=1}^s {\Bbb Z} v({\cal O}_{C_i}(b_i))$ 
be a Mukai vector.
\begin{enumerate}
\item[(1)]
There is a $E_{{\bf b}}$-twisted stable sheaf $E$ with
$v(E)=v+v({\Bbb C}_x)$ if and only if 
$v \in \oplus_{i=1}^s {\Bbb Z}_{\geq 0} v({\cal O}_{C_i}(b_i))$ and  
$\langle v^2 \rangle \geq -2$.
\item[(2)]
There is a $E_{{\bf b}}$-twisted stable sheaf $E$ with
$v(E)=-(v+v({\Bbb C}_x))$ if and only if 
$v \in \oplus_{i=1}^s {\Bbb Z}_{\leq 0} v({\cal O}_{C_i}(b_i))$ and  
$\langle v^2 \rangle = -2$.
\end{enumerate}
\end{lem}

\begin{proof}
(1) 
We set $v_i:=v({\cal O}_{C_i}(b_i))$.
Let $E$ be an $E_{{\bf b}}$-twisted stable sheaf
with $v(E)=v+v({\Bbb C}_x)$.
Then since $E \otimes K_X \cong E$,
$\Ext^2(E,E) \cong \Hom(E,E)^{\vee} \cong {\Bbb C}$.
Hence $0 \geq \langle v^2 \rangle =
\langle v(E)^2 \rangle \geq -2$.
If $\langle v^2 \rangle=0$, then
$v=0$ and if $\langle v^2 \rangle=-2$, then
by the theory of root system of $A,D,E$ type,
$\pm v \in \oplus_{i=1}^s {\Bbb Z}_{\geq 0} v_i$.
Since $\Supp(E)$ is effective, we get
$v \in \oplus_{i=1}^s {\Bbb Z}_{\geq 0} v_i$.

Conversely assume that $v$ satisfies the assumptions.
Obviously ${\Bbb C}_x$ and ${\cal O}_{C_i}(b_i+1)$ are
$E_{{\bf b}}$-twisted stable sheaf
with Mukai vectors $v({\Bbb C}_x)$ and $v_i+v({\Bbb C}_x)$.
For $w=v+v({\Bbb C}_x)$ and $v_i$ with
$\langle w,v_i \rangle=1$, by using \cite[Lem. 4.1]{Y:action},
we have an isomorphism
$M_H^{E_{{\bf b}}}(w) \to M_H^{E_{{\bf b}}}(w+v_i)$
by sending $E \in M_H^{E_{{\bf b}}}(w)$ to
the sheaf $F$ fitting in the universal extension
\begin{equation}
0 \to \Ext^1(E,{\cal O}_{C_i}(b_i))^{\vee} \otimes {\cal O}_{C_i}(b_i)
\to F \to E \to 0.
\end{equation} 
Since all positive roots are obtained from $v_i$ by these operations,
 we get the claim. 

(2) The proof is almost similar.
We only remark that
${\cal O}_{C_i}(b_i-1)$ is an $E_{{\bf b}}$-twisted stable
sheaf with $v(E)=v_i-v({\Bbb C}_x)$ and
for $w=v+v({\Bbb C}_x)$ and $v_i$ with
$\langle w,v_i \rangle=1$, by using \cite[Lem. 4.2]{Y:action},
we have an isomorphism
$M_H^{E_{{\bf b}}}(w) \to M_H^{E_{{\bf b}}}(w+v_i)$
by sending $E \in M_H^{E_{{\bf b}}}(w)$ to
the sheaf $F$ fitting in the universal extension
\begin{equation}
0 \to E \to F \to \Ext^1({\cal O}_{C_i}(b_i),E) \otimes {\cal O}_{C_i}(b_i)
\to 0.
\end{equation} 

\end{proof}

\begin{lem}\label{lem:B-stability2}
\begin{enumerate}
\item[(1)]
For an $E_{\bf b}$-twisted stable sheaf $E$
in Lemma \ref{lem:B-stability} (1) and an exact sequence
\begin{equation}\label{eq:B-stability1}
\begin{CD}
0 @>>> F @>>> E @>>> {\Bbb C}_x @>>> 0,\\
\end{CD}
\end{equation}
$F$ is an $E_{\bf b}$-twisted semi-stable sheaf 
with $\chi(E_{\bf b},F)=0$.
In particular, 
$F$ is a successive extension of 
${\cal O}_{C_i}(b_i)$, $i=1,...,s$.
\item[(2)]
For an $E_{\bf b}$-twisted stable sheaf $E$
in Lemma \ref{lem:B-stability} (2) and a nontrivial extension
\begin{equation}\label{eq:B-stability2}
\begin{CD}
0 @>>> E @>>> F @>>> {\Bbb C}_x @>>> 0,\\
\end{CD}
\end{equation}
$F$ is an $E_{\bf b}$-twisted semi-stable sheaf 
with $\chi(E_{\bf b},F)=0$.
In particular, $F$ is a successive extension of 
${\cal O}_{C_i}(b_i)$, $i=1,...,s$.
\item
[(3)]
Conversely, let 
$E$ be a coherent sheaf fitting in the exact sequences
\eqref{eq:B-stability1} or \eqref{eq:B-stability2}.
If $\Hom(E,{\cal O}_{C_i}(b_i))=0$ (resp. $\Hom({\cal O}_{C_i}(b_i),E)=0$)
for $i=1,...,s$, then $E$ is $E_{\bf b}$-twisted stable.
\end{enumerate}
\end{lem}

\begin{proof}
(1) For a non-trivial subsheaf $E'$ of $E$,
By the minimality of $\chi(E_{\bf b},E)$, we have
$\chi(E_{\bf b},E') \leq 0$.
Hence $F$ is $E_{\bf b}$-twisted semi-stable.

(2)
For a non-trivial subsheaf $E'$ of $E$,
$\chi(E_{\bf b},E') \leq -\rk E_{\bf b}$.
Since $\chi(E_{\bf b},{\Bbb C}_x)=\rk E_{\bf b}$, 
we see that $F$ is $E_{\bf b}$-twisted semi-stable.

(3)
If the sheaf $E$ in \eqref{eq:B-stability1} has a quotient
$E \to E'$ such that $E'$ is a $E_{\bf b}$-twisted stable sheaf
with $\chi(E_{\bf b},E') \leq 0$, then
we have a non-zero homomorphism $F \to E'$.
By the assumption of $F$, $E' \cong {\cal O}_{C_i}(b_i)$,
which implies that $\Hom(E,{\cal O}_{C_i}(b_i))\ne 0$.
Therefore $E$ is $E_{\bf b}$-twisted stable.  

Assume that the sheaf $E$ in \eqref{eq:B-stability2} has a subsheaf
$E' \subset E$ such that $E'$ is a $E_{\bf b}$-twisted stable sheaf
with $\chi(E_{\bf b},E') \geq 0$.
Since $E'$ is a subsheaf of $F$, our assumption of $F$ implies that
$E' \cong {\cal O}_{C_i}(b_i)$, which implies 
$\Hom({\cal O}_{C_i}(b_i),E) \ne 0$. Therefore
$E$ is $E_{\bf b}$-twisted stable.
\end{proof}

\begin{rem}
There is a filtration
\begin{equation}
0 \subset E_1 \subset E_2 \subset \cdots \subset E_n=E
\end{equation}
such that $E_i/E_{i-1} \cong {\cal O}_{C_{n_i}}(b_{n_i})$, $i>1$
and $E_1 \cong {\cal O}_{C_{n_1}}(b_{n_1}-1)$: 
We take a quotient $\phi:F' \to {\cal O}_{C_i}(b_i)$.
If $\psi:\ker \phi \to F' \to {\Bbb C}_x$ is not surjective,
then $\ker \phi \subset E$, which contradicts the assumption.
Hence $\psi$ is surjective. Then 
we have two extensions
\begin{equation}
\begin{split}
0 \to \ker \psi \to E \to {\cal O}_{C_i}(b_i) \to 0,\\
0 \to \ker \psi \to \ker \phi \to {\Bbb C}_x \to 0.
\end{split}
\end{equation}
By using an induction, we see that 
$E$ is twisted stable.
\end{rem}

\begin{rem}
We note that $\langle v^2 \rangle \leq 0$ and the equality holds if
$v=v({\Bbb C}_x)$.
%For the proof of Lemma \ref{lem:B-stability2} (2),
%we note that our assumption implies that
\end{rem}

\begin{defn}\label{defn:A_0}
Let $A_0({\bf b})$ and $A_0({\bf b})^*$ 
be the $E_{{\bf b}}$-twisted stable sheaves
with 
\begin{equation}
\begin{split}
v(A_0({\bf b})) & =v({\Bbb C}_x)+
\sum_{i=1}^s a_i v({\cal O}_{C_i}(b_i)),\\
v(A_0({\bf b})^*) & =-v({\Bbb C}_x)+
\sum_{i=1}^s a_i v({\cal O}_{C_i}(b_i)).
\end{split}
\end{equation}
Since $\langle v(A_0({\bf b}))^2 \rangle=
\langle v(A_0({\bf b})^*)^2 \rangle=-2$,
$A_0({\bf b})$ and $A_0({\bf b})^*$ are uniquely determined.
\end{defn}

\begin{rem}
For a $(-2)$-vector $v$, there is at most one stable object
$E$ with $v(E)=v$.
Indeed if $E$ and $E'$ are stable object with
$v(E)=v(E')=v$, then
$\chi(E,E')=2$ implies that
there is a non-zeto homomorphism
$E \to E'$ or $E' \to E$. Then
the stability implies that $E \cong E'$. 
\end{rem}

\begin{rem}\label{rem:A_0}
By the proof of Lemma \ref{lem:B-stability},
we also see that
there is a surjective homomorphism
$A_0({\bf b}) \to {\cal O}_{C_i}(b_i+1)$
and injective homomorphism
${\cal O}_{C_i}(b_i-1) \to A_0({\bf b})^*$
for all $i$.
Then for any point $x \in Z$,
we have exact sequences 
\eqref{eq:B-stability1} and \eqref{eq:B-stability1}
for $A_0({\bf b})$ and
$A_0({\bf b})^*$ respectively.  
\end{rem}

\begin{lem}\label{lem:B-stability=G-stability}
Let $G$ be a locally free sheaf such that
$\chi(G,{\cal O}_{C_i}(b_i))<0$, $i>0$ and $\chi(G,A_0({\bf b}))>0$.
Then $E_{\bf b}$-twisted stable sheaves in Lemma \ref{lem:B-stability} (1)
are $G$-twisted stable. 
\end{lem}

\begin{proof}
Let $E$ be an $E_{\bf b}$-twisted stable sheaf such that
$v(E)={\Bbb C}_x+v$, 
$v \in \oplus_{i=1}^s {\Bbb Z}_{\geq 0} v({\cal O}_{C_i}(b_i))$ and  
$\langle v^2 \rangle \geq -2$.
We may assume that $\langle v^2 \rangle=-2$.
Then $\chi(G,E)>0$.
Let $E_1$ be a proper subsheaf of $E$.
We shall prove that
$E_1 \in {\cal S}({\bf b})$.
For a point $x \in \Supp(E/E_1)$,
we take a surjection $\phi:E \to E/E_1 \to {\Bbb C}_x$.
Then Lemma \ref{lem:B-stability2} (1) implies
that 
$\ker \phi$ is a successive extension of
${\cal O}_{C_i}(b_i)$.
Thus we get $\ker \phi \in {\cal S}({\bf b})$.
Since $E_1 \in \ker \phi$, we get the claim.
Then we have $\chi(G,E_1)<0$, which implies that
$E$ is $G$-twisted stable. 
\end{proof}

\begin{rem}
For a purely 1-dimensional sheaf $E$ and 
a subsheaf $E' \subset E$ such that
$E/E'$ is purely 1-dimensional,
we have a quotient $E^{\vee}[1] \to {E'}^{\vee}[1]$
and 
\begin{equation}
\frac{\chi(E_{{\bf b}},E')}{(c_1(E'),L)} \leq
 \frac{\chi(E_{{\bf b}},E)}{(c_1(E),L)}
\end{equation}
if and only if
\begin{equation}
\frac{\chi(E_{{\bf b}}^{\vee},{E'}^{\vee}[1])}{(c_1({E'}^{\vee}[1]),L)} \geq
 \frac{\chi(E_{{\bf b}}^{\vee},E^{\vee}[1])}{(c_1(E^{\vee}[1]),L)}.
\end{equation}
Hence $E$ is $E_{{\bf b}}$-twisted stable if and only if
$E^{\vee}[1]$ is $E_{{\bf b}}^{\vee}$-twisted stable.
In particular,
\begin{equation}
A_0({\bf b})^{\vee}[1] \cong A_0(-{\bf b}+2{\bf b}_0)^*,\;
(A_0({\bf b})^*)^{\vee}[1] \cong A_0(-{\bf b}+2{\bf b}_0).
\end{equation}
\end{rem}

We collect easy facts on $A_0({\bf b})$ and
$A_0({\bf b})^*$.
\begin{lem}\label{lem:NB:A_0}
\begin{equation}
\begin{split}
\Hom(A_0({\bf b}),{\cal O}_{C_i}(b_i))=
\Ext^1(A_0({\bf b}),{\cal O}_{C_i}(b_i))&=0,\\ 
\Hom({\cal O}_{C_i}(b_i),A_0({\bf b})^*)=
\Ext^1({\cal O}_{C_i}(b_i),A_0({\bf b})^*)&=0
\end{split}
\end{equation}
for all $i$.
\begin{equation}
\begin{split}
\Hom(A_0({\bf b}),{\Bbb C}_x)
=\Hom(A_0({\bf b}),A_0({\bf b})) &={\Bbb C},\\
\Hom(A_0({\bf b})^*,{\Bbb C}_x)
=\Hom(A_0({\bf b})^*,A_0({\bf b})^*)&={\Bbb C}.
\end{split}
\end{equation}
\end{lem}

\begin{proof}
We only prove the assertions for $A_0({\bf b})$. 
Obviously $\Hom(A_0({\bf b}),{\cal O}_{C_i}(b_i))=0$
and $\Hom(A_0({\bf b}),A_0({\bf b}))={\Bbb C}$.
If there is a non-trivial extension
\begin{equation}
0 \to {\cal O}_{C_i}(b_i) \to E \to A_0({\bf b}) \to 0,
\end{equation}
then $E$ is a $E_{{\bf b}}$-twisted stable sheaf.
Hence $\langle v(E)^2 \rangle=-2$.
On the other hand, by our choice of
$v( A_0({\bf b}))$, we have
$\langle v( A_0({\bf b})),v({\cal O}_{C_i}(b_i)) \rangle \leq 0$,
which implies that
$\langle v(E)^ \rangle \leq -4$.
Therefore $\Ext^1(A_0({\bf b}),{\cal O}_{C_i}(b_i))=0$.
Then we have $\Hom(A_0({\bf b}),{\Bbb C})=\Hom(A_0({\bf b}),A_0({\bf b}))$
by using \eqref{eq:B-stability1}.
\end{proof}

\begin{rem}
Let ${\cal M}_i$ be full sheaves in an analytic neighborhood
of the fundamental class $Z$
such that $(c_1({\cal M}_i),C_j)=\delta_{ij}$.
We set ${\cal L}:=\bigotimes_{i=1}^s \det{\cal M}_i^{\otimes (b_i+1)}$.
%We set
%
%\begin{equation}
%\Per(X/Y,{\bf b})_Z:=\{ E \in \Per(X/Y,{\bf b})| \Supp(E) \subset Z \}.
%\end{equation}
%
Then we have an equivalence
\begin{equation}
\begin{matrix}
\Per(X/Y,{\bf b}_0)_Z & \to & \Per(X/Y,{\bf b})_Z\\
E & \mapsto & E \otimes {\cal L}
\end{matrix}
\end{equation}
By this equivalence,
$A_0({\bf b}) \cong {\cal O}_Z \otimes {\cal L}$ and
$A_0({\bf b})^* \cong {\cal O}_Z(Z) \otimes {\cal L}$.
\end{rem}

%
%Let $A_0$ be the simple sheaf on $X$ fitting in an exact sequence
%\begin{equation}
%0 \to F \to A_0 \to {\Bbb C}_p \to 0, 
%\end{equation}
%where $F$ is a successive extensions of
%${\cal O}_{C_i}(b_i)$ with 
%$c_1(F)=Z$ and $p \in Z$.
%$A_0 \cong 
%{\cal O}_Z \otimes {\cal L}$
%and is uniquely determined.

\begin{lem}\label{lem:generate}
Assume that $E \in {\cal T}({\bf b})$ satisfies $\Supp(E) \subset Z$.
Then $E$ is a successive extension of quotients of $A_0({\bf b})$.
\end{lem}

\begin{proof}
For a 0-dimensional sheaf on $X$,
obviously the claim holds.
By dividing its 0-dimensional submodule $T$ of $E$,
we assume that $E$ is purely 1-dimensional.
Then for a non-zero homomorphism
$\phi:A_0({\bf b}) \to E$, $(c_1(\coker \phi),L)<(c_1(E),L)$,
where $L$ is an ample divisor on $Z$.
Thus by the induction on $(c_1(E),L)$, it is sufficient to prove that
$\Hom(A_0({\bf b}),E) \ne 0$.
We assume that $\Hom(A_0({\bf b}),E) = 0$. 
Since the intersection matrix of $C_i, i=1,...,s$ is negative definite,
$\chi({\cal O}_{C_{i_1}}(b_{i_1}),E)=-(c_1(E),C_{i_1})>0$ for an $i_1$.
Since $E \in {\cal T}({\bf b})$ and 
$\Hom(A_0({\bf b}),{\cal O}_{C_{i_1}}(b_{i_1}))=0$,
we have a homomorphism $\phi_1:{\cal O}_{C_{i_1}}(b_{i_1}) \to E$
and $\Hom(A_0({\bf b}),\coker \phi_1)=0$.
If $\coker \phi_1 \ne 0$, then we have
a homomorphism $\phi_2:{\cal O}_{C_{i_2}}(b_{i_2}) \to \coker \phi_1$
and $\Hom(A_0({\bf b}),\coker \phi_2)=0$.
Continueing this procedure, we see that 
$\coker \phi_n=0$ for an $n$.
Then $E$ 
is a successive extensions of ${\cal O}_{C_i}(b_i)$,
which contradicts to $E \in {\cal T}({\bf b})$.
Therefore $\Hom(A_0({\bf b}),E) \ne 0$ and our claim holds.
\end{proof}

\begin{cor}\label{cor:generate}
Let $E$ be a coherent sheaf on $X$ such that
$\Supp(E) \subset Z$.
Then $E \in {\cal S}({\bf b})$ if and only if
$\Hom(A_0({\bf b}),E)=0$.
\end{cor}

\begin{prop}\label{prop:irreducible-obj}
\begin{enumerate}
\item[(1)]
$A_0({\bf b})$ and ${\cal O}_{C_i}(b_i)[1]$, $i=1,2,\dots,s$ are 
irreducible objects of $\Per(X/Y,{\bf b})$.
\item[(2)]
Every $E \in \Per(X/Y,{\bf b})_Z$ is a successive extensions of
$A_0({\bf b})$ and ${\cal O}_{C_i}(b_i)[1]$, $i=1,2,\dots,s$ in 
$\Per(X/Y,{\bf b})$.
\item[(3)]
Let $E$ be an object of $\Per(X/Y,{\bf b})$.
If $\Hom(E,A_0({\bf b}))=\Hom(E,{\cal O}_{C_i}(b_i)[1])=
\Hom(E,{\Bbb C}_x)=0$ for all $x \in X \setminus Z$ and
all $i$, then $E=0$.  
\end{enumerate}
\end{prop}

\begin{proof}
(1)
Let $E_1 \ne 0$ be a subobject of $A_0({\bf b})$ in $\Per(X/Y,{\bf b})$. 
Then $E_1 \in {\cal T}({\bf b})$, which implies that
$\Hom(A_0({\bf b}),E_1) \ne 0$ by Lemma \ref{lem:generate}.
Since $\Hom(A_0({\bf b}), A_0({\bf b})) \cong {\Bbb C}$,
we have $E_1=A_0({\bf b})$.

Let $E_2 \ne 0$ be a quotient object of ${\cal O}_{C_i}(b_i)[1]$ 
in $\Per(X/Y,{\bf b})$. 
Then $E_2 \in {\cal S}({\bf b})$, which implies that
$\Hom(E_2,{\cal O}_{C_j}(b_j)[1]) \ne 0$ for a $j$.
Since $\Hom({\cal O}_{C_i}(b_i),{\cal O}_{C_j}(b_j)) \cong {\Bbb C}$
or $0$ according as $j=i$ or $j \ne i$,
Hence we have $E_2={\cal O}_{C_i}(b_i)[1]$.

(2)
We first note that the claim holds for ${\Bbb C}_x$, $x \in Z$.
For a subsheaf ${\cal O}_{C_i}(b_i-n)$ of ${\cal O}_{C_i}(b_i)$,
${\cal O}_{C_i}(b_i-n)[1]$ fits in an exact sequence
\begin{equation}
0 \to A \to {\cal O}_{C_i}(b_i-n)[1] \to {\cal O}_{C_i}(b_i)[1] \to 0
\end{equation}
in $\Per(X/Y,{\bf b})$, where $A \in {\cal T}({\bf b})$ 
is a 0-dimensional sheaf.
Hence the claim holds for all objects of ${\cal S}({\bf b})$. 
For a quotient $\phi:A_0({\bf b}) \to E$, the proof of 
Lemma \ref{lem:B-stability=G-stability} implies
$\ker \phi \in {\cal S}({\bf b})$.
Then $E$ is an extension of $(\ker \phi)[1]$ by
$A_0({\bf b})$.
Therefore the claim follows from Lemma \ref{lem:generate}.

(3)
Since $\Hom(E,{\Bbb C}_x)=0$ for $x \in X \setminus Z$,
we have $\Supp(H^{-1}(E)), \Supp(H^0(E)) \subset Z$.
Then the claim follows from (2). 
\end{proof}

\begin{rem}
A proof in $\Coh(X)$:
By the above argument,
$\phi:E_1 \to A_0({\bf b})$ is surjective in $\Coh(X)$.
Since $\Ext^1(A_0({\bf b}),{\cal O}_{C_i}(b_i))=0$ for all $i$,
we see that $\Hom(\ker \phi,{\cal O}_{C_i}(b_i))=0$ for all $i$.
\end{rem}

By Proposition \ref{prop:tilting:-1Per-irred}, we have the following.
\begin{prop}\label{prop:irreducible-obj}
\begin{enumerate}
\item[(1)]
$A_0({\bf b})$ and ${\cal O}_{C_i}(b_i)[1]$, $i=1,2,\dots,s$ are 
irreducible objects of $\Per(X/Y,{\bf b})$.
\item[(2)]
${\Bbb C}_x, x \in Z$ is generated by
$A_0({\bf b})$ and ${\cal O}_{C_i}(b_i)[1]$, $i=1,2,\dots,s$.
In particular, irreducible objects of
$\Per(X/Y,{\bf b})$ are ${\Bbb C}_x, x \in X \setminus Z$,
$A_0({\bf b})$ and ${\cal O}_{C_i}(b_i)[1]$, $i=1,2,\dots,s$. 
\end{enumerate}
\end{prop}

Then
the following result follows from Proposition \ref{prop:tilting:generator}
and Lemma \ref{lem:tilting:R1=0}.

\begin{prop}\label{prop:Per-generator}
Let $G$ be a locally free sheaf of rank $r$ on $X$ such that
\begin{equation}
\chi(G,{\cal O}_{C_i}(b_i))=-r_i<0,\;
\sum_i a_i r_i<r.
\end{equation}
\begin{enumerate}
\item[(1)]
If
$G \in {\cal T}({\bf b}),\;
\Ext^1(A_0({\bf b}),G)=0$, then
$R^1 \pi_*(G^{\vee} \otimes G)=0$ and
$G$ is a local projective generator of 
$\Per(X/Y,{\bf b})$. In particular,
$E \in {\cal T}({\bf b})$ if and only if
$R^1 \pi_*(G^{\vee} \otimes E)=0$ and 
$E \in {\cal S}({\bf b})$ if and only if
$\pi_*(G^{\vee} \otimes E)=0$.
\item[(2)]
$\Ext^1(G,A_0({\bf b}))=
\Hom(G,{\cal O}_{C_i}(b_i))=0$ for all $i$
if and only if
$R^1 \pi_*(G^{\vee} \otimes G)=0$.
\end{enumerate}
\end{prop}

\begin{rem}
Assume that ${\bf R}\pi_*(G^{\vee} \otimes E)=\pi_*(G^{\vee} \otimes E)$.
Let $V$ ve a locally free sheaf on $Y$ with a surjection
$V \to \pi_*(G^{\vee} \otimes E)$.
Then we have a morphism $V \to {\bf R}\pi_*(G^{\vee} \otimes E)$
in ${\bf D}(Y)$ which induces the surjective homomorphism
$V \to H^0({\bf R}\pi_*(G^{\vee} \otimes E))=\pi_*(G^{\vee} \otimes E)$.
Then we have a morphism
$\pi^*(V) \otimes G \to 
{\bf L}\pi^*({\bf R}\pi_*(G^{\vee} \otimes E)) \otimes G
\overset{e}{\to} E$ such that
$(e \otimes 1_{G^{\vee}}) \circ \phi$ in the following diagram 
induces an isomorphsim ${\bf R}\pi_*((e \otimes 1_{G^{\vee}}) \circ \phi)$:
\begin{equation}
\begin{CD}
\pi^*(V) @>>>
{\bf L}\pi^*({\bf R}\pi_*(G^{\vee} \otimes E)) 
@. \\
@VVV @VV{\phi}V @.\\
\pi^*(V) \otimes G \otimes G^{\vee} @>>>
{\bf L}\pi^*({\bf R}\pi_*(G^{\vee} \otimes E)) \otimes G \otimes G^{\vee}
@>{e \otimes 1_{G^{\vee}}}>> E \otimes G^{\vee}.
\end{CD}
\end{equation} 
Hence $F:=\mathrm{Cone}(\pi^*(V) \otimes G \to E)[1]$ satisfies
${\bf R}\pi_*(G^{\vee} \otimes F) \in \Coh(Y)$.
\end{rem}

\begin{lem}
For a bounded complex $E^{\bullet}$ of coherent sheaves
$E^i$ on $X$, 
we take a locally free resolution 
$0 \to V_{-1} \to V_0 \to {\cal O}_X \to 0$
of ${\cal O}_X$ such that
$R^1 \pi_*((G \otimes V_0)^{\vee} \otimes E^i)=0$ for all $i$.
Then
$R^1 \pi_*((G \otimes V_{-1})^{\vee} \otimes E^i)=0$ for all $i$
and
$E^{\bullet}$ is quasi-isomorphic to
$E^{\bullet} \otimes (V_0^{\vee} \to V_{-1}^{\vee})$.
\end{lem}

Let $E^{\bullet}$ be a bounded complex such that
$E^i \in {\cal T}({\bf b})$ for all $i$.
Then if $E^{\bullet}$ is exact in $\Coh(X)$, then
$E^{\bullet}$ is also exact in $\Per(X/Y,{\bf b})$.
Indeed for $d^i:E^i \to E^{i+1}$,
$\ker d^i=\im d^{i-1} \in {\cal T}({\bf b})$ for all $i$.
Hence $\ker d^i$ coincides with the kernel in $\Per(X/Y,{\bf b})$.
Then we see that $E^{\bullet}$ is exact in $\Per(X/Y,{\bf b})$. 
\end{NB}

\begin{NB}
\subsection{The category $\Per(X/Y,{\bf b})^*$.}
\label{subsect:Per(X/Y)^*}

\begin{defn}
We set
\begin{equation}
{\cal S}({\bf b})^*:=
\{E \in {\cal S}({\bf b})| 
\Hom({\cal O}_{C_i}(b_i),E)=0,\; i>0 \}.
\end{equation}
Let ${\cal T}({\bf b})^*$ be the full subcategory
of $\Coh(X)$ such that $E \in \Coh(X)$ belongs to
${\cal T}({\bf b})^*$ if and only if 
there is an exact sequence
\begin{equation}
0 \to E_1 \to E \to E_2 \to 0
\end{equation}
where $E_1 \in  {\cal T}({\bf b})$ and
$E_2$ is a succesive extensions of ${\cal O}_{C_i}(b_i)$,
$1 \leq i \leq s$.
\end{defn}

\begin{lem}
For $E \in {\cal S}({\bf b})$, there is a unique decomposition
\begin{equation}
0 \to E_1 \to E \to E_2 \to 0
\end{equation}
such that $E_1$ is a successive extensions of
${\cal O}_{C_i}(b_i)$, $i>0$ and
$E_2 \in  {\cal S}({\bf b})^*$.
\end{lem}

\begin{proof}
The uniqueness is obvious. Hence we prove the existence.
For $E \in  {\cal S}({\bf b})$,
assume that there is a non-zero homomorphism
$\phi:{\cal O}_{C_i}(b_i) \to E$.
Since $E$ is purely 1-dimensional,
$\phi$ is injective.
Since $\Hom(A_0({\bf b}),{\cal O}_{C_i}(b_i))=
\Ext^1(A_0({\bf b}),{\cal O}_{C_i}(b_i))=0$,
$\Hom(A_0({\bf b}),E) \cong \Hom(A_0({\bf b}),\coker \phi)$.
By Corollary \ref{cor:generate}, we get
$\coker \phi \in {\cal S}({\bf b})$.
If $\Hom({\cal O}_{C_j}(b_j),\coker \phi) \ne 0$,
then we apply the same procedure and finally 
we get a desired decomposition.
\end{proof}

\begin{lem}\label{lem:generate*}
Let $E$ be a coherent sheaf on $X$.
\begin{enumerate}
\item[(1)]
Every $E \in {\cal S}({\bf b})^*$ is
 a successive extension of subsheaves of $A_0({\bf b})^*$.
\item[(2)]
$E \in {\cal T}({\bf b})^*$ if and only if
$\Hom(E,A_0({\bf b})^*)=0$.
\item[(3)]
Every $E \in {\cal T}({\bf b})^*$ with
$\Supp(E) \subset Z$ is
 a successive extension of ${\cal O}_{C_i}(b_i)$, $i>0$
and ${\Bbb C}_x$, $x \in Z$.
\item[(4)]
$E \in {\cal S}({\bf b})^*$ if and only if
$\Supp(E) \subset Z$,
$\Hom({\cal O}_{C_i}(b_i),E)=0, i>0$ and
$\Hom({\Bbb C}_x,E)=0$, $x \in Z$. 
\end{enumerate}
\end{lem}

\begin{rem}
$\Hom({\Bbb C}_x,E)=0$, $x \in Z$ is not necessary from
Remark \ref{rem:A_0}.
\end{rem}

\begin{proof}
(1)
We note that $E \in {\cal S}({\bf b})^*$ 
is a purely 1-dimensional sheaf with 
$\Supp(E) \subset Z$.
By the induction on $c_1(E)$,
it is sufficient to prove that there is a non-zero homomorphism
$E \to A_0({\bf b})^*$.
Assume that $\Hom(E,A_0({\bf b})^*)=0$.
Since the intersection matrix of $C_i, i=1,...,s$ is negative definite,
$\chi(E,{\cal O}_{C_i}(b_i))=-(c_1(E),C_i)>0$.
By our assumption,
there is a homomorphism $\phi:E \to {\cal O}_{C_i}(b_i)$.
Since 
$\Hom({\cal O}_{C_i}(b_i-1),A_0({\bf b})^*) \ne 0$,
$\phi$ is surjective.
Since $\Ext^1({\cal O}_{C_i}(b_i),A_0({\bf b})^*)=0$ by
Lemma \ref{lem:A_0}, $\Hom(\ker \phi,A_0({\bf b})^*)=0$.
Applying the same procedure, we finally get a subsheaf $E'$ 
such that $E' \cong {\cal O}_{C_j}(b_j)$.
Then $\Hom({\cal O}_{C_j}(b_j),E) \ne 0$, which is a contradiction.
(2) is a consequence of (1).

(3)
Assume that $E \in {\cal T}({\bf b})^*$.
Since the 0-dimensional subsheaf of $E$ satisfies the claim,
we may assume that $E$ is purely 1-dimensional.
Assume that $E \ne 0$.
then $\chi({\cal O}_{C_i}(b_i),E)=-(C_i,c_1(E))>0$ for an $i$.
If we have a non-zero homomorphism 
$\phi:E \to {\cal O}_{C_i}(b_i)$,
then since $\Hom(E,{\cal O}_{C_i}(b_i-1))=0$,
$\phi$ is surjective.
Since $\Hom({\cal O}_{C_i}(b_i),A_0({\bf b})^*)=
\Ext^1({\cal O}_{C_i}(b_i),A_0({\bf b})^*)=0$,
(2) implies that $\ker \phi \in {\cal T}({\bf b})^*$. 
If  we have a non-zero homomorphism 
$\psi:{\cal O}_{C_i}(b_i) \to E$.
Then it is injective and $\coker \psi \in {\cal T}({\bf b})^*$. 
By the induction on $(c_1(E),L)$, we get our claim.
(4) follows from (3).
\end{proof}

\begin{prop}\label{prop:torsion-pair*}
$T({\bf b})^*:=({\cal T}({\bf b})^*,{\cal S}({\bf b})^*)$ 
is a torsion pair of $\Coh(X)$.
\end{prop}

\begin{defn}
Let $\Per(X/Y,{\bf b})^*$ be the tilting of 
$\Coh(X)$ with respect to $T({\bf b})^*$.
\end{defn}

Let ${\cal M}_i$ be full sheaves in an analytic neighborhood
of the fundamental class $Z$.
We set ${\cal L}:=\bigotimes_{i=1}^s \det{\cal M}_i^{\otimes (b_i+1)}$.
We set
\begin{equation}
\Per(X/Y,{\bf b})_Z^*:=\{ E \in \Per(X/Y,{\bf b})^*| \Supp(E) \subset Z \}.
\end{equation}
Then we have an equivalence
\begin{equation}
\begin{matrix}
\Per(X/Y,{\bf b}_0)_Z^* & \to & \Per(X/Y,{\bf b})_Z^*\\
E & \mapsto & E \otimes {\cal L}
\end{matrix}
\end{equation}

\begin{prop}\label{prop:span*}
$\Per(X/Y,{\bf b})_Z^*$ is generated by $A_0({\bf b})^*[1]$ and
${\cal O}_{C_i}(b_i)$, $i=1,2,\dots,n$. 
In particular, if
$E \in \Per(X/Y,{\bf b})^*$ satisfies
$\Hom(E,A_0({\bf b})^*)=\Hom(E,{\cal O}_{C_i}(b_i))=
\Hom(E,{\Bbb C}_x)=0$ for all $i=1,2,\dots,n$ and
$x \in X$, then $E=0$. 
\end{prop}

\begin{proof}
We first note that the claim holds for ${\Bbb C}_x$, $x \in Z$.
By Lemma \ref{lem:generate*} (3),
the claim also holds for all objects of ${\cal T}({\bf b})^*$.
Finally we prove the claim for
$E[1]$, $E \in {\cal S}({\bf b})^*$.
By Lemma \ref{lem:generate*} (1), 
we may assume that $E$ is a subsheaf of $A_0({\bf b})^*$.
We have a subsheaf $E'$ of $A_0({\bf b})^*$ such that
$A_0({\bf b})^*/E'$ is purely 1-dimensional
and $E'/E$ is 0-dimensional.
Then $E[1]$ is an extension of $E'[1]$ by $E'/E$ 
in $\Per(X/Y,{\bf b})^*$.
Hence we prove the claim for $E'$.
We take a non-trivial extension
\begin{equation}\label{eq:E'}
0 \to E' \to F' \to {\Bbb C}_x \to 0.
\end{equation}
By the inclusion $E' \to A_0({\bf b})^*$,
we have the following diagram
\begin{equation}
\begin{CD}
@. 0 @. 0 @.\\
@. @AAA @AAA @. @.\\
 @. A_0({\bf b})^*/E' @>>> F/F' @.\\
@. @AAA @AAA @. @.\\
0 @>>> A_0({\bf b})^* @>>> F @>>> {\Bbb C}_x @>>> 0\\
@. @AAA @AAA @| @.\\
0 @>>> E' @>>> F' @>>> {\Bbb C}_x @>>> 0\\
\end{CD}
\end{equation}
Since $A_0({\bf b})^*/E'$ is purely 1-dimensional,
$\Ext^1({\Bbb C}_x,E') \to \Ext^1({\Bbb C}_x,A_0({\bf b})^*)$
is injective. Thus 
the middle horizontal sequence gives a non-trivial extension.
By Lemme \ref{lem:B-stability2} (2),
$F$ is a successive extension of
${\cal O}_{C_i}(b_i)$, $i=1,2,\dots,n$.
Then $A_0({\bf b})^*/E' \cong  F/F'$ belongs to
${\cal T}({\bf b})^*$. 
Therefore $E'[1]$ is generated by $A_0({\bf b})^*[1]$
and objects of ${\cal T}({\bf b})^*$ in $\Per(X/Y,{\bf b})^*$. 
\end{proof}

\begin{prop}
Let $G$ be a locally free sheaf of rank $r$ on $X$.
Then $G$ is a local projective generator of $\Per(X/Y,{\bf b})^*$
if and only if 
\begin{enumerate}
\item[(i)]
$\Ext^1(G,{\cal O}_{C_i}(b_i))=\Hom(G,A_0({\bf b})^*)=0$, $i>0$,
\item[(ii)]
$r_i:=\chi(G,{\cal O}_{C_i}(b_i))>0$, $i>0$ and
$ r-\sum_i a_i r_i=
\chi(G,A_0({\bf b})^*[1])>0$.
\end{enumerate}
\end{prop}

\begin{proof}
We first note that 
the conditions are equivalent to 
\begin{equation}\label{eq:G*}
\begin{split}
\Hom(G,{\cal O}_{C_i}(b_i)) &\ne 0,\;i>0\\
\Ext^1(G,{\cal O}_{C_i}(b_i))&=0,\;i>0\\
\Hom(G,A_0({\bf b})^*)&= 0,\\
\Ext^1(G,A_0({\bf b})^*) &\ne 0.
\end{split}
\end{equation}
Since ${\cal O}_{C_i}(b_i), A_0({\bf b})^*[1] \in \Per(X/Y,{\bf b})^*$,
the conditions are necessary.
Conversely we assume that 
$G$ satisfies \eqref{eq:G*}. Then the claim follows from
the following two lemmas.
\end{proof}

\begin{lem}\label{lem:chi(G,E)*}
Assume that a locally free sheaf $G$ satisfies 
\eqref{eq:G*}.
Let $E \ne 0$ be a coherent sheaf on $X$. 
\begin{enumerate}
\item[(1)]
If $E \in {\cal T}({\bf b})^*$ and $\Supp(E) \subset Z$, then
$\pi_*(G^{\vee} \otimes E) \ne 0$ and $R^1 \pi_*(G^{\vee} \otimes E)=0$.
In particular $\chi(G,E)>0$.
\item[(2)]
If $E \in {\cal S}({\bf b})^*$, then
$\pi_*(G^{\vee} \otimes E)= 0$ and $R^1 \pi_*(G^{\vee} \otimes E) \ne 0$.
In particular $\chi(G,E)<0$.
\item[(3)]
If $R^1 \pi_*(G^{\vee} \otimes E)=0$, then
$E \in {\cal T}({\bf b})^*$. 
\end{enumerate}
\end{lem}

\begin{proof}
By \eqref{eq:G*}, we have
$G_{|C_i} \cong {\cal O}_{C_i}(b_i)^{\oplus r_i}
\oplus {\cal O}_{C_i}(b_i+1)^{\oplus (r-r_i)}$,
where $r_i=\chi(G,{\cal O}_{C_i}(b_i))$.
(1)
By Lemma \ref{lem:generate*} (3), it is sufficient to prove the claim
for ${\Bbb C}_x, x \in Z$ and
${\cal O}_{C_i}(b_i)$, $i>0$.
These follows from 
the definition of $G$. 

(3)
By Lemma \ref{lem:generate*},
we may assume that $E$ is a subsheaf of $A_0({\bf b})^*$.
We set $F:=A_0({\bf b})^*/E$.
Since $A_0({\bf b})^*$ is simple,
$\Hom(F,A_0({\bf b})^*)=0$.
Hence $F \in {\cal T}({\bf b})^*$.
Then the claim follows from (1).
(3) follows from (2) and Proposition \ref{prop:torsion-pair*}.

\end{proof}

\begin{lem}
Assume that a locally free sheaf $G$ satisfies 
\eqref{eq:G*}.
Then the following holds.
\begin{enumerate}
\item[(1)]
$R^1 \pi_*(G^{\vee} \otimes G)=0$.
\item[(2)]
$E \in {\cal T}({\bf b})^*$ if and only if
$R^1 \pi_*(G^{\vee} \otimes E)=0$, and
$E \in {\cal S}({\bf b})^*$ if and only if
$\pi_*(G^{\vee} \otimes E)=0$.
\item[(3)]
$G$ is a local projective generator of 
$\Per(X/Y,{\bf b})^*$
\end{enumerate}
\end{lem}

\begin{proof}
(1)
It is sufficient to prove that
$H^1(X,G^{\vee} \otimes G_{|Z})=0$.
We note that $\Hom({\Bbb C}_x,G_{|Z}(Z))=0$.
Since $\Hom({\cal O}_{C_i}(b_i),G_{|Z}(Z))\cong
\Hom(G^{\vee}_{|C_i}(b_i),{\cal O}_Z(Z))$ and
$\Hom({\cal O}_{C_i}(-1),{\cal O}_Z(Z))=0$,
we get $\Hom({\cal O}_{C_i}(b_i),G_{|Z}(Z))=0$,
which implies that $G_{|Z}(Z) \in {\cal S}({\bf b})^*$.
By Lemma \ref{lem:generate*} and $\Hom(G,A_0({\bf b})^*)=0$,
we see that 
\begin{equation}
H^1(X,G^{\vee} \otimes G_{|Z})\cong
\Ext^1(G_{|Z},G)^{\vee} \cong 
H^0(Z, G^{\vee} \otimes G \otimes {\cal O}_Z(Z))^{\vee} \cong  
\Hom(G,G_{|Z}(Z))^{\vee}=0.
\end{equation}

The proof of (2) and (3) are similar to those of
Proposition \ref{prop:Per-generator}.
\end{proof}

\begin{lem}\label{lem:Per-generator*}
Let $G$ be a local projective generator of 
$\Per(X/Y,{\bf b})$.
Then $G^{\vee}$ is a local projective generator of 
$\Per(X/Y,-{\bf b}+2{\bf b}_0)^*=\Per(X/Y,{\bf b})^D$.
\end{lem}

\begin{proof}
We note that
$D({\cal O}_{C_i}(b_i)[1])[2]={\cal O}_{C_i}(-b_i-2)$ and
$D(A_0({\bf b}))[2]=A_0(-{\bf b}+2{\bf b}_0)^*[1]$.
Then Lemma \ref{lem:generate*} (1), (2) and Proposition
\ref{prop:tilting:S-T} imply that
$\Per(X/Y,-{\bf b}+2{\bf b}_0)^*=\Per(X/Y,{\bf b})^D$.
Then Lemma \ref{lem:tilting:dual} implies the claim.
\end{proof}

\begin{proof}
We note that $A_0({\bf b})^{\vee}[1] \cong A_0(-{\bf b}+2{\bf b}_0)^*$.
Then we have the claim by the following relations:
\begin{equation}
\begin{split}
\Hom(G^{\vee},A_0(-{\bf b}+2{\bf b}_0)^*) & 
\cong \Ext^1(A_0({\bf b}),G) \cong
\Ext^1(G,A_0({\bf b}))^{\vee}= 0,\\
\Ext^1(G^{\vee},A_0(-{\bf b}+2{\bf b}_0)^*) & 
\cong \Ext^2(A_0({\bf b}),G) \cong
\Hom(G,A_0({\bf b}))^{\vee} \ne 0,\\
\Hom(G^{\vee},{\cal O}_{C_i}(-b_i-2)) & \cong
\Ext^1({\cal O}_{C_i}(b_i),G) \cong
\Ext^1(G,{\cal O}_{C_i}(b_i))^{\vee} \ne 0,\\
\Ext^1(G^{\vee},{\cal O}_{C_i}(-b_i-2)) & \cong
\Ext^2({\cal O}_{C_i}(b_i),G) \cong
\Hom(G,{\cal O}_{C_i}(b_i))^{\vee}=0.
\end{split}
\end{equation}
\end{proof}

$E \in \Coh(X)$ satisfies that
$R^1 \pi_*(E)=0$ if and only if 
there is an exact sequence
\begin{equation}
0 \to E_1 \to E \to E_2 \to 0
\end{equation}
such that $E_1 \in \Per(X/Y)$ and 
$E_2$ is a successive extension of 
${\cal O}_{C_i}(-1)$, $i=1,...,s$.

We set 
\begin{equation}
E_1:=\im(\pi^*(\pi_*(E)) \to E),\;
E_2:=\coker(\pi^*(\pi_*(E)) \to E).
\end{equation}
Then $E_1 \in \Per(X/Y)$ and 
${\bf R}\pi_*(E_2)=0$.
Hence $E_2$ is a successive extension of 
${\cal O}_{C_i}(-1)$, $i=1,...,s$.

We set
\begin{equation}
{\cal S}(-{\bf b})^*:=
\{E \in {\cal S}(-{\bf b}+2{\bf b}_0)| 
\Hom({\cal O}_{C_i}(-b_i-2),E)=0,\; i>0 \}.
\end{equation}
\begin{lem}
Every $E \in {\cal S}(-{\bf b})^*$ is
 a successive extension of subsheaves of $A_0^{\vee}[1]$.
\end{lem}

\begin{proof}
For an object $E \in {\cal S}(-{\bf b})^*$,
$E^{\vee}[1]$ is a purely 1-dimensional sheaf
such that
$\Supp(E^{\vee}[1]) \subset Z$
and $\Hom(E^{\vee}[1],{\cal O}_{C_i}(b_i))=0$.
In particular, $E^{\vee}[1] \in {\cal T}({\bf b})$.
Then we have $\Hom(E,A_0^{\vee}[1])=\Hom(A_0,E^{\vee}[1]) \ne 0$.
Let $E'$ be the kernel of a non-zero homomorphism
$\phi:E \to A_0^{\vee}[1]$. Then we have 
$E' \in {\cal S}(-{\bf b})^*$.
Hence $E$ is a successive extension of subsheaves of $A_0^{\vee}[1]$.
\end{proof}

\begin{rem}
In the above proof,
assume that $\coker \phi$ is purely 1-dimensional.
Then
$(\coker \phi)^{\vee}[1]$ is a purely 1-dimensional and
$A_0 \to \phi(E)^{\vee}[1]$ is surjective. 
We see that $(\coker \phi)^{\vee}[1]$ is a successive 
extension of ${\cal O}_{C_i}(b_i)$.
Therefore $\coker \phi$ is a successive extension
of ${\cal O}_{C_i}(-b_i-2)$.
\end{rem}

For ${\bf b}=(b_1,b_2,\dots,b_s)$,
%let ${\cal S}(-{\bf b})^*$ be the full subcategory of $\Coh(X)$ 
%generated by subsheaves of $A_0^{\vee}[1]$.
let ${\cal T}(-{\bf b})^*$ be the full subcategory
of $\Coh(X)$ such that $E \in \Coh(X)$ belongs to
${\cal T}(-{\bf b})^*$ if and only if 
there is an exact sequence
\begin{equation}
0 \to E_1 \to E \to E_2 \to 0
\end{equation}
where $E_1 \in  {\cal T}({\bf b})$ and
$E_2$ is a succesive extensions of ${\cal O}_{C_i}(-b_i-2)$,
$1 \leq i \leq s$.
$T(-{\bf b})^*:=({\cal T}(-{\bf b})^*,{\cal S}(-{\bf b})^*)$ 
is a torsion pair of $\Coh(X)$.
Let $\Per(X/Y,-{\bf b})^*$ be the tilting of 
$\Coh(X)$ with respect to $T$.
If ${\bf b}_0=(-1,-1,\dots,-1)$, $\Per(X/Y,-{\bf b}_0)=\Per(X/Y)^*$.

Let ${\cal M}_i$ be full sheaves in an analytic neighborhood
of the fundamental class $Z$.
We set ${\cal L}:=\bigotimes_{i=1}^s \det{\cal M}_i^{\otimes (b_i+1)}$.
We set
\begin{equation}
\Per(X/Y,-{\bf b})_Z^*:=\{ E \in \Per(X/Y,-{\bf b})^*| \Supp(E) \subset Z \}.
\end{equation}
${\cal E}xt^1(A_0,{\cal O}_X) \cong {\cal O}_Z(Z) \otimes {\cal L}^{\vee}$
and we have an exact sequence
\begin{equation}
0 \to {\cal E}xt^1(A_0,{\cal O}_X) \to {\cal E}xt^1(F,{\cal O}_X) \to
{\Bbb C}_p \to 0.
\end{equation}
Then we have an equivalence
\begin{equation}
\begin{matrix}
\Per(X/Y,-{\bf b}_0)_Z^* & \to & \Per(X/Y,-{\bf b})_Z^*\\
E & \mapsto & E \otimes {\cal L}^{\vee}
\end{matrix}
\end{equation}

\end{NB}

\subsection{Moduli spaces of 0-dimensional objects.}
\label{subsect:wall-chamber}

Let $\pi:X \to Y$ be the minimal resolution of a normal projective
surface $Y$ and $p_1,p_2,\dots,p_n$ the rational double points of
$Y$ as in \ref{subsect:Per(X/Y)}. 
%$Z_i:=\pi^{-1}(p_i)=\sum_{j=0}^{s_i} a_{ij}C_{ij}$ is the fundamental 
%cycle of $,
We set $Z:=\cup_i Z_i$.
%From now on, we assume that $\pi:X \to Y$ is the minimal resolution of 
%a rational double point $p \in Y$ and ${\cal C}=\Per(X/Y,{\bf b})$.
Let $G$ be a locally free sheaf on $X$ which is a tilting generator of
the category ${\cal C}:={\cal C}_G$ in Lemma \ref{lem:tilting}.
For $\alpha \in \NS(X) \otimes {\Bbb Q}$, we define $\alpha$-twisted
semi-stability as $\gamma^{-1}((0,\alpha,0))$-twisted stability, where
$\gamma$ is the homomorphism \eqref{eq:gamma}.
In this subsection,
we shall study the moduli of $\alpha$-twisted semi-stable objects.
For simplicity, we say that $\alpha$-twisted semi-stability as 
$\alpha$-semi-stability. 
%$\gamma:K(X)_{\mathrm{top}} \otimes {\Bbb Q} 
%\to ({\Bbb Z} \oplus \NS(X) \oplus {\Bbb Z}) \otimes {\Bbb Q}$.
For simplicity, we set 
$X^{\alpha}:=\overline{M}_{{\cal O}_X(1)}^{G,\alpha}(\varrho_X)$.
Since every 0-dimensional object is $0$-semi-stable,
we have a natural morphism $\pi_\alpha:X^\alpha \to X^0$.

\begin{lem}\label{lem:obstruction}
For a $0$-dimensional object $E$ of ${\cal C}$,
there is a proper subspace $T(E)$ of
$\Ext^2(E,E)$ such that all obstructions for infinitesimal deformations
of $E$
belong to $T(E)$.
\end{lem}

\begin{proof}
Let $E$ be a 0-dimensional object of ${\cal C}$. 
We first assume that there is a curve $C \in |K_X|$ such that
$C \cap \Supp(E) =\emptyset$.
Then
$H^0(X,K_X) \to \Hom(E,E(K_X))$ is non-trivial, which implies that
the trace map
\begin{equation}
\tr:\Ext^2(E,E) \to H^2(X,{\cal O}_X),
\end{equation}
is non-trivial.
Since the obstruction for infinitesimal deformations of $E$ lives in 
$\ker \tr$,
$T(E) \subset \ker \tr$ is a proper subspace of
$\Ext^2(E,E)$.
For a general case, we use the covering trick.
Let $D$ be a very ample divisor on $Y$ such that
there is a smooth curve $B \in |2D|$ with
$B \cap \pi(\Supp(E) \cup Z) =\emptyset$ and
$|K_Y+D|$ contains a curve $C$ with 
$C \cap \pi(\Supp(E)\cup Z) =\emptyset$.
Since $\pi$ is isomorphic over
$Y \setminus \pi(Z)$, we may
regard $B$ and $C$ as divisors 
on $X$. Let $\phi:\widetilde{Y} \to Y$ be the double covering branced along
$B$ and set $\widetilde{X}=X \times_Y \widetilde{Y}$.
We also denote $\widetilde{X} \to X$ by $\phi$.
Then $|K_{\widetilde{X}}|=|\phi^*(K_X+D)|$ contains
$\phi^*(C)$.
Since $\phi$ is \'{e}tale over $Y \setminus B$, we have a decomposition
$\pi^*(E)=E_1 \oplus E_2$ and
$\Ext^2(E,E) \to \Ext^2(E_i,E_i)$ are isomorphism for $i=1,2$.
Under these isomorphisms, $T(E)$ is mapped into $T(E_i)$. 
Since $\tr_i:\Ext^2(E_i,E_i) \to H^2(\widetilde{X},{\cal O}_{\widetilde{X}})$
are non-trivial, $\ker \tr_i$ are proper subspaces of $\Ext^2(E_i,E_i)$.
Hence $T(E)$ is a proper subspace of $\Ext^2(E,E)$. 
\end{proof}

\begin{prop}\label{prop:0-dim:smooth}
\begin{enumerate}
\item[(1)]
For a 0-dimensional
object $E$ of ${\cal C}$,
$E \otimes K_X \cong E$.
In particular,
$\Ext^2(E,E) \cong \Hom(E,E)^{\vee}$.
\item[(2)]
For a 0-dimensional Mukai vector $v$,
$M_{{\cal O}_X(1)}^{G,\alpha}(v)$
is smooth of dimension $\langle v^2 \rangle+2$.
\end{enumerate}
\end{prop}

\begin{proof}
(1)
Since $K_X=\pi^*(K_Y)$ and $\dim \pi(\Supp(E))=0$,
we get $E \otimes K_X \cong E$. 
(2)
For $E \in M_{{\cal O}_X(1)}^{G,\alpha}(v)$,
we have $\Hom(E,E)={\Bbb C}$.
Then Lemma \ref{lem:obstruction} implies that
$T(E)=0$.
Since $\dim \Ext^1(E,E)=\langle v^2 \rangle+2$,
$M_{{\cal O}_X(1)}^{G,\alpha}(v)$
is smooth of dimension $\langle v^2 \rangle+2$.
\end{proof}

\begin{rem}
There is another argument to prove the smoothness due to
Bridgeland \cite{Br:1}. We shall use the argument later.
So for stable objects, we do not need Lemma \ref{lem:obstruction},
but it is necessary for the study of properly semi-stable
objects (see Proposition \ref{prop:lci-stack}).   
\begin{NB}
Bridgeland's argument:
If $\langle v(E)^2 \rangle =-2$, then $\Ext^1(E,E)=0$,
and hence the moduli space is smooth.
Assume that $v(E)=\varrho$.
Let $M$ be an irreducible component of $M_{{\cal O}_X(1)}^{G,\alpha}(v)$
containing $X \setminus \cup_i Z_i$.
Then $\dim \Ext^1(E,E)=2$ for any point $E \in M$ and 
$\dim M=2$. Hence $M$ is smooth and is a connected component of
$M_{{\cal O}_X(1)}^{G,\alpha}(v)$. Then by the standard argument
of the Fourier-Mukai transform, we get $M=M_{{\cal O}_X(1)}^{G,\alpha}(v)$.
\end{NB}  
\end{rem}

\begin{lem}\label{lem:crepant1}
Assume that $\alpha \in \NS(X) \otimes {\Bbb Q}$ satisfies that
\begin{equation}\label{eq:weakly-general}
\text{ $(\alpha, D) \ne 0$
 for all $D \in \NS(X)$ with $(D^2)=-2$ and
$(c_1({\cal O}_X(1)),D)=0$.}
\end{equation}
Then 
$X^{\alpha}=
M_{{\cal O}_X(1)}^{G,\alpha}(\varrho_X)$.
\end{lem}

\begin{proof}
Assume that $E \in X^{\alpha}$ is $S$-equivalent to
$\oplus_{i=1}^t E_i$, where $E_i$ are $\alpha$-stable objects.
Then $(\alpha,c_1(E_i))=0$, $(c_1({\cal O}_X(1)),c_1(E_i))=0$ and
$(c_1(E_i)^2)=\langle v(E_i)^2 \rangle \geq -2$ for all $i$.
Since $\langle v(E_i),v(E_j) \rangle \geq 0$ for $E_i \not \cong E_j$
and $\sum_{i,j} \langle v(E_i),v(E_j) \rangle=\langle v(E)^2 \rangle=0$,
(i) $\langle v(E_i)^2 \rangle=-2$ for an $i$, or
(ii) $\langle v(E_i)^2 \rangle=0$ for all $i$.
By our choice of $\alpha$, the case (i) does not occur. 
In the second case, we see that
$v(E_i)=a_i \varrho_X$, $a_i>0$.
Then $\varrho_X=(\sum_i a_i)\varrho_X$, which implies $t=1$.
Therefore $E$ is $\alpha$-stable.
\end{proof}

\begin{lem}\label{lem:simple-generator}
Let ${\cal E}$ be an object of ${\bf D}(X \times X')$
such that 
$\Phi_{X \to X'}^{{\cal E}^{\vee}}:
{\bf D}(X) \to {\bf D}(X')$ is an equivalence,
${\cal E}_{|X \times \{ x' \}} \in {\cal C}$ for all $x' \in X'$
and $v({\cal E}_{|X \times \{ x' \}})=\varrho_X$.
Then every irreducible object of ${\cal C}$ appears as a direct summand of
the $S$-equivalence class of
${\cal E}_{|X \times \{ x' \}}$. 
\end{lem}

\begin{proof}
Let $E$ be an irreducible object of ${\cal C}$.
If $\Supp(E) \not \subset Z$, then we have a non-trivial morphism
$E \to {\Bbb C}_x$, $x \not \in Z$.
Since $({\cal C})_{|X \setminus Z}=\Coh(X \setminus Z)$,
${\Bbb C}_x$ is an irreducible object. Hence
$E\cong {\Bbb C}_x$.
Since $\chi({\cal E}_{|X \times \{ x' \}},{\Bbb C}_x)=0$ and
$\Phi_{X \to X'}^{{\cal E}^{\vee}}$ is an equivalence,
there is a point $x' \in X'$ such that 
$\Hom({\cal E}_{|X \times \{ x' \}},{\Bbb C}_x) \ne 0$
or $\Hom({\Bbb C}_x,{\cal E}_{|X \times \{ x' \}}) \ne 0$. 
Since $v({\Bbb C}_x)=v({\cal E}_{|X \times \{ x' \}})=\varrho_X$,
we get ${\Bbb C}_x \cong {\cal E}_{|X \times \{ x' \}}$.
If $\Supp(E) \subset \cup_i Z_i$, then we still have
$\chi({\cal E}_{|X \times \{ x' \}},E)=0$, since
${\cal E}_{|X \times \{ x' \}}={\Bbb C}_{x}$, $x \not \in Z$ for
a point $x' \in X'$.
Then we have
$\Hom({\cal E}_{|X \times \{ x' \}},E) \ne 0$
or $\Hom(E,{\cal E}_{|X \times \{ x' \}}) \ne 0$. 
Therefore our claim holds.
\end{proof}

\begin{lem}\label{lem:0-dim:irreducible}
If $\alpha$ is general, then
$X^{\alpha}$ is irreducible.
\end{lem}
\begin{proof} 
Let $X'$ be a connected component of
$X^{\alpha}$.
Then we have an equivalence
$\Phi_{X \to X'}^{{\cal E}^{\vee}}:{\bf D}(X) \to
{\bf D}(X')$, where ${\cal E}$ is the universal family.
By the same argument as in the proof of Lemma \ref{lem:simple-generator},
we see that every $E \in X^{\alpha}$ belongs to $X'$.
\end{proof}

\begin{prop}\label{prop:lci-stack}
Let ${\cal X}^0$ be the moduli stack of 0-semi-stable objects $E$ 
with $v(E)=\rho_X$. Then 
${\cal X}^0$ is a locally complete intersection stack of dimension 1
and irreducible. In particlar ${\cal X}^0$ is a reduced stack.   
%we have the following
%\begin{enumerate}
%\item[(1)]
%${\cal X}^0$ is irreducible. 
%\item[(2)]
%${\cal X}^0$ is a 
%locally complete intersection stack of dimension 1. 
%\end{enumerate}
\end{prop}

\begin{proof}
Let $Q$ be an open subscheme
of a perverse quot-scheme such that
$X^0$ is a GIT-quotient of a suitable
$GL(N)$-action.
Then ${\cal X}^0$ is the quotient stack
$[Q/GL(N)]$.
Let ${\cal E}$ be the family of 0-dimensional
objects of ${\cal C}$ on $Q \times X$.
For any point $q \in Q$,
we set $n_1:=\dim \Hom({\cal K}_q,{\cal E}_q)$
and $n_2:=\dim T({\cal E}_q)$, where ${\cal K}$ is the
universal subobject on $Q \times X$.
Then an analytic neighborhood of 
$Q$ is an intersection of $n_2$ hypersurfaces in ${\Bbb C}^{n_1}$.
Hence $\dim Q \geq n_1-n_2$ and
$\dim [Q/GL(N)] \geq -\chi({\cal E}_q,{\cal E}_q)+1=1$.   
We take a general $\alpha$ and set 
$Q^u:=\{ q \in Q| \text{ ${\cal E}_q$ is not
$\alpha$-semi-stable } \}$.
By the proof of \cite[Prop. 2.16]{O-Y:1},
we see that $\dim [Q^u/GL(N)]=0$.
Since $[(Q \setminus Q^u)/GL(N)]$ is the moduli stack of
$\alpha$-stable objects,
it is a smooth and irreducible stack
of dimension 1. 
Hence $[Q/GL(N)]$ is a locally complete intersection stack of dimension 1
and irreducible. In particlar $[Q/GL(N)]$ is a reduced stack.   
\end{proof}

\begin{lem}\label{lem:one-point}
Let $E$ be a 0-semi-stable object with $v(E)=\varrho_X$.
Then $\Supp(\pi_*(G^{\vee} \otimes E))$ is a point of $Y$.
\end{lem}

\begin{proof}
For $E$, we have a decomposition
$E=\oplus_{i=1}^t E_i$ such that $\Supp(\pi_*(G^{\vee} \otimes E_i))$,
$i=1,...,t$ are distinct $t$ points of $Y$.
We set $v(E_i)=(0,D_i,a_i)$.
Since $D_i$ are contained in the exceptional loci,
$0=\langle v(E)^2 \rangle=\sum_i (D_i^2)$ implies that
$(D_i^2)=0$ for all $i$. Thus 
we have $v(E_i)=a_i \varrho_X$ for all $i$, which implies that
$\varrho_X=(\sum_i a_i)\varrho_X$.
Since $\chi(G,E_i)>0$, we have $a_i>0$. 
Therefore $t=1$.
\end{proof}
By Lemma \ref{lem:tilting:irreducible}, we get the following. 

\begin{lem}\label{lem:contraction-varphi}
\begin{enumerate}
\item[(1)]
${\Bbb C}_x \in {\cal C}$ for all $x \in X$.
In particular,
we have a morphism $\varphi:X \to X^0$ by sending $x \in X$ to 
the $S$-equivalence class of ${\Bbb C}_x$.
\item[(2)]
$\varphi(Z_i)$ is a point.
\end{enumerate}
\end{lem}
\begin{NB}
If there is an $\alpha$ such that $X \cong X^\alpha$,
$\varphi=\pi_\alpha$. 
\end{NB}

\begin{NB}
\begin{proof}
(1)
We note that ${\bf R}\pi_*(G^{\vee} \otimes {\Bbb C}_x)=
 \pi_*(G^{\vee} \otimes {\Bbb C}_x)$. Hence ${\Bbb C}_x \in {\cal C}$.
(2)
Since the restriction map
$R^1 \pi_*(G^{\vee} \otimes G) \to
R^1 \pi_*(G^{\vee} \otimes G_{|C_{ij}})$ is surjective,
$R^1 \pi_*(G^{\vee} \otimes G_{|C_{ij}})=0$.
Thus we can write 
$G_{|C_{ij}} \cong {\cal O}_{C_{ij}}(d_{ij})^{\oplus r_{ij}}
\oplus {\cal O}_{C_{ij}}(d_{ij}+1)^{\oplus r_{ij}'}$.
Since $R^1 \pi_*(G^{\vee} \otimes {\cal O}_{C_{ij}}(d_{ij}))=0$
and $\pi_*(G^{\vee} \otimes {\cal O}_{C_{ij}}(d_{ij}-1))=0$,
${\cal O}_{C_{ij}}(d_{ij}),{\cal O}_{C_{ij}}(d_{ij}-1)[1] \in
{\cal C}$.
For $x \in C_{ij}$, we have an exact sequence in ${\cal C}$
\begin{equation}
0 \to {\cal O}_{C_{ij}}(d_{ij}) \to {\Bbb C}_x \to
{\cal O}_{C_{ij}}(d_{ij}-1)[1] \to 0.
\end{equation} 
Hence ${\Bbb C}_x$ is $S$-equivalent to
${\cal O}_{C_{ij}}(d_{ij}) \oplus {\cal O}_{C_{ij}}(d_{ij}-1)[1]$,
which implies that $\varphi(C_{ij})$ is a point.
Since $Z_i$ is connected, we see that $\varphi(Z_i)$ is a point.
\end{proof}
\end{NB}

If ${\Bbb C}_x$ is properly 0-semi-stable, then
${\Bbb C}_x$ is $S$-equivalent to
$\oplus_j E_{ij}^{\oplus a_{ij}'}$ for an $i$.

\begin{prop}\label{prop:Y=X^0}
There is an isomorphism $\psi:X^0 \to Y$ such that
$\psi \circ \varphi:X \to Y$ coincides with $\pi$.  
In particular, $X^0$ is a normal projective surface.
\end{prop}

\begin{proof}
We keep the notation in the proof of Proposition
\ref{prop:lci-stack}.
By Lemma \ref{lem:one-point},
${\cal F}:=\pi_*(G^{\vee} \otimes {\cal E})$ is a flat family  
of coherent sheaves on $Y$ such that
$\Supp({\cal F}_q)$ is a point for every $q \in Q$. 
Since the characteristic of the base field is zero,
we have a morphism $Q \to S^r Y$, where $r=\rk G$
(cf. \cite{F:1}, \cite{F:2}).
Since the image is contained in the diagonal $Y$,
we have a morphism $Q \to Y$.
\begin{NB}
For the diagonal embedding $Y \to Y^r$,
we have an embedding $Y \cong Y/{\frak S}_r \to Y^r/{\frak S}_r$.
\end{NB}
Hence we have a morphsim $\psi:X^0 \to Y$.
By the construction of $\varphi$ and $\psi$,
$\pi=\psi \circ \varphi$. 
Since $\varphi$ and $\psi$ are projective birational morphisms
between irreducible surfaces, 
$\varphi$ and $\psi$ are contractions.
By using Lemma \ref{lem:contraction-varphi}, we see that
$\psi$ is injective. Hence $\psi$ is a finite morphism.
Since $Y$ is normal, $\psi$ is an isomorphism.
\end{proof}

\begin{lem}\label{lem:crepant}
Assume that $\alpha \in \NS(X) \otimes {\Bbb Q}$ satisfies
\eqref{eq:weakly-general}.
Then 
$K_{X^{\alpha}}$ is the pull-back of a line bundle on $X^0$.
\end{lem}

\begin{proof}
Let ${\cal E}$ be the universal family on $X^{\alpha} \times X$.
Let $p_S:S \times X \to S$ be the projection.
Since $X^{\alpha}$ is smooth, the base change theorem 
implies that $\Ext^i_{p_{X^{\alpha}}}({\cal E},{\cal E})$,
$i=0,1,2$ are locally free sheaves on $X^{\alpha}$ and
compatible with base changes.
Since $\Ext^1_{p_{X^{\alpha}}}({\cal E},{\cal E})$ is the tangent bundle
of $X^{\alpha}$,
we show that there is a symplectic form on
$\Ext^1_{p_{X^{\alpha}}}({\cal E},{\cal E})$. 
For any point $y \in Y$, we take a very ample divisor $D_2$ on $Y$
such that $y \not \in D_2$, 
$|K_Y+D_2|$ contains a divisor $D_1$
with $y \not \in D_1$.
We set $U:=Y \setminus (D_1 \cup D_2)$.
Then $U$ is an open neighborhood of $y$ such that
$K_Y$ is trivial over $U$.
Let $\widetilde{D}_i$ be the pull-back of $D_i$ to
$X$.
Then we have $K_X={\cal O}_X(\widetilde{D}_1 -\widetilde{D}_2)$.
We set $V:=\pi_{\alpha}^{-1}(\psi^{-1}(U))$. We shall prove that
(i) the alternating pairing
\begin{equation}
\Ext^1_{p_V}({\cal E},{\cal E}) \times 
\Ext^1_{p_V}({\cal E},{\cal E}) \to 
\Ext^2_{p_V}({\cal E},{\cal E}) 
\end{equation}
is non-degenerate and (ii) 
$\Ext^2_{p_V}({\cal E},{\cal E})\cong {\cal O}_V$. 
Since $\Ext^1_{p_{X^{\alpha}}}({\cal E},{\cal E})$ is the tangent bundle,
this means that $K_V \cong {\cal O}_V$. Thus the claim holds. 

We first note that there are isomorphisms 
\begin{equation}\label{eq:canonical-bundle}
\Ext^i_{p_V}({\cal E},{\cal E}) \cong 
\Ext^i_{p_V}({\cal E},{\cal E}(\widetilde{D}_1)) \cong 
\Ext^i_{p_V}({\cal E},{\cal E}(\widetilde{D}_1-\widetilde{D}_2)),\; i=0,1,2,
\end{equation}
which is compatible with the base change.
By the Serre duality, the trace map 
$\tr:\Ext^2({\cal E}_y,{\cal E}_y(K_X)) \to H^2(X,K_X)$
is an isomorphism for $y \in V$. Hence (ii) holds.
%$\Ext^2_{p_V}({\cal E},{\cal E}) \cong {\cal O}_V$.
By the Serre duality, the pairing 
$\Ext^1({\cal E}_y,{\cal E}_y) \times \Ext^1({\cal E}_y,{\cal E}_y(K_X)) 
\to \Ext^2({\cal E}_y,{\cal E}_y(K_X)) \cong
H^2(X,K_X)$
is non-degenerate.
Combining this with \eqref{eq:canonical-bundle},
we get (i).
%the alternating pairing
%\begin{equation}
%\Ext^1_{p_V}({\cal E},{\cal E}) \times 
%\Ext^1_{p_V}({\cal E},{\cal E}) \to 
%\Ext^2_{p_V}({\cal E},{\cal E}) 
%\end{equation}
%is also non-degenerate.
%Therefore $V$ has a symplectic structure, which implies the claim.
\end{proof}

\begin{lem}\label{lem:alpha=0}
Assume that $\alpha=0$.
\begin{enumerate}
\item[(1)]
Assume that $p_i \in Y$ corresponds to 
$\bigoplus_{j=0}^{s_i}E_{ij}^{\oplus a_{ij}}$ via
$\psi$, where $E_{ij}$ are 0-stable
objects.
Then ${\Bbb C}_x, x \in Z_i$ are $S$-equivalent to
$\bigoplus_{j=0}^{s_i}E_{ij}^{\oplus a_{ij}}$.
\item[(2)]
Let $E \in {\cal C}$ be a 0-twisted stable object.
Then $E$ is one of the following:
\begin{equation}\label{eq:0-stable-object}
{\Bbb C}_x\; ( x \in X \setminus Z),\quad
E_{ij},\; (1 \leq i \leq n,0 \leq j \leq s_i).
\end{equation}
\item[(3)]
Every $0$-dimensional object is generated by \eqref{eq:0-stable-object}. 
\end{enumerate}
\end{lem}

\begin{proof}
By Proposition \ref{prop:Y=X^0}, (1) holds.
We shall apply Lemma \ref{lem:simple-generator}
to ${\cal E}={\cal O}_{\Delta} \in {\bf D}(X \times X)$. 
Then (2) is a consequence of (1).
It also follows from Lemma \ref{lem:tilting:irreducible} (3).
(3) follows from (2). 
\end{proof}

\begin{rem}
If ${\bf b}={\bf b}_0$, then $\pi_*({\cal E})$
is a flat family of coherent sheaves on $Y$
such that $\pi({\cal E})_q$ is a point sheaf.
Then we have a morphism $Q \to Y$. 
Thus we do not need the reducedness of $Q$ in this case.

\begin{NB}
We take an analytic neighborhood $U$ of $Z$ such that
there are line bundles $L_i$ such that
$\deg(L_{i|C_j})=\delta_{ij}$.
Let $Q_U$ be an analytic open subset of $Q$
which is  the pull-back of $X^0 \setminus (X \setminus U)$.
Here we identified ${M}_{{\cal O}_X(1)}^0(\varrho_X)$ with
$X \setminus Z$.

For a general $\alpha$, we have a morphism 
$X^{\alpha} \to X^0$. Thus we get a morphism 
$X \to Y \to X^0$.
We shall show that there is a morphism
$Q_U \to Y$ which is the lift of $Q \to X^0$.
********
\end{NB}
\end{rem}

\begin{defn}
We set $Z^{\alpha}_i:=
\pi_\alpha^{-1}(\bigoplus_j E_{ij}^{\oplus a_{ij}})=
\pi^{-1}_\alpha \circ \psi^{-1}(p_i)$
and $Z^{\alpha}:=\cup_i Z_i^{\alpha}$.
\end{defn}

%Next lemma corresponds to \cite[Lem. 2.4]{O-Y:1}.
\begin{lem}(cf. \cite[Lem. 2.4]{O-Y:1})\label{lem:0-stable:key}
Let $E_{ij}$ be 0-stable objects in Lemma \ref{lem:alpha=0}.
Assume that $-(\alpha,c_1(E_{ij}))>0$ for all $j>0$.
\begin{NB}
Since the first Chern classes of 
$\{ E_{ij} \}$ are linearly independent,
we can choose such an $\alpha$.
\end{NB}
Let $F$ be a 0-semi-stable object such that 
$v(F)=v(E_{i0} \oplus \bigoplus_{j>0}E_{ij}^{\oplus b_j})$,
$0 \leq b_j \leq a_{ij}$.
\begin{enumerate}
\item[(1)]
If $v(F) \ne \varrho_X$, then $F$ is
$S$-equivalent to $E_{i0} \oplus \bigoplus_{j>0}E_{ij}^{\oplus b_j}$
with respect to $0$-stability.
\item[(2)]
Assume that $F$ is $S$-equivalent to
$E_{i0} \oplus \bigoplus_{j>0}E_{ij}^{\oplus b_j}$.
Then the following conditions are equivalent. 
\begin{enumerate}
\item
$F$ is $\alpha$-stable 
\item
$F$ is $\alpha$-semi-stable 
\item
$\Hom(E_{ij},F)=0$ for all $j>0$.
\end{enumerate}
\item[(3)]
Assume that $F$ is $\alpha$-stable.
For a non-zero homomorphism $\phi:F \to E_{ij}$, $j>0$,
$\phi$ is surjective and $F':=\ker \phi$ is an $\alpha$-stable object.
\item[(4)]
If there is a non-trivial extension
\begin{equation}
0 \to F \to F'' \to E_{ij} \to 0
\end{equation} 
and $b_k+\delta_{jk} \leq a_{ik}$, then
$F''$ is an $\alpha$-stable object, where $\delta_{jk}=0,1$
according as $j \ne k$, $j=k$.
\end{enumerate}
\end{lem}

\begin{proof}
(1) 
%Since $v(F) \ne \varrho_X$,
%$c_1(F)$ is a non-zero divisor... 
%Assume that $F$ is $S$-equivalent to $\oplus_j E_j$, where 
%$E_j$ are 0-stable objects.....
%
Since $E:=F \oplus \bigoplus_{j>0}E_{ij}^{\oplus (a_{ij}-b_j)}$
is a 0-semi-stable object with $v(E)=\varrho_X$ and
$\Supp(\pi_*(G^{\vee} \otimes E))=
\Supp(\pi_*(G^{\vee} \otimes F)) \cup \{ p_i \}$,
Lemma \ref{lem:one-point} and Proposition \ref{prop:Y=X^0}
imply that
the $S$-equivalence class of $E$ corresponds to 
$p_i \in Y$.
Hence $E$ is $S$-equivalent to $\bigoplus_{j \geq 0}E_{ij}^{\oplus a_{ij}}$,
which implies that $F$ is $S$-equivalent to 
$E_{i0} \oplus \bigoplus_{j>0}E_{ij}^{\oplus b_j}$.

(2) It is sufficient to prove 
that (c) implies (a). 
Let $\psi:F \to I$ be a quotient of $F$.
Since $I$ and $\ker \psi$ are
0-dimensional objects,
they are 0-semi-stable.
Since $\Hom(E_{ij},\ker \psi)=0$ for $j>0$,
(1) implies that $E_{i0}$ is a subobject of $\ker \psi$.
Hence $v(I)=\sum_{j>0}b_j' v_{ij}$, which implies that
$F$ is $\alpha$-stable.

(3)
Since $E_{ij}$ is irreducible, $\phi$ is surjective.
By (1), $\ker \phi$ also satisfies the assumption of (2).
Let $\psi:\ker \phi \to I$ be a quotient object.
Since $\Hom(E_{ik},F)=0$ for $k>0$,
(2) implies that $\ker \phi$ is $\alpha$-stable.

(4)
Since $v(F) \ne \varrho_X$, (1) implies that
$F''$ satisfies the assumption of (2). 
If $\Hom(E_{ik},F'') \ne 0$, then 
$\Hom(E_{ik},F)=0$ implies that $k=j$ and 
we have a splitting of the exact sequence.
Hence $\Hom(E_{ik},F'')=0$ for $k>0$.
Then (2) implies the claim.
\end{proof}

\begin{cor}\label{cor:0-dim:reflection}
Assume that $-(\alpha,c_1(E_{ij}))>0$ for all $j>0$.
We set $v:=v(E_{i0} \oplus \oplus_{j>0}E_{ij}^{\oplus b_j})$,
$0 \leq b_j \leq a_{ij}$ with $\langle v^2 \rangle=-2$.
\begin{enumerate}
\item[(1)]
$\dim \Hom(E,E_{ij})=\max\{-\langle v,v(E_{ij}) \rangle,0 \}$.
\item[(2)]
If $-\langle v,v(E_{ij}) \rangle>0$, then 
$M_{{\cal O}_X(1)}^{G,\alpha}(v) \cong 
M_{{\cal O}_X(1)}^{G,\alpha}(w)$,
where $w= v+\langle v,v(E_{ij}) \rangle v(E_{ij})$.
\end{enumerate}
\begin{NB}
For $v:=v(E_{i0} \oplus \oplus_{j>0}E_{ij}^{\oplus b_j})$,
$0 \leq b_j \leq a_{ij}$ with $\langle v^2 \rangle=-2$,
$-1 \leq \langle v,v(E_{ij}) \rangle \leq 2$.
If $\langle v,v(E_{ij})=\pm 1$, then
$M_{{\cal O}_X(1)}^{G,\alpha}(v) \cong 
M_{{\cal O}_X(1)}^{G,\alpha}(w)$,
if $\langle v,v(E_{ij}) \rangle<0$.
where $w= v+\langle v,v(E_{ij}) \rangle v(E_{ij})$.
\end{NB}
\end{cor}

\begin{proof}
(1) 
For $E \in M_{{\cal O}_X(1)}^{G,\alpha}(v)$, 
we set $n:=\dim \Hom(E,E_{ij})$.
Then we have a surjective morphism $\phi:E \to E_{ij}^{\oplus n}$.
Then $F:=\ker \phi$ is $\alpha$-stable.
Since $-2 \leq \langle v(F)^2 \rangle=\langle v(E)^2 \rangle
-2n(n+\langle v,v(E_{ij}) \rangle)$,
$n=-\langle v,v(E_{ij}) \rangle$ or $n=0$.

(2)
If $-\langle v,v(E_{ij}) \rangle>0$, then
$\dim \Hom(E,E_{ij})=-\langle v,v(E_{ij}) \rangle$, 
$\Ext^p(E,E_{ij})=0$, $p>0$,
and we have a morphism
$\sigma:M_{{\cal O}_X(1)}^{G,\alpha}(v) \to 
M_{{\cal O}_X(1)}^{G,\alpha}(w)$.
Conversely for $F \in M_{{\cal O}_X(1)}^{G,\alpha}(w)$,
$\langle v(F),v(E_{ij}) \rangle
=-\langle v,v(E_{ij}) \rangle>0$.
Hence $\Hom(F,E_{ij})=0$, which implies that
$\dim \Ext^1(E_{ij},F)=\langle v(F),v(E_{ij}) \rangle$ 
and the universal extension
gives an $\alpha$-stable object $E$ with $v(E)=v$.
Therefore we also have the inverse of $\sigma$.
\end{proof}
We come to the main result of this subsection.
\begin{thm}(cf. \cite[Thm. 0.1]{O-Y:1})\label{thm:RDP-desing}
\begin{enumerate}
\item[(1)]
$X^0 \cong Y$ and the singular points 
$p_1,p_2,\dots,p_n$ of $X^0$ correspond to
the $S$-equivalence classes of properly $0$-twisted semi-stable objects.
\item[(2)]  
Assume that $\alpha$ satisfies that
$(\alpha, D) \ne 0$ for all $D \in \NS(X)$ with $(D^2)=-2$ and
$(c_1({\cal O}_X(1)),D)=0$.
Then $X^{\alpha}=
M_{{\cal O}_X(1)}^{G,\alpha}(\varrho_X)$. In particular
$\pi_\alpha:X^{\alpha} \to X^0$ 
is the minimal resolution of the singularities.
\item [(3)]
Let $\oplus_{j=0}^{s_i} E_{ij}^{\oplus a_{ij}}$ be 
the $S$-equivalence class corresponding to $p_i$.
Then the matrix $(-\langle v(E_{ij}),v(E_{ik}) \rangle)_{j,k \geq 0}$ 
is of affine
type $\tilde{A},\tilde{D}, \tilde{E}$.
Assume that $a_{i0}=1$. Then 
the singularity of $X^0$ at 
$p_i$ is a
rational double point of type $A,D,E$ according as the type of 
the matrix $(-\langle v(E_{ij}),v(E_{ik}) \rangle)_{j,k \geq 1}$.
%\item[(4)]
%We assume that $
%We set
%\begin{equation}
%C_{ij}:=\{y \in Y|\Hom(E_{ij},{\cal E}_{|\{y \} \times X}) \ne 0 \}.
%\end{equation}
%Then $C_{ij}$ is a smooth rational curve such that
%$(C_{ij},C_{i' j'})=-\chi(E_{ij},E_{i' j'})$ and
%$\pi^{-1}(p_i)=\sum_{j \geq 1}a_{ij}C_{ij}$. 
\end{enumerate}
\end{thm}

\begin{proof}
(1) By Proposition \ref{prop:Y=X^0},
$X^0 \cong Y$. 
Since $\varphi:X \to X^0$ is surjective,
$y \in Y$ corresponds to the $S$-equivalence class of
${\Bbb C}_x$, $x \in \pi^{-1}(y)$.
By Lemma \ref{lem:contraction-varphi},
${\Bbb C}_x$, $x \in \pi^{-1}(p_i)$ is not irreducible.
Hence $p_i$ corresponds to a properly 0-semi-stable objects.
For a smooth point $y \in Y$,
${\Bbb C}_x$, $x \in \pi^{-1}(y)$ is irreducible.
Therefore the second claim also holds.
(2) is a consequence of Proposition \ref{prop:0-dim:smooth}
and Lemma \ref{lem:crepant}.

(3) We note that 
\begin{equation}
\begin{split}
& \langle \varrho_X,v(E_{ij}) \rangle=0,\\
& \langle v(E_{ij}),v(E_{ij}) \rangle =-2,\\
& \langle v(E_{ij}),v(E_{kl}) \rangle \geq 0,
\;(E_{ij} \ne E_{kl}).
\end{split}
\end{equation} 
As we see in Example \ref{ex:1} in appendix,
we can apply Lemma \ref{lem:appendix:lattice} (1) to our
situation.
Hence the matrix $(-\langle v(E_{ij}),v(E_{ik}) \rangle)_{j,k \geq 0}$ 
is of affine
type $\tilde{A},\tilde{D}, \tilde{E}$.
Then we may assume that $a_{i0}=1$ for all $i$.
By Lemma \ref{lem:appendix:lattice} (2),
we can choose an $\alpha$ with $-\langle v(E_{ij}),\alpha \rangle>0$
for all $j>0$.
Let ${\cal E}^\alpha$ be the universal family on 
$X \times X^\alpha$. 
(3) is a consequence of the following lemma.
\end{proof}

\begin{lem}\label{lem:0-stable:exceptional}
Assume that $\alpha$ satisfies 
$-\langle v(E_{ij}),\alpha \rangle>0$
for all $j>0$.
\begin{enumerate}
\item[(1)]
We set 
\begin{equation}
C_{ij}^\alpha:=\{x^\alpha \in X^\alpha| 
\Hom({\cal E}_{|X \times \{x^\alpha \} },E_{ij}) \ne 0 \}, j>0.
\end{equation}
Then 
$C_{ij}^\alpha$ is a smooth rational curve.
\item[(2)]
\begin{equation}
\begin{split}
Z_i^\alpha=&
\{ x^\alpha \in X^\alpha | \Hom(E_{i0},
{\cal E}_{|X \times \{x^\alpha \}}) \ne 0 \}
= \cup_j C_{ij}^\alpha.
\end{split}
\end{equation}
\item[(3)]
$\cup_j C_{ij}^\alpha$ is simple normal crossing
and $(C_{ij}^\alpha,C_{ik}^{\alpha})=\langle v(E_{ij}),v(E_{ik}) \rangle$.
\end{enumerate}
\end{lem}

\begin{proof}
(1) By our choice of $\alpha$, 
$\Hom(E_{ij},{\cal E}_{|X \times \{x^\alpha \} })=0$ for all 
$x^\alpha \in X^\alpha$.
If $C_{ij}^\alpha=\emptyset$, 
then $\chi(E_{ij},{\cal E}_{|X \times \{x^\alpha \} })=0$
implies that
$\Hom({\cal E}_{|X \times \{x^\alpha \}},E_{ij})=
\Ext^1({\cal E}_{|X \times \{x^\alpha \}},E_{ij})=0$.
Then $\Phi_{X \to X^\alpha}^{{\cal E}^{\vee}}(E_{ij})=0$, 
which is a contradiction.
Therefore $C_{ij}^\alpha \ne \emptyset$.
In order to prove the smoothness,
we consider the moduli space of coherent systems
\begin{equation}
N(\varrho_X,v(E_{ij})):=
\{(E,V)|E \in X^{\alpha}, V \subset \Hom(E,E_{ij}), \dim_{\Bbb C} V=1 \}.
\end{equation}
We have a natural projection $\iota:N(\varrho_X,v(E_{ij})) \to X^\alpha$
whose image is $C_{ij}^\alpha$.
For $(E,V) \in N(\varrho_X,v(E_{ij}))$, 
we have a homomorphism $\xi:E \to E_{ij} \otimes V^{\vee}$.
The Zariski tangent space at $(E,V)$ is
$\Hom(E,E \to E_{ij} \otimes V^{\vee})$.
By Lemma \ref{lem:0-stable:key} (3),
$\xi$ is surjective and 
$\ker \xi \in M_{{\cal O}_X(1)}^{G,\alpha}(\varrho_X-v(E_{ij}))$.
In particular 
$\Hom(E,E \to E_{ij} \otimes V^{\vee}) \cong \Ext^1(E,\ker \xi)$.
\begin{NB}
Since $E \to E_{ij} \otimes V^{\vee}=\ker \xi[1]$,
$\Hom(E,E \to E_{ij} \otimes V^{\vee}) \cong 
\Hom(E,\ker \xi[1])=\Ext^1(E,\ker \xi)$.
\end{NB}
Conversely for $F \in M_{{\cal O}_X(1)}^{G,\alpha}(\varrho_X-v(E_{ij}))$
and a non-trivial extension
\begin{equation}
0 \to F \to E \to E_{ij} \to 0,
\end{equation}
Lemma \ref{lem:0-stable:key} (4) implies that $E \in X^\alpha$
and $E \to E_{ij}$ defines an element of $N(\varrho_X,v(E_{ij}))$.
By Corollary \ref{cor:0-dim:reflection} (1) and our choice of $\alpha$,
$\Hom(F,E_{ij})=\Hom(E_{ij},F)=0$.
Hence $\dim \Ext^1(E_{ij},F)=2$.
Since $M_{{\cal O}_X(1)}^{G,\alpha}(\varrho_X-v(E_{ij}))$ 
is a reduced one point, we see that 
$N(\varrho_X,v(E_{ij}))$ is isomorphic
to ${\Bbb P}^1$.
We show that $\iota:N(\varrho_X,v(E_{ij})) \to X^\alpha$
is a closed immersion.
For $(E,V) \in N(\varrho_X,v(E_{ij}))$,
$\dim \Hom(E,E_{ij})=\dim \Hom(\ker \xi,E_{ij})+1=1$.
Hence $\iota$ is injective.
We also see that $\iota_*:\Ext^1(E,\ker \xi) \to \Ext^1(E,E)$
is injective. Therefore $\iota$ is a closed immersion.
\begin{NB}
 and the obstructions 
for infinitesimal deformations belong to 
\begin{equation}
\ker(\Ext^1(E,E \to E_{ij} \otimes V^{\vee}) \to
\Ext^2(E,E)). 
\end{equation}
By Lemma \ref{lem:0-stable:key} (3),
$\xi$ is surjective and $\ker \xi$ is an $\alpha$-stable
object. In particular
$\Ext^1(E,E \to E_{ij} \otimes V^{\vee}) \cong 
\Ext^2(E,\ker \xi) \cong \Hom(\ker \xi,E)^{\vee}$.
In order to prove the smoothness of 
$N(\varrho_X,v(E_{ij}))$, it is sufficient to prove
$\Hom(E,E) \to \Hom(\ker \xi,E)$ is surjective.
Since $\langle v(\ker \xi),v(E_{ij})\rangle =2$,
Corollary \ref{cor:0-dim:reflection} (1) implies that
$\Hom(\ker \xi,E_{ij})=0$.
Then
we see that $\Hom(\ker \xi,E) \cong \Hom(\ker \xi,\ker \xi) \cong {\Bbb C}$.
Therefore $\Ext^2(E,\ker \xi) \to \Ext^2(E,E)$ is injective.
$N(\varrho_X,v(E_{ij}))$ is isomorphic to ${\Bbb P}^1$ and 
$N(\varrho_X,v(E_{ij})) \to C_{ij}^\alpha \subset X^{\alpha}$ 
is a closed immersion.
\end{NB}

(2) By our choice of $\alpha$,
$\Hom(E_{i0},{\cal E}_{|X \times \{x^\alpha \} }) \ne 0$ for
 $x^\alpha \in Z_i^\alpha$.
Conversely if $\Hom(E_{i0},{\cal E}_{|X \times \{x^\alpha \}}) \ne 0$, then
Lemma \ref{lem:one-point} implies that
$\Supp(\pi_*(G^{\vee} \otimes {\cal E}_{|X \times \{x^\alpha \}}))=\{p_i \}$.
Since $\Supp(\pi_*(G^{\vee} \otimes {\cal E}_{|X \times \{x^\alpha \}}))$ 
depends only on the $S$-equivalence class of 
${\cal E}_{|X \times \{x^\alpha \}}$,
we have $\psi(\pi_\alpha(x^\alpha))=p_i$. Thus
$x^\alpha \in Z_i^\alpha$.
%The proof of Lemma \ref{lem:0-stable:key} (1) 
%implies that ${\cal E}_{|X \times \{x^\alpha \}}$ is $S$-equivalent
%to $\bigoplus_{j \geq 0} E_{ij}^{\oplus a_{ij}}$.
Therefore we have the first equality.
By the choice of $\alpha$, 
we also get $Z_i^\alpha \subset 
\cup_j C_{ij}^\alpha$.
If $\Hom({\cal E}_{|X \times \{x^\alpha \}},E_{ij}) \ne 0$, $j>0$,
then we see that 
$\Supp(\pi_*(G^{\vee} \otimes {\cal E}_{|X \times \{x^\alpha \}}))=\{p_i \}$,
which implies that $x^\alpha \in Z_i^\alpha$.
Thus the second claim also holds.
\begin{NB}
Old version:
In the same way, 
by the choice of $\alpha$, 
we get $Z_i^\alpha \subset 
\cup_j C_{ij}^\alpha$ and by Lemma \ref{lem:0-stable:key} (1) and (3), 
we get
$\cup_j C_{ij}^\alpha \subset Z_i^\alpha$.
\end{NB}

(3)
Since $(-\langle v(E_{ij}),v(E_{ik}) \rangle)_{j,k \geq 1}$
is of $ADE$-type,
by using Corollary \ref{cor:0-dim:reflection}, we can show that
$M_{{\cal O}_X(1)}^{G,\alpha}(v) \cong 
M_{{\cal O}_X(1)}^{G,\alpha}(v(E_{i0}))$ for
$v=v(E_{i0} \oplus \oplus_{j>0}E_{ij}^{\oplus b_j})$,
$0 \leq b_j \leq a_{ij}$ with $\langle v^2 \rangle=-2$.
In particular, they are non-empty. 
Then by a similar arguments in \cite[Prop. 2.9]{O-Y:1},
we can also show that $\cup_j C_{ij}^\alpha$ is simple
normal crossing
and $(C_{ij}^\alpha,C_{ik}^{\alpha})=\langle v(E_{ij}),v(E_{ik}) \rangle$.
For another proof, see Corollary \ref{cor:0-dim:NC}. 
\end{proof}

\begin{NB}
For a general case, we use the covering trick.
Let $D$ be a very ample divisor on $Y$ such that
there is a smooth curve $B \in |2D|$ with
$B \cap \pi(\cup_i Z_i) =\emptyset$ and
$|K_Y+D|$ contains a curve $C$ with $C \cap \pi(\cup_i Z_i) =\emptyset$.
Since $\pi$ is an isomorphic over
$Y \setminus \pi(\cup_i Z_i)$, we regards $B$ and $C$ as divisors 
on $X$. Let $\phi:\widetilde{Y} \to Y$ be the double covering branced along
$B$ and set $\widetilde{X}=X \times_Y \widetilde{Y}$.
We also denote $\widetilde{X} \to X$ by $\phi$.
Then $|K_{\widetilde{X}}|=|\phi^*(K_X+D)|$ contains
$\pi^*(C)$.
Since $\phi$ is \'{e}tale over $\cup_i Z_i$, we have a decomposition
$\pi^*(E)=E_1 \oplus E_2$ and 
the deformation theory of
$E$ and $E_1$ are the same. Therefore
we also get
$\dim \Def_E \geq \dim \Ext^1(E,E)-(\dim \Ext^2(E,E)-1)$.
\end{NB}

\begin{NB}
Let $D$ be an effective divisor on $X \setminus \cup_i Z_i$. 
For a 0-dimensional object $E$ on $\cup_i Z_i$,
$\Ext^2(E,E) \to \Ext^2(E,E(-D))$ is an isomorphism. Hence 
we have a homomorphism
\begin{equation}
\tau:\Ext^2(E,E) \to \Ext^2(E,E(-D)) \overset{\tr}{\to} 
H^2(X,{\cal O}_X(-D)),
\end{equation}
where $\tr$ is the trace map.
The obstruction for infinitesimal deformations of $E$ lives in 
$\ker \tau$.
If $C \in |K_X+D|$ does not intersect $\cup_i Z_i$, then
$\dim \ker \tau \leq \dim \Ext^2(E,E)-1$. 
\end{NB}

\begin{NB}
$H^2(X,{\cal O}_X(-D))$ is the space of obstructions for
infinitesimal deformations of line bundle $L$ with an isomorphism
$L_{|D} \to F$, where $F$ is a fixed line bundle on $D$.  
In our case, $L=\det E$ and $F={\cal O}_D$.
Since $\det E$ is trivial on $D$, every 
\end{NB}

\subsection{Fourier-Mukai transforms on $X$.}
\label{subsect:0-dim:FM}

We keep the notations in subsection \ref{subsect:wall-chamber}.
Assume that $X^\alpha$ consists of $\alpha$-stable objects.
Let ${\cal E}^{\alpha}$ be a universal family
on $X \times X^{\alpha}$.
We have an equivalence
$\Phi_{X \to X^{\alpha}}^{({\cal E}^{\alpha})^{\vee}}:
{\bf D}(X) \to {\bf D}(X^{\alpha})$.
If ${\cal F}^{\alpha}$ be another universal family, then
we see that 
\begin{equation}\label{eq:universal-family}
\Phi_{X \to X^{\alpha}}^{({\cal E}^{\alpha})^{\vee}} \circ
\Phi_{X^{\alpha} \to X}^{{\cal F}^{\alpha}}=
\Phi_{X^{\alpha} \to X^{\alpha}}^{{\cal O}_{\Delta}(L)}[-2],
L \in \Pic(X^{\alpha}).
\end{equation}
Let $\Gamma^{\alpha}$ be the closure of the graph of
the rational map $\pi_\alpha^{-1} \circ \pi$:
\begin{equation}
\begin{CD}
\Gamma^\alpha @>>> X^\alpha \\
@VVV @VV{\pi_\alpha}V\\
X @>>{\pi}> Y.
\end{CD}
\end{equation}

\begin{lem}\label{lem:univ-characterize}
\begin{enumerate}
\item[(1)]
We may assume that 
${\cal E}^{\alpha}_{|X \times (X^{\alpha} \setminus Z^\alpha)} \cong 
{\cal O}_{{\Gamma^\alpha}|X \times (X^{\alpha} \setminus Z^\alpha)}$.
\item[(2)]
${\cal E}^{\alpha}$ is characterized
by 
${\cal E}^{\alpha}_{|X \times (X^{\alpha} \setminus Z^\alpha)}$ and
$\det \Phi_{X \to X^{\alpha}}^{({\cal E}^{\alpha})^{\vee}}(G)$.
\end{enumerate}
\end{lem}

\begin{proof}
(1)
We note that
${\cal E}^{\alpha}_{|X \times (X^{\alpha} \setminus Z^\alpha)} \cong 
({\cal O}_{{\Gamma^\alpha}}\otimes p_{X^{\alpha}}^*(L))
_{|X \times (X^{\alpha} \setminus Z^\alpha)} $,
where $L \in \Pic(X^{\alpha} \setminus Z^\alpha)$.
We also denote an extension of $L$ to $X^{\alpha}$
by $L$. Then
${\cal E}^{\alpha}\otimes p_{X^{\alpha}}^*(L^{\vee})$
is a desired universal family. 

(2)
Assume that
${\cal E}^{\alpha}_{|X \times (X^{\alpha} \setminus Z^\alpha)} \cong 
({\cal E}^{\alpha}\otimes p_{X^{\alpha}}^*(L))
_{|X \times (X^{\alpha} \setminus Z^\alpha)}$ and
$\det \Phi_{X \to X^{\alpha}}^{({\cal E}^{\alpha})^{\vee}}(G)
\cong \det 
\Phi_{X \to X^{\alpha}}^{({\cal E}^{\alpha}
\otimes p_{X^{\alpha}}^*(L))^{\vee}}(G)$. 
Then $L_{|X^{\alpha} \setminus Z^\alpha} \cong 
{\cal O}_{X^{\alpha} \setminus Z^\alpha}$
and $L^{\otimes \rk G} \cong {\cal O}_{X^{\alpha}}$.
In order to prove $L \cong {\cal O}_{X^{\alpha}}$, it is
sufficient to prove the injectivity of
the restriction map
\begin{equation}
r:\Pic(X^{\alpha}) \to \Pic(X^{\alpha} \setminus Z^{\alpha}) 
\times \prod_{i,j} \Pic(C_{ij}^\alpha).
\end{equation}
If $L_{|X^{\alpha} \setminus Z^{\alpha}} 
\cong {\cal O}_{X^{\alpha} \setminus Z^{\alpha}}$,
then we can write $L={\cal O}_X(\sum_{i,j} r_{ij} C_{ij}^\alpha)$.
Since the intersection matrix $((C_{ij}^\alpha,C_{ik}^\alpha))_{j,k}$ 
is negative definite,
$\deg(L_{|C_{ij}^\alpha})=
\sum_k r_{ik} (C_{ik}^\alpha,C_{ij}^\alpha)=0$ for all $i,j$ 
implies that
$r_{ij}=0$ for all $i,j$.
Thus $r$ is injective.
\end{proof}
%
%
%Hence if 
%$\det \Phi_{X \to X^{\alpha}}^{({\cal F}^{\alpha})^{\vee}}({\cal O}_X)
%\cong 
%\det \Phi_{X \to X^{\alpha}}^{({\cal E}^{\alpha})^{\vee}}({\cal O}_X)$,
%then $\Phi_{X \to X^{\alpha}}^{({\cal F}^{\alpha})^{\vee}} \cong
%\Phi_{X \to X^{\alpha}}^{({\cal E}^{\alpha})^{\vee}}$.
%Replacing ${\cal E}^{\alpha}$ by
%${\cal E}^{\alpha} \otimes p_{X^{\alpha}}^*(L)$, $L \in \Pic(X^{\alpha})$,
%we may assume that 
%$\det \Phi_{X \to X^{\alpha}}^{({\cal E}^{\alpha})^{\vee}}({\cal O}_X)
%\cong {\cal O}_{X^{\alpha}}$.

\begin{defn}
We set $\Lambda^{\alpha}:=
\Phi_{X \to X^{\alpha}}^{({\cal E}^{\alpha})^{\vee}}[2]$.
\end{defn}

\begin{lem}\label{lem:compatible2}
${\cal O}_X(n) \otimes \underline{\;\;}$ and 
$\Lambda^{\alpha}$ are commutative. 
\end{lem}

\begin{proof}
Let $D$ be an effective divisor on $X$
such that $D \cap Z =\emptyset$.
It is sufficient to prove that
\begin{equation}
{\cal E}^{\alpha} \otimes 
({\cal O}_X(-D) \boxtimes {\cal O}_{X^{\alpha}}(D))
\cong {\cal E}^{\alpha}.
\end{equation}
We note that ${\cal E}^{\alpha} \cong {\cal O}_{\Gamma^\alpha}$
over $X^{\alpha} \setminus Z^\alpha$.
Obviously the claim holds over $X^{\alpha} \setminus Z^\alpha$.
By Lemma \ref{lem:univ-characterize},
we shall show that
$\det \Lambda^{\alpha}(G(D)) \cong 
\det (\Lambda^{\alpha}(G)(D))$.
We have an exact triangle
\begin{equation}
({\cal E}^{\alpha})^{\vee} \to ({\cal E}^{\alpha})^{\vee}(D) \to
({\cal E}^{\alpha})^{\vee}_{|D}(D) \to ({\cal E}^{\alpha})^{\vee}[1].
\end{equation} 
Since $({\cal E}^{\alpha})^{\vee}_{|D}(D) \cong 
{\cal O}_{\Delta|D}(D)[-2]$,
we have an exact triangle
\begin{equation}
\Lambda^{\alpha}(G) \to 
\Lambda^{\alpha}(G(D)) \to
G_{|D}(D) \to \Lambda^{\alpha}(G)[1].
\end{equation}
Hence we get 
$\det \Lambda^{\alpha}(G(D)) \cong 
(\det \Lambda^{\alpha}(G))( (\rk G) D) \cong
\det (\Lambda^{\alpha}(G)(D))$.
%Therefore 
%${\cal O}_X(D) \otimes \underline{\;\;}$ and 
%$\Lambda^{\alpha}$ are commutative. 
\end{proof}

\begin{prop}\label{prop:Phi-alpha}
\begin{enumerate}
\item[(1)]
$G^\alpha:=\Lambda^{\alpha}(G)$ is a locally free sheaf and
${\bf R}\pi_{\alpha*}({G^\alpha}^{\vee} \otimes G^\alpha)=
\pi_{\alpha*}({G^\alpha}^{\vee} \otimes G^\alpha)$.
\item[(2)]
$\Lambda^{\alpha}(E_{ij})[j]$ is a sheaf, where
$j=-1$ or 0 according as $(\alpha,c_1(E_{ij}))<0$ or $(\alpha,c_1(E_{ij}))>0$.
\item[(3)]
We set ${\cal A}^\alpha:=\pi_{\alpha*}({G^\alpha}^{\vee} \otimes G^\alpha)$.
Then ${\cal A}^\alpha$ is a reflexive sheaf on $Y$.
Under the identification $X^{\alpha} \setminus Z^{\alpha}
\cong X \setminus Z$, 
$G_{|X^{\alpha} \setminus Z^{\alpha}}^\alpha$ corresponds to
$G_{|X \setminus Z}$.
Hence we have an isomorphism ${\cal A} \cong {\cal A}^\alpha$.
\item[(4)]
We identify $\Coh_{{\cal A}}(Y)$ with $\Coh_{{\cal A}^\alpha}(Y)$
via $ {\cal A} \cong {\cal A}^\alpha$.
Then we have a commutative diagram
\begin{equation}
\begin{CD}
{\cal C} @>{\Lambda^{\alpha}}>> \Lambda^{\alpha}({\cal C})\\
@V{{\bf R}\pi_*{\cal H}om(G,\;\;)}VV 
@VV{{\bf R}\pi_{\alpha*}{\cal H}om(G^\alpha,\;\;)}V\\
\Coh_{{\cal A}}(Y) @= \Coh_{{\cal A}^\alpha}(Y) 
\end{CD}
\end{equation}
In particular $G^\alpha$ gives a local projective generator of 
$\Lambda^{\alpha}({\cal C})$.
\item[(5)]
We set
\begin{equation}
\begin{split}
S^{\alpha}:=&\{ \Lambda^{\alpha}(E_{ij})[-1]| i,j\} 
\cap \Coh(X^{\alpha}),\\
{\cal T}^{\alpha}:=& \{E \in  \Coh(X^{\alpha})|
\Hom(E,c)=0, c \in S^{\alpha} \},\\
{\cal S}^{\alpha}:=& \{E \in  \Coh(X^{\alpha})|
\text{ $E$ is a successive extension of subsheaves of 
$c \in S^{\alpha}$ }\}.
\end{split}
\end{equation}
Then $({\cal T}^{\alpha},{\cal S}^{\alpha})$ is a torsion pair of
$\Coh(X^{\alpha})$ and
$\Lambda^{\alpha}({\cal C})$ is the tilting of
$\Coh(X^{\alpha})$ with respect to $({\cal T}^{\alpha},{\cal S}^{\alpha})$.
\item[(6)]
Let $G'$ be a local projective generator of ${\cal C}$.
Then $\Lambda^\alpha$ induces an isomorphism
${\cal M}_H^{G'}(v)^{ss} \to 
{\cal M}_H^{\Lambda^\alpha(G')}(\Lambda^\alpha(v))^{ss}$.
\end{enumerate}
\end{prop}

\begin{proof}
(1)
We note that $\Hom({\cal E}_{|X \times \{x^\alpha \} }^{\alpha},G[i])
\cong \Hom(G,{\cal E}_{|X \times \{x^\alpha \}}^{\alpha}[2-i])^{\vee}=0$
for $i \ne 2$ and $x^\alpha \in X^{\alpha}$.
By the base change theorem, $G^\alpha$ is a locally free sheaf.
%We prove that $R^i \pi_*({G^\alpha}^{\vee} \otimes G^\alpha)=0$, $i \ne 0$.
%Let ${\cal O}_X(1)$ be the pull-back of an ample
%line bundle on $Y$.
%By the Leray spectral sequence, 
%it is sufficient to prove that
%$\Hom(\Lambda^{\alpha}(G),\Lambda^{\alpha}(G)(n)[1])=0$ for $n \gg 0$.
By using Lemma \ref{lem:compatible2} and the ampleness of 
${\cal O}_Y(1)$, we have 
\begin{equation}
\begin{split}
H^0(Y,R^i \pi_{\alpha*}({G^\alpha}^{\vee} \otimes G^\alpha)(n))=& 
\Hom(\Lambda^{\alpha}(G),\Lambda^{\alpha}(G)(n)[i])\\
=& \Hom(\Lambda^{\alpha}(G),\Lambda^{\alpha}(G(n))[i])\\
=& 
\Hom(G,G(n)[i])=
H^0(Y,R^i \pi_{*}({G}^{\vee} \otimes G)(n))=0
\end{split}
\end{equation}
for $n \gg 0$ and $i \ne 0$.
Therefore $R^i \pi_*({G^\alpha}^{\vee} \otimes G^\alpha)=0$, $i \ne 0$ and
the claim holds.

(2)
If $(\alpha,c_1(E_{ij}))<0$, then
$\Hom({\cal E}_{|X \times \{x^\alpha \} }^{\alpha},E_{ij}[2])
\cong 
\Hom(E_{ij},{\cal E}_{|X \times \{x^\alpha \}}^{\alpha})^{\vee}=0$
for $x^\alpha \in X^{\alpha}$.
Since $\Hom({\cal E}_{X \times |\{x^\alpha \} }^{\alpha},E_{ij})
=0$ if $x^\alpha \not \in Z_i^\alpha$,
we see that $\Lambda^{\alpha}(E_{ij})[-1]$ is 
a torsion sheaf whose support is contained in $Z_i^{\alpha}$.

If $(\alpha,c_1(E_{ij}))>0$, then
$\Hom({\cal E}_{|X \times \{x^\alpha \} }^{\alpha},E_{ij})=0$
for $x^\alpha \in X^{\alpha}$.
Since $\Hom({\cal E}_{|X \times \{x^\alpha \}}^{\alpha},E_{ij}[2])
=0$ if $x^\alpha \not \in Z_i^\alpha$,
we see that $\Lambda^{\alpha}(E_{ij})$ is 
a torsion sheaf whose support is contained in $Z_i^{\alpha}$.

(3)
By the claim (1) and \cite[Lem. 2.1]{E:1},
${\cal A}^\alpha$ is a reflexive sheaf. 
Since ${\cal E}^{\alpha}$ is isomorphic
to ${\cal O}_{\Gamma^\alpha}$
over $X^{\alpha} \setminus Z^{\alpha}$, we get
$\Lambda^{\alpha}(G)_{|X^\alpha \setminus Z^\alpha} 
\cong \pi_\alpha^{-1} \circ \pi(G_{|X \setminus Z})$. 
Hence the second claim also follows.

(4)
For $E \in {\cal C}$,
we first prove that ${\bf R}\pi_*({G^\alpha}^{\vee} 
\otimes \Lambda^{\alpha}(E))
\in \Coh_{{\cal A}^\alpha}(Y)$.
As in the proof of (1), we have
\begin{equation}
\begin{split}
H^i(Y,{\bf R}\pi_*({G^\alpha}^{\vee} \otimes \Lambda^{\alpha}(E))(n))
=&
\Hom(G^\alpha,\Lambda^{\alpha}(E)(n)[i])\\
=&\Hom(G,E(n)[i])=0
\end{split}
\end{equation}
 for
$i \ne 0$, $n \gg 0$.
Therefore $H^i({\bf R}\pi_*({G^\alpha}^{\vee} \otimes \Lambda^{\alpha}(E)))=0$
for $i \ne 0$.
For $E \in {\cal C}$,
we take an exact sequence
\begin{equation}
G(-m)^{\oplus M} \to G(-n)^{\oplus N} \to E \to 0
\end{equation}
Then we have a diagram
\begin{equation}\label{eq:A-A'}
\begin{CD}
{\cal A}(-m)^{\oplus M} @>>> {\cal A}(-n)^{\oplus N} @>>> 
\pi_*(G^{\vee} \otimes E) @>>> 0 \\
@V{\phi}VV @VV{\psi}V @. @.\\
{\cal A}^\alpha(-m)^{\oplus M} @>>> {\cal A}^\alpha(-n)^{\oplus N} @>>> 
\pi_*({G^\alpha}^{\vee} \otimes \Lambda^{\alpha}(E)) @>>> 0
\end{CD}
\end{equation}
which is commutative over $Y^*:=Y \setminus \{p_1,p_2,\dots,p_n\}$,
where $\phi$ and $\psi$ are the isomorphisms induced by
${\cal A} \cong {\cal A}^\alpha$.
Let $j:Y^* \hookrightarrow Y$ be the inclusion.
Since ${\cal H}om({\cal A},{\cal A}^\alpha) \to
j_* j^*{\cal H}om({\cal A},{\cal A}^\alpha)$ is an isomorphism,
\eqref{eq:A-A'} is commutative,
which induces an isomorphism
$\xi:\pi_*(G^{\vee} \otimes E) \to 
 \pi_*({G^\alpha}^{\vee} \otimes \Lambda^{\alpha}(E))$.
It is easy to see that 
the construction of $\xi$ is functorial and defines an isomorphism
${\bf R}\pi_*{\cal H}om(G,\;\;) \cong 
{\bf R}\pi_*{\cal H}om(G^\alpha,\;\;) \circ \Lambda^{\alpha}$.

(5)
Since $\Lambda^\alpha$ is an equivalence,
$\Lambda^\alpha(E_{ij})$ are irreducible objects of 
$\Lambda^\alpha({\cal C})$.
By Lemma \ref{lem:tilting:C} and Proposition \ref{prop:tilting:S-T},
we get the claim.
\begin{NB}
For $E \in \Coh(X^{\alpha})$, we consider
$\phi:G^\alpha \otimes \pi^*(\pi_*({G^\alpha}^{\vee} \otimes E)) \to E$.
We set $E_1:=\im \phi$ and $E_2:=\coker \phi$.
Since $\Hom(G^\alpha,\Phi^{\alpha}(F)[-1])=
\Hom(G,F[-1])=0$ for $F \in {\cal C}$,
$G^\alpha \in {\cal T}^{\alpha}$.
Hence $E_1 \in {\cal T}^{\alpha}$.
We shall show that
$E_2 \in {\cal S}^{\alpha}$.
We note that ${\bf R}\pi_*({G^\alpha}^{\vee} \otimes E_1)=
\pi_*({G^\alpha}^{\vee} \otimes E_1)$ and
${\bf R}\pi_*({G^\alpha}^{\vee} \otimes E_2)=
R^1 \pi_*({G^\alpha}^{\vee} \otimes E_1)[-1]$. 
Then $E_1, E_2[1] \in \Lambda^{\alpha}({\cal C})$.
Since $\Supp(E_2) \subset Z^{\alpha}$,
$E_2[1]$ is generated by $\Lambda^{\alpha}(E_{ij})$.
Hence if $E_2 \ne 0$, then $\Hom(E_2[1],c[1])\ne 0$ for an object
$c \in S^{\alpha}$.
By the induction on the support of $E_2$,
we see that $E_2 \in {\cal S}^{\alpha}$.
Therefore $({\cal T}^{\alpha},{\cal S}^{\alpha})$ is a torsion pair
of $\Coh(X^{\alpha})$.
We also see that 
\begin{equation}
\begin{split}
{\cal T}^{\alpha}&=\{E \in \Coh(X^{\alpha})|
R^1 \pi_*({G^\alpha}^{\vee} \otimes E)=0 \},\\
{\cal S}^{\alpha}&=\{E \in \Coh(X^{\alpha})|
\pi_*({G^\alpha}^{\vee} \otimes E)=0 \}
\end{split}
\end{equation} 
and $\Lambda^{\alpha}({\cal C})$ is the tilting of 
$\Coh(X^{\alpha})$.
\end{NB}

(6) We note that the proof of (1) implies that
$\Lambda^\alpha(G')$ is a local projective generator of
$\Lambda^\alpha({\cal C})$.
By Lemma \ref{lem:compatible2}, 
$\chi(G',E(n))=\chi(\Lambda^\alpha(G'),\Lambda^\alpha(E)(n))$.
Hence the claim holds.
\end{proof}

\begin{NB}
\begin{lem}\label{lem:reflexive}
Let $E$ be a locally free sheaf on $X$ such that
$R^1 \pi_*(E)=R^1 \pi_*(E^{\vee})=0$.
Then $\pi_*(E)$ is a reflexive sheaf on $Y$. 
\end{lem}

\begin{proof}
We note that 
$\Hom(E,{\cal O}_{C_i}(-2))=0$.
Hence we have an exact sequence
\begin{equation}
0 \to E' \to E \to {\cal O}_{C_i}(-1) \otimes 
\Hom(E,{\cal O}_{C_i}(-1))^{\vee} \to 0.
\end{equation}
Then $E'$ also satisfies the assumptions.
Applying the same procedure, we get a subsheaf $E''$ of $E$ 
such that $\pi_*(E'') \cong \pi_*(E)$,
$R^1 \pi_*({E''}^{\vee})=0$ and
$\Hom(E'',{\cal O}_{C_i}(-1))=0$ for all $i$.
Then $E''$ is a reflexive sheaf.
\end{proof}
\end{NB}

\begin{rem}
If ${\cal C}=^{-1} \Per(X/Y)$, then 
${\cal O}_X \in ^{-1} \Per(X/Y)$ and
$\Lambda^{\alpha}({\cal O}_X)$ is a line bundle on
$X^{\alpha}$.
Hence we may assume that
$\Lambda^{\alpha}({\cal O}_X)
 \cong {\cal O}_{X^{\alpha}}$.
Then $\Hom({\cal O}_{X^{\alpha}},
\Lambda^{\alpha}({\cal O}_{C_{ij}}(-1))[n])=0$ 
for all $n$.
Thus $\Lambda^{\alpha}({\cal O}_{C_{ij}}(-1))[n]$ is a successive
extensions of ${\cal O}_{C_{ik}}(-1)$.
We also get
$\Hom({\cal O}_{X^{\alpha}},\Lambda^{\alpha}({\cal O}_{Z_i}))={\Bbb C}$ and
$\Hom({\cal O}_{X^{\alpha}},\Lambda^{\alpha}({\cal O}_{Z_i})[n])=0$
for $n \ne 0$. 
\begin{NB}
If $-(\alpha,c_1(E_{ij}))>0$ unless $j \ne j_i$, $j_i >0$,
then $\Lambda^\alpha(E_{i0})[-1]$ is a line bundle on the support.
Thus $\Lambda^\alpha(E_{i0})={\cal O}_{C_{ik}}(-2)[1]$.
\end{NB}
\end{rem} 

Since $\Lambda^\alpha$ is an equivalence with
$\Lambda^\alpha(\varrho_X)=\varrho_{X^\alpha}$,
we have the following corollary.
\begin{cor}\label{cor:Phi-alpha}
For a general $\alpha$,
the equivalence 
$$
\Lambda^\alpha:{\cal C}
\to \Lambda^\alpha({\cal C})
$$
induces an isomorphism:
$$
\Lambda^\alpha:{\cal M}_{{\cal O}_X(1)}^{G,\beta}(\varrho_X)^{ss} \to
{\cal M}_{{\cal O}_{X^\alpha}(1)}^{G^\alpha,\Lambda^\alpha(\beta)}
(\varrho_{X^\alpha})^{ss},
$$ 
where 
$\beta \in \varrho_X^{\perp}$.
\end{cor}

\subsubsection{Wall and chambers}

For the $0$-stable objects $E_{ij}$ in Theorem \ref{thm:RDP-desing}, 
we set $v_{ij}:=v(E_{ij})$.
By Lemma \ref{lem:simple-generator},
$\{E_{ij}\}$ is the set of
irreducible objects $E$ with $\Supp(E) \subset \cup_i Z_i$. 
Let ${\frak g}_i$ be the finite Lie algebra whose Cartan matrix is
$(-\langle v_{ij},v_{ik} \rangle_{j,k \geq 1})$ and
\begin{equation}
R_i:=\left\{\left. u=\sum_{j>0} n_{ij}' v_{ij} \right| 
\langle u^2 \rangle=-2, n_{ij}' \geq 0 \right \}. 
\end{equation}
Then $R_i$ is identified with the set of positive roots of
${\frak g}_i$. In particular, $R_i$ is a finite set.

\begin{defn}
For $u \in \cup_i R_i$,
we define the wall as
\begin{equation}
W_u:=\left\{\alpha \in \NS(X) \otimes {\Bbb R} \left|
\frac{\langle u,\alpha \rangle}{\langle u,v(G) \rangle}
=\frac{\langle v,\alpha \rangle}{\langle v,v(G) \rangle}
\right. \right\}.
\end{equation}
A connected component of 
$\NS(X) \otimes {\Bbb R} \setminus \cup_u W_u$ is called a chamber.
\end{defn}

\begin{rem}
If $v=\varrho_X$, then $W_u=u^{\perp}$.
\end{rem}

\begin{lem}\label{lem:Hodge}
Let $v$ be the Mukai vector of a 0-dimensional object $E$, 
which is primitive.
\begin{enumerate}
\item[(1)]
$\overline{M}_{{\cal O}_X(1)}^{G,\alpha}(v)$ consists of
$\alpha$-twisted stable objects
if and only if $\alpha \not \in \cup_u W_u$.
We say that $\alpha$
is general with respect to $v$. 
\item[(2)]
If $\alpha$ is general with respect to $v$,
then the virtual Hodge number of
$M_{{\cal O}_X(1)}^{G,\alpha}(v)$ does not depend on the choice
of $\alpha$. 
In particular, the non-emptyness of 
$M_{{\cal O}_X(1)}^{G,\alpha}(v)$ 
does not depend on the choice
of $\alpha$.
\end{enumerate} 
\end{lem}

\begin{proof}
(1)
For $E \in \overline{M}_{{\cal O}_X(1)}^{G,\alpha}(v)$,
we assume that $E$ is $S$-equivalent to
$\oplus_{i=1}^n E_i$.
If $\langle v(E_i)^2 \rangle=0$ for all $i$, then
$v(E_i) \in {\Bbb Z}_{>0}\varrho_X$. Hence
$v=\sum_{i=1}^n v(E_i)$ is not primitive.
Therefore we may assume that $\langle v(E_1)^2 \rangle=-2$. 
By the $\alpha$-stability of $E_1$,
$\Supp(E_1) \subset Z_i$ for an $i$.
Since $E_1$ is generated by $\{E_{ij}|0 \leq j \leq s_i \}$,
$v(E_1) \in \oplus_{j=0}^{s_i} {\Bbb Z}_{\geq 0}v_{ij}$.
Then we see that 
$v(E_1) \in \pm R_i +{\Bbb Z}\varrho_X$.
Therefore the claim holds.
(2)
The proof is similar to that of \cite[Prop. 2.6]{Y:11}.
\end{proof}

\begin{lem}
\begin{enumerate}
\item[(1)]
Let $w_1:=v_{i0}+\sum_{j=1}^{s_i} n_{ij} v_{ij}$, $n_{ij} \geq 0$  
be a Mukai vector with
$\langle w_1^2 \rangle \geq -2$.
Then there is an $\alpha$-twisted stable object $E$ with
$v(E)=w_1$ for a general $\alpha$.
\item[(2)]
Let $w_2 \in R_i$
be a non-zero Mukai vector.
Then there is an $\alpha$-twisted stable object $E$ with
$v(E)=w_2$ for a general $\alpha$.
\end{enumerate}
\end{lem}

\begin{proof}
(1)
By Proposition \ref{prop:0-dim:equivalence} below,
we may assume that 
${\cal C}=\Per(X'/Y,{\bf b}_1,...,{\bf b}_n)$.
\begin{NB}
$\alpha$ is changed to $\Lambda_\beta(\alpha)$, ($\beta=\alpha$).
\end{NB}
The claim follows from Lemme \ref{lem:alpha>0} below
and Lemma \ref{lem:Hodge} (2).
Instead of using Lemma \ref{lem:alpha>0},
we can also use Corollary \ref{cor:0-dim:reflection}
to show the claim for a special $\alpha$.

(2)
We set $w_1:=\sum_{j=0}^{s_i} a_{ij} v_{ij}-w_2$.
Then $w_1$ is the Mukai vector in (1).
We can take a general element $\alpha \in \NS(X) \otimes {\Bbb Q}$ 
such that $\langle \alpha,w_1 \rangle=0$.
Then $\alpha$ is general with respect to $w_1$ and
we have a $\alpha$-twisted stable object $E$
with $v(E)=w_1$.
We consider $X^{\alpha'}$ such that $\alpha'$ is sufficiently close to
$\alpha$ and 
$\langle \alpha',v(E) \rangle>0$.
Since $\Lambda^{\alpha'}$ is an equivalence,
there is a morphism
$\phi:E \to {\cal E}_{|\{y \} \times X}^{\alpha'}$,
where $y \in X^{\alpha'}$. 
By our choice of $\alpha$, $\coker \phi$ is an $\alpha$-twisted
stable object with $v(\coker \phi)=w_2$. 
Then the claim follows from Lemma \ref{lem:Hodge} (2).
\end{proof}

\subsubsection{A special chamber}

We take $\alpha \in \varrho_X^{\perp}$
with $-\langle v(E_{ij}),\alpha \rangle>0$, $j>0$. 
%Let ${\cal E}$ be the universal family on $X \times X^{\alpha}$.
%We set
%\begin{equation}
%\Phi:=\Phi_{X \to X^{\alpha}}^{{\cal E}^{\vee}}.
%\end{equation}
%\begin{equation}
%C_{ij}':=\{x' \in X'| \Hom({\cal E}_{|\{x' \} \times X},E_{ij}) \ne 0\}, \;j>0
%\end{equation}
%is a smooth rational curve and
%$\Phi(E_{ij})={\cal O}_{C_{ij}^\alpha}(b_{ij}')[-1]$. 
\begin{lem}\label{lem:Lambda(E)}
$\Lambda^{\alpha}(E_{ij})[-1]$, $j>0$ is a line bundle on $C_{ij}^\alpha$.
We set $\Lambda^{\alpha}(E_{ij}):={\cal O}_{C_{ij}^\alpha}(b_{ij}^\alpha)[1]$.
\end{lem}

\begin{proof}
We note that
$\Lambda^{\alpha}(E_{ij}) \overset{{\bf L}}{\otimes} {\Bbb C}_{x^\alpha}
={\bf R}\Hom({\cal E}^{\alpha}_{|X \times \{ x^\alpha \}},E_{ij}[2])$.
Then $H^k(\Lambda(E_{ij}) \overset{{\bf L}}{\otimes} {\Bbb C}_{x^\alpha})=0$
for $k \not =-1,-2$. Hence $H^k(\Lambda^{\alpha}(E_{ij}))=0$
for $k \not =-1,-2$ and $H^{-2}(\Lambda^{\alpha}(E_{ij}))$ is a locally free
sheaf.  
By the proof of Theorem \ref{thm:RDP-desing} (3),
$\Supp(H^k(\Lambda^{\alpha}(E_{ij}))) \subset C_{ij}^\alpha$ for all $k$.
Hence $H^{-2}(\Lambda^{\alpha}(E_{ij}))=0$, which implies that
$\Lambda^{\alpha}(E_{ij})[-1] \in \Coh(X^{\alpha})$.
Since $\Hom({\Bbb C}_{x^\alpha},\Lambda^{\alpha}(E_{ij})[-1])=
\Hom({\cal E}^{\alpha}_{|X \times \{ x^\alpha \}},E_{ij}[-1])=0$,
$\Lambda^{\alpha}(E_{ij})[-1]$ is purely 1-dimensional.
We set $C:=\Div(\Lambda^{\alpha}(E_{ij})[-1])$. Then
$(C^2)=\langle v(\Lambda^{\alpha}(E_{ij})[-1])^2 \rangle
=\langle v(E_{ij})^2 \rangle=-2$, which implies that $C=C_{ij}^\alpha$.
Therefore $\Lambda^{\alpha}(E_{ij})[-1]$ is a 
line bundle on $C_{ij}^\alpha$.
%${\cal O}_{C_{ij}^\alpha}$-module  
\end{proof}

\begin{cor}\label{cor:0-dim:NC}
\begin{enumerate}
\item[(1)]
$(C_{ij}^\alpha,C_{i'j'}^\alpha)=\langle v(E_{ij}),v(E_{i'j'}) \rangle$.
\item[(2)]
$\{C_{ij}^\alpha \}$ is a simple normal crossing divisor. 
\end{enumerate}
\end{cor}

\begin{proof}
(1)
By Lemma \ref{lem:Lambda(E)},
$(C_{ij}^\alpha,C_{i'j'}^\alpha)=
\langle v(\Lambda^\alpha(E_{ij})),v(\Lambda^\alpha(E_{i'j'})) \rangle=
\langle v(E_{ij}),v(E_{i'j'}) \rangle$.
Then (2) also follows.
\end{proof}

%Assume that $\alpha \in v_0^{\perp}$
%satisfies $-\langle \alpha,v(E_{ij}) \rangle>0$ for $j>0$.
$E_{i0}$ is a subobject of ${\cal E}_{|X \times \{x^\alpha \}}$
for $x^\alpha \in Z_i^\alpha$ 
and we have an exact sequence
\begin{equation}
0 \to E_{i0} \to {\cal E}_{|X \times \{x^\alpha \} } \to F \to 0,\;
 x^\alpha \in Z_i^\alpha
\end{equation}
where $F$ is a $0$-semi-stable object with 
$\gr(F)=\oplus_{j= 1}^{s_i} E_{ij}^{\oplus a_{ij}}$.
Then we get an exact sequence
\begin{equation}\label{eq:0-dim:Z}
0 \to \Lambda^\alpha(F)[-1] \to \Lambda^\alpha(E_{i0})
 \to {\Bbb C}_{x^\alpha} \to 0
\end{equation}
in $\Coh(X^\alpha)$.
Thus $\Lambda^\alpha(E_{i0}) \in \Coh(X^\alpha)$.

\begin{defn}
We set 
$A_{i0}^\alpha:=\Lambda^\alpha(E_{i0})$ and
$A_{ij}^\alpha:=\Lambda^\alpha(E_{ij})=
{\cal O}_{C_{ij}^\alpha}(b_{ij}^\alpha)[1]$ for $j>0$.
\end{defn}

\begin{lem}\label{lem:A_alpha}
\begin{enumerate}
\item[(1)]
$\Hom(A_{i0}^\alpha,A_{ij}^\alpha[-1])=
\Ext^1(A_{i0}^\alpha,A_{ij}^\alpha[-1])=0$.
\item[(2)]
We set ${\bf b}_i^\alpha:=
(b_{i1}^\alpha,b_{i2}^\alpha,\dots,b_{i s_i}^\alpha)$.
Then $A_{i0}^\alpha \cong A_0({\bf b}_i^\alpha)$.
In particular,
$\Hom(A_{i0}^\alpha,{\Bbb C}_{x^\alpha})=
{\Bbb C}$ for $x^\alpha \in Z_i^\alpha$. 
\end{enumerate}
\end{lem}
\begin{proof}
(1)
We have
\begin{equation}
\begin{split}
\Hom(A_{i0}^\alpha,A_{ij}^\alpha[k])&=
\Hom(\Lambda^\alpha(E_{i0}),\Lambda^\alpha(E_{ij})[k])\\
&=\Hom(E_{i0},E_{ij}[k])=0
\end{split}
\end{equation} 
for $k=-1,0$.

(2)
By \eqref{eq:0-dim:Z} and (1), 
we can apply Lemma \ref{lem:G-1per:characterize}
and get
$A_{i0}^\alpha=A_0({\bf b}_i^\alpha)=A_{p_i}$.
%Since $\Hom(A_{i0},A_{i0})={\Bbb C}$,
%the claim follows from (1) and \eqref{eq:0-dim:Z}.
\end{proof}

%We set $G':=\Phi(G)[2]$. Then
%$G'$ is a locally free sheaf and is a local projective generator 
%of $\Per(X'/Y')$.
%$\Hom({\cal E}_{|X \times \{ x' \}},G[i])=
%\Hom(G,{\cal E}_{|X \times \{ x' \}}[2-i])=0$ for $i \ne 2$.
%By the base change theorem,
%$G'$ is a locally free sheaf.
%Since $\Hom(G',A_{ij}'[i])=\Hom(G,E_{ij}[i])$,
%the claim follows.
%
%For $E \in {\cal C}$,
%$\Hom(G,E(n))$
%
%
%We have a morphism $\psi:X^{\alpha} \to Y$, which is an isomorphism
%over $X_1:=\psi^{-1}(Y \setminus \{ p_1,...,p_n \})$.
%
%We note that $\Phi(E)[2]=E$ over $X_1$. Hence
%$H^0(\Phi(E))$ and $H^1(\Phi(E))$ are supported on
%$X' \setminus X_1$. 
%Since $\Hom({\cal E}_{|X \times \{ x' \}},E)=0$ for
%a general $x' \in X'$, we get
%$H^0(\Phi(E))=0$.
%Hence $\Hom(A_{i0}',H^1(\Phi(E)))=
%\Hom(A_{i0}',\Phi(E)[1])=\Hom(\Phi(E_{i0}[2],\Phi(E)[1])
%=\Hom(E_{i0},E[-1])=0$. Thus 

\begin{rem}
Assume that $\alpha \in v_0^{\perp}$ satisfies 
$-\langle v(E_{ij}),\alpha \rangle<0$, $j>0$.
Then $\Phi(E_{ij})[2]={\cal O}_{C_{ij}^\alpha}(b_{ij}'')$, $j>0$
and $\Phi(E_{i0})[2]=A_0({\bf b}_i'')[1]$ belong to
$\Per(X^\alpha/Y,{\bf b}_1'',...,{\bf b}_n'')^*$,
where ${\bf b}_i'':=(b_{i1}'',...,b_{i s_i}'')$. 
\end{rem}

By Proposition \ref{prop:Phi-alpha}, we have the following result.

\begin{prop}\label{prop:0-dim:equivalence}
If $-\langle \alpha,v(E_{ij}) \rangle >0$ for all $j>0$, then
$\Lambda^\alpha$ induces an equivalence
 $$
{\cal C} \to
\Per(X^{\alpha}/Y,{\bf b}_1^\alpha,...,{\bf b}_n^\alpha),
$$ 
where ${\bf b}_i^\alpha=(b_{i1}^\alpha,...,b_{is_i}^\alpha)$.
\end{prop}

\begin{prop}\label{prop:0-dim:duality}
Assume that there is a $\beta \in \varrho_X^{\perp}$
such that ${\Bbb C}_x$ are $\beta$-stable for all $x \in X$.
\begin{enumerate}
\item[(1)]
We set ${\cal F}:={{\cal E}^{\alpha}}^{\vee}[2]$. 
Then we have an isomorphism 
\begin{equation}
\begin{matrix}
X & \to & 
M_{{\cal O}_{X^\alpha}(1)}^{G^\alpha,\Lambda^\alpha(\beta)}
(\varrho_{X^\alpha})=(X^\alpha)^{\Lambda^\alpha(\beta)}\\
x & \mapsto & {\cal F} 
\overset{{\bf L}}{\otimes} {\Bbb C}_x. 
\end{matrix}
\end{equation}
Since $\Phi_{X^\alpha \to X}^{{\cal F}^{\vee}[2]}=
\Phi_{X^\alpha \to X}^{{\cal E}^\alpha}$,
we have ${\cal C}=
\Phi_{X^\alpha \to X}^{{\cal F}^{\vee}[2]}
(\Per(X^{\alpha}/Y,{\bf b}_1^\alpha,...,{\bf b}_n^\alpha))$.
\item[(2)]
We also have an isomorphism
\begin{equation}
\begin{matrix}
X & \to & 
M_{{\cal O}_{X^\alpha}(1)}^{(G^\alpha)^{\vee},
-D_{X^\alpha} \circ \Lambda^\alpha(\beta)}
(\varrho_{X^\alpha})\\
x & \mapsto & {\cal E}^\alpha \overset{{\bf L}}{\otimes} {\Bbb C}_x,  
\end{matrix}
\end{equation}
where $M_{{\cal O}_{X^\alpha}(1)}^{(G^\alpha)^{\vee},
-D_{X^\alpha} \circ \Lambda^\alpha(\beta)}
(\varrho_{X^\alpha})$ is the moduli 
of stable objects of $\Lambda^\alpha({\cal C})^D$.
\end{enumerate}
Thus $X$ and $X^\alpha$ are Fourier-Mukai dual.
\end{prop}

\begin{proof}
(1) is a consequence of Corollary \ref{cor:Phi-alpha}.
(2) is a consequence of (1) and the isomorphism
${\cal M}_{{\cal O}_{X^\alpha}(1)}^{G^\alpha,\gamma}
(\varrho_{X^\alpha})^{ss} \to
{\cal M}_{{\cal O}_{X^\alpha}(1)}^{(G^\alpha)^{\vee},
-D_{X^\alpha}(\gamma)}
(\varrho_{X^\alpha})^{ss}$
defined by $E \mapsto D_{X^\alpha}(E)[2]$.
\end{proof}

\begin{NB}
$\Lambda^{\Lambda^\alpha(\beta)}(\Lambda^\alpha(\beta))
=(\Lambda^\alpha)^{-1}(\Lambda^\alpha(\beta))=\beta$.
\end{NB}

The following proposition explains the
condition of the stability of ${\Bbb C}_x$. 
\begin{prop}\label{prop:0-dim:duality2}
${\cal C}=\Lambda^\gamma(\Per(X'/Y,{\bf b}_1,...,{\bf b}_n))$
with $X=(X')^\gamma$ if and only if
there is a $\beta \in \varrho_X^{\perp}$
such that
${\Bbb C}_x$ are $\beta$-stable for all $x \in X$.
\end{prop}

\begin{proof}
For $X=(X')^\gamma$,
$\gamma$-stability of 
${\cal E}^\gamma_{|X' \times \{x \}}$
and Corollary \ref{cor:Phi-alpha} imply 
the $\beta$-stability of
${\Bbb C}_x$, where
$\beta:=\Lambda^\gamma(\gamma)$.
\begin{NB}
Since 
$\Lambda({\cal E}^\gamma_{|X' \times \{ x \}})={\Bbb C}_x$,
${\Bbb C}_x$ is $\Lambda^\gamma(\gamma)$-stable.
\end{NB} 
Conversely if
${\Bbb C}_x$ are $\beta$-stable for all $x \in X$,
then Proposition \ref{prop:0-dim:duality} (1) implies the claim,
where $X':=X^\alpha$ and $\gamma:=\Lambda^\alpha(\beta)$.
\end{proof}

We give two examples of ${\cal C}$ satisfying the stability
condition of ${\Bbb C}_x$. 
\begin{lem}\label{lem:alpha>0}
\begin{enumerate}
\item[(1)]
Assume that ${\cal C}=
\Per(X/Y,{\bf b}_1,...,{\bf b}_n)$.
If $-\langle \alpha,v({\cal O}_{C_{ij}}(b_{ij})[1]) \rangle>0$ 
for all $j>0$, then
$X \cong X^{\alpha}$ by sending $x \in X$ to 
${\Bbb C}_x \in X^{\alpha}$.
Moreover $A_{p_i} \otimes {\cal O}_C$ such that
${\cal O}_C$ is a purely 1-dimensional
${\cal O}_{Z_i}$-module with $\chi({\cal O}_C)=1$ 
are $\alpha$-stable.
\item[(2)]
Assume that ${\cal C}=
\Per(X/Y,{\bf b}_1,...,{\bf b}_n)^*$.
If $-\langle \alpha,v({\cal O}_{C_{ij}}(b_{ij})) \rangle<0$ for all $j>0$, 
then
$X \cong X^{\alpha}$ by sending $x \in X$ to 
${\Bbb C}_x \in X^{\alpha}$.
\begin{NB}
Moreover $A_{p_i} \otimes \omega_C[1]$ such that
${\cal O}_C$ is a purely 1-dimensional
${\cal O}_{Z_i}$-module with $\chi({\cal O}_C)=1$ 
are $\alpha$-stable.
\end{NB}
\end{enumerate}
\end{lem}

\begin{proof}
We only prove (1).
Since ${\Bbb C}_x$, $x \in  X \setminus \cup_{i=1}^n Z_i$ is irreducible,
it is $\alpha$-twisted stable for any $\alpha$. 
For $x \in Z_i$,
assume that there is an exact sequence
\begin{equation}
0 \to E_1 \to {\Bbb C}_x \to E_2 \to 0
\end{equation}
such that $E_1 \ne 0$, $E_2 \ne 0$ and
$-\langle \alpha, v(E_1) \rangle=\chi(v^{-1}(\alpha),E_1) \geq 0$.
We note that $-\langle \alpha,v(E_{ij}) \rangle >0$ for all $j>0$.
Since $\langle \alpha,\varrho_X \rangle =0$,
$\langle \alpha, v(A_0({\bf b}_i)) \rangle=
-\sum_{j>0} a_{ij}  \langle \alpha,v(E_{ij}) \rangle$.
As a 0-semi-stable object, $E_1$ is $S$-equivalent to
$\oplus_{j>0} {\cal O}_{C_{ij}}(b_{ij})[1]^{\oplus a_{ij}'}$, 
$a_{ij}' \leq a_{ij}$.
Since $\Hom({\cal O}_{C_{ij}}(b_{ij})[1],{\Bbb C}_x)=0$,
this is impossible.
Therefore ${\Bbb C}_x$ is $\alpha$-twisted stable. 
Then we have an injective morphism
$\phi:X \to X^{\alpha}$ by sending $x \in X$ to ${\Bbb C}_x$.   
By using the Fourier-Mukai transform
$\Phi_{X \to X}^{{\cal O}_{\Delta}^{\vee}}:{\bf D}(X)
\to {\bf D}(X)$, we see that
$\phi$ is surjective.
Since both spaces are smooth, $\phi$ is an isomorphism. 
The last claim also follows by a similar argument.
\end{proof}

\subsubsection{Relation with the twist functor \cite{S-T:1}.}

Let $F$ be a spherical object of ${\bf D}(X)$ and set 
\begin{equation}
{\cal E}:=\mathrm{Cone}(F^{\vee} \boxtimes F \to {\cal O}_{\Delta})[1].
\end{equation}
Then 
$T_F:=\Phi_{X \to X}^{{\cal E}}$
is an autoequivalence of ${\bf D}(X)$.  

\begin{lem}\label{lem:reflection-Pi}
Let $\Pi:{\bf D}(X) \to {\bf D}(Y)$ be a Fourier-Mukai transform.
Then
\begin{equation}
\Pi \circ T_F \cong T_{\Pi(F)} \circ \Pi.
\end{equation}
\end{lem}

\begin{proof}
Let ${\bf E} \in {\bf D}(X \times Y)$ be an object
such that $\Pi=\Phi_{X \to Y}^{{\bf E}}$. 
It is sufficient to prove 
$\Pi({\cal E}) \cong T_{\Pi(F)}({\bf E})$.
We set $X_i:=X$, $i=1,2$.
We note that $F^{\vee} \cong 
\Hom_{p}( {\cal O}_{X_1}\boxtimes F,{\cal O}_{\Delta})$,
where $p:X_1 \times X_2 \to X_1$ is the projection and
$\Delta \subset X_1 \times X_2$ the diagonal.
Then
\begin{equation}
{\cal E} \cong 
\mathrm{Cone}(
 \Hom_p({\cal O}_{X_1} \boxtimes F,{\cal O}_{\Delta})\boxtimes F
\to {\cal O}_{\Delta})[1].
\end{equation}
Let $p_{X_2}:Y \times X_2 \to X_2$,
$p_Y:Y \times X_2 \to Y$ and
$q:X_1 \times Y  \to X_1$  
be the projections.
We have a morphism
\begin{multline}
\Hom_p({\cal O}_{X_1} \boxtimes F,{\cal O}_{\Delta})
\to
\Hom_{q'}({\cal O}_{X_1} \boxtimes ({\bf E} \otimes p_{X_2}^*(F)),
({\cal O}_{X_1}\boxtimes {\bf E})_{|\Delta'})\\
\to
\Hom_{q}({\cal O}_{X_1} \boxtimes 
{\bf R}p_{Y*}({\bf E} \otimes p_{X_2}^*(F)),{\bf E}),
\end{multline}
where $\Delta'=\Delta \times Y$
and 
$q':X_1 \times Y \times X_2 \to X_1$ is the projection.   
We also have a commutative diagram in ${\bf D}(Y \times X_1)$:
\begin{equation}
\begin{CD}
\Hom_p({\cal O}_{X_1} \boxtimes F,{\cal O}_{\Delta})\boxtimes \Pi(F) 
@>{\alpha}>> {\bf E}\\
@V{\gamma}VV @|\\
\Hom_{q}({\cal O}_{X_1} \boxtimes \Phi_{X \to Y}^{{\bf E}}(F),
{\bf E})\boxtimes \Pi(F)
@>{\beta}>> {\bf E}.
\end{CD}
\end{equation}
Since $\Pi$ is an equivalence,
$\gamma$ is an isomorphism.
Since $\Pi({\cal E}) \cong \mathrm{Cone}(\alpha)[1]$ and
$T_{\Pi(F)}({\bf E})\cong \mathrm{Cone}(\beta)[1]$, we get
$\Pi({\cal E}) \cong T_{\Pi(F)}({\bf E})$.
\end{proof}

\begin{NB}
\begin{lem}
Let $F$ be a spherical object
such that $\Supp(F) \subset Z$.
Let $D$ be 
a divisor on $X$ such that
${\cal O}_{|U}(D) \cong {\cal O}_U$ on a neighborhood $U$ of $Z$.
We set
${\cal E}:=\mathrm{Cone}(F \boxtimes F^{\vee} \to {\cal O}_{\Delta})$.
Then
$({\cal O}_X(D) \boxtimes {\cal O}_X(-D)) \otimes {\cal E} \cong {\cal E}$.
\end{lem}

\begin{proof}
We may assume that $D$ is effective.
Let $V^{\bullet}$ be a bounded complex of locally free sheaves
such that $\Supp(H^i(V^{\bullet})) \subset Z$.
Then $V^{\bullet}(D)\cong V^{\bullet}$.
We have an exact triangle
\begin{equation}
V^{\bullet} \to V^{\bullet}(D) \to V^{\bullet}_{|D}(D)
\to V^{\bullet}[1].
\end{equation}
By our assumption, 
$V^{\bullet}_{|D}(D)$ is quasi-isomorphic to 0.
Hence we get our claim.
Since we have the commutative diagram
\begin{equation}
\begin{CD}
V^{\bullet}(-D) \boxtimes (V^{\bullet})^{\vee} @>>> 
V^{\bullet} \boxtimes (V^{\bullet})^{\vee}  \\
@VVV @VVV\\
V^{\bullet}(-D) \boxtimes (V^{\bullet}(-D))^{\vee}
@>>> {\cal O}_{\Delta},
\end{CD}
\end{equation}
we get
$\mathrm{Cone}(V^{\bullet} \boxtimes (V^{\bullet})^{\vee} 
\to {\cal O}_{\Delta}) \cong
\mathrm{Cone}(V^{\bullet}(-D) \boxtimes (V^{\bullet}(-D))^{\vee}
\to {\cal O}_{\Delta}) \cong
\mathrm{Cone}(V^{\bullet} \boxtimes (V^{\bullet})^{\vee} 
\to {\cal O}_{\Delta}) \otimes 
({\cal O}_X(-D) \boxtimes {\cal O}_X(D))$.
\end{proof}

$\Hom(F \boxtimes F^{\vee} \to {\cal O}_{\Delta})=
\Hom(F \boxtimes {\cal O}_X,{\cal O}_{\Delta} \otimes F)$.

Let $V^{\bullet}$ be a bounded complex of locally free sheaves
on $X$ representing $F$.
Then
$\Hom(F \boxtimes (V^{\bullet})^{\vee} \to {\cal O}_{\Delta})=
\Hom(F \boxtimes {\cal O}_X,{\cal O}_{\Delta} \otimes V^{\bullet})$.

\end{NB}

\begin{cor}\label{cor:compatible}
Assume that $\Supp(H^i(F)) \subset Z$ for all $i$.
Let $D$ be the pull-back of a divisor on $Y$.
Then 
$T_F(E(D)) \cong T_F(E)(D)$.
\end{cor}  

\begin{proof}
We apply Lemma \ref{lem:reflection-Pi} to
$\Pi=\Phi_{X \to X}^{{\cal O}_{\Delta}(D)}$.
Since $\Pi(F) \cong F$, we get our claim.
\end{proof}

\begin{prop}\label{prop:A-A'}
Assume that $G^{\vee} \otimes G$ 
satisfies $R^1\pi_*(G^{\vee} \otimes G)=0$.
Assume that $G':=T_F(G)$ is a locally free sheaf up to shift.
\begin{enumerate}
\item[(1)]
${\bf R}^1 \pi_*({G'}^{\vee} \otimes G')=0$
and $\pi_*({G'}^{\vee} \otimes G')
\cong \pi_*(G^{\vee} \otimes G)$.
\item[(2)]
We set ${\cal A}':=\pi_*({G'}^{\vee} \otimes G')$.
We identify $\Coh_{{\cal A}}(Y)$ with $\Coh_{{\cal A}'}(Y)$
via $ {\cal A} \cong {\cal A}'$.
Then we have a commutative diagram
\begin{equation}
\begin{CD}
\Per(X/Y,{\bf b}_1,...,{\bf b}_n) @>{T_F}>> 
T_F(\Per(X/Y,{\bf b}_1,...,{\bf b}_n))\\
@V{{\bf R}\pi_*{\cal H}om(G,\;\;)}VV @VV{{\bf R}\pi_*{\cal H}om(G',\;\;)}V\\
\Coh_{{\cal A}}(Y) @= \Coh_{{\cal A}'}(Y) 
\end{CD}
\end{equation}
\end{enumerate}
\end{prop}

\begin{proof}
The proof is almost the same as that of Proposition \ref{prop:Phi-alpha}.
\end{proof}

\begin{defn}
For an $\alpha \in H^{\perp} \otimes {\Bbb Q}$, 
${\cal X}^\alpha$ denotes the moduli stack
of $\alpha$-semi-stable objects $E$ of ${\cal C}$
such that $v(E)=\varrho_X$.
\end{defn}
For an $\alpha \in H^{\perp} \otimes {\Bbb Q}$, 
let $F$ be an $\alpha$-stable object such that
(i) $\langle v(F)^2 \rangle=-2$ and (ii)
$\langle \alpha, v(F)\rangle=0$.
By (i), $F$ is a spherical object. 
By the same proof of \cite[Prop. 1.12]{O-Y:1}, we have the following result.
\begin{prop}\label{prop:0-dim:O-Y}
We set $\alpha^{\pm}:=\pm \epsilon v(F)+\alpha$, where
$0<\epsilon \ll 1$.
Then $T_F$ induces an isomorphism
\begin{equation}
\begin{matrix}
{\cal X}^{\alpha^-} &\to &
{\cal X}^{\alpha^+}\\
E & \mapsto & T_F(E)
\end{matrix}
\end{equation}
which preserves the $S$-equivalence classes.
Hence we have an isomorphism
\begin{equation}
X^{\alpha^-} \to
X^{\alpha^+}.
\end{equation}
\end{prop}

Combining Proposition \ref{prop:0-dim:O-Y} with Lemma \ref{lem:reflection-Pi},
we get the following corollary.
\begin{cor}\label{cor:0-dim:O-Y}
Assume that $\alpha$ belongs to exactly one wall defined by
$F$. Then $T_F$ induces an isomorphism
$X^{\alpha^-} \to X^{\alpha^+}$.
Under this isomorphism, we have 
\begin{equation}
\Phi_{X^{\alpha^-} \to X}^{{\cal E}^{\alpha^+}} \cong 
T_F \circ \Phi_{X^{\alpha^-} \to X}^{{\cal E}^{\alpha^-}}
\cong \Phi_{X^{\alpha^-} \to X}^{{\cal E}^{\alpha^-}} \circ T_{A},
\end{equation}
where $A:=\Phi_{X \to X^{\alpha^-}}^{({\cal E}^{\alpha^-})^{\vee}[2]}(F)$.
\end{cor}

\subsection{Construction of a local projective generator.}

We return to the general situation in section \ref{subsect:Per(X/Y)}.
We shall construct local projective generators 
for $\Per(X/Y,\{L_{ij} \})$.
\begin{prop}\label{prop:generator-exist}
Let $\beta$ be a 2-cocycle of ${\cal O}_X^{\times}$
defining a torsion element of $H^2(X,{\cal O}_X^{\times})$.
Assume that $E \in K^{\beta}(X)$ satisfies
\begin{equation}\label{eq:local-projective:condition}
\begin{split}
0 \leq -\chi(E,L_{ij}),\;1 \leq j \leq s_i,\\
-\sum_j a_{ij}\chi(E,L_{ij}) \leq r
\end{split}
\end{equation}
for all $i$.
\begin{enumerate}
\item[(1)]
There is a locally free $\beta$-twisted sheaf
$G$ on $X$ 
such that $R^1 \pi_*(G^{\vee} \otimes G)=0$,
${\bf R}\pi_*(G^{\vee} \otimes F) \in \Coh(Y)$ for
$F \in \Per(X/Y,\{L_{ij} \})$,
$G$ is $\mu$-stable and
$\tau(G)=\tau(E)-n \tau({\Bbb C}_x)$, $n \gg 0$.
\item[(2)]
There is a locally free $\beta$-twisted sheaf
$G$ on $X$ 
such that $R^1 \pi_*(G^{\vee} \otimes G)=0$,
${\bf R}\pi_*(G^{\vee} \otimes F) \in \Coh(Y)$ for
$F \in \Per(X/Y,\{L_{ij} \})$ and
$\tau(G)=2\tau(E)$.
\item[(3)]
Moreover if the inequalities in 
\eqref{eq:local-projective:condition} are strict, then
$G$ in (1) and (2) are local projective generators of
$\Per(X/Y,\{L_{ij}\})$.
\end{enumerate}

\end{prop}

\begin{cor}
Assume that $(r,\xi) \in {\Bbb Z}_{>0} \oplus \NS(X)$
satisfies 
\begin{equation}
\begin{split}
& 0<(\xi,C_{ij})-r(b_{ij}+1),\;1 \leq j \leq s_i,\\
& \sum_j a_{ij}(\xi,C_{ij})-r \sum_j a_{ij}(b_{ij}+1)<r,
\end{split}
\end{equation}
for all $i$.
\begin{enumerate}
\item[(1)]
For any sufficiently large $n$,
there is a local projective generator
$G$ of $\Per(X/Y,{\bf b}_1,...,{\bf b}_n)$
such that $G$ is a $\mu$-stable sheaf with respect to $H$ and
$(\rk G,c_1(G),c_2(G))=(r,\xi,c_2)$.
\item[(2)]
For any ${\bf e} \in K(X)_{\mathrm{top}}$ with
$(\rk {\bf e},c_1({\bf e}))=(r,\xi)$,
there is a local projective generator
$G$ such that $\tau(G)=2{\bf e}$. 
\end{enumerate}
\end{cor}

{\it Proof of Proposition \ref{prop:generator-exist}.}

(1) We assume that $H$ is represented by a smooth connected curve
with $Z \cap H=\emptyset$, where $Z=\sum_{i=1}^n Z_i$.
We take a torsion free sheaf $E$ 
such that $\Ext^2(E,E(-Z-H))_0=0$.
By the construction of $E$,
we may assume that $E$ is locally free on $Z \cup H$.
We consider the restriction morphism
of the local deformation spaces
\begin{equation}
\phi:\Def(X,E) \to \Def(Z,E_{|Z}) \times \Def(H,E_{|H}).
\end{equation}
Then $\Def(X,E)$ and $\Def(Z,E_{|Z}) \times \Def(H,E_{|H})$
are smooth, and $\phi$ is submersive. 
\begin{NB}
Since $E_{|Z}$ and $E_{|H}$ are locally free sheaves on curves,
$\Def(Z,E_{|Z}) \times \Def(H,E_{|H})$ is smooth. 
By $\Ext^2(E,E(-Z-H))_0=0$, we also get 
$\Ext^2(E,E)_0=0$. Thus $\Def(X,E)$ is also smooth.
We consider the homomorphism of the tangent spaces 
\begin{equation}
\Ext^1(E,E)_0 \to \Ext^1(E_{|Z},E_{|Z})_0 
\oplus \Ext^1(E_{|H},E_{|H})_0  \to \Ext^2(E,E(-Z-H))_0=0.
\end{equation}
Hence $\phi$ is submersive.
\end{NB}
In particular, by using Lemma \ref{lem:local-deform}
below,
we see that $E$ deforms to a locally free $\beta$-twisted sheaf
$G$ such that $G$ is $\mu$-stable with respect to $H$ and
$\Hom(G,L_{ij})=\Ext^1(G,A_{p_i})=0$
for all $i,j$.
%Then $G_{|Z_i} \in T$ is a 0-dimensional object
%of $\Per(X/Y,{\bf b}_1,...,{\bf b}_n)$.
%Hence it is generated by $L_{ij}[1]$ and $A_{p_i}$.
By Remark \ref{rem:tilting:generator},
 Proposition \ref{prop:generator-exist} (1) holds.

(2) By (1), we have locally free sheaves
$E_i$, $i=1,2$ such that $R^1 \pi_*(E_i^{\vee} \otimes E_i)=0$,
${\bf R}\pi_*(G_i^{\vee} \otimes F) \in \Coh(Y)$ for
$F \in \Per(X/Y,\{L_{ij} \})$,
$\tau(E_i)=\tau(E)-n_i \tau({\Bbb C}_x)$ and
$n_1+n_2=n^2(H^2)\rk E$.
Then $G=E_1(nH) \oplus E_2(-nH)$ 
satisfies the claim.

(3) 
The claim follows from Proposition \ref{prop:tilting:generator}. 
\qed

\begin{lem}\label{lem:local-deform}
\begin{enumerate}
\item[(1)]
$E_{|Z}$ deforms to a locally free $\beta$-twisted sheaf such that
\begin{equation}\label{eq:local-deform}
H^0(C_{ij},E^{\vee} \otimes L_{ij})=
H^1(Z_i,E^{\vee} \otimes A_{p_i})=0
\end{equation}
for all $i,j$.
\item[(2)]
$E_{|H}$ deforms to a $\mu$-stable locally free $\beta$-twisted sheaf on $H$.
\end{enumerate}
\end{lem}

\begin{proof}
(1)
Since $E_{|Z}=\oplus_{i=1}^n E_{|Z_i}$, we shall prove the claims
for each $E_{|Z_i}$.
Since $H^2(Z,{\cal O}_Z^{\times})=\{ 1\}$,
there is a $\beta$-twisted line bundle ${\cal L}$ on $Z_i$
which induces an equivalence
$\varphi:\Coh^{\beta}(Z) \cong \Coh(Z)$ in \eqref{eq:twisted-equiv}. 
Since $\Pic(Z_i) \to {\Bbb Z}^{s_i}$ 
($L \mapsto \prod_{j=1}^{s_i}\deg(L_{|C_{ij}})$)
is an isomorphism,
we may assume that $\varphi(L_{ij})={\cal O}_{C_{ij}}(-1)$.
Thus we may assume that $\beta$ is trivial and
$L_{ij}={\cal O}_{C_{ij}}(-1)$. 
In this case, we have $A_{p_i}={\cal O}_{Z_i}$.
Then
we have $\deg(E_{|C_{ij}}) \geq 0$ for all $j>0$ and
$\deg(E_{|Z_i})  \leq r$.
Let $D$ be an effective Cartier divisor on $Z_i$ such that
$(D,C_{ij})=\deg(E_{|C_{ij}})$.
\begin{NB}
This means that $c_1(E_{|Z_i})=D$.
\end{NB}
Then 
\begin{equation}
K:=\ker(H^0({\cal O}_{Z_i \cap D}) \otimes {\cal O}_{Z_i} \to
{\cal O}_{Z_i \cap D})
\end{equation}
is a locally free sheaf on $Z_i$ such that
$H^1(Z_i,K)=0$ and $H^0(C_{ij},K_{|C_{ij}}(-1))=0$.
Since $\rk K=\dim H^0({\cal O}_{Z_i \cap D})=\deg_{Z_i}(D)=
\deg(E_{|Z_i}) \leq r$, 
we set 
$F:=K^{\vee} \oplus {\cal O}_{Z_i}^{\oplus (\rk E-\rk K)}$.
Since $F$ is a locally free sheaf with
$(\rk F,\det(F^{\vee}))=(\rk E_{|Z_i},\det(E_{|Z_i}))$,
we get the claim by Lemma \ref{lem:RDP:irred-C}
and the openness of the condition \eqref{eq:local-deform}.
\begin{NB}
$F$ and $E_{|Z_i}$ are parametrized by a suitable
affine space (cf. \cite[Lem. 4.2]{I:1}). 
By the upper semi-continuity of the dimension
of the cohomology groups,
$E_{|Z_i}$ deformes to
a locally free sheaf with desired properties.
\end{NB}
\begin{NB}
Let $F$ be the full sheaf in a formal neighborhood of $Z_i$
such that $\rk F=r$ and $\deg(F_{|C_{ij}})=\deg(E_{|C_{ij}})$ for all $j>0$.  
For a sufficiently ample divisor $D$ on $Z_i$,
we have exact sequences
\begin{equation*}
\begin{split}
0 \to {\cal O}_{Z_i}(-D)^{\oplus (r-1)} \to F \to L \to 0\\
0 \to {\cal O}_{Z_i}(-D)^{\oplus (r-1)} \to E_{|Z_i} \to L \to 0,
\end{split}
\end{equation*} 
where $L \in \Pic(Z_i)$.
Hence $F$ and $E_{Z_i}$ are parametrized by 
$\Ext^1_{{\cal O}_{Z_i}}(L,{\cal O}_{Z_i}(-D)^{\oplus (r-1)})$.
For a locally free sheaf ${\cal E}$ on $S \times Z_i$,
${\cal E}^{\vee}$ and ${\cal E}^{\vee}_{|C_{ij}}$ are 
locally free sheaves on $S \times Z$ and $S \times C_{ij}$
respectively. Since projections are flat morphisms,
they are flat over $S$.
By the base change theorem,
the required conditions are open condition.
Theefore a small deformation of $E_{Z_i}$ satisfies the claim.  
\end{NB}

(2) is well-known.
\end{proof}

\begin{cor}\label{cor:generator-exist}
Assume that $\pi$ is the minimal resolution
of rational double points $p_1,...,p_n$. 
Let ${\cal C}$ be the category in Lemma \ref{lem:tilting}
and $E_{ij}$, $1 \leq i \leq n$, $0 \leq j \leq s_i$
the 0-stable objects in Lemma \ref{lem:alpha=0} (2).
For an element $E \in K(X)$
satisfying $\chi(E,E_{ij})>0$ for all $i,j$,
there is a local projective generator $G$ of ${\cal C}$
such that $\tau(G)=2 \tau(E)$.
\end{cor}

\begin{proof}
We consider the equivalence $\Lambda^\alpha$
in Proposition \ref{prop:0-dim:equivalence}.
Then since $\chi(\Lambda^\alpha(E),\Lambda^\alpha(E_{ij}))>0$
for all $i,j$, Proposition \ref{prop:generator-exist}
implies that 
there is a local projective generator
$G^\alpha$ of $\Lambda^\alpha({\cal C})$ such that
$\tau(G^\alpha)=2\tau(\Lambda^\alpha(E))$.
We set $G:=(\Lambda^{\alpha})^{-1}(G^\alpha) \in {\cal C}$.
Then 
\begin{equation}
\begin{split}
H^0(X,H^k(G \overset{{\bf L}}{\otimes}{\Bbb C}_x))=&
H^k(X,G \overset{{\bf L}}{\otimes}{\Bbb C}_x)\\
=&\Hom({\Bbb C}_x,G[k+2])\\
=&\Hom(\Lambda^\alpha({\Bbb C}_x),G^\alpha[k+2])\\
=& \Hom(G^\alpha,\Lambda^\alpha({\Bbb C}_x)[-k])=0
\end{split}
\end{equation}
 for all $x \in X$ and $k \ne 0$.
Therefore $G$ is a locally free sheaf on $X$.
Since $G^\alpha$ is a local projective generator of 
$\Lambda^\alpha({\cal C})$ and 
$\Lambda^\alpha$ is an equivalence,
$G$ is a local projective generator of ${\cal C}$.
%(see the proof of Proposition \ref{prop:Phi-alpha} (1)). 
\end{proof}

\begin{NB}
For $\Per(X/Y,\{L_{ij}\})^*$, 
the condition in Corollary \ref{cor:generator-exist} is
\begin{equation}
\begin{split}
0 < \chi(E,L_{ij}),\;1 \leq j \leq s_i,\\
\sum_j a_{ij}\chi(E,L_{ij}) < r
\end{split}
\end{equation}
for all $i$.
\end{NB}

\subsubsection{More results on the structure of ${\cal C}$.}
Let ${\cal C}$ be the category of perverse coherent sheaves 
in Lemma \ref{lem:tilting}.
Assume that there is $\beta \in \NS(X) \otimes {\Bbb Q}$
such that 
${\Bbb C}_x$ is $\beta$-stable for all $x \in X$.
By Proposition 
\ref{prop:0-dim:duality2},
${\cal C}=\Lambda^\alpha(\Per(X'/Y,{\bf b}_1,...,{\bf b}_n))$.
So we first assume that
${\cal C}=\Per(X/Y,{\bf b}_1,...,{\bf b}_n)$
and set 
\begin{equation}
E_{ij}:=
\begin{cases}
{\cal O}_{C_{ij}}(b_{ij})[1], & j>0,\\
A_0({\bf b}),& j=0.
\end{cases}
\end{equation}
We set
$v_{ij}:=v(E_{ij})$.
Let $u_0$ be an isotropic Mukai vector such that
$r_0:=\rk u_0>0$, $\langle u_0, v_{ij} \rangle=0$ for all
$i,j$.
We set
\begin{equation}
L:={\Bbb Z}u_0+\sum_{i=1}^n \sum_{j=0}^{s_i}{\Bbb Z}v_{ij}.
\end{equation}
Then $L$ is a sublattice of $H^*(X,{\Bbb Z})$ 
and we have a decomposition
%which contains the sublattice
%$\oplus_{i=1}^n \oplus_{j=1}^{s_i}{\Bbb Z} v_{ij}$ and
\begin{equation} 
L
=({\Bbb Z}u_0 \oplus {\Bbb Z}\varrho_X) \perp
(\oplus_{i=1}^n \oplus_{j=1}^{s_i}{\Bbb Z}v_{ij}).
\end{equation} 
\begin{NB}
$v_{i0}+\sum_{j=1}^{s_i} a_{ij} v_{ij}=\varrho_X$.
\end{NB}
We set 
\begin{equation}
\begin{split}
T_i:=& \oplus_{j=1}^{s_i} {\Bbb Z}C_{ij},\\
T:=& \oplus_{i=1}^n T_i. 
\end{split}
\end{equation}
Then we have an isometry
\begin{equation}
\begin{matrix}
\psi:& \oplus_{i=1}^n \oplus_{j=1}^{s_i} {\Bbb Z}v_{ij} & \to &
T\\
& v & \mapsto & c_1(v).
\end{matrix}
\end{equation}
Combining the isometry
${\Bbb Z}u_0 \oplus {\Bbb Z}\varrho_X \to 
{\Bbb Z}r_0 \oplus {\Bbb Z}\varrho_X$
($x u_0+z\varrho_X \mapsto x r_0+z\varrho_X$), we also have an isometry 
\begin{equation}
\widetilde{\psi}:({\Bbb Z}u_0 \oplus {\Bbb Z}\varrho_X) \perp
(\oplus_{i=1}^n \oplus_{j=1}^{s_i} {\Bbb Z}v_{ij}) \to 
({\Bbb Z}r_0 \oplus {\Bbb Z}\varrho_X) \perp T
\end{equation}
Let ${\frak g}_i$ (resp. $\widehat{\frak g}_i$)
be the finite Lie algebra (resp. affine Lie algebra)
associated to the lattice $\oplus_{j=1}^{s_i}{\Bbb Z}v_{ij}$
(resp. $\oplus_{j=0}^{s_i}{\Bbb Z}v_{ij}$).
Let ${\frak g}$ (resp. $\widehat{\frak g}$)
be the Lie algebra associated to 
$\oplus_{i=1}^n \oplus_{j=1}^{s_i}{\Bbb Z}v_{ij}$
(resp. $\oplus_{i=1}^n \oplus_{j=0}^{s_i}{\Bbb Z}v_{ij}$).
 
Let $W({\frak g}_i)$ (resp.  $W({\frak g})$) 
be the Weyl group of ${\frak g}_i$ (resp. ${\frak g}$)
and ${\cal W}_i$ (resp. ${\cal W}$)
the set of Weyl chambers of $W({\frak g}_i)$ 
(resp. $W({\frak g})$).
Since ${\frak g}=\oplus_{i=1}^n {\frak g}_i$,
$W({\frak g})=\prod_{i=1}^n W({\frak g}_i)$ and
${\cal W}=\prod_{i=1}^n {\cal W}_i$.
By the action of $W({\frak g})$, 
${\Bbb Q}u_0+{\Bbb Q}\varrho_X$ is fixed.
Let $W(\widehat{\frak g}_i)$ (resp.
$W(\widehat{\frak g})$) be the Weyl group of
$\widehat{\frak g}_i$ (resp. $\widehat{\frak g}$). 
We have the following decompositions
\begin{equation}
\begin{split}
W(\widehat{\frak g}_i)= & T_i \rtimes W({\frak g}_i),\\
W(\widehat{\frak g})= & T \rtimes W({\frak g}),
\end{split}
\end{equation}
and the action of $D \in T$ on $L$ is the multiplication 
by $e^D$.
Indeed 
$$
T_{{\cal O}_{C_{ij}}(b_{ij}+1)} \circ
T_{{\cal O}_{C_{ij}}(b_{ij})[1]}=e^{-C_{ij}} 
$$ 
as an isometry of $L$.
\begin{NB}
\begin{proof}
For $v \in u_0^{\perp} \cap \varrho_X^{\perp}$ with $\langle v^2 \rangle=2$,
we shall prove that
$T_{\varrho_X-v} \circ T_v=e^v$:

For $x \frac{u_0}{\rk u_0}+y+z \varrho_X, y \in \oplus_{i,j>0} {\Bbb Q}v_{ij}$,
\begin{equation}
\begin{split}
T_{\varrho_X-v} \circ T_v(x \frac{u_0}{\rk u_0}+y+z \varrho_X)=&
T_{\varrho_X-v}(x \frac{u_0}{\rk u_0}+y+z \varrho_X+(v,y)v)\\
=& x \frac{u_0}{\rk u_0}+(xv+y)+
(z+(y,v)+\frac{\langle v^2 \rangle}{2}) \varrho_X \\
=& e^v (x \frac{u_0}{\rk u_0}+y+z \varrho_X).
\end{split}
\end{equation}
\end{proof}
\end{NB} 

\begin{NB}
$(T_{{\cal O}_{C_{ij}}(b_{ij})[1]}) \circ e^D \circ 
(T_{{\cal O}_{C_{ij}}(b_{ij})[1]})=e^{D+(D,C_{ij})C_{ij}}$.
We write $D=\sum_{i,j}n_{ij}v_{ij}+a v({\Bbb C}_x)=
\sum_{i,j}n_{ij}v_{ij}-\frac{(D,c_1(u_0))}{\rk u_0}\varrho_X$
and 
$$
u_0 e^D=u_0 e^{\sum_{i,j}n_{ij}v_{ij}}-a \rk u_0
=u_0+\rk u_0 \sum_{i,j}n_{ij}v_{ij}+b v({\Bbb C}_x).
$$
Since $0=\langle u_0 e^D,u_0 e^D \rangle=
\langle (\rk u_0 \sum_{i,j}n_{ij}v_{ij})^2 \rangle -2\rk u_0 b$,
$$
u_0 e^D=u_0+\rk u_0 \sum_{i,j}n_{ij}v_{ij}+
\frac{\rk u_0}{2}\langle ( \sum_{i,j}n_{ij}v_{ij})^2 \rangle 
\rho_X.
$$
We also have
$$
v_{kl}e^D=v_{kl}+\langle v_{kl},\sum_{i,j}n_{ij}v_{ij} \rangle \varrho_X.
$$
\end{NB}
We shall study the category $\Lambda^\alpha({\cal C})$.
We may assume that $\alpha \in \NS(X) \otimes {\Bbb Q}$ is
$\alpha=\sum_i \alpha_i$ with $\alpha_i \in T_i \otimes {\Bbb Q}$.
Via the identification $\psi$,
we have an action of $W$ on $T \otimes {\Bbb Q}$.
We set
\begin{equation}
\begin{split}
C_i^{\mathrm{fund}}:=&
\{ \alpha \in T_i \otimes {\Bbb R}|(\alpha,C_{ij})>0, 
1 \leq j \leq s_i \},\\
C^{\mathrm{fund}}:=& \prod_{i=1}^n C_i^{\mathrm{fund}}.
\end{split}
\end{equation}
$C^{\mathrm{fund}}$ is the fundamental Weyl chamber.
If $\alpha \in C^{\mathrm{fund}}$, then
Lemma \ref{lem:alpha>0} implies that ${\Bbb C}_x$ is $\alpha$-stable
for all $x \in X$.
By the action of $W({\frak g}_i)$,
we have ${\cal W}_i=W({\frak g}_i)C_i^{\mathrm{fund}}$.
We also set
\begin{equation}
\begin{split}
C_{\mathrm{alcove}}^{\mathrm{fund}}:=&
\{ \alpha \in T \otimes {\Bbb R}|(\alpha,C_{ij})>0, 
1 \leq j \leq s_i,\; (\alpha,Z_i)<1 \}.
\end{split}
\end{equation}
By the isometry $\widetilde{\psi}^{-1}$, we have
\begin{equation}
\begin{split}
(\alpha,C_{ij})=& -\langle \psi^{-1}(\alpha),v_{ij} \rangle\\
=&
-\langle (\frac{u_0}{\rk u_0}+\psi^{-1}(\alpha)+
\frac{(\alpha^2)}{2}\varrho_X),v_{ij} \rangle
=-\langle e^{\frac{c_1(u_0)}{\rk u_0}+\alpha},v_{ij} \rangle
\end{split}
\end{equation}
for $j>0$ and $1-(\alpha,Z_i)=
1+\sum_{j=1}^{s_i} a_{ij}
\langle e^{\frac{c_1(u_0)}{\rk u_0}+\alpha},v_{ij} \rangle
=-\langle e^{\frac{c_1(u_0)}{\rk u_0}+\alpha},v_{i0} \rangle$.
Hence we have
 \begin{equation}
\begin{split}
C_{\mathrm{alcove}}^{\mathrm{fund}}
%=&
%\{ \alpha \in T \otimes {\Bbb R}|
%-\langle (\frac{u_0}{\rk u_0}+\psi^{-1}(\alpha)+
%\frac{(\alpha^2)}{2}\varrho_X),v_{ij} \rangle>0 \}\\
=& \{ \alpha \in T \otimes {\Bbb R}|
-\langle e^{\frac{c_1(u_0)}{\rk u_0}+\alpha},v_{ij} \rangle>0 \}.
\end{split}
\end{equation}

Applying Corollary \ref{cor:0-dim:O-Y} successively,
we get the following result.
\begin{prop}
If $\alpha \in T \otimes {\Bbb Q}$ belongs to a chamber
$C=\prod_{i=1}^n C_i$, $C_i \subset T_i \otimes {\Bbb Q}$,
 then
there are rigid objects $F_1,...,F_n \in {\cal C}$ such that
$X^\alpha \cong X$ and
$\Phi_{X \to X}^{{\cal E}^\alpha}=T_{F_n} \circ T_{F_{n-1}} \circ
\cdots \circ T_{F_1}$.
Thus $\Lambda^\alpha=(\Phi_{X \to X}^{{\cal E}^\alpha})^{-1}$ 
induces an isometry 
$w(\alpha)$ of
$L$.
\end{prop}

Then we have a map
\begin{equation}
\begin{matrix}
\phi:& {\cal W} & \to & W(\widehat{\frak g})/T\\
& C(\alpha) & \mapsto & [w(\alpha) \mod T],
\end{matrix}
\end{equation}
where $C(\alpha)$ is the chamber containing $\alpha$.

\begin{lem}\label{lem:weyl-chamber}
$\phi:{\cal W} \to W(\widehat{\frak g})/T \cong W({\frak g})$ is bijective.
\end{lem}

\begin{proof}
There is an element $\alpha_0$ in the fundamental Weyl chamber
such that
$\alpha=\Phi_{X \to X}^{{\cal E}^\alpha}(\alpha_0)$. 
Hence $w(\alpha)(C(\alpha))=C(\alpha_0)$.
Thus $\phi$ is injective.
Since $\# {\cal W}_i=\# W({\frak g}_i)$,
$\phi$ is bijective.
\end{proof}

We set 
\begin{equation}
T^*:=\{D \in T \otimes {\Bbb Q}| (D,C_{ij}) \in {\Bbb Z} \}.
\end{equation}
Then $\widetilde{W}:=T^* \rtimes W({\frak g})$ is the extended Weyl group.
\begin{NB}
We define $C_{kl}^* \in T^*$ by
$(C_{kl}^*,C_{ij})=\delta_{ik}\delta_{jl}$.
Then $T^* =\oplus_{i,j}{\Bbb Z}C_{ij}^*$.
\end{NB}
By the action of $\widetilde{W}$,
we can change $({\bf b}_1,...,{\bf b}_n)$ to any sequence
$({\bf b}_1',...,{\bf b}_n')$.

\begin{prop}\label{prop:perverse=-1}
Let ${\cal C}$ be the category in Lemma \ref{lem:tilting} and assume that
there is $\beta \in \NS(X) \otimes {\Bbb Q}$
such that ${\Bbb C}_x$ is $\beta$-stable for all $x \in X$.
Then ${\cal C}$ is equivalent to $^{-1} \Per(X/Y)$.
In particular, 
$\Per(X/Y,{\bf b}_1,...,{\bf b}_n) \cong {^{-1} \Per(X/Y)}$.
\end{prop}

\begin{proof}
We may assume that
${\cal C}=\Per(X/Y,{\bf b}_1,...,{\bf b}_n)$.
We set 
\begin{equation}
u_{ij}:=
\begin{cases}
v({\cal O}_{C_{ij}}(-1)[1]), &  j>0,\\
v({\cal O}_{Z_i}), & j=0.
\end{cases}
\end{equation} 
By the theory of affine Lie algebras,
there is an element $w \in W(\widehat{\frak g})$ such that 
\begin{equation}
\begin{split}
& w(\{\beta \in T \otimes {\Bbb R}| -\langle e^\beta,v_{ij} \rangle>0,
i,j \geq 0\})\\
=&\{\beta \in T \otimes {\Bbb R}| -\langle e^\beta,u_{ij} \rangle>0,
i,j \geq 0 \}.
\end{split}
\end{equation}
Then we have
$$
\{ w(v_{ij})|0 \leq j \leq s_i \}
=\{u_{ij} |0 \leq j \leq s_i \}
$$
for all $i$.
\begin{NB}
Since $w(v_{ij})$ is a $(-2)$-vector,
the half plane $P:=\{\gamma|-\langle e^\gamma,w(v_{ij}) \rangle \geq 0 \}$
determine $w(v_{ij})$:
We set $w(v_{ij})=u+a \varrho_X$, $u \in T \otimes {\Bbb Q}$.
Then $(u,\gamma) \leq a$ for $\gamma \in P$.
Since the boundary is a hyperplane,
${\Bbb Q}u$ is determined by $P$.
Since $\langle u^2 \rangle=-2$,
$\pm u$ is determined by $P$.
Since $(-u,\gamma) \leq a$ implies $(u,\gamma) \geq a$,
$u$ itself is determined by $P$. 
\end{NB}

\begin{NB}
\begin{equation}
\begin{split}
w(e^\beta)=& w(1+\beta+\frac{(\beta^2)}{2}\varrho_X)\\
=& 1+w(\beta)+\frac{(w(\beta)^2)}{2}\varrho_X
=e^{w(\beta)}.
\end{split}
\end{equation}
\end{NB}
For each $i$, there is an integer $j_i$ such that
(1) $c_1(w(v_{i j_i}))$ is effective and 
(2) $-c_1(w(v_{i j}))$, $j \ne j_i$ are effective.
By Lemma \ref{lem:weyl-chamber},
we have $w=e^D \phi(\alpha)$, $D, \alpha \in T$.
Since $v(\Lambda^\alpha(E_{ij})\otimes{\cal O}_X(D))=
e^D v(\Lambda^\alpha(E_{ij}))=e^D \phi(\alpha)(v_{ij})$,
Proposition \ref{prop:Phi-alpha} (2) implies that
$-(\alpha,c_1(E_{ij}))>0$ unless $j =j_i$.
By Lemma \ref{lem:0-stable:exceptional} and
Lemma \ref{lem:Lambda(E)},
$\Lambda^\alpha(E_{ij})[-1]$, $j \ne j_i$ 
is a line bundle on a smooth rational curve
and $\Lambda^\alpha(E_{ij_i})$ is a line bundle on $Z_i$.
%Then we have $\{ v(\Lambda^\alpha(E_{ij})\otimes{\cal O}_X(D))|j \ne j_i \}
%=\{u_{ij}|0< j \leq s_i \}$. 
Thus 
\begin{equation}
\begin{split}
\{ \Lambda^\alpha(E_{ij})\otimes{\cal O}_X(D)|j \ne j_i \}
=& \{{\cal O}_{C_{ij}}(-1)[1] |0< j \leq s_i \},\\
\Lambda^\alpha(E_{i j_i})\otimes{\cal O}_X(D)= & {\cal O}_{Z_i}.
\end{split}
\end{equation}
By Proposition \ref{prop:Phi-alpha} (5), we get 
$\Lambda^\alpha({\cal C})\otimes{\cal O}_X(D) \cong ^{-1} \Per(X/Y)$. 
\end{proof}

\begin{NB}
Let $D$ be an effective divisor on $X$ such that
$\pi(D)$ is a point of $Y$.
Then $K_{X|D} \cong {\cal O}_D$, which implies that
$\omega_D \cong K_X(D)_{|D} \cong {\cal O}_D(D)$..
Since $R^1 \pi_*({\cal O}_X)=0$,
we have 
$H^0(\omega_D) \cong H^1({\cal O}_D)^{\vee}=0$.
Therefore $H^0({\cal O}_X) \to H^0({\cal O}_X(D))$
is an isomorphism. 
\end{NB}

\begin{rem}
For the derived category of coherent twisted sheaves,
we also see that the equivalence classes of
$\Per(X/Y,\{L_{ij} \})$ does not depend on the choice of $\{L_{ij} \}$.
\end{rem}

\begin{prop}\label{prop:category-generator-exist}
We set $v=(r,\xi,a) \in H^{ev}(X,{\Bbb Z})_{\alg}$, $r>0$.
Assume that $(\xi,D) \not \in r{\Bbb Z}$
for all $D \in T$ with $(D^2)=-2$.
Then there is a category of perverse coherent sheaves
${\cal C}_v$ and a locally free sheaf $G$ on $X$ such that
$G$ is a local projective generator of ${\cal C}_v$
with $v(G)=2v$. 
We also have a local projective generator $G'$ of ${\cal C}_v$
such that $G'$ is $\mu$-stable with respect to $H$ and
$v(G')=v-b \varrho_X$, $b \gg 0$.
\end{prop}

\begin{proof}
We set ${\cal C}=\Per(X/Y,{\bf b}_1,...,{\bf b}_n)$ 
and keep the notation as above.
By our assumption,
$\langle v,u \rangle \not \in r {\Bbb Z}$
for all $(-2)$-vectors $u \in L$.
Then there is $w \in W$ such that
$v=w(v_f)$ and 
$v_f/r$ belongs to the fundamental alcove, that is,
$-\langle v_f/r,v_{ij} \rangle>0$ for all $i,j$. 
\begin{NB}
We write $\frac{v_f}{\rk v_f}=\frac{u_0}{\rk u_0}+u+b \varrho_X$,
$u \in \oplus_{i,j >0}{\Bbb Q}v_{ij}$.
Then $-\langle u,v_{ij} \rangle>0,
\sum_{j=1}^{s_i}-a_{ij} \langle u,v_{ij} \rangle<1$. 
\end{NB}
By Lemma \ref{lem:weyl-chamber}, we have an element $\alpha$ such that
$w=e^{D}\phi(\alpha)$, $D \in T$.
By Proposition \ref{prop:generator-exist},
there is a local projective generator $G_f$ of
${\cal C}$
such that $v(G_f)=2v_f$. 
We set ${\cal C}_v:=
\Lambda^\alpha({\cal C}) 
\otimes {\cal O}_X(D)$.
Then $G^\alpha:=\Lambda^\alpha(G_f)$ is a local
projective generator of ${\cal C}_v \otimes {\cal O}_X(-D)$.
Hence $G:=G^\alpha(D)$ is a local projective generator of ${\cal C}_v$
such that $v(G)=2v$.
\end{proof}

\begin{NB}
\begin{rem}
Assume that there is a $\beta$ such that ${\Bbb C}_x$ is $\beta$-stable
for every $x \in X$. Then we can choose ${\cal E}^\alpha$
so that $c_1(\Phi_{X \to X}^{{\cal E}^\alpha}(v))=c_1(v)$.
\end{rem}
\end{NB}

\subsection{Deformation of a local projective generator.}

Let $f:({\cal X},{\cal L}) \to S$ be a flat family of
polarized surfaces over $S$.
For a point $s_0 \in S$, we set $X:={\cal X}_{s_0}$.
Let ${\cal H}$ be a relative Cartier divisor on $X$
such that $H:={\cal H}_{s_0}$ gives a contraction
$f:X \to Y$ to a normal surface $Y$ with
${\bf R} \pi_*({\cal O}_X)={\cal O}_Y$.
We shall construct a family of contractions
$f:{\cal X} \to {\cal Y}$ over a neighborhood of $s_0$.

Replacing $H$ by $mH$, we may assume that
$H^i(X,{\cal O}_X(mH))=H^i(Y,{\cal O}_Y(mH))=0$
for $m> 0$.
We shall find an open neighborhood $S_0$ of $s_0$ such that
$R^i f_*({\cal O}_{{\cal X}_{S_0}}(m{\cal H}))=0$, $i>0,m>0$
and $f_*({\cal O}_{{\cal X}_{S_0}}(m{\cal H}))$ is locally free:
%We take a neighborhood $S_1$ of $s_0$ such that
%${\cal H}_s$ are smooth divisors
%for all $s \in S_0$.
We consider the exact sequence  
\begin{equation}
0 \to {\cal O}_{{\cal X}}(m{\cal H}) \to
{\cal O}_{{\cal X}}((m+1){\cal H}) \to {\cal O}_{{\cal H}}((m+1){\cal H})
\to 0.
\end{equation}
Since ${\cal H} \to S$ is a flat morphism,
the base change theorem implies that 
$R^i f_*({\cal O}_{{\cal X}}(m{\cal H})) \to
R^i f_*({\cal O}_{{\cal X}}((m+1){\cal H}))$ is surjective,
if $(m+1)(H^2)>(H^2)+(H,K_X)$.
We take an open neighborhood $S_0$ of $s_0$ such that
$R^i f_*({\cal O}_{{\cal X}_{S_0}}(m{\cal H}))=0$, 
$i>0, (H,K_X)/(H^2) \geq m>0$.
Then the claim holds. 
\begin{NB}
Let $R$ be the stalk of ${\cal O}_S$ at $s_0$
and set ${\cal X}_R:={\cal X} \times_S \Spec(R)$.
By the base change theorem,
$R^i f_*({\cal O}_{{\cal X}_R}(m{\cal H}))=0$, $i>0,m>0$
and $f_*({\cal O}_{{\cal X}_R}(m{\cal H}))$ is a free $R$-module
with $\phi_*({\cal O}_{{\cal X}_R}(m{\cal H})) \otimes_R k(s_0)
\cong H^0(X,{\cal O}_X(mH))$.
\end{NB}
We replace $S$ by $S_0$ and set
${\cal Y}:=\Proj(\oplus_m f_*({\cal O}_{{\cal X}}(m{\cal H})))$.
Then ${\cal Y}$ is flat over $S$ and
${\cal Y}_{s_0} \cong Y$.
By the construction,
${\cal Y} \to S$ is a flat family of normal surfaces.

Let ${\cal Z}:=\{x \in {\cal X}|\dim \pi^{-1}( \pi(x)) \geq 1 \}$ 
be the exceptional locus.
Then $\{({\cal Z}_s,{\cal L}_s)| s \in S \}$ is a bounded set.
Hence ${\cal D}:=\{ D \in \NS({\cal X}_s)| s \in S, (D,{\cal H}_s)=0 \}$
is a finite set.
Replacing $S$ by an open neighborhood of $s_0$,
we may assume that $D \in {\cal D}$ is a deformation of
$D_0 \in \NS(X)$ (i.e., $D$ belongs to $\NS(X)$
via the identification $H^2({\cal X}_s,{\Bbb Z}) \cong
H^2(X,{\Bbb Z})$).

\begin{NB}
Let ${\frak H}$ be a connected component
of $\Hilb_{{\cal X}/S}^d:=\{D \subset {\cal X}_s|(D,{\cal L}_s)=d \}$.
Then the image of ${\frak H} \to S$ is a closed subset of $S$.
So we remove all closed subset of $S$ which do not
contain $s_0$. 
\end{NB}

\begin{lem}
Assume that there is a locally free sheaf $G$ on ${\cal X}$
such that $R^1 \pi_*(G^{\vee} \otimes G)=0$ and 
$\rk G \nmid (c_1(G)_{s_0},D)$ for all
$(-2)$-curves with $(D,{\cal H}_{s_0})=0$.
Then replacing $S$ by an open neighborhood of $s_0$,
we may assume that $\rk G \nmid (c_1(G)_{s},D)$ for all
$(-2)$-curves with $(D,{\cal H}_s)=0$.   
Thus $G$ is a family of tilting generators.
\end{lem}

As an example, we consider a family of $K3$ surfaces.
Let $X$ be a $K3$ surface and $\pi:X \to Y$ a contraction.
Let $p_i$, $i=1,2,...,n$ be the singular points and
$Z_i:=\sum_j a_{ij} C_{ij}$ their fundamental cycles. 
Let $H$ be the pull-back of an ample divisor on $Y$.
Assume that $(r,\xi) \in {\Bbb Z}_{>0} \times \NS(X)$
satisfies $r \nmid (\xi,D)$ for all $(-2)$-curves $D$ with
$(D,H)=0$.
By Proposition \ref{prop:category-generator-exist},
 there is a category of perverse coherent sheaves ${\cal C}$
and a local projective generator $G$ of ${\cal C}$
such that $G$ is $\mu$-stable with respect to $H$ and
$(\rk G,c_1(G))=(r,\xi)$. 
\begin{NB}
Old version:
We fix coherent sheaves
${\cal O}_{C_{ij}}(b_{ij})$ on $X$.
We consider the tilting with respect to these sheaves.
Let $v=(r,\xi,a)$ be a Mukai vector such that 
\begin{equation}\label{eq:generator-cond}
\langle v^2 \rangle \geq -2,\;
0<(\xi,C_{ij})-r(b_{ij}+1),\;
\sum_j a_{ij}(\xi,C_{ij})-r \sum_j a_{ij}(b_{ij}+1)<r.
\end{equation}
Then there is a local projective generator $G$ of 
the tilted category $\Per(X/Y,\{{\cal O}_{C_{ij}}(b_{ij})\})$
such that $\rk G=r$ and
$c_1(G)=\xi$.
We may assume that $G$ is $\mu$-stable with respect to $H$.
\end{NB}
Replacing $G$ by $G \otimes L^{\otimes m}$,
$L \in \Pic(X)$
and ${\cal C}$ by 
${\cal C} \otimes L^{\otimes m}$,
we assume that $\xi$ is ample.
If $({\Bbb Q}\xi+{\Bbb Q}H) \cap H^{\perp}$
does not contain a $(-2)$-curve, then
we have a deformation 
$({\cal X},{\cal L}) \to S$ 
of $(X,\xi)$ such that
${\cal H}_s$ is ample for a general $s \in S$.
Since $G$ is simple, replacing $S$ by a smooth covering $S' \to S$,
we also have a deformation ${\cal G}$ of $G$ over $S$.
By shrinking $S$, we may assume that
${\cal G}$ is a family of tilting generators.  
Then we can construct a family of moduli spaces
$f:\overline{M}_{({\cal X},{\cal H})/S}^{\cal G}(v) \to S$
of ${\cal G}_s$-twisted semi-stable objects on ${\cal X}_s$, $s \in S$
(for the twisted cases, see Step 3, 4 of the proof of 
\cite[Thm. 3.16]{Y:twisted}).
By our assumption, a general fiber of $f$ is the moduli space
of ${\cal G}_s$-twisted semi-stable sheaves,
which is non-empty by Lemma \ref{lem:appendix:existence}.
Hence we get the following lemma.
\begin{lem}\label{lem:surjective}
Assume that $v$ is primitive and $\langle v^2 \rangle \geq -2$.
Then $f$ is surjective.
In particular,
$\overline{M}_{({\cal X},{\cal H})/S}^{\cal G}(v)_{s_0} \ne \emptyset$. 
\end{lem}

\begin{rem}\label{rem:assumption-surj}
We note that
$R:=\{C \in \NS(X)|(C,H)=0, (C^2)=-2 \}$ is a finite set.
If $\rho(X) \geq 3$, then
$\cup_{C \in R} ({\Bbb Q}H+{\Bbb Q}C)$ is a proper subset
of $\NS(X) \otimes {\Bbb Q}$.
Hence $({\Bbb Q}\xi+{\Bbb Q}H) \cap R =\emptyset$
for a general $\xi$.
In general, we have a deformation
$({\cal X},{\cal L}) \to S$ of $(X,\xi)$
such that
${\cal G}$ is a family of tilting generators and 
$\rho({\cal X}_s) \geq 3$ for infinitely many points $s \in S$.
\end{rem}

\begin{rem}
By the usual deformation theory of objects, we note that 
$M_{({\cal X},{\cal H})/S}^{\cal G}(v) \to S$ is a smooth morphism.
If $\overline{M}_{({\cal X},{\cal H})/S}^{\cal G}(v)_{s_0}=
M_{({\cal X},{\cal H})/S}^{\cal G}(v)_{s_0}$, then
we have a smooth deformation
$\overline{M}_{({\cal X},{\cal H})/S}^{\cal G}(v) \to S$
of $\overline{M}_{({\cal X},{\cal H})/S}^{\cal G}(v)_{s_0}$.
In particular, $\overline{M}_{({\cal X},{\cal H})/S}^{\cal G}(v)_{s_0}$
deforms to a usual moduli of semi-stable sheaves. 
\end{rem}

\begin{cor}\label{cor:K3-non-empty}
Let $v_0=(r,\xi,a)$ be a primitive isotropic Mukai vector
such that $r \not|(\xi,D)$ for all $(-2)$-curves $D$
with $(D,H)=0$.
Let ${\cal C}$ be the category 
in Proposition \ref{prop:category-generator-exist}.
Then $M_H^{v_0}(v_0) \ne \emptyset$.
\end{cor}

\begin{proof}
By Lemma \ref{lem:surjective} and Remark \ref{rem:assumption-surj},
we see that 
$\overline{M}_H^{v_0}(v_0) \ne \emptyset$.
By the same proof of \cite[Lem. 2.17]{O-Y:1},
we see that 
$\overline{M}_H^{v_0+\alpha}(v_0) \ne \emptyset$
for a general $\alpha$.
Then $\overline{M}_H^{v_0+\alpha}(v_0)$ is a $K3$ surface.
In the same way as in the proof of \cite[Prop. 2.11]{O-Y:1},
we see that $M_H^{v_0}(v_0) \ne \emptyset$.  
\end{proof}

\section{Fourier-Mukai transform on a $K3$ surface.}\label{sect:K3}

\subsection{Basic results on the moduli spaces of dimension 2.}

\begin{NB}
Let $E$ be a local projective generator
of $\Per(X/Y,{\bf b})$
which is $\mu$-stable.
We set 
\begin{equation}
\begin{split}
E_0:=&\mathrm{Cone}(E \otimes \Hom(E,A_0({\bf b})) \to
A_0({\bf b}))[1],\\ 
E_i:=&\mathrm{Cone}(E \otimes \Hom(E,{\cal O}_{C_i}(b_i)[1]) \to
{\cal O}_{C_i}(b_i))[1],\;i>0. 
\end{split}
\end{equation}
We note that $A_0({\bf b})$ and ${\cal O}_{C_i}(b_i)[1]$
are 0-dimensional.
Since $E$ is a local projective generator,
$E_i$, $i \geq 0$ are objects of $\Per(X/Y,{\bf b})$.
Then
$E_i$ are $E$-twisted stable objects.

Let $E'$ be a $E$-twisted stable subobject of 
$E_0$ such that
$\deg_E(E')=0$
Then we have a filtration
$0 \subset F_1 \subset F_2 \subset \cdots \subset F_t=E'$
such that $A_j:=F_j/F_{j-1}$ are subobject of $E$ with
$\deg_{E}(A_j)=0$.
We set $B_j:=E/A_j$.
Then $B_i$ is 0-dimensional and 
$\chi(E,A_j)= \chi(E,E)-\chi(E,B_j) \geq \chi(E,E)$.
Since $\Hom(E,B_j[1])=H^1(Y,\pi_*(E^{\vee} \otimes B_j))=0$,
we have an exact and commutative diagram:
\begin{equation}
\begin{CD}
0 @>>> A_1 @>>> E @>>> B_1 @>>> 0\\
@. @VVV @VVV @VVV @.\\
0 @>>> E_0 @>>> E \otimes \Hom(E,A_0({\bf b})) @>>> A_0({\bf b}) @>>> 0 
\end{CD}
\end{equation}
Since $\Hom(E,E_0)=0$, $E \to A_0({\bf b})$ is not zero.
Hence $\Hom(B_1,A_0({\bf b})) \ne 0$.
Since $A_0({\bf b})$ is a simple object,
we have a surjective morphism
$B_1 \to A_0({\bf b})$, which implies that
$\chi(E,B_1) \geq \chi(E, A_0({\bf b}))$.
Therefore $\chi(E,E') \leq t\chi(E,E)-\chi(E, A_0({\bf b}))$.
Then 
\begin{equation}
\begin{split}
\chi(E,E')/\rk E' \leq & (t\chi(E,E)-\chi(E,A_0({\bf b})))/t\rk E\\
 < & \chi(E,E)/\rk E-\chi(E,A_0({\bf b}))/\rk E_0=\chi(E,E_0)/\rk E_0
\end{split}
\end{equation}
if $t \rk E<\rk E_0$.
Hence $E_0$ is $E$-twisted stable.

Let $E'$ be a $E$-twisted stable subobject of 
$E_i$ such that
$\deg_E(E')=0$
Then we have a filtration
$0 \subset F_1 \subset F_2 \subset \cdots \subset F_t=E'$
such that $A_j:=F_j/F_{j-1}$ are subobject of $E$ with
$\deg_{E}(A_j)=0$.
We set $B_j:=E/A_j$.
Then $B_i$ is 0-dimensional and 
$\chi(E,A_j)= \chi(E,E)-\chi(E,B_j) \geq \chi(E,E)$.
Since $\Hom(E,B_j[1])=H^1(Y,\pi_*(E^{\vee} \otimes B_j))=0$,
we have an exact and commutative diagram:
\begin{equation}
\begin{CD}
0 @>>> A_1 @>>> E @>>> B_1 @>>> 0\\
@. @VVV @VVV @VVV @.\\
0 @>>> E_0 @>>> E \otimes \Hom(E,{\cal O}_{C_i}(b_i)[1]) 
@>>>{\cal O}_{C_i}(b_i)[1] @>>> 0 
\end{CD}
\end{equation}
Since $\Hom(E,E_i)=0$, $E \to {\cal O}_{C_i}(b_i)[1]$ is not zero.
Hence $\Hom(B_1,{\cal O}_{C_i}(b_i)[1]) \ne 0$.
Since ${\cal O}_{C_i}(b_i)[1]$ is a simple object,
we have a surjective morphism
$B_1 \to {\cal O}_{C_i}(b_i)[1]$, which implies that
$\chi(E,B_1) \geq \chi(E, {\cal O}_{C_i}(b_i)[1])$.
Therefore $\chi(E,E') \leq t\chi(E,E)-\chi(E, {\cal O}_{C_i}(b_i)[1])$.
Then 
\begin{equation}
\begin{split}
\chi(E,E')/\rk E' \leq & (t\chi(E,E)-\chi(E,{\cal O}_{C_i}(b_i)[1]))/t\rk E\\
 < & \chi(E,E)/\rk E-\chi(E,{\cal O}_{C_i}(b_i)[1])/\rk E_0
=\chi(E,E_i)/\rk E_i
\end{split}
\end{equation}
if $t \rk E<\rk E_i$.
Hence $E_i$ is $E$-twisted stable.

\end{NB}

Let $Y$ be a normal $K3$ surface and
$\pi:X \to Y$ the minimal resolution.
Let $p_1,p_2,\dots,p_n$ be the singular points of $Y$
and $Z_i:=\pi^{-1}(p_i)=\sum_{j=0}^{s_i}a_{ij}C_{ij}$
the fundamental cycle, where $C_{ij}$ are smooth rational curves
on $X$
and $a_{ij} \in {\Bbb Z}_{>0}$.
We shall study moduli of stable objects in the category ${\cal C}$
in Lemma \ref{lem:tilting} satisfying the following assumption.

\begin{assume}\label{ass:stability}
There is a $\beta \in \varrho_X^{\perp} \otimes {\Bbb Q}$
such that ${\Bbb C}_x$ is $\beta$-stable for all $x \in X$.
\end{assume}
By Proposition \ref{prop:0-dim:duality2},
there are 
${\bf b}_i:=(b_{i1},b_{i2},\dots,b_{is_i}) \in {\Bbb Z}^{\oplus s_i}$ 
and an autoequivalence
$\Phi_{X \to X}^{{\cal F}^{\vee}[2]}:{\bf D}(X) \to {\bf D}(X)$
such that
$\Phi_{X \to X}^{{\cal F}^{\vee}[2]}
(\Per(X/Y))={\cal C}$,
where $\Per(X/Y):=\Per(X/Y,{\bf b}_1,\dots,{\bf b}_n)$ 
and
${\cal F}$ is the family of
$\Phi_{X \to X}^{\cal F}(\beta)$-stable objects of $\Per(X/Y)$
in Proposition \ref{prop:0-dim:duality2}.  
We set

\begin{equation}
A_{ij}:=
\begin{cases}
\Phi_{X \to X}^{{\cal F}^{\vee}[2]}(A_0({\bf b}_i)), & j=0,\\
\Phi_{X \to X}^{{\cal F}^{\vee}[2]}({\cal O}_{C_{ij}}(b_{ij})[1]),& j>0.
\end{cases}
\end{equation}
Throughout this section,
we assume the following:

\begin{assume}\label{ass:v_0}
$v_0:=r_0+\xi_0+a_0 \varrho_X$,
$r_0>0,\xi_0 \in \NS(X)$ is a primitive isotropic Mukai vector
such that
$\langle v_0,v(A_{ij}) \rangle<0$ for all $i,j$.
\end{assume}
By Corollary \ref{cor:generator-exist}, we have the following.
\begin{lem}\label{lem:assumption}
There is a local projective generator $G$ of
${\cal C}$ whose Mukai vector is $2v_0$.
More generally, for a sufficiently small $\alpha \in 
(v_0^{\perp} \cap \varrho_X^{\perp}) \otimes {\Bbb Q}$,
there is a local projective generator $G$ of
${\cal C}$ such that $v(G) \in {\Bbb Q}_{>0}(v_0+\alpha)$.
\end{lem}
Let $H$ be the pull-back of an ample divisor on $Y$.
For a sufficiently small $\alpha \in 
(v_0^{\perp} \cap \varrho_X^{\perp}) \otimes {\Bbb Q}$,
we take a local projective generator $G$ of ${\cal C}$ with
$v(G) \in {\Bbb Q}_{>0}(v_0+\alpha)$.
We define $v_0+\alpha$-twisted semi-stability in a usual way.
Since it is equivalent to the $G$-twisted semi-stability, 
we have the moduli space $\overline{M}_H^{v_0+\alpha}(v_0)$.
Let $M_H^{v_0+\alpha}(v_0)$ be the moduli space
of $v_0+\alpha$-stable objects.
By Corollary \ref{cor:K3-non-empty}, 
\begin{NB}
or applying it to $\Per(X/Y)$ and use
Proposition \ref{prop:Phi-alpha} (6)
\end{NB}
$M_H^{v_0}(v_0) \ne \emptyset$. Hence we see that 
$M_H^{v_0+\alpha}(v_0)$ is also non-empty.
Then we have the following which is well-known for the moduli
of stable sheaves on $K3$ surfaces.

\begin{prop}\label{prop:K3:smooth}
\begin{enumerate}
\item[(1)]
$M_H^{v_0+\alpha}(v_0)$ is a smooth surface.
If $\alpha$ is general, then 
$\overline{M}_H^{v_0+\alpha}(v_0)=M_H^{v_0+\alpha}(v_0)$
is projective.
\item[(2)]
If $\overline{M}_H^{v_0+\alpha}(v_0)=M_H^{v_0+\alpha}(v_0)$, then
it is a $K3$ surface.
\end{enumerate}
\end{prop}
For the structure of $\overline{M}_H^{v_0}(v_0)$,
as in \cite{O-Y:1}, we have the following.

\begin{thm}\label{thm:K3-desing}(cf. \cite[Thm. 0.1]{O-Y:1})
\begin{enumerate}
\item[(1)]
$\overline{M}_H^{v_0}(v_0)$ is normal
and the singular points 
$q_1,q_2,\dots,q_m$ of $\overline{M}_H^{v_0}(v_0)$ correspond to
the $S$-equivalence classes of properly $v_0$-twisted semi-stable objects.
\item[(2)]  
For a suitable choice of $\alpha$ with $|\langle \alpha^2 \rangle| \ll 1$,
there is a surjective morphism
$\pi:\overline{M}_H^{v_0+\alpha}(v_0)=
M_H^{v_0+\alpha}(v_0) \to \overline{M}_H^{v_0}(v_0)$ 
which becomes a minimal resolution of the singularities.
\item [(3)]
Let $\bigoplus_{j \geq 0} E_{ij}^{\oplus a_{ij}'}$ be 
the $S$-equivalence class corresponding to $q_i$, where
$E_{ij}$ are $v_0$-twisted stable objects. 
\begin{enumerate}
\item
Then the matrix $(-\langle v(E_{ij}),v(E_{ik}) \rangle)_{j,k \geq 0}$ 
is of affine
type $\tilde{A},\tilde{D}, \tilde{E}$.
\item
Assume that $a_{i0}'=1$. Then 
the singularity of $\overline{M}_H^{v_0}(v_0)$ at 
$q_i$ is a
rational double point of type $A,D,E$ according as the type of 
the matrix $(-\langle v(E_{ij}),v(E_{ik}) \rangle)_{j,k \geq 1}$.
\end{enumerate}
%\item[(4)]
%Assume that $a_{i0}'=1$ for all $i$.
%There is an $\alpha$ such that $-\langle \alpha,v(E_{ij}) \rangle>0$
%for all $j>0$.
%We set
%\begin{equation}
%C_{ij}':=\{x' \in X'|\Hom({\cal E}_{|\{x' \} \times X},E_{ij})\ne 0 \}.
%$\end{equation}
%Then $C_{ij}$ is a smooth rational curve 
%and $\Phi^\alpha(E_{ij})={\cal O}_{C_{ij}'}(b_{ij}')[-1]$.
%In particular,
%$(C_{ij},C_{i' j'})=-\chi(E_{ij},E_{i' j'})$ and
%$\pi^{-1}(p_i)=\sum_{j \geq 1}a_{ij}C_{ij}$. 
\end{enumerate}
\end{thm}

\begin{rem}\label{rem:K3-desing}
A $(-2)$-vector $u \in L:=v_0^{\perp} \cap \widehat{H}^{\perp}
\cap H^*(X,{\Bbb Z})_{\alg}$ is {\it numerically irreducible}, if
there is no decomposition
$u=\sum_i b_i u_i$ such that $u_i \in L$,
$\langle u_i^2 \rangle=-2$, $\rk u> \rk u_i>0$, $b_i \in {\Bbb Z}_{>0}$.  
If $u$ is numerically irreducible,
as we shall see in 
Proposition \ref{prop:K3-exceptional-exist},
there is a $v_0$-twisted stable object $E$ with $v(E)=u$.
In particular, if there is a decomposition
$v_0=\sum_{i \geq 0} a_i u_i$
such that $u_i \in L$ are numerically irreducible,
%$v_0^{\perp} \cap \widehat{H}^{\perp}\cap H^*(X,{\Bbb Z})_{\alg}$, 
$\langle u_i^2 \rangle=-2$, 
$\rk u_i>0$ and $a_i \in {\Bbb Z}_{>0}$, then
there are $v_0$-stable objects $E_i$ such that
$v(E_i)=u_i$, and hence 
$v_0=v(\oplus_i E_i^{\oplus a_i})$. 
Thus the types of the singularities are determined by the sublattice
$L$ of $H^*(X,{\Bbb Z})$.
%$v_0^{\perp} \cap \widehat{H}^{\perp}\cap H^*(X,{\Bbb Z})_{\alg}$.
\end{rem}
We shall give a proof of this theorem
in subsection \ref{subsect:K3:proof}.
We assume that 
$\alpha \in (v_0^{\perp} \cap \varrho_X^{\perp}) \otimes {\Bbb Q}$ 
is general and set $X':=M_H^{v_0+\alpha}(v_0)$.
$X'$ is a $K3$ surface.
We have a morphism $\phi:X' \to \overline{M}_H^{v_0}(v_0)$.
%
%Let $Y'$ be the normalization of
%$\overline{M}_H^{v_0}(v_0)$. Then we have a morphism
%$\pi':X' \to Y'$. 
%We also treat the case where $X'=Y'$, that is, $m=0$.
%
We shall explain some cohomological properties of
the Fourier-Mukai transform associated to $X'$. 
Let ${\cal E}$ be a universal family as a twisted object on
$X' \times X$.
For simplicity, we assume that
${\cal E}$ is an untwisted object on
$X' \times X$.
But all results hold even if ${\cal E}$ is a twisted object.
%
%
%We shall define a canonical polarization of $X'$.
We set 
\begin{equation}\label{eq:G_i}
\begin{split}
G_1:=&{\cal E}_{|\{ x' \} \times X} \in K(X),\\
G_2:=&{\cal E}_{|X' \times \{x \}}^{\vee} \in K(X'),\\
G_3:=&{\cal E}_{|X' \times \{x \}} \in K(X')
\end{split}
\end{equation}
for some $x \in X$ and $x' \in X'$.
We also set 
\begin{equation}\label{eq:w_0}
w_0:=v({\cal E}_{|X' \times \{x \}}^{\vee})=
r_0+\widetilde{\xi}_0+\widetilde{a}_0 \varrho_{X'}, 
\widetilde{\xi}_0 \in \NS(X').
\end{equation}
% 
%
%Let $p_X:Y \times X \to X$ (resp. $p_Y:Y \times X \to Y$) be the projection.
We set $\Phi^{\alpha}:=\Phi_{X \to X'}^{{\cal E}^{\vee}}$ and
$\widehat{\Phi}^{\alpha}:=\Phi_{X' \to X}^{{\cal E}}$.
Thus
%$\Phi^{\alpha}:{\bf D}(X) \to {\bf D}(X')$ by
\begin{equation}
\Phi^{\alpha}(x):={\bf R}\Hom_{p_{X'}}({\cal E}, p_X^*(x)),
x \in {\bf D}(X),
\end{equation}
and
$\widehat{\Phi}^{\alpha}:{\bf D}(X') \to {\bf D}(X)$ by
\begin{equation}
\widehat{\Phi}^{\alpha}(y):={\bf R}\Hom_{p_{X}}({\cal E}^{\vee}, p_{X'}^*(y)),
y \in {\bf D}(X'),
\end{equation}
where $\Hom_{p_{Z}}(-,-)=p_{Z*}{\cal H}om_{{\cal O}_{X' \times X}}(-,-)$,
$Z=X,X'$
are the sheaves of relative homomorphisms. 
\begin{thm}[\cite{Br:2}, \cite{Or:1}]
$\Phi^{\alpha}$ is an equivalence of
categories and the inverse is given by $\widehat{\Phi}^{\alpha}[2]$.
\end{thm}
%We denote the $i$-th cohomology sheaf $H^i({\cal F}_{\cal E}(x))$ by
%${\cal F}_{\cal E}^i(x)$.
%${\cal F}_{\cal E}$ also induces an isometry of
%the Mukai lattices ${\cal F}_{\cal E}:H^{ev}(X,{\Bbb Z})
%\to H^{ev}(Y,{\Bbb Z})$.
For $D \in H^2(X,{\Bbb Q})$,
we set 
\begin{equation}\label{eq:mu-map}
\begin{split}
\widehat{D}:
%& -\nu(D)\\
%-f_{w_0}^{-1} \circ {\cal F}_{\cal E} \circ f_{v_0^{\perp}}(D)\\
=& -\left[\Phi^{\alpha}\left(D+\frac{(D,\xi_0)}{r_0}\varrho_X \right)
\right]_1\\ %\in H^2(Y,{\Bbb Q})\\
=& \left[
p_{X'*} \left(\left(c_2({\cal E})-\frac{r_0-1}{2r_0}(c_1({\cal E})^2)\right) 
\cup p_X^*(D) \right) \right]_1
\in H^2(X',{\Bbb Q}),
\end{split}
\end{equation}
where $[\;\;]_1$ means the projection to
$H^2(X',{\Bbb Q})$.

\begin{lem}(cf. \cite[Lem. 1.4]{Y:Stability})
$r_0 \widehat{H}$ is a nef and big divisor on $X'$
which defines a contraction
$\pi':X' \to Y'$ of $X'$ to a normal surface $Y'$.
There is a morphism $\psi:Y' \to \overline{M}_H^{v_0}(v_0)$
such that $\phi=\psi \circ \pi'$.
\end{lem}

\begin{proof}
Let $G$ be a local projective generator
of ${\cal C}$ such that $\tau(G)=2\tau(G_1)$
(Lemma \ref{lem:assumption}).
Applying Lemma \ref{lem:det-bdle}, we have an ample line bundle
${\cal L}(\zeta)$ on $\overline{M}_H^G(v_0)=\overline{M}_H^{v_0}(v_0)$.
By the definition of $\widehat{H}$,
$c_1(\phi^*({\cal L}(\zeta)))=r_0 \widehat{H}$.
Hence our claim holds.
%Thus we can show that the natural polarization of
%$\overline{M}_H^{G}(v_0)$ associated to the $G$-twisted 
%semi-stability coincides with $\widehat{H}$.
\end{proof}

\begin{prop}(cf. \cite[Prop. 1.5]{Y:Stability})\label{prop:deg-preserve}
\begin{enumerate}
\item[(1)]
Every element $v \in H^*(X,{\Bbb Z})$ 
can be uniquely written as 
\begin{equation*}
 v=l v_0 +a \varrho_X+
d \left(H+\frac{1}{r_0}(H,\xi_0)\varrho_X \right)+
\left(D
 +\frac{1}{r_0}(D,\xi_0)\varrho_X \right),
\end{equation*}
where 
\begin{equation}\label{eq:deg}
\begin{split}
l=&\frac{\rk v}{\rk v_0}=-\frac{\langle v,\varrho_X \rangle}{\rk v_0}
\in \frac{1}{r_0}{\Bbb Z},\\
a=&-\frac{\langle v, v_0 \rangle}{\rk v_0} \in
\frac{1}{r_0}{\Bbb Z},\\ 
d=&\frac{\deg_{G_1}(v)}{\rk v_0(H^2)}
\in \frac{1}{r_0 (H^2)}{\Bbb Z}
\end{split}
\end{equation}
and $D \in H^2(X,{\Bbb Q})\cap H^{\perp}$.
Moreover $v \in v({\bf D}(X))$ if and only if
$D \in \NS(X)\otimes {\Bbb Q} \cap H^{\perp}$.
\item[(2)]
\begin{equation}\label{eq:deg-preserve}
\begin{split}
& \Phi^{\alpha} \left(l v_0 +a \varrho_X+
\left(dH+D+\frac{1}{r_0}(dH+D,\xi_0)\varrho_X \right)\right)\\
=&
l \varrho_{X'}+a w_0-
\left(d\widehat{H}+\widehat{D}+
\frac{1}{r_0}(d\widehat{H}+\widehat{D},\widetilde{\xi}_0)\varrho_{X'}
\right)
\end{split}
\end{equation}
where $D \in H^2(X,{\Bbb Q}) \cap H^{\perp}$.
\item[(3)]
\begin{equation*}
\deg_{G_1}(v)=-\deg_{G_2}(\Phi^{\alpha}(v)).
\end{equation*}
In particular, $\deg_{G_2}(w) \in {\Bbb Z}$ for
$w \in H^*(X',{\Bbb Z})$ and
\begin{equation*}
\min\{\deg_{G_1}(E)>0| E \in K(X)\}
=\min\{\deg_{G_2}(F)>0| F \in K(X')\}.
\end{equation*}
\end{enumerate}
\end{prop}

\subsection{Proof of Theorem \ref{thm:K3-desing}}\label{subsect:K3:proof}

We shall choose a special $\alpha$ and study the structure of the moduli
spaces.

We first prove the following.
The normalness of $\overline{M}_H^{v_0}(v_0)$ will be proved 
in Proposition \ref{prop:K3-normal}.
\begin{prop}\label{prop:K3-desing}
\begin{enumerate}
\item[(1)]
$\psi:Y' \to \overline{M}_H^{v_0}(v_0)$ is bijective.
\item[(2)]
The singular points of $Y'$ correspond to properly
$v_0$-twisted semi-stable objects.
\item[(3)]
Let $\bigoplus_{j \geq 0} E_{ij}^{\oplus a_{ij}'}$ be 
the $S$-equivalence class of a properly $v_0$-twisted semi-stable
object, where $E_{ij}$ are $v_0$-twisted stable.
Then the matrix $(-\langle v(E_{ij}),v(E_{ik}) \rangle)_{j,k \geq 0}$ 
is of affine
type $\tilde{A},\tilde{D}, \tilde{E}$.
We assume that $a_{i0}=1$.
Then $\psi^{-1}(\bigoplus_{j \geq 0} E_{ij}^{\oplus a_{ij}'})$ 
is a rational double point of type $A,D,E$ according as the type
of the matrix $(-\langle v(E_{ij}),v(E_{ik}) \rangle)_{j,k \geq 1}$.
\end{enumerate}
\end{prop}

\subsubsection{Proof of Proposition \ref{prop:K3-desing}.}
We note that $M_H^{v_0}(v_0)$ is smooth and
$\phi, \psi$ are isomorphic over $M_H^{v_0}(v_0)$.
Hence the singular points of $Y'$ are in the inverse image of
$\overline{M}_H^{v_0}(v_0) \setminus M_H^{v_0}(v_0)$. 
Thus we may concentrate on the locus of properly
$v_0$-twisted semi-stable objects.
The first claim of Proposition \ref{prop:K3-desing} (3) follows
from the following.
\begin{lem}\label{lem:K3:lattice}
Assume that $E$ is $S$-equivalent to 
$\bigoplus_{j \geq 0} E_{ij}^{\oplus a_{ij}'}$, where
$E_{ij}$ are $v_0$-twisted stable objects.
Then the matrix 
$(-\langle v(E_{ij}),v(E_{ik}) \rangle)_{j,k \geq 0}$ is of type 
$\widetilde{A},\widetilde{D}, \widetilde{E}$.
Moreover $\langle v(E_{ij}),v(E_{kl}) \rangle =0$, 
if $\bigoplus_{j \geq 0} E_{ij}^{\oplus a_{ij}'} \not \cong
\bigoplus_{l \geq 0} E_{kl}^{\oplus a_{kl}'}$.
\end{lem}

\begin{proof}
Since $\deg_{G_1}(E)=\chi(G_1,E)=0$,
$\deg_{G_1}(E_{ij})=\chi(G_1,E_{ij})=0$, which implies
that $v(E_{ij}) \in v_0^{\perp} \cap \widehat{H}^\perp$.
Since $(v_0^{\perp} \cap \widehat{H}^\perp)/{\Bbb Z}v_0$ is negative
definite, applying Lemma \ref{lem:appendix:lattice} (1),
we see that the matrix is of type 
$\widetilde{A},\widetilde{D}, \widetilde{E}$.
We note that 
$\bigoplus_{j \geq 0} E_{ij}^{\oplus a_{ij}'} \not \cong
\bigoplus_{l \geq 0} E_{kl}^{\oplus a_{kl}'}$ implies that
$\{ E_{i0},E_{i1},...,E_{i s_i'} \}\ne
\{ E_{k0},E_{k1},...,E_{k s_k'} \}$. 
Since $\chi(E_{ij},E_{kl})>0$ implies that
$E_{ij} \cong E_{kl}$, 
$\{ v(E_{i0}),v(E_{i1}),...,v(E_{i s_i'}) \}\ne
\{ v(E_{k0}),v(E_{k1}),...,v(E_{k s_k'}) \}$. 
Then the second claim follows from Lemma \ref{lem:appendix:lattice} (2).
\end{proof}

By this lemma,
we may assume that $a_{i0}'=1$ for all $i$. 
THen we can choose a sufficiently small $\alpha \in v_0^{\perp}$ 
such that
$-\langle \alpha,v(E_{ij}) \rangle>0$ for all $j>0$.
We have the following.
\begin{lem}\label{lem:K3:key}
Lemma \ref{lem:0-stable:key} holds, if we replace 
$\varrho_X$ by $v_0$ and the $\alpha$-stability 
by the $v_0 +\alpha$-twisted stability.
\end{lem}

\begin{proof}
(1) Assume that $F$ is $S$-equivalent to
$\bigoplus_{j \geq 0}F_{ij}^{\oplus c_{ij}}$, where
$F_{ij}$ are $v_0$-twisted stable objects.
If $v(F)=v(\oplus_{j \geq 0} E_{ij}^{\oplus b_{ij}})$, $b_{i0}=1$, then
applying Lemma \ref{lem:K3:lattice}
to $\bigoplus_{j \geq 0}F_{ij}^{\oplus c_{ij}}
 \oplus \bigoplus_{j \geq 0}E_{ij}^{\oplus (a_{ij}-b_{ij})}$ and 
$\bigoplus_{j \geq 0}E_{ij}^{\oplus a_{ij}}$, we get
$ \bigoplus_{j \geq 0}F_{ij}^{\oplus c_{ij}}
\oplus \bigoplus_{j>0}E_{ij}^{\oplus (a_{ij}-b_{ij})}
\cong \bigoplus_{j \geq 0}E_{ij}^{\oplus a_{ij}}$, which implies the
claim.
Then the proofs of (2), (3) and (4) are the same. 
\end{proof}

\begin{lem}\label{lem:K3:exceptional}
\begin{enumerate}
\item[(1)]
We set 
\begin{equation}
C_{ij}':=\{x' \in X'| 
\Hom({\cal E}_{|\{x' \} \times X},E_{ij}) \ne 0 \}, j>0.
\end{equation}
Then 
$C_{ij}'$ is a smooth rational curve.
\item[(2)]
\begin{equation}
\begin{split}
\phi^{-1}(\bigoplus_{j \geq 0} E_{ij}^{\oplus a_{ij}'})=&
\{ x' \in X'| \Hom(E_{i0},{\cal E}_{|\{x' \} \times X}) \ne 0 \}
= \cup_j C_{ij}'.
\end{split}
\end{equation}
In particular, $\phi$ and $\psi$ are surjective.
\end{enumerate}
\end{lem}

\begin{proof}
The proof is the same as in Lemma \ref{lem:0-stable:exceptional}.
\begin{NB}(1) By our choice of $\alpha$, 
$\Hom(E_{ij},{\cal E}_{|\{x' \} \times X})=0$ for all $x' \in X'$.
If $C_{ij}'=\emptyset$, then $\chi(E_{ij},{\cal E}_{|\{x' \} \times X})=0$
implies that
$\Ext^1({\cal E}_{|\{x' \} \times X},E_{ij})=0$.
Then $\Phi(E_{ij})=0$, which is a contradiction.
Therefore $C_{ij}' \ne \emptyset$.
The smoothness follows from the same arguments in the proof of 
Theorem \ref{thm:RDP-desing}.

(2) By our choice of $\alpha$,
$\Hom(E_{i0},{\cal E}_{|\{x' \} \times X}) \ne 0$ for
 $x' \in \phi^{-1}(\bigoplus_{j \geq 0} E_{ij}^{\oplus a_{ij}'})$.
Conversely if $\Hom(E_{i0},{\cal E}_{|\{x' \} \times X}) \ne 0$, then
Lemma \ref{lem:K3:key} 
implies that ${\cal E}_{|\{x' \} \times X}$ is $S$-equivalent
to $\bigoplus_{j \geq 0} E_{ij}^{\oplus a_{ij}'}$.
Therefore we have the first equality.
In the same way, 
by the choice of $\alpha$, 
we get $\phi^{-1}(\bigoplus_{j \geq 0} E_{ij}^{\oplus a_{ij}'}) 
\subset \cup_j C_{ij}'$ and by Lemma \ref{lem:K3:key}, we get
$\cup_j C_{ij}' \subset 
\phi^{-1}(\bigoplus_{j \geq 0} E_{ij}^{\oplus a_{ij}'})$. 
\end{NB}
\end{proof}

We also have the following lemma whose proof is the same as of 
Lemma \ref{lem:Lambda(E)}.

\begin{lem}\label{lem:K3:Phi(E)}
$\Phi^\alpha(E_{ij})[1]$ is a line bundle on $C_{ij}'$.
In particular, $\langle v(E_{ij}),v(E_{kl}) \rangle=
(C_{ij}',C_{kl}')$. 
We define $b_{ij}'$ by
$\Phi^\alpha(E_{ij})={\cal O}_{C_{ij}'}(b_{ij}')[-1]$. 
\end{lem}

This lemma shows that the configuration of $\{C_{ij}'|j>0\}$ is of type
$A,D,E$.
Since $(\widehat{H},C_{ij}')=0$,
$\cup_j C_{ij}'$ is contracted to a rational double point
of $Y'$. Hence Proposition \ref{prop:K3-desing} (2) and (3) hold.
Since $\psi^{-1}(\bigoplus_{j \geq 0} E_{ij}^{\oplus a_{ij}'})$ is a point,
$\psi$ is injective. Thus Proposition \ref{prop:K3-desing} (1) also holds.

\begin{NB}
In particular, all fibers of $\pi$ are connected.
Since $Y \to \overline{M}_H^{v_0}(v_0)$ is finite,
$Y \to  \overline{M}_H^{v_0}(v_0)$ is set-theoretically injective and
the image of $\cup_j C_{ij}'$ in $Y$ is a rational double point.
\end{NB}
We shall prove the normality in Proposition \ref{prop:K3-normal}.

\subsubsection{Perverse coherent sheaves on $X'$ and the normality of
$\overline{M}_H^{v_0}(v_0)$.}
We set $Z_i':=\pi^{-1}(q_i)=\sum_{j=1}^{s_i'}a_{ij}' C_{ij}'$.
%Assume that $\alpha \in v_0^{\perp}$
%satisfies $-\langle \alpha,v(E_{ij}) \rangle>0$ for $j>0$.
Then $E_{i0}$ is a subobject of ${\cal E}_{|\{x' \} \times X}$
for $x' \in Z_i'$ 
and we have an exact sequence
\begin{equation}
0 \to E_{i0} \to {\cal E}_{|\{x' \} \times X} \to F \to 0,\;
 x' \in Z_i'
\end{equation}
where $F$ is a $v_0$-twisted semi-stable object with 
$\gr(F)=\oplus_{j= 1}^{s_i'} E_{ij}^{\oplus a_{ij}'}$.
Then we get an exact sequence
\begin{equation}\label{eq:y}
0 \to \Phi^\alpha(F)[1] \to \Phi^\alpha(E_{i0})[2] \to {\Bbb C}_{x'} \to 0
\end{equation}
in $\Coh(X')$.
Thus $\WIT_2$ holds for $E_{i0}$ with respect to
$\Phi^\alpha$.

\begin{defn}
We set 
$A_{i0}':=\Phi^\alpha(E_{i0})[2]$ and
$A_{ij}':=\Phi^\alpha(E_{ij})[2]={\cal O}_{C_{ij}'}(b_{ij}')[1]$ for $j>0$.
\end{defn}

\begin{lem}\label{lem:A_*}
\begin{enumerate}
\item[(1)]
$\Hom(A_{i0}',A_{ij}'[-1])=\Ext^1(A_{i0}',A_{ij}'[-1])=0$.
\item[(2)]
We set ${\bf b}_i':=(b_{i1}',b_{i2}',\dots,b_{i s_i'}')$.
Then $A_{i0}' \cong A_0({\bf b}_i')$.
In particular,
$\Hom(A_{i0}',{\Bbb C}_{x'})={\Bbb C}$ for $x' \in Z_i'$. 
\item[(3)]
Irreducible objects of $\Per(X'/Y',{\bf b}_1',...,{\bf b}_m')$
are 
\begin{equation}
A_{ij}' \;(1 \leq i \leq m,0 \leq j \leq s_i'),\; {\Bbb C}_{x'} \;
(x' \in X' \setminus \cup_i Z_i').
\end{equation}
\end{enumerate}
\end{lem}
\begin{proof}
(1)
We have
\begin{equation}
\begin{split}
\Hom(A_{i0}',A_{ij}'[k])&=
\Hom(\Phi^\alpha(E_{i0})[2],\Phi^\alpha(E_{ij})[2+k])\\
&=\Hom(E_{i0},E_{ij}[k])=0
\end{split}
\end{equation} 
for $k=-1,0$.

(2)
By \eqref{eq:y} and (1), 
we can apply Lemma \ref{lem:A_0}
to prove 
$A_{i0}'=A_0({\bf b}_i')=A_{q_i}$. 
(3) is a consequence of (2) and
Proposition \ref{prop:tilting:G-1Per-irred}.
%Since $\Hom(A_{i0},A_{i0})={\Bbb C}$,
%the claim follows from (1) and \eqref{eq:y}.
\end{proof}

\begin{defn}\label{defn:K3:Per^D}
We set
\begin{equation}
\begin{split}
\Per(X'/Y'):=&\Per(X'/Y',{\bf b}_1',\dots,{\bf b}_m'),\\
\Per(X'/Y')^D:=&
\Per(X'/Y',-{\bf b}_1'+2{\bf b}_0,\dots,-{\bf b}_m'+2{\bf b}_0)^*,
{\bf b}_0:=(-1,-1,...,-1).
\end{split}
\end{equation}
\end{defn}

\begin{rem}
Assume that $\alpha \in v_0^{\perp}$ satisfies 
$-\langle v(E_{ij}),\alpha \rangle<0$, $j>0$.
Then $\Phi(E_{ij})[2]={\cal O}_{C_{ij}'}(b_{ij}'')$, $j>0$
and $\Phi(E_{i0})[2]=A_0({\bf b}_i'')[1]$ belong to
$\Per(X'/Y',{\bf b}_1'',\dots,{\bf b}_m'')^*$, where
${\bf b}_i''=(b_{i0}'',...,b_{i s_i'}'')$. 
\end{rem}

\begin{lem}\label{lem:assumption-G_2}
There is a local projecive generator $G$ of $\Per(X'/Y')$
such that $\tau(G)=2 \tau(G_2)$.  Moreover
$G^{\vee}$ is a local projective generator of $\Per(X'/Y')^D$.
\end{lem}

\begin{proof}
Since
$\chi(G_2,A_{ij})=\chi({\Bbb C}_x,E_{ij})=\rk E_{ij}>0$,
we get our claim by Proposition \ref{prop:generator-exist}.
The second claim follows from the definition
of $\Per(X'/Y')^D$ and Lemma \ref{lem:tilting:dual}.
\end{proof}

\begin{lem}\label{lem:semi-stable-objects}
Let $E$ be an object of ${\cal C}$ such that
$E$ is $G_1$-twisted stable and $\deg_{G_1}(E)=\chi(G_1,E)=0$.
Then $E \cong E_{ij}$ or $E \cong {\cal E}_{|\{ x' \} \times X}$,
$x' \in X' \setminus \cup_i Z_i'$. 
\end{lem}

\begin{proof}
Since $\chi(G_1,E)=0$, there is a point $x' \in X'$ such that
$\Hom({\cal E}_{|\{x' \} \times X},E) \ne 0$ or
$\Hom(E,{\cal E}_{|\{x' \} \times X}) \ne 0$.
Then $E$ is a quotient object or a subobject of
${\cal E}_{|\{x' \} \times X}$, which implies the claim.
\end{proof}

\begin{defn}
\begin{enumerate}
\item[(1)]
Let
${\cal C}_{v_0}$ be the full subcategory of
${\cal C}$ generated by
$E_{ij}$ and ${\cal E}_{|\{x' \} \times X}$, $x' \in X'$.
That is ${\cal C}_{v_0}$ consists of $v_0$-twisted semi-stable
objects $E$ with $\deg_{G_1}(E)=\chi(G_1,E)=0$.
\item[(2)]
Let $\Per(X'/Y')_0$ be the full subcategory of 
$\Per(X'/Y')$ consisting of 0-dimensional objects.
\end{enumerate}
\end{defn}

\begin{prop}\label{prop:K3-normal}
\begin{enumerate}
\item[(1)]
$\Phi^\alpha[2]$ induces an equivalence
${\cal C}_{v_0} \to \Per(X'/Y')_0$.
\item[(2)]
Moreover $\Phi^\alpha[2]$ induces an isomorphism
${\cal M}_H^{v_0+\beta}(v_0)^{ss} \cong 
{\cal M}_{\widehat{H}}^{G,\Phi^\alpha(\beta)}(\varrho_{X'})^{ss}$,
where $\beta \in (v_0^{\perp} \cap \varrho_X^{\perp}) \otimes {\Bbb Q}$
is sufficiently small and $G$ an arbitrary projective generator of 
$\Per(X'/Y')$.
\item[(3)]
$\overline{M}_H^{v_0+\beta}(v_0) \cong
\overline{M}_{\widehat{H}}^{G,\Phi^\alpha(\beta)}(\varrho_{X'})$.
In particular, $\overline{M}_H^{v_0}(v_0)$
is a normal surface.
\end{enumerate}
\end{prop}

\begin{proof}
(1) We note that $\Phi^\alpha(E_{ij})[2]=A_{ij}'$ and
$\Phi^\alpha({\cal E}_{|\{x' \} \times X})[2]={\Bbb C}_{x'}$, $x' \in X'$.
Hence the claim holds.
(2) We note that
$E \in {\cal M}_H^{v_0}(v_0)^{ss}$ is 
$v_0+\beta$-twisted semi-stable, if 
$\chi(\beta,F)=\chi(v_0+\beta,F) \leq 0$ for all subsheaf $F$ of $E$ with
$\deg_{G_1}(F)=\chi(G_1,F)=0$.
Since $\chi(\Phi^\alpha(\beta),\Phi^\alpha(F))=\chi(\beta,F)$,
$\Phi^\alpha(E)[2]$ is $(G_2,\Phi^\alpha(\beta))$-twisted semi-stable.
Then Remark \ref{rem:G-indep} implies that 
$\Phi^\alpha(E)[2]$ is $(G,\Phi^\alpha(\beta))$-twisted semi-stable
for any $G$.
The first claim of (3) follows from (2).  
In the notation of subsection \ref{subsect:wall-chamber},
$\overline{M}_{\widehat{H}}^{G,0}(\varrho_{X'}) 
\cong (X')^0$.
Hence the second claim of (3) follows from 
Proposition \ref{prop:Y=X^0}.
\end{proof}

\begin{prop}\label{prop:K3-exceptional-exist}
Let $u \in H^{ev}(X,{\Bbb Z})_{\alg}$ be a Mukai vector such that 
$u \in v_0^{\perp} \cap \widehat{H}^{\perp}$, $0<\rk u<\rk v_0$
and $\langle u^2 \rangle=-2$.
Then $u=\sum_j b_j v(E_{ij})$, $0 \leq b_j \leq a_{ij}$.
In particular, $\overline{M}^{v_0}_H(u) \not =\emptyset$.
\end{prop}

\begin{proof}
Since $u \in v_0^{\perp} \cap \widehat{H}^{\perp}$,
$\Phi^\alpha(u)=(0,D,b)$, $D \in \NS(X')$, $b \in {\Bbb Z}$ and
$(D,\widehat{H})=0$.
Since $(D^2)=-2$,
$D$ or $-D$ is an effective divisor supported on an exceptional locus
$Z_i'$.
Hence $\Phi^\alpha(u) \in 
\oplus_{j=0}^{s_i'} {\Bbb Z}\Phi^\alpha(E_{ij})
=\oplus_{j=1}^{s_i'}{\Bbb Z}C_{ij} \oplus {\Bbb Z}\varrho_X$.
By the basic properties of the root systems of affine Lie algebra, 
$\Phi^\alpha(u)=c\Phi^\alpha(v_0)\pm \sum_{j>0} c_j \Phi^\alpha(E_{ij})$, 
$0 \leq c_j \leq a_{ij}$.
Then $\rk u=cr \pm \sum_{j>0} c_j \rk E_{ij}$.
Since $\sum_{j>0} c_j \rk E_{ij} \leq \sum_{j>0} a_{ij} \rk E_{ij}<r$,
we get
$u=\sum_{j>0} c_j v(E_{ij})$ or 
$u=v_0-\sum_{j>0} c_j v(E_{ij})$.
Therefore the claim holds. 
\end{proof}

\subsection{Walls and chambers for the moduli spaces of dimension 2.}

We shall study the dependence of $\overline{M}_H^w(v_0)$ on $w$.
We set
\begin{equation}
\begin{matrix}
\delta: & \NS(X) \otimes {\Bbb Q}& \to & H^*(X,{\Bbb Q}) \\
& D & \mapsto & D+\frac{(D,\xi_0)}{r_0}\varrho_X.
\end{matrix}
\end{equation}
We may assume that $w=v_0+\alpha, \alpha \in \delta(H^{\perp})$
(cf. \cite[sect. 1.1]{O-Y:1}).
We set 
\begin{equation}
 {\cal U}:=\left\{u \in v({\bf D}(X)) \left|
\begin{split}
&\langle u^2 \rangle=-2, \langle v_0,u \rangle \leq 0,
\langle \delta(H),u \rangle=0,\\
&0< \rk u <\rk v_0 
\end{split}
\right. \right\}.
\end{equation}
For a fixed $v_0$ and $H$, ${\cal U}$ is a  finite set.
%Let $u$ be a Mukai vector such that $0<\rk u<\rk v$, 
%$\langle v,u \rangle \leq 0$, $\langle u^2 \rangle=-2$ and
%$\langle u,\widehat{H} \rangle=0$.
For $u \in {\cal U}$,
we define a wall $W_{u} \subset 
\delta(H^{\perp}) \otimes_{\Bbb Q} {\Bbb R}$ 
with respect to $v$
by 
\begin{equation}
W_u:=\{\alpha \in \delta(H^{\perp}) \otimes {\Bbb R}|\;
\langle v_0+\alpha,u \rangle=0 \}.
\end{equation}
A connected component of 
$\delta(H^{\perp}) \otimes_{\Bbb Q} {\Bbb R} 
\setminus \cup_{u \in {\cal U}} W_u$ is said to be a chamber.

\begin{lem}
If $\alpha$ does not lie on any wall $W_u$, $u \in {\cal U}$,
then $\overline{M}_H^{v_0+\alpha}(v_0)=M_H^{v_0+\alpha}(v_0)$.
In particular, $\overline{M}_H^{v_0+\alpha}(v_0)$ is a K3 surface.
\end{lem}

We are interested in the $v_0+\alpha$-twisted stability with 
a sufficiently small $|\langle \alpha^2 \rangle|$.
So we may assume that
\begin{equation}
u \in {\cal U}':=\{u \in {\cal U}| \langle v_0,u \rangle=0 \}.
\end{equation}
For an $\alpha \in \delta(H^{\perp})$ with 
$|\langle \alpha^2 \rangle| \ll 1$, 
let $F$ be a $v_0+\alpha$-twisted stable torsion free object such that
\begin{enumerate}
\item
$\langle v(F)^2 \rangle=-2$,
\item
$\langle v(F),\delta(H) \rangle/\rk F
=(c_1(F),H)/\rk F-(\xi_0,H)/r_0=0$ and 
\item
$\langle v_0,v(F) \rangle=\langle \alpha, v(F)\rangle=0$.
\end{enumerate}
By (i), $F$ is a rigid torsion free object.

\begin{prop}(\cite[Prop. 1.12]{O-Y:1})\label{prop:O-Y}
We set $\alpha^{\pm}:=\pm \epsilon v(F)+\alpha$, where
$0<\epsilon \ll 1$.
Then $T_F$ induces an isomorphism
\begin{equation}
\begin{matrix}
{\cal M}_H^{v+\alpha^-}(v)^{ss} &\to &
{\cal M}_H^{v+\alpha^+}(v)^{ss}\\
E & \mapsto & T_F(E)
\end{matrix}
\end{equation}
which preserves the $S$-equivalence classes.
Hence we have an isomorphism
\begin{equation}
\overline{M}_H^{v+\alpha^-}(v) \to
\overline{M}_H^{v+\alpha^+}(v).
\end{equation}
\end{prop}

\begin{rem}
In \cite{O-Y:1}, we considered the functor $T_F[-1]$.  
\end{rem}

Combining Proposition \ref{prop:O-Y} with Lemma \ref{lem:reflection-Pi},
we get the following Corollary.
\begin{cor}\label{cor:O-Y}
\begin{equation}
\Phi_{X' \to X}^{{\cal E}^{v_0+\alpha^+}} \cong 
T_F \circ \Phi_{X' \to X}^{{\cal E}^{v_0+\alpha^-}}
\cong \Phi_{X' \to X}^{{\cal E}^{v_0+\alpha^-}} \circ T_{A},
\end{equation}
where $A:=\Phi_{X \to X'}^{({\cal E}^{v_0+\alpha^-})^{\vee}[2]}(F)$.
\end{cor}
Assume that ${\cal E}^{v_0+\alpha}_{|\{x' \} \times X}$ is
$S$-equivalent to $\oplus_i {E_i'}^{\oplus a_i'}$.
Then $\alpha \in (\sum_i {\Bbb Q}v(E_i'))^{\perp}$.

\begin{rem}\label{rem:comm}
If $\alpha$ belongs to exactly one wall $W_u$, $u \in {\cal U}$, then
there is a $v+\alpha$-twisted stable object $F$ with $v(F)=u$. 
So we can apply Propositions \ref{prop:O-Y}.
Moreover $A={\cal O}_{C}(b)$, where $C$ is a smooth rational curve
defined by 
\begin{equation}
C:=\{x' \in X'| \Ext^2({\cal E}^{v_0+\alpha^-}_{|\{x' \} \times X},F) \ne 0\}.
\end{equation}

\end{rem}

%Let $G$ be an object of ${\bf D}(X)$ such that
%\begin{equation}
%\begin{split}
%\Hom(G,E_{ij}[k])=&0, k \ne 2\\
%\Hom(G,E_{ij}[k]) \ne &0, k=2.\\
%\Hom(G,E[k])=&0, k \ne 2\\
%\Hom(G,E[k]) \ne &0, k=2\\
%\end{split}
%\end{equation}
%for all $E \in M_H^{G_1}(v_0)$ and $i,j$.

\begin{prop}\label{prop:choice-of-alpha}
Let $G$ be an object of ${\bf D}(X)$ such that
$\chi(G,E_{ij})> 0$ for all $i,j$ and
\begin{equation}
\begin{split}
\Hom(G,E_{ij}[k])=&\Hom(G,E[k])=0, k \ne 2
%\chi(G,E_{ij})> &0,
\end{split}
\end{equation}
for all $E \in M_H^{G_1}(v_0)$ and $i,j$.
Assume that $\alpha \in 
\delta(H^{\perp}) \setminus \cup_{u \in {\cal U}'}W_u$ is sufficiently small.
\begin{enumerate}
\item[(1)]
$G^{\alpha}:=\Phi^{\alpha}(G)$ is a locally free sheaf on $X'$
and ${\cal A}':=\pi_*((G^{\alpha})^{\vee} \otimes G^{\alpha})$
is a reflexive sheaf on $Y'$ which
is independent of the choice of $\alpha$. 
\item[(2)]
${\bf R}\pi_*((G^{\alpha})^{\vee} \otimes \underline{\;\;})
\circ \Phi^{\alpha}:
{\bf D}(X) \to {\bf D}_{{\cal A}'}(Y')$
is independent of the choice of $\alpha$.
\end{enumerate}
\end{prop}

\begin{proof}
We take a small $\alpha \in 
\delta(H^{\perp})$
with $-\langle \alpha,v(E_{ij}) \rangle>0$, $j>0$.
By the base change theorem,
$G^{\alpha}$ is a locally free sheaf on $X'$.
Let $A_{ij}'$ be objects of $\Per(X'/Y')$ 
in subsection \ref{subsect:K3:proof}.
Then we have $\Hom(G^\alpha,A_{ij}'[k])=0$ for $k \ne 0$
and $\Hom(G^\alpha,A_{ij}')\ne 0$. 
Assume that
$\alpha' \in \delta(H^{\perp})$
belongs to another chamber.
We set $X'':=M_H^{v_0+\alpha'}(v_0)$.
By Proposition \ref{prop:K3-normal} (2),
$X'' \cong 
M_{\widehat{H}}^{G^\alpha, \Phi^\alpha(\alpha')}(\varrho_{X'})$
and
${\cal F}:=\Phi_{X \to X'}^{({\cal E}^\alpha)^{\vee}[2]}({\cal E}^{\alpha'})$
is the universal family of $\Phi^\alpha(\alpha')$-twisted stable objects,
where ${\cal E}^{\alpha'}$ is the universal family
associated to $\alpha'$.
We have $\Phi^{\alpha'}=
\Phi_{X' \to X''}^{{\cal F}^{\vee}[2]} \circ \Phi^\alpha$.
In particular, $G^{\alpha'}=\Phi_{X' \to X''}^{{\cal F}^{\vee}[2]}(G^\alpha)$.
Then the claim follows from Proposition \ref{prop:Phi-alpha}. 
%and Proposition \ref{prop:A-A'}.
\end{proof}

\subsection{A tilting appeared in \cite{Br:3} and its generalizations.}
From now on, we assume that $\alpha$ satisfies
$-\langle \alpha,v(E_{ij}) \rangle>0$ for all $j>0$
and set
\begin{equation}
\Phi:=\Phi^\alpha,\;\widehat{\Phi}:=\widehat{\Phi}^\alpha.
\end{equation}
By Proposition \ref{prop:choice-of-alpha},
the assumption is not essential.

\begin{defn}%\label{defn:HNF}
We set
\begin{equation}
{\frak C}_i:=
\begin{cases}
{\cal C},& i=1,\\
\Per(X'/Y'), & i=2,\\
\Per(X'/Y')^D,& i=3.
\end{cases}
\end{equation}
For an object $E \in {\frak C}_i$, we define the $G_i$-twisted Hilbert
polynomial by
\begin{equation}
\chi(G_i,E(n)):=\sum_j (-1)^j \dim \Hom(G_i,E(n)[j]),
\end{equation} 
where $E(n):=E(nH)$, $i=1$ and $E(n):=E(n\widehat{H})$, $i=2,3$.
\end{defn}
Then Lemma \ref{lem:assumption} and Lemma \ref{lem:assumption-G_2}
imply the following.
\begin{lem}\label{lem:Hilbert-Poly}
$\chi(G_i,E(n)) > 0$ for $ E \ne 0$ and $n \gg 0$,
that is,
(i) $\rk E>0$ or (ii) $\rk E=0, \deg_{G_i}(E)>0$ or
(iii) $\rk E=\deg_{G_i}(E)=0, \chi(G_i,E)>0$.
\end{lem}

\begin{defn}
Let $E \ne 0$ be an object of ${\frak C}_i$. 
\begin{enumerate}
\item[(1)]
There is a (unique) filtration
\begin{equation}\label{eq:HNF}
0 \subset F_1 \subset F_2 \subset \cdots \subset F_s=E
\end{equation}
such that each $E_j:=F_j/F_{j-1}$ is a torsion object or a torsion free 
$G$-twisted semi-stable object and
\begin{equation}
(\rk E_{j+1})\chi(G_i,E_j(n))>(\rk E_j) \chi(G_i,E_{j+1}(n)), n \gg 0.
\end{equation}
We call it the {\it Harder-Narasimhan filtration} of $E$.  
\item[(2)]
In the notation of (1), we set
\begin{equation}
\begin{split}
\mu_{\max,G_i}(E):=
&\begin{cases}
\mu_{G_i}(E_1),&\rk E_1>0\\
\infty,&\rk E_1=0,
\end{cases}
\\
\mu_{\min,G_i}(E):=
&\begin{cases}
\mu_{G_i}(E_s),&\rk E_s>0\\
\infty,&\rk E_s=0.
\end{cases}
\end{split}
\end{equation}
\end{enumerate}
\end{defn}

\begin{rem}
An object $E \ne 0$ has a torsion if and only if
$\mu_{\max,G_i}(E)=\infty$ and
$E$ is a torsion object if and only if
$\mu_{\min,G_i}(E)=\infty$.
\end{rem}

We define several torsion pairs of ${\frak C}_i$.
\begin{defn}
\begin{enumerate}
\item[(1)]
Let ${\frak T}_i^{\mu}$ (resp. $\overline{\frak T}_i^{\mu}$) 
be the full subcategory of ${\frak C}_i$
such that $E \in {\frak C}_i$ belongs to ${\frak T}_i^{\mu}$
(resp. $\overline{\frak T}_i^{\mu}$) 
if (i) $E$ is a torsion object or (ii) 
$\mu_{\min,G_i}(E) > 0$ (resp. $\mu_{\min,G_i}(E) \geq 0$).
\item[(2)]
Let ${\frak F}_i^{\mu}$ (resp. $\overline{\frak F}_i^{\mu}$)
be the full subcategory of ${\frak C}_i$
such that $E \in {\frak C}_i$ belongs to ${\frak T}_i^{\mu}$
(resp. $\overline{\frak F}_i^{\mu}$)
if $E=0$ or $E$ is a torsion free object 
with $\mu_{\max,G_i}(E) \leq 0$ (resp. $\mu_{\max,G_i}(E) < 0$).
\end{enumerate}
\end{defn}

\begin{defn}
\begin{enumerate}
\item[(1)]
Let ${\frak T}_i$ (resp. $\overline{\frak T}_i$) 
be the full subcategory of ${\frak C}_i$
such that $E \in {\frak C}_i$ belongs to ${\frak T}_i$
(resp. $\overline{\frak T}_i$) 
if (i) $E$ is a torsion object or (ii) 
for the Harder-Narasimhan filtration \eqref{eq:HNF} of $E$,
$E_s$ satisfies
$\mu_{G_i}(E_s)>0$ or $\mu_{G_i}(E_s)=0$ and $\chi(G_i,E_s)>0$
(resp. $\mu_{G_i}(E_s)=0$ and $\chi(G_i,E_s) \geq 0$).
\item[(2)]
Let ${\frak F}_i$ (resp. $\overline{\frak F}_i$) 
be the full subcategory of ${\frak C}_i$
such that $E \in {\frak C}_i$ belongs to 
${\frak F}_i$ (resp. $\overline{\frak F}_i$)
if $E$ is a torsion free object and   
for the Harder-Narasimhan filtration \eqref{eq:HNF} of $E$,
$E_1$ satisfies
$\mu_{G_i}(E_1)<0$ or $\mu_{G_i}(E_1)=0$ and $\chi(G_i,E_1) \leq 0$
(resp. $\mu_{G_i}(E_1)=0$ and $\chi(G_i,E_1) < 0$).
\end{enumerate}
\end{defn}

\begin{defn}\label{defn:category}
$({\frak T}_i^{\mu},{\frak F}_i^{\mu})$,
$(\overline{\frak T}_i^{\mu},\overline{\frak F}_i^{\mu})$,
$({\frak T}_i,{\frak F}_i)$ 
 and $(\overline{\frak T}_i,\overline{\frak F}_i)$ 
are torsion pairs of ${\frak C}_i$.
We denote the tiltings of ${\frak C}_i$ by
${\frak A}_i^{\mu}$, $\overline{\frak A}_i^{\mu}$,
${\frak A}_i$ and 
$\overline{\frak A}_i$ respectively. 
\end{defn}

We note that ${\frak T}_1^{\mu} \subset {\frak T}_1$.
We shall study the condition ${\frak T}_1^{\mu}= {\frak T}_1$.
We start with the following lemma.
\begin{lem}\label{lem:maximal-object}
Let $E$ be a local projective generator of ${\frak C}_i$.
Then $\Ext^1(E,F)=0$ for all 0-dimensional objects $F$ of 
${\frak C}_i$.
In particular, if $E$ is a subobject of a torsion free object
$E'$ such that $E'/E$ is 0-dimensional, then $E'=E$.
\end{lem}

\begin{proof}
We only treat the case where $i=1$.
Then ${\bf R}\pi_*(E^{\vee} \otimes F)=\pi_*(E^{\vee} \otimes F)$
is a 0-dimensional sheaf on $Y$.
Hence we get $\Ext^1(E,F)=H^1(Y,\pi_*(E^{\vee} \otimes F))=0$. 
\end{proof}

\begin{lem}\label{lem:generic-mu-stable}
Assume that ${\cal E}_{|\{x' \} \times X}$ is a $\mu$-stable 
local projective generator of ${\cal C}$ for a general $x' \in X'$.
\begin{enumerate}
\item[(1)]
 ${\frak T}_1={\frak T}_1^{\mu}$.
\item[(2)]
Every $\mu$-semi-stable object $E \in {\cal C}$ with 
$\deg_{G_1}(E)=\chi(G_1,E)=0$
is $G_1$-twisted semi-stable.
Moreover if $E$ is $G_1$-twisted stable, then
it is $\mu$-stable.
\item[(3)]
Let $E$ be a $\mu$-semi-stable object $E \in {\cal C}$ with 
$\rk E>0$, $\deg_{G_1}(E)=\chi(G_1,E)=0$. Then
 $\Ext^i(E,S)=0$, $i \ne 0$ for any irreducible object $S \in {\cal C}$. 
\item[(4)]
${\cal E}_{|\{x' \} \times X}$ is a local projective generator
of ${\cal C}$ for any $x' \in X'$.
\end{enumerate}
\end{lem}

\begin{proof}
(1)
Let $E$ be a $\mu$-stable object of ${\cal C}$ with
$\deg_{G_1}(E)=0$ and $\chi(G_1,E)>0$.
Since $\Hom(E,{\cal E}_{|\{x' \} \times X})=0$
for all $x' \in X'$,
$\Hom({\cal E}_{|\{x' \} \times X},E) \ne 0$
for all $x' \in X'$.
Assume that ${\cal E}_{|\{x' \} \times X}$ is a 
$\mu$-stable local projective generator.
By Lemma \ref{lem:maximal-object} and
$\Hom({\cal E}_{|\{x' \} \times X},E) \ne 0$, we get
$E \cong {\cal E}_{|\{x' \} \times X}$.
Therefore $\chi(G_1,E) \leq 0$ for 
all $\mu$-stable object $E \in {\cal C}$ with
$\deg_{G_1}(E)=0$.
Hence we get
${\frak T}_1={\frak T}_1^{\mu}$.

(2)
Let $E'$ be a subobject of $E$ with $\deg_{G_1}(E)=0$.
Then (1) implies that $\chi(G_1,E') \leq 0$.
Hence
$E$ is $G_1$-twisted semi-stable.
If $E/E'$ is torsion free, then
we also have $\chi(G_1,E/E') \leq 0$, which implies that
$\chi(G_1,E')=\chi(G_1,E/E')=0$.
Thus $E$ is properly $G_1$-twisted semi-stable.
Therefore the second claim also holds.
\begin{NB}
For $E$, we take the Harder-Narasimhan filtration
$$
0 \subset F_1 \subset F_2 \subset \cdots \subset F_s=E.
$$
Then $F_i/F_{i-1}$ are $G_1$-twisted semi-stable objects with
$\deg_{G_1}(F_i/F_{i-1})=0$.
By (1), we have $\chi(G_1,F_i/F_{i-1}) \leq 0$.
Hence $\chi(G_1,E) \leq 0$.
\end{NB}

(3)
If $\Ext^1(S,E)=\Ext^1(E,S)^{\vee} \ne 0$, then a non-trivial extension
\begin{equation}
0 \to E \to E' \to S \to 0
\end{equation}
gives a $\mu$-semi-stable object $E'$ with
$\chi(G_1,E')=\chi(G_1,S)>0$.
On the other hand, (1) implies that $\chi(G_1,E') \leq 0$.
Therefore $\Ext^1(E,S)=0$.
Since $S$ is a torsion object,
$\Ext^2(E,S)\cong \Hom(S,E)^{\vee}=0$.

(4)
Since ${\cal E}_{|\{x' \} \times X}$ is a 
$\mu$-semi-stable object with 
$\deg_{G_1}({\cal E}_{|\{x' \} \times X})=
\chi(G_1,{\cal E}_{|\{x' \} \times X})=0$,
${\cal E}_{|\{x' \} \times X} \in {\cal C}$ and 
satisfies the assertion of (3).
By Lemma \ref{lem:Hilbert-Poly},
$\chi({\cal E}_{|\{x' \} \times X},S)=\chi(G_1,S)>0$ for
any irreducible object $S$.
Then ${\cal E}_{|\{x' \} \times X}$ is locally free
and is a local projective generator
by Proposition \ref{prop:tilting:generator}.
\end{proof}

\begin{rem}\label{rem:generic-mu-stable}
By the proof of Lemma \ref{lem:generic-mu-stable},
${\cal E}_{|\{x' \} \times X}$, $x' \in X'$ 
is a local projective generator
of ${\cal C}$ if ${\frak T}_1={\frak T}_1^{\mu}$.
Indeed if ${\frak T}_1={\frak T}_1^{\mu}$, then
the same proofs of (2), (3) and (4) work. 
\end{rem}

\subsection{Equivalence between ${\frak A}_1$ and ${\frak A}_2^{\mu}$.}

\begin{lem}\label{lem:vanish-Phi}
\begin{enumerate}
\item[(1)]
If $E \in {\frak T}_1$, then $\Hom(E,E_{ij})=
\Hom(E,{\cal E}_{|\{x' \} \times X})=0$
for all $i,j$ and $x' \in X'$. In particular,
$H^2(\Phi(E))=0$.
\item[(2)]
If $E \in {\frak F}_1$, then $\Hom({\cal E}_{|\{x' \} \times X},E)=0$
for a general $x' \in X'$. In particular,
$H^0(\Phi(E))=0$.
\end{enumerate}
\end{lem}

\begin{proof}
(1) The first claim is obvious.
The second claim is a consequence of the Serre duality and 
the base change theorem
(see the proof of Lemma \ref{lem:Perverse-1} (2)).

(2)
If there is a non-zero morphism $\phi:{\cal E}_{|\{ x' \} \times X}
\to E$, we see that $\phi$ is injective and
$\coker \phi \in {\frak F}_1$.
By the induction on $\rk E$,
we get the first claim. The second claim follows by
the base change theorem. 
\end{proof}

\begin{lem}\label{lem:Perverse-1}
Let $E$ be an object of ${\cal C}$.
\begin{enumerate}
\item[(1)]
Assume that $\Hom(E_{ij},E[q])=
\Hom({\cal E}_{|\{x' \} \times X},E[q])=0$
for all $i,j$, $x' \in X'$ and $q>0$.
Then $\Phi(E) \in \Per(X'/Y')$.
\item[(2)]
There is a complex 
\begin{equation}
0 \to W_0 \to W_1 \to W_2 \to 0
\end{equation}
such that $W_i$ are local projective objects of
$\Per(X'/Y')$ and
$\Phi(E)$ is quasi-isomorphic to this complex.  
\item[(3)]
$H^0(^p H^2(\Phi(E)))=H^2(\Phi(E))$ and
$^p H^0(\Phi(E)) \subset H^0(\Phi(E))$.
In particular, $^p H^0(\Phi(E))$ is torsion free. 
\item[(4)]
If $\Hom(E,E_{ij})=0$ for all $i,j$
and $\Hom(E,{\cal E}_{|\{x' \} \times X})=0$ for all $x' \in X'$, then
$^p H^2(\Phi(E))=0$.
In particular, if $E \in {\frak T}_1$,
then $^p H^2(\Phi(E))=0$.
\item[(5)]
If $E \in {\frak F}_1$, then
$^p H^0(\Phi(E))=0$.
 
%If $\Hom(E,{\cal E}_{|\{x' \} \times X})=0$ for a general $x' \in X'$,
%then $^p H^0(\Psi(E))=0$.
%In particular, if $E \in {\frak T}^*$, then $^p H^0(\Psi(E))=0$.
\end{enumerate}
\end{lem}

\begin{proof}
(1)
We note that
$F \in \Per(X'/Y')$ is 0 if and only if  
$\Hom(F, A_{ij}')=\Hom(F,A_{i0}')=\Hom(F,{\Bbb C}_{x'})=0$
for all $i$, $j>0$ and $x' \in X'$. 
Since
\begin{equation}
\begin{split}
\Hom(\Phi(E)[q],\Phi(E_{ij})[2])& \cong \Hom(E[q],E_{ij}[2]) \cong
\Hom(E_{ij},E[q])^{\vee},\\
\Hom(\Phi(E)[q],\Phi({\cal E}_{|\{x' \} \times X})[2])
& \cong \Hom(E[q],{\cal E}_{|\{x' \} \times X}[2])
 \cong \Hom({\cal E}_{|\{x' \} \times X},E[q])^{\vee},
\end{split}
\end{equation}
we have ${^p H}^q(\Phi(E))=0$ for $q>0$, which implies that
$\Phi(E) \in \Per(X'/Y')$.
Thus the claim (1) holds.

(2)
\begin{NB}
Old (wrong) version:
We take a resolution of ${\cal O}_X$
\begin{equation}
0 \to V_{-2} \to V_{-1} \to V_0 \to {\cal O}_X \to 0
\end{equation}
such that $V_{-i}:={\cal O}_X(-n_i)^{\oplus N_i}$, $n_i \gg 0$
for $i=0,1$.
Then $V_{-i}^{\vee} \otimes E$, $i=0,1$ satisfy the conditions
in (1). It is easy to see that 
$V_{-2}^{\vee} \otimes E$ also satisfies the conditions
in (1). 
We set $W_i:=\Phi(V_{-i}^{\vee} \otimes E)$.
Then $W_i$ defines the required complex.  
\end{NB}

We take a resolution of $E$
\begin{equation}
0 \to V_{-2} \to V_{-1} \to V_0 \to E \to 0
\end{equation}
such that $V_{-k}=G(-n_k)^{\oplus N_k}$, $n_k \gg 0$
for $k=0,1$, where
$G$ is a local projective generator of ${\cal C}$.
By using the Serre duality, our choice of $n_k$ implies that
 $\Hom({\cal E}_{|\{x' \} \times X},V_{-k}[q])=
\Hom(E_{ij},V_{-k}[q])=0$
for $q \ne 2$ and 
$k=0,1$. 
Then we also have
$\Hom({\cal E}_{|\{x' \} \times X},V_{-2}[q])=\Hom(E_{ij},V_{-2}[q])=0$
for $q \ne 2$. 
Hence $\Phi(V_{-k})[2]$, $k=0,1,2 $ are
locally free sheaves on $X'$.
Since $\Hom(\Phi(V_{-k})[2],A_{ij}'[q])=
\Hom(\Phi(V_{-k})[2],\Phi(E_{ij})[2+q])=
\Hom(V_{-k},E_{ij}[q])=0$, $q>0$,
$W_{2-k}:=\Phi(V_{-k})[2]$, $k=0,1,2 $
are local projective objects of $\Per(X'/Y')$ and
the associated complex
$W_{\bullet}$ defines the required complex.  

(3) is obvious.
(4) follows from the proof of (1) and Lemma \ref{lem:vanish-Phi} (1). 
(5) follows from (3) and Lemma \ref{lem:vanish-Phi} (2).
\end{proof}

\begin{defn}
\begin{enumerate}
\item[(1)]
We set $\Phi^i(E):= {^p H^i}(\Phi(E)) \in \Per(X'/Y')$ and
$\widehat{\Phi}^i(E):={^p H^i}(\widehat{\Phi}(E)) \in {\cal C}$. 
\item[(2)]
We say that
$\WIT_i$ holds for $E \in {\cal C}$ (resp. $F \in \Per(X'/Y')$) 
with respect to $\Phi$ (resp. $\widehat{\Phi}$), if
$\Phi^j(E)=0$ (resp. $\widehat{\Phi}^j(F))=0$) for $j \ne i$.
\end{enumerate}
\end{defn}

\begin{lem}\label{lem:slope}
Let $E$ be an object of ${\cal C}$.
\begin{enumerate}
\item[(1)]
If $\WIT_0$ holds for $E$ with respect to $\Phi$,
then $E \in {\frak T}_1$.
\item[(2)]
If $\WIT_2$ holds for $E$ with respect to $\Phi$,
then $E \in {\frak F}_1$.
In particular, $E$ is torsion free.
Moreover if $\Phi^2(E)$ does not contain a 0-dimensional object,
then $E \in \overline{\frak F}_1^{\mu}$.
\end{enumerate}
\end{lem}

\begin{proof}
For an object $E \in {\cal C}$,
there is an exact sequence
\begin{equation}
0 \to E_1 \to E \to E_2 \to 0
\end{equation}
such that $E_1 \in {\frak T}_1$ and
$E_2 \in {\frak F}_1$.
Applying $\Phi$ to this exact sequence, 
we get a long exact sequence
\begin{equation}\label{eq:FM-Phi}
\begin{CD}
0 @>>> \Phi^0(E_1) @>>> 
\Phi^0(E) @>>> 
\Phi^0(E_2)\\
@>>>
\Phi^1(E_1) @>>> 
\Phi^1(E) @>>> 
\Phi^1(E_2)\\
@>>>
\Phi^2(E_1) @>>> 
\Phi^2(E) @>>> 
\Phi^2(E_2)@>>> 0.
\end{CD}
\end{equation}
\begin{NB}
Old version:
By our choice of $E_1, E_2$ and the $G_1$-twisted stability of 
${\cal E}_{|\{x' \} \times X}$, we see that 
\end{NB}
By Lemma \ref{lem:Perverse-1} (4),(5),
$\Phi^0(E_2)=\Phi^2(E_1)=0$.
If $\WIT_0$ holds for $E$, then we get $\Phi(E_2)=0$.
Hence (1) holds.
If $\WIT_2$ holds for $E$, 
then we get $\Phi(E_1)=0$.
Thus the first part of (2) holds.
Assume that there is an exact sequence
\begin{equation}
0 \to E_2' \to E \to E_2'' \to 0
\end{equation}
such that $E_2'$ is a $\mu$-semi-stable object with $\deg_{G_1}(E_2')=0$
and $E_2'' \in \overline{\frak F}_1^{\mu}$.
By the first part of (2), we get 
$\chi(G_1,E_2') \leq 0$.
By Lemma \ref{lem:vanish-Phi} (2),
$\Phi^0(E_2'')=0$.
Then we see that $\WIT_2$ holds for $E_2'$ and
$\deg_{G_2}(\Phi^2(E_2'))=\deg_{G_1}(E_2')=0$.
Since $\rk \Phi^2(E_2')=\chi(G_1,E_2') \leq 0$,
$\Phi^2(E_2')$ is a 0-dimensional object.
By our assumption, we get that
$\Phi^1(E_2'') \to \Phi^2(E_2')$ is an isomorphism.
By Lemma \ref{lem:spectral} in the appendix,
we have $\widehat{\Phi}^0(\Phi^1(E_2''))=0$, 
which implies that $E_2' \cong \widehat{\Phi}^0(\Phi^2(E_2'))=0$. 
\end{proof}

%\begin{lem}\label{lem:F=0-Phi}
%Let $F \in \Per(Y/W)$ be a $\mu$-semi-stable object with
%$\deg_{G_2}(F)=0$.
%If $\WIT_0$ holds for $F$ with respect to $\widehat{\Phi}$, 
%then $F=0$.  
%\end{lem}

\begin{lem}\label{lem:deg-Phi}
For an object $E \in {\cal C}$,
$\deg_{G_2}(\Phi^0(E)) \leq 0$ and 
$\deg_{G_2}(\Phi^2(E)) \geq 0$.
\end{lem}

\begin{proof}
We note that
\begin{equation}
\widehat{\Phi}(\Phi^0(E))=\widehat{\Phi}^2(\Phi^0(E))[-2],\;
\widehat{\Phi}(\Phi^2(E))=\widehat{\Phi}^0(\Phi^2(E))\;
\end{equation}
and
\begin{equation}
\deg_{G_2}(\Phi^0(E))=-\deg_{G_1}(\widehat{\Phi}^2(\Phi^0(E))), \;
\deg_{G_2}(\Phi^2(E))=-\deg_{G_1}(\widehat{\Phi}^0(\Phi^2(E))).
\end{equation} 
Since $\widehat{\Phi}^2(\Phi^0(E))$ satisfies $\WIT_0$ 
with respect to $\Phi$, $\widehat{\Phi}^2(\Phi^0(E)) \in {\frak T}_1$,
which implies that $\deg_{G_1}(\widehat{\Phi}^2(\Phi^0(E)))\geq 0$.
Since $\widehat{\Phi}^0(\Phi^2(E))$ satisfies $\WIT_2$ 
with respect to $\Phi$, $\widehat{\Phi}^0(\Phi^2(E)) \in {\frak F}_1$,
which implies that $\deg_{G_1}(\widehat{\Phi}^0(\Phi^2(E)))\leq 0$.
Therefore our claims hold.
\end{proof}

\begin{lem}\label{lem:widehat-Phi}
\begin{enumerate}
\item[(1)]
If $F \in {\frak T}_2^{\mu}$, then 
$\widehat{\Phi}^2(F)=0$.
\item[(2)]
If $\WIT_0$ holds for $F \in \Per(X'/Y')$ with respect to $\widehat{\Phi}$,
then $F \in {\frak T}_2^{\mu}$.
\item[(3)]
If $F \in {\frak F}_2^{\mu}$, then 
$\widehat{\Phi}^0(F)=0$.
\item[(4)]
If $\WIT_2$ holds for $F \in \Per(X'/Y')$ with respect to $\widehat{\Phi}$,
then $F \in {\frak F}_2^{\mu}$.
\end{enumerate}
\end{lem}

\begin{proof}
(1)
By Lemma \ref{lem:spectral} in the appendix,
we have an exact sequence
\begin{equation}
F \to \Phi^0(\widehat{\Phi}^2(F)) \overset{\phi}{\to}
 \Phi^2(\widehat{\Phi}^1(F))
\to 0.
\end{equation}
By Lemma \ref{lem:deg-Phi},
$\deg_{G_2}(\ker \phi) \leq 0$.
Since $\Phi^0(\widehat{\Phi}^2(F))$ is torsion free,
$\ker \phi$ is also torsion free.
By our assumption of $F$,
we have $\ker \phi=0$.
Then $\Phi^0(\widehat{\Phi}^2(F)) \cong
 \Phi^2(\widehat{\Phi}^1(F))$ satisfies $\WIT_0$ and
$\WIT_2$, which implies that
$\Phi^0(\widehat{\Phi}^2(F)) \cong
 \Phi^2(\widehat{\Phi}^1(F)) \cong 0$.
Therefore $\widehat{\Phi}^2(F)=0$.

(2)
Assume that there is an exact sequence
\begin{equation}
0 \to F_1 \to F \to F_2 \to 0
\end{equation}
such that $F_1 \in {\frak T}_2^{\mu}$ 
and $F_2 \in {\frak F}_2^{\mu}$.
By (1), we have $\widehat{\Phi}^2(F_1)=0$.
By a similar exact sequence to \eqref{eq:FM-Phi},
we see that $\WIT_0$ holds for $F_2$ and
$\deg_{G_1}(\widehat{\Phi}^0(F_2))=-\deg_{G_2}(F_2) \geq 0$.
On the other hand,
since $\WIT_2$ holds for $\widehat{\Phi}^0(F_2)$,
Lemma \ref{lem:slope} implies that 
$\widehat{\Phi}^0(F_2) \in {\frak F}_1$.
Hence $\deg_{G_1}(\widehat{\Phi}^0(F_2))=0$ and
$\chi(G_1,\widehat{\Phi}^0(F_2)) \leq 0$.
Since $\chi(G_1,\widehat{\Phi}^0(F_2))=\rk F_2$,
we have $\rk F_2=0$.
Since ${\frak F}_2^{\mu}$ contains no torsion object except
$0$, we conclude that $F_2=0$.  

(3)
By Lemma \ref{lem:spectral},
we have an exact sequence
\begin{equation}
0 \to \Phi^0(\widehat{\Phi}^1(F)) \overset{\psi}{\to}
\Phi^2(\widehat{\Phi}^0(F)) \to F.
\end{equation}
By (2), $\Phi^2(\widehat{\Phi}^0(F)) \in {\frak T}_2^{\mu}$, 
which implies that $\coker \psi=0$.
Then $\Phi^0(\widehat{\Phi}^1(F)) \cong
\Phi^2(\widehat{\Phi}^0(F))$ satisfies $\WIT_0$ and
$\WIT_2$, which implies that
$\Phi^0(\widehat{\Phi}^1(F)) \cong
\Phi^2(\widehat{\Phi}^0(F)) \cong 0$.
Therefore $\widehat{\Phi}^0(F)=0$.

(4)
Assume that there is an exact sequence
\begin{equation}
0 \to F_1 \to F \to F_2 \to 0
\end{equation}
such that $0 \ne F_1 \in {\frak T}_2^{\mu}$ 
and $F_2 \in {\frak F}_2^{\mu}$.
By (3), $\widehat{\Phi}^0(F_2)=0$.
By a similar exact sequence to \eqref{eq:FM-Phi},
we see that $\WIT_2$ holds for $F_1$ and
$\deg_{G_1}(\widehat{\Phi}^2(F_1))=-\deg_{G_2}(F_1) \leq 0$.
Moreover if $\rk F_1>0$, then 
$\deg_{G_1}(\widehat{\Phi}^2(F_1))<0$.
On the other hand,
since $\WIT_0$ holds for $\widehat{\Phi}^2(F_1)$,
Lemma \ref{lem:slope} implies that 
$\widehat{\Phi}^2(F_1) \in {\frak T}_1$.
Hence $\rk F_1=0$ and $\deg_{G_1}(\widehat{\Phi}^2(F_1))=0$.
Then 
$\widehat{\Phi}^2(F_1) \in {\frak T}_1$
implies that $0<\chi(G_1,\widehat{\Phi}^2(F_1))=\rk F_1$,
which is a contradiction.
Therefore $F_1=0$.
\end{proof}

\begin{lem}\label{lem:equiv-Phi}
\begin{enumerate}
\item[(1)]
Assume that $E \in {\frak T}_1$.
Then
\begin{enumerate}
\item
$\Phi^0(E)\in {\frak F}_2^{\mu}$.
\item
$\Phi^1(E)\in {\frak T}_2^{\mu}$.
\item
$\Phi^2(E)=0$. 
\end{enumerate}
\item[(2)]
Assume that $E \in {\frak F}_1$.
Then
\begin{enumerate}
\item
$\Phi^0(E)=0$.
\item
$\Phi^1(E)\in {\frak F}_2^{\mu}$. 
\item
$\Phi^2(E)\in {\frak T}_2^{\mu}$.
\end{enumerate}
\end{enumerate}
\end{lem}

\begin{proof}
We take a decomposition
\begin{equation}
0 \to F_1 \to \Phi^1(E) \to F_2 \to 0
\end{equation}
with
$F_1 \in {\frak T}_2^{\mu}$ and
$F_2 \in {\frak F}_2^{\mu}$.
Applying $\widehat{\Phi}$, we have an exact sequence
\begin{equation}\label{eq:FM-FM}
\begin{CD}
0 @>>> \widehat{\Phi}^0(F_1) @>>> 
\widehat{\Phi}^0(\Phi^1(E)) @>>> 
\widehat{\Phi}^0(F_2)\\
@>>>\widehat{\Phi}^1(F_1) @>>> 
\widehat{\Phi}^1(\Phi^1(E)) @>>> 
\widehat{\Phi}^1(F_2)\\
@>>> \widehat{\Phi}^2(F_1) @>>> 
\widehat{\Phi}^2(\Phi^1(E)) @>>> 
\widehat{\Phi}^2(F_2) @>>> 0.
\end{CD}
\end{equation}
By Lemma \ref{lem:widehat-Phi},
we have $\widehat{\Phi}^0(F_2)=\widehat{\Phi}^2(F_1)=0$.

(1) Assume that $E \in {\frak T}_1$.
Then (a) follows from Lemma \ref{lem:widehat-Phi} (4), and (c)
follows from Lemma \ref{lem:Perverse-1} (4).
We prove (b).
We assume that $F_2 \ne 0$.
By Lemma \ref{lem:spectral} and (c),
we have $\widehat{\Phi}^2(\Phi^1(E))=0$.
Then $\WIT_1$ holds for $F_2$ and
$\deg_{G_1}(\widehat{\Phi}^1(F_2))=\deg_{G_2}(F_2) \leq 0$.
By Lemma \ref{lem:spectral}, we have a surjective homomorphism
\begin{equation}
E \to \widehat{\Phi}^1(\Phi^1(E)).
\end{equation} 
Hence $\widehat{\Phi}^1(F_2)$ is a quotient object of $E$.
Since $E \in {\frak T}_1$, we see that
$\deg_{G_1}(\widehat{\Phi}^1(F_2)) \geq 0$.
Hence $\deg_{G_1}(\widehat{\Phi}^1(F_2))=0$.
If $\rk \widehat{\Phi}^1(F_2)>0$, 
then 
since $\chi(G_1,\widehat{\Phi}^1(F_2))=-\rk F_2<0$,
we get $E \not \in {\frak T}_1$.
Hence $\rk \widehat{\Phi}^1(F_2)=0$. Then
$\chi(G_1,\widehat{\Phi}^1(F_2))=-\rk F_2<0$ implies that
the $G_1$-twisted Hilbert polynomial of $\widehat{\Phi}^1(F_2)$
is not positive.
By Lemma \ref{lem:Hilbert-Poly}, this is impossible.
Therefore $F_2=0$.

(2)
Assume that $E \in {\frak F}_1$.
By Lemma \ref{lem:Perverse-1} and Lemma \ref{lem:widehat-Phi}, (a) and 
(c) hold. We prove (b).
Assume that $F_1 \ne 0$.
By $\Phi^0(E)=0$ and Lemma \ref{lem:spectral},
we have $\widehat{\Phi}^0(\Phi^1(E))=0$.
Then $\WIT_1$ holds for $F_1$
and we have an injective morphism
$\widehat{\Phi}^1(F_1) \to \widehat{\Phi}^1(\Phi^1(E)) \to E$.
Assume that $\dim F_1 \geq 1$. 
Since $\deg_{G_1}(\widehat{\Phi}^1(F_1))=\deg_{G_2}(F_1)> 0$, 
this is impossible.
Assume that $\dim F_1 =0$.
Then $\chi(G_2,F_1)>0$, which implies that
$\rk \widehat{\Phi}^1(F_1)=-\chi(G_2,F_1)<0$.
This is a contradiction.
Therefore $F_1=0$.
\end{proof}

The following is a generalization of
a result in \cite{H:category} (see Remark \ref{rem:equiv-Phi} below).

\begin{thm}\label{thm:equiv-Phi}
$\Phi$ induces an equivalence
${\frak A}_1 \to {\frak A}_2^{\mu}[-1]$.
Moreover $\widehat{\Phi}^0(F) \in \overline{\frak F}_1^{\mu}$
if $F \in {\frak T}_2^{\mu}$ does not contain
a 0-dimensional object.
\end{thm}

\begin{proof}
For $E \in {\frak A}_1$, we have an exact sequence in
${\frak A}_1$
\begin{equation}
0 \to H^{-1}(E)[1] \to E \to H^0(E) \to 0.
\end{equation}
Then we have an exact triangle
\begin{equation}
 \Phi(H^{-1}(E))[2] \to \Phi(E[1]) \to \Phi(H^0(E))[1]
\to \Phi(H^{-1}(E))[3].
\end{equation} 
Hence
$\Phi^i(E[1])=0$ for $i \ne -1,0$ and we have an exact sequence
\begin{equation}
\begin{CD}
0 @>>> \Phi^1(H^{-1}(E)) @>>> \Phi^{-1}(E[1]) @>>> \Phi^0(H^0(E))\\
@>>>  \Phi^2(H^{-1}(E)) @>>> \Phi^0(E[1]) @>>> \Phi^1(H^0(E)) @>>> 0.
\end{CD}
\end{equation}
By Lemme \ref{lem:equiv-Phi}, $\Phi^{-1}(E[1]) \in {\frak F}_2^{\mu}$ 
and $\Phi^0(E[1]) \in {\frak T}_2^{\mu}$. Therefore
$\Phi(E[1]) \in {\frak A}_2^{\mu}$.

Conversely for $F \in {\frak A}_2^{\mu}$ and
$E_1 \in {\frak A}_1$,
$\Phi(E_1)[1] \in {\frak A}_2^{\mu}$ implies that
\begin{equation}
\begin{split}
\Hom(\widehat{\Phi}(F)[1],E_1[p])&=
\Hom(F,(\Phi(E_1)[1])[p])=0,\;p<0,\\
\Hom(E_1[p],\widehat{\Phi}(F)[1])&=
\Hom((\Phi(E_1)[1])[p],F)=0,\;p>0.
\end{split}
\end{equation}
Hence $\widehat{\Phi}(F)[1] \in {\frak A}_1$.
Therefore the first claim holds.

For the last claim, we note that there is an exact sequence
\begin{equation}
0 \to \Phi^0(\widehat{\Phi}^1(F)) \to 
\Phi^2(\widehat{\Phi}^0(F)) \to F
\end{equation} 
by Lemma \ref{lem:spectral}.
By Lemma \ref{lem:Perverse-1} (3),
$\Phi^0(\widehat{\Phi}^1(F))$ is torsion free.
Hence $\Phi^2(\widehat{\Phi}^0(F))$ does not contain a 0-dimensional
object. 
Then Lemma \ref{lem:slope} (2) implies the
claim.
\end{proof}

\begin{rem}\label{rem:equiv-Phi}
In \cite{Y:Stability},
we gave a different proof of \cite[Prop. 4.2]{H:category}.
Since we used different notations in \cite{Y:Stability},
we explain the correspendence of the terminologies: 
$\Phi$ corresponds to ${\cal F}_{\cal E}$ in \cite{Y:Stability},
${\frak A}_2^{\mu}$ corresponds to ${\frak A}_1$ in
\cite[Thm. 2.1]{Y:Stability} and
${\frak A}_1$ corresponds to ${\frak A}_2$ or
${\frak A}_2'$ in \cite[Thm. 2.1, Prop. 2.7]{Y:Stability}.
\end{rem}

\subsection{Fourier-Mukai duality for a $K3$ surface.}

In this subsection, we shall prove a kind of duality property
between $(X,H)$ and $(X',\widehat{H})$.
In other words, we show that $X$ is the moduli space of some objects on $X'$
and $H$ is the natural determinant line bundle on the moduli space. 
\begin{thm}\label{thm:duality}
Assume that ${\Bbb C}_x$ is $\beta$-stable for all
$x \in X$.
%$\alpha' \in K(X') \otimes {\Bbb Q}$ such that
%$\chi(\alpha',\Psi({\Bbb C}_x))=0$ and
%$\chi(\alpha',\Psi(A_{ij}))<0$ for $j>0$.
\begin{enumerate}
\item[(1)]
${\cal E}_{|X' \times \{x \}} \in \Per(X'/Y')^D$ 
is $G_3-\Phi(\beta)^{\vee}$-twisted stable 
for all $x \in X$ and 
we have an isomorphism
$\phi:X \to M_{\widehat{H}}^{G_3-\Phi(\beta)^{\vee}}(w_0^{\vee})$ 
by sending $x \in X$
to ${\cal E}_{|X' \times \{x \}} \in 
M_{\widehat{H}}^{G_3-\Phi(\beta)^{\vee}}(w_0^{\vee})$.
Moreover we have $H=\widehat{(\widehat{H})}$ under this isomorphism.
\item[(2)]
Assume that ${\cal E}_{|\{x' \} \times X}$ is a $\mu$-stable 
local projective generator of ${\cal C}$ for a general $x' \in X'$.
Then ${\cal E}_{|X' \times \{x \}}$ is a $\mu$-stable local projective
generator of $\Per(X'/Y')^D$ for 
$x \in X \setminus \cup_i Z_i$.
\end{enumerate}
\end{thm}

The proof is similar to
that in \cite[Thm. 2.2]{Y:Stability}.
In particular, if ${\cal E}_{|\{ x' \} \times X}$ is a
$\mu$-stable locally free sheaf for a general $x' \in X'$, 
then the same proof in \cite{Y:Stability} works.
However if ${\cal E}_{|\{ x' \} \times X}$ is not
a $\mu$-stable locally free sheaf for any $x' \in X'$, 
then we need to introduce a (contravariant) Fourier-Mukai 
transforms and study their properties.
We set
\begin{equation}
\begin{split}
\Psi(E):=&{\bf R}\Hom_{p_{X'}}(p_X^*(E),{\cal E})
=\Phi(E)^{\vee}[-2],\; E \in {\bf D}(X),\\
\widehat{\Psi}(F):=&{\bf R}\Hom_{p_X}(p_{X'}^*(F),{\cal E}),\;
F \in {\bf D}(X').
\end{split}
\end{equation}
We shall first study the properties of
$\Psi$ and $\widehat{\Psi}$ which are similar to
those of $\Phi$ and $\widehat{\Phi}$.

We set
\begin{equation}
\begin{split}
\Psi(E_{ij})[2]& =B_{ij}',\; j>0\\
\Psi(E_{i0})[2]& =B_{i0}'.
\end{split}
\end{equation}
Then the following claims follow from 
Definition \ref{defn:K3:Per^D} and Lemma \ref{lem:A_*}.
\begin{lem}\label{lem:K3:irreducible-objD}
\begin{enumerate}
\item[(1)]
 $B_{ij}'={\cal O}_{C_{ij}'}(-b_{ij}'-2) \in \Per(X'/Y')^D$
and $B_{i0}' =A_0(-{\bf b}'+2{\bf b}_0)^*[1]\in \Per(X'/Y')^D$.
\item[(2)]
Irreducible objects of $\Per(X'/Y')^D$ are
\begin{equation}
B_{ij}' \;(1 \leq i \leq m,0 \leq j \leq s_i'),\;
{\Bbb C}_{x'} (x' \in X \setminus \cup_i Z_i').
\end{equation}
\end{enumerate}
\end{lem}
\begin{NB}
Not correct:
Since
\begin{equation}
\Hom({\cal E}_{|X' \times \{x \}},A_0(-{\bf b}'+2{\bf b}_0)^*)
=\Hom({\cal E}_{|X' \times \{x \}},B_{i0}'[-1])
=\Hom(E_{i0},{\Bbb C}_x[-1])=0,
\end{equation}
we have ${\cal E}_{|X' \times \{x \}} \in \Per(X'/Y')^D$.

If ${\cal E}_{|X' \times \{ x\}} \in \Coh(X')$, then
the argument is OK, but is not trivial. 
\end{NB}

\begin{lem}\label{lem:vanish}
\begin{enumerate}
\item[(1)]
Assume that $E \in \overline{\frak T}_1$.
Then $\Hom(E,{\cal E}_{|\{x' \} \times X})=0$
for a general $x' \in X'$.
\item[(2)]
Assume that $E \in \overline{\frak F}_1$.
Then $\Hom({\cal E}_{|\{x' \} \times X},E)=0$
for all $x' \in X'$.
\end{enumerate}
\end{lem}

\begin{proof}
We only prove (1).
Let $E$ be a $G_1$-twisted stable object of ${\cal C}$.
If $\deg_{G_1}(E)>0$ or $\deg_{G_1}(E)=0$ and
$\chi(G_1,E)>0$, then 
$\Hom(E,{\cal E}_{|\{x' \} \times X})=0$ for all $x' \in X'$.
\begin{NB}
If $E$ is a torsion object, then
$\deg_{G_1}(E)>0$ or $\deg_{G_1}(E)=0$ and
$\chi(G_1,E)>0$.
\end{NB}
Assume that $\deg_{G_1}(E)=0$ and
$\chi(G_1,E)=0$.
Then a non-zero homomorphism
$E \to {\cal E}_{|\{x' \} \times X}$ is an isomorphism if 
$x' \not \in \cup_i Z_i'$. 
Therefore $\Hom(E,{\cal E}_{|\{x' \} \times X})=0$ for a general 
$x' \in X'$.
\end{proof}

\begin{lem}\label{lem:Perverse}
Let $E$ be an object of ${\cal C}$.
\begin{enumerate}
\item[(1)]
$^p H^i(\Psi(E))=0$ for $i \geq 3$.
\item[(2)]
$H^0(^p H^2(\Psi(E)))=H^2(\Psi(E))$.
\item[(3)]
$^p H^0(\Psi(E)) \subset H^0(\Psi(E))$.
In particular, $^p H^0(\Psi(E))$ is torsion free. 
\item[(4)]
If $\Hom(E,E_{ij}[2])=0$ for all $i,j$
and $\Hom(E,{\cal E}_{|\{x' \} \times X}[2])=0$ for all $x' \in X'$, 
then
$^p H^2(\Psi(E))=0$.
In particular, if $E \in \overline{\frak F}_1$,
then $^p H^2(\Psi(E))=0$.
\item[(5)]
If $E$ satisfies $E \in \overline{\frak T}_1$, then
$^p H^0(\Psi(E))=0$.
 
%If $\Hom(E,{\cal E}_{|\{y \} \times X})=0$ for a general $y \in Y$,
%then $^p H^0(\Psi(E))=0$.
%In particular, if $E \in {\frak T}^*$, then $^p H^0(\Psi(E))=0$.
\end{enumerate}
\end{lem}

\begin{proof}
Let $W_{\bullet}$ be the complex in Lemma \ref{lem:Perverse-1} (2).
By Remark \ref{rem:tilting:dual},
$W_i^{\vee}$ are local projective objects of
$\Per(X'/Y')^D$.
Since $\Psi(E)$ is represented by the
complex $W_{\bullet}^{\vee}[-2]$,
 (1), (2) and (3) follow.

By Lemma \ref{lem:K3:irreducible-objD},
$F \in \Per(X'/Y')^D$ is 0 if and only if  
$\Hom(F, B_{ij}')=\Hom(F,{\Bbb C}_{x'})=0$
for all $i$, $j$ and $x' \in X'$. 

Since
\begin{equation}
\begin{split}
\Hom(E,E_{ij}[2-p])^{\vee} & \cong \Hom(\Psi(E)[2-p],\Psi(E_{ij})[2]),\\
\Hom(E,{\cal E}_{|\{x' \} \times X}[2-p])^{\vee}
& \cong \Hom(\Psi(E)[2-p],\Psi({\cal E}_{|\{x' \} \times X})[2]),
\end{split}
\end{equation}
we have (4).
%(3) is obvious.
(5) follows from (3) and Lemma \ref{lem:vanish} (1).
\end{proof}

\begin{defn}
We set $\Psi^i(E):= {^p H^i}(\Psi(E)) \in \Per(X'/Y')^D$ and
$\widehat{\Psi}^i(E):={^p H^i}(\widehat{\Psi}(E)) \in {\cal C}$. 
\end{defn}

\begin{lem}\label{lem:slope3}
Let $E$ be an object of ${\cal C}$.
\begin{enumerate}
\item[(1)]
If $\WIT_0$ holds for $E$ with respect to $\Psi$,
then $E \in \overline{\frak F}_1$.
\item[(2)]
If $\WIT_2$ holds for $E$ with respect to $\Psi$,
then $E \in \overline{\frak T}_1$.
If $\Psi^2(E)$ does not contain a $0$-dimensional object, then
$E \in {\frak T}_1$.
\end{enumerate}
\end{lem}

\begin{proof}
For an object $E$ of ${\cal C}$, 
there is an exact sequence
\begin{equation}
0 \to E_1 \to E \to E_2 \to 0
\end{equation}
such that $E_1 \in \overline{\frak T}_1$ and
$E_2 \in \overline{\frak F}_1$.
Applying $\Psi$ to this exact sequence, 
we get a long exact sequence
\begin{equation}\label{eq:FM-Psi}
\begin{CD}
0 @>>> \Psi^0(E_2) @>>> 
\Psi^0(E) @>>> 
\Psi^0(E_1)\\
@>>>
\Psi^1(E_2) @>>> 
\Psi^1(E) @>>> 
\Psi^1(E_1)\\
@>>>
\Psi^2(E_2) @>>> 
\Psi^2(E) @>>> 
\Psi^2(E_1)@>>> 0
\end{CD}
\end{equation}
By Lemma \ref{lem:Perverse},
we have
$\Psi^0(E_1)=\Psi^2(E_2)=0$.
If $\WIT_0$ holds for $E$, then we get $\Psi(E_1)=0$.
Hence (1) holds.
If $\WIT_2$ holds for $E$, 
then we get $\Psi(E_2)=0$.
Thus the first part of (2) holds.
Assume that $\Psi^2(E)$ does not have a non-zero 
0-dimensional subobject.
We take a decomposition
\begin{equation}
0 \to E_1 \to E \to E_2 \to 0
\end{equation}
such that $E_1 \in {\frak T}_1$ and
$E_2$ is a $G_1$-twisted semi-stable object with
$\deg_{G_1}(E_2)=\chi(G_1,E_2)=0$.
Then $\Psi^0(E_1)=\Psi^0(E_2)=\Psi^1(E_2)=0$.
In particular, $\WIT_2$ holds for $E_2$ with respect to
$\Psi$. Then 
$\Psi^2(E_2)$ is a torsion object with 
$\deg_{G_3}(\Psi^2(E_2))=0$, which implies that
$\Psi^2(E_2)$ is 0-dimensional.
Our assumption implies that
$\Psi^1(E_1) \cong \Psi^2(E_2)$.
By Lemma \ref{lem:spectral2} and
$\widehat{\Psi}^0(\Psi^0(E_1))=0$, we get
$E_2=\widehat{\Psi}^2(\Psi^2(E_2))=\widehat{\Psi}^2(\Psi^1(E_1))=0$.
\end{proof}

\begin{lem}\label{lem:E=0}
Let $E$ be a $\mu$-semi-stable object with
$\deg_{G_1}(E)=0$.
If $\WIT_0$ holds for $E$, then $E=0$.  
\end{lem}

\begin{proof}
If $\WIT_0$ holds for $E \ne 0$, then 
$\chi(G_1,E)=\rk \Psi(E) \geq 0$.
On the other hand,
Lemma \ref{lem:slope3} implies that $\chi(G_1,E)<0$.
Therefore $E=0$.
\end{proof}

\begin{lem}\label{lem:slope4}
If $\WIT_0$ holds for $E$ with respect to $\Psi$,
then $E \in \overline{\frak F}_1^{\mu}$.
\end{lem}

\begin{proof}
Assume that there is an exact sequence
\begin{equation}
0 \to E_1 \to E \to E_2 \to 0
\end{equation}
such that $E_1$ is a $\mu$-semi-stable object
with $\deg_{G_1}(E_1)=0$ and $E_2 \in \overline{\frak F}_1^{\mu}$.
Then we have $\Psi^2(E_2)=0$.
By the exact sequence \eqref{eq:FM-Psi},
$\WIT_0$ holds for $E_1$.
Then Lemma \ref{lem:E=0} implies that 
$E_1=0$. 
\end{proof}

\begin{lem}\label{lem:Psi}
If $E \in \overline{\frak T}_1^{\mu}$,
then $\Psi^0(E)=0$.
\end{lem}

\begin{proof}
We may assume that $E$ is a $\mu$-semi-stable object or a torsion object.
If $\deg_{G_1}(E)>0$, then the claim holds by the base change theorem.
Assume that $\deg_{G_1}(E)=0$.
By Lemma \ref{lem:spectral2},
we have an exact sequence
\begin{equation}
E \to \widehat{\Psi}^0(\Psi^0(E)) \to \widehat{\Psi}^2(\Psi^1(E))
\to 0.
\end{equation}
By Lemma \ref{lem:slope4},
$\widehat{\Psi}^0(\Psi^0(E)) \in \overline{\frak F}_1^{\mu}$.
Since $E$ is a $\mu$-semi-stable object with $\deg_{G_1}(E)=0$,
$E \to \widehat{\Psi}^0(\Psi^0(E))$ is a zero map.
Then $\widehat{\Psi}^0(\Psi^0(E)) \cong \widehat{\Psi}^2(\Psi^1(E))$
satisfies $\WIT_0$ and $\WIT_2$,
which implies that 
$\widehat{\Psi}^0(\Psi^0(E)) \cong \widehat{\Psi}^2(\Psi^1(E)) \cong 0$.
Therefore $\Psi^0(E)=0$.
\end{proof}

\begin{lem}\label{lem:deg-Psi}
\begin{equation}
\deg_{G_3}(\Psi^0(E)) \leq 0,\,
\deg_{G_3}(\Psi^2(E)) \geq 0.
\end{equation}
\end{lem}

\begin{proof}
We note that
\begin{equation}
\deg_{G_3}(\Psi^i(E))=\deg_{G_1}(\widehat{\Psi}^i(\Psi^i(E)))
\end{equation}
for $i=0,2$ by Lemma \ref{lem:spectral2}.
Then the claim follows from Lemma \ref{lem:slope3}.
\end{proof}

\begin{NB}
\begin{lem}
If $\WIT_0$ holds for $E$ with respect to $\Psi$,
then $E \in \overline{\frak F}_1^{\mu}$.

If $E \in \overline{\frak T}_1^{\mu}$, then 
$\Psi^0(E)=0$.
\end{lem}

\begin{proof}

We take a decomposition
\begin{equation}
0 \to E_1 \to E \to E_2 \to 0
\end{equation}
such that $E_1 \in \overline{\frak T}_1^{\mu}$ and
$E_2 \in \overline{\frak F}_1^{\mu}$.
By Lemma \ref{lem:Perverse},
$\Psi^2(E_2)=0$.
Hence $\WIT_0$ holds for $E_1$.
Then $E_1 \in \overline{\frak F}_1 \cap 
\overline{\frak T}_1^{\mu}$ implies that
$\deg_{G_1}(E_1)=0$ and $\chi(G_1,E_1) \leq 0$.
Since $\Psi^0(E_1)$ is torsion free, $\Psi^0(E_1)=0$.
Therefore $E_1=0$.

By Lemma \ref{lem:spectral2},
we have an exact sequence
\begin{equation}
E \to \widehat{\Psi}^0(\Psi^0(E)) \to  \widehat{\Psi}^2(\Psi^1(E))
\to 0. 
\end{equation}
Since $\WIT_0$ holds for $\widehat{\Psi}^0(\Psi^0(E))$,
$\widehat{\Psi}^0(\Psi^0(E)) \in \overline{\frak F}_1^{\mu}$,
which implies that 
$\widehat{\Psi}^0(\Psi^0(E)) \to  \widehat{\Psi}^2(\Psi^1(E))$ is an
isomorphism.
Then we have 
$\widehat{\Psi}^0(\Psi^0(E)) \cong  \widehat{\Psi}^2(\Psi^1(E))
\cong 0$, which implies that $\Psi^0(E)=0$.
\end{proof}
\end{NB}
%\begin{cor}
%If $\deg_{G_1}(\widehat{\Psi}^0(E))=0$, then
%$\widehat{\Psi}^0(E)=0$.
%\end{cor}
%
%\begin{proof}
%If $\deg_{G_1}(\widehat{\Psi}^0(E))=0$, 
%then Lemma \ref{lem:slope3} implies that
%$\widehat{\Psi}^0(E)$ is $\mu$-semi-stable sheaf, which implies the
%claim.
%\end{proof}

%\begin{NB}
%$\Hom(E_i,E[p])=\Hom(\Psi(E)[-p],\Psi(E_i))=
%\Hom(\Psi(E)[-p],{\cal O}_{C_i}(-1)[-2])$ for $i>0$.
%$\Hom(E_0,E)=\Hom(\Psi(E),\Psi(E_0))=
%\Hom(\Psi(E),{\cal O}_Z(Z)[1])$.
%$\Hom(\Psi(E),{\cal O}_{C_i}(-1)[-2])$.
%$\Hom(\Psi(E)[p],{\cal O}_{C_i}(-1))=\Hom(\Psi(E)[p],
%\Hom(\Psi(E)[p],{\cal O}_{Z}(Z))=0$ for all $i$.
%\end{NB}

{\it Proof of Theorem \ref{thm:duality}.}

(1) We first prove the $G_3$-twisted semi-stability
of ${\cal E}_{|X' \times \{ x \}}$ for all $x \in X$.
It is sufficient to prove the following lemma.
\begin{lem}\label{lem:WIT-simple}
Let $E$ be a 0-dimensional object of ${\cal C}$.
Then $\WIT_2$ holds for $E$ with respect to $\Psi$
and $\Psi^2(E)$ is a $G_3$-twisted semi-stable object
such that $\deg_{G_3}(\Psi^2(E))=
\chi(G_3,\Psi^2(E))=0$.
Moreover if $E$ is irreducible, then $\Psi^2(E)$ is $G_3$-twisted stable.
\end{lem}

\begin{proof}
We first prove that $E$ satisfies $\WIT_2$ with respect to
$\Psi$. We may assume that $E$ is irreducible.
Then we get $\Hom(E,{\cal E}_{|\{x' \} \times X})=0$ for all $x'$.
Hence $\Psi^0(E)=0$.
We shall prove that $\Psi^1(E)=0$ by showing
$\widehat{\Psi}^i(\Psi^1(E))=0$
for $i=0,1,2$. 
By Lemma \ref{lem:spectral2},
$\widehat{\Psi}^2(\Psi^1(E))=0$ and
we have an exact sequence
\begin{equation}
0 \to \widehat{\Psi}^0(\Psi^1(E)) \to
\widehat{\Psi}^2(\Psi^2(E)) \to E \to \widehat{\Psi}^1(\Psi^1(E)) \to 0.
\end{equation}
By Lemma \ref{lem:slope3} and Lemma \ref{lem:spectral2},
$\widehat{\Psi}^0(\Psi^1(E)) \in \overline{\frak F}_1$ and
$\widehat{\Psi}^2(\Psi^2(E)) \in \overline{\frak T}_1$.
Since $E$ is 0-dimensional,
$\widehat{\Psi}^0(\Psi^1(E))$ is $\mu$-semi-stable and
$\deg_{G_1}(\widehat{\Psi}^0(\Psi^1(E)))=
\deg_{G_1}(\widehat{\Psi}^2(\Psi^2(E)))=0$.
By Lemma \ref{lem:E=0},
$\widehat{\Psi}^0(\Psi^1(E))=0$.
Since $E$ is an irreducible object,
$\widehat{\Psi}^2(\Psi^2(E))=0$ or 
$\widehat{\Psi}^1(\Psi^1(E))=0$.
If $\widehat{\Psi}^2(\Psi^2(E))=0$, then 
$\Psi^2(E)=0$. Since $\chi(G_1,E)>0$, we get a contradiction.
Hence we also have
$\widehat{\Psi}^1(\Psi^1(E))=0$, which implies that
$\Psi^1(E)=0$. Therefore $\WIT_2$ holds for $E$ 
with respect to $\Psi$.

We next prove that $\Psi^2(E)$ is $G_3$-twisted semi-stable.
Assume that there is an exact sequence
\begin{equation}\label{eq:K3:destabilize}
0 \to F_1 \to \Psi^2(E) \to F_2 \to 0
\end{equation}
such that $F_1 \in \Per(X'/Y')^D$,
$\deg_{G_3}(F_1) \geq 0$ and
$F_2 \in \Per(X'/Y')^D$.
\begin{NB}
$\rk F_1$ may be 0.
\end{NB}
Applying $\widehat{\Psi}$ to this exact sequence, 
we get a long exact sequence
\begin{equation}\label{eq:Huybrechts}
\begin{CD}
0 @>>> \widehat{\Psi}^0(F_2) @>>> 
0 @>>> 
\widehat{\Psi}^0(F_1)\\
@>>>\widehat{\Psi}^1(F_2) @>>> 
0 @>>> 
\widehat{\Psi}^1(F_1)\\
@>>> \widehat{\Psi}^2(F_2) @>>> 
E @>>> 
\widehat{\Psi}^2(F_1) @>>> 0.
\end{CD}
\end{equation}
By Lemma \ref{lem:spectral2},
$\WIT_2$ holds for $\widehat{\Psi}^2(F_2)$.
Hence $\widehat{\Psi}^2(F_2) \in \overline{\frak T}_1$, 
in particular, we have 
$\deg_{G_1}(\widehat{\Psi}^2(F_2)) \geq 0$.
By Lemma \ref{lem:spectral2},
$\WIT_0$ holds for $\widehat{\Psi}^1(F_2) \cong 
\widehat{\Psi}^0(F_1)$. Hence 
$\widehat{\Psi}^1(F_2) \in \overline{\frak F}_1$,
which implies that 
$\deg_{G_1}(\widehat{\Psi}^1(F_2)) \leq 0$.
Therefore 
$\deg_{G_1}(\widehat{\Psi}(F_2))\geq 0$.
On the other hand, $\deg_{G_1}(\widehat{\Psi}(F_2))
=\deg_{G_3}(F_2) \leq 0$.
Hence $\widehat{\Psi}^1(F_2)$ is a $\mu$-semi-stable object
with $\deg_{G_1}(\widehat{\Psi}^1(F_2))=0$
and
$\deg_{G_3}(F_2)=0$.
Then Lemma \ref{lem:E=0} implies that
$\widehat{\Psi}^1(F_2)=0$.
\begin{NB}
Lemma \ref{lem:slope4} implies that 
$\deg_{G_1}(\widehat{\Psi}^1(F_2))<0$ or $\widehat{\Psi}^1(F_2)=0$.
Therefore 
$\deg_{G_1}(\widehat{\Psi}(F_2))\geq 0$.
On the other hand, $\deg_{G_1}(\widehat{\Psi}(F_2))
=\deg_{G_3}(F_2) \leq 0$.
Hence $\widehat{\Psi}^1(F_2)=0$
and $\deg_{G_1}(\widehat{\Psi}^2(F_2))=\deg_{G_3}(F_2)=0$.
\end{NB}
If $\chi(G_3,F_2) \leq 0$, then 
 $\rk \widehat{\Psi}^2(F_2)=
\chi(G_3,F_2)$ implies that
$\chi(G_3,F_2)=0$ and 
$\widehat{\Psi}^2(F_2)$ is a torsion object.
This in particular means that
$\Psi^2(E)$ is $G_2$-twisted semi-stable.
We further assume that $E$ is irreducible. 
Since $\deg_{G_1}(\widehat{\Psi}^2(F_2))=0$,
$\widehat{\Psi}^2(F_2)$ is a 0-dimensional object.
Then $\WIT_2$ holds for $\widehat{\Psi}^1(F_1)$,
$\widehat{\Psi}^2(F_1)$ and 
$\widehat{\Psi}^2(F_2)$ with respect to
$\Psi$.
Since $\Psi^2(\widehat{\Psi}^1(F_1))=0$,
$\widehat{\Psi}^1(F_1)=0$.
\begin{NB}
Use $\widehat{\Psi}^0(F_1)=\widehat{\Psi}^1(F_2)=0$
and the spectral sequence to show
$\Psi^2(\widehat{\Psi}^1(F_1))=0$.
\end{NB}
Then $\widehat{\Psi}^2(F_2)=0$ or $\widehat{\Psi}^2(F_1)=0$,
which implies that $F_1=0$ or $F_2=0$.
Therefore $\Psi^2(E)$ is $G_3$-twisted stable.
\end{proof}
We continue the proof of (1).
Assume that 
there is an exact sequence in $\Per(X'/Y')^D$
\begin{equation}
0 \to F_1 \to {\cal E}_{|X' \times \{x\}} \to F_2 \to 0
\end{equation}
such that $\deg_{G_3}(F_1)=\chi(G_3,F_1)=0$.
By the proof of Lemma \ref{lem:WIT-simple},
$\WIT_2$ holds for $F_1$ and $F_2$.
Thus we get an exact sequence
\begin{equation}
0 \to 
\widehat{\Psi}^2(F_2) \to 
{\Bbb C}_x \to
\widehat{\Psi}^2(F_1) \to 0 
\end{equation}
\begin{NB}
$\rk F_1$ may be 0.
But since there is a local projective generator of
${\frak C}_3$ whose Mukai vector is $2v(G_3)$,
it is impossible.
\end{NB}
Since ${\Bbb C}_x$ is $\beta$-stable,
$\chi(\beta,\widehat{\Psi}^2(F_2))<0$,
which implies that $\chi(-\Psi(\beta),F_2)>0$.
\begin{NB}
Since $\Hom(A_{ij},{\Bbb C}_x)=0$,
$\Hom(A_{ij},\widehat{\Psi}^2(F_2))=0$ for all $i$ and $j>0$.
Thus $\Hom(A_{i0},\widehat{\Psi}^2(F_2)) \ne 0$.
Then
$v(\widehat{\Psi}^2(F_1))=\sum_{j>0} n_j v(A_{ij})$,
$n_j \geq 0$.
Hence we have $\chi(\alpha',F_1)<0$.
\end{NB}
Therefore 
${\cal E}_{|X' \times \{x \}}$ is $G_3-\Psi(\beta)$-twisted stable.
Then we have an injective morphism 
$\phi:X \to \overline{M}_{\widehat{H}}^{G_3+\alpha'}(w_0^{\vee})$
by sending $x \in X$ to
${\cal E}_{|X' \times \{x \}}$,
where $\alpha'=-\Psi(\beta)$.
By a standard argument, we see that $\phi$ is an isomorphism. 
We note that
$[\widehat{\Psi}(\widehat{H}+(\widehat{H},
\widetilde{\xi}_0)/r_0 \varrho_{X'})]_1$ is the pull-back of 
the canonical polarization
on $\overline{M}_{\widehat{H}}^{G_3}(w_0^{\vee})$.
Hence under the identification 
$M_{\widehat{H}}^{G_3 +\alpha'}(w_0^{\vee}) \cong X$,
$\widehat{(\widehat{H})}=H$.

(2)
Assume that ${\cal E}_{|\{x' \} \times X}$ is a $\mu$-stable
local projective generator for a general $x' \in X'$.
By Lemma \ref{lem:free} (2) below, we only need to prove the 
$\mu$-stability of
${\cal E}_{|X' \times \{x \}}$ for
 $x  \in X \setminus \cup_i Z_i$.
We shall study the exact sequence \eqref{eq:K3:destabilize}
in Lemma \ref{lem:WIT-simple}, where $E={\Bbb C}_x$.
We may assume that $F_2$ 
satisfies $\deg_{G_3}(F_2)=0$ and
$\chi(G_3,F_2)>0$. 
Then $\WIT_2$ holds for $F_2$ by the proof of Lemma 
\ref{lem:WIT-simple}. 
We shall first prove that $\widehat{\Psi}^1(F_1)$ does not contain
a 0-dimensional object.
Let $T_1$ be the 0-dimensional subobject of
$\widehat{\Psi}^1(F_1)$.
Then we have a surjective morphism
$\Psi^2(\widehat{\Psi}^1(F_1)) \to \Psi^2(T_1)$.
Since $\WIT_2$ holds for $T_1$ with respect to 
$\Psi$ and 
$\Psi^0(\widehat{\Psi}^0(F_1)) \to \Psi^2(\widehat{\Psi}^1(F_1))$ is
surjective, we get $T_1=0$.
\begin{NB}
Use
$\widehat{\Psi}^0(F_1)=\widehat{\Psi}^1(F_2)=0$.
\end{NB}
\begin{NB} 
Since ${\Bbb C}_x$ is a simple object,
$\widehat{\Psi}^2(F_2) \to {\Bbb C}_x$ is surjective.
Then $\widehat{\Psi}^2(F_1)=0$.
By Lemma \ref{lem:spectral2},
$\Psi^0(\widehat{\Psi}^1(F_1))=0$.
Thus $\WIT_1$ holds for $F_1$.
\end{NB}
By Lemma \ref{lem:slope3}, 
$\widehat{\Psi}^2(F_2) \in \overline{\frak T}_1$.
Then Lemma \ref{lem:generic-mu-stable} 
and $\deg_{G_1}(\widehat{\Psi}^2(F_2))=0$ imply that
$\widehat{\Psi}^2(F_2)$ 
is an extension of a 
$G_1$-semi-stable object $E_1$ with $\deg_{G_1}(E_1)=\chi(G_1,E_1)=0$ by
a 0-dimensional object $T$.
Since $T \cap \widehat{\Psi}^1(F_1)=0$,
$T={\Bbb C}_x$ or $0$.
By our assumption, $\Psi^2(E_1)$ is a torsion object.
By the exact sequence
\begin{equation}
\Psi^2(E_1) \to F_2 \to \Psi^2(T) \to 0,
\end{equation}
we have 
$\rk F_2=(\rk {\cal E}_{|X' \times \{x \}})\dim T$, which implies that
$\rk F_2=\rk {\cal E}_{|X' \times \{x \}}$ or
$\rk F_2=0$.
Therefore ${\cal E}_{|X' \times \{x \}}$ is $\mu$-stable.
\qed

\begin{lem}\label{lem:dual-generator}
If ${\cal E}_{|\{x' \} \times X}, x' \in X'$ and 
$E_{ij}$ are locally free
on an open subset $X^0$ of $X$, then
${\cal E}_{|X' \times \{x \}}$ is a local projective generator
of $\Per(X'/Y')^D$ for $x \in X^0$.
\end{lem}

\begin{proof}
We first note that ${\cal E}_{|X' \times \{ x \}}
\in \Coh(X')$ by Theorem \ref{thm:duality}.
The claim follows from the following equalities:
\begin{equation}
\begin{split}
\Hom({\cal E}_{|X' \times \{x \}},{\Bbb C}_{x'}[k])
&=\Hom(\Psi({\Bbb C}_x),\Psi({\cal E}_{|\{x' \} \times X})[k])=
\Hom({\cal E}_{|\{x' \} \times X},{\Bbb C}_x[k])=0,\\
\Hom({\cal E}_{|X' \times \{x \}},B_{ij}'[k])
&=\Hom(\Psi({\Bbb C}_x),\Psi(E_{ij})[k])=
\Hom(E_{ij},{\Bbb C}_x[k])=0
\end{split}
\end{equation}
for $x \in X^0$, $x' \in X'$ and
$k \ne 0$.
\end{proof}

\begin{lem}\label{lem:free}
\begin{enumerate}
\item[(1)]
If $X=Y$ and $Y'$ is not smooth, then 
${\cal E}_{|X' \times \{ x \}}$
is a local projective generator of $\Per(X'/Y')^D$ 
for all $x \in X$.
\item[(2)]
If ${\cal E}_{|\{ x' \} \times X}$ is a 
$\mu$-stable local projective object of 
${\cal C}$ for a general $x' \in X'$,
then ${\cal E}_{|X' \times \{x \}}$ is a local projective generator
of $\Per(X'/Y')^D$ for all $x \in X$.
\end{enumerate}
\end{lem}

\begin{proof}
(1)
We first note that $E_{ij} \in \Coh(X)={\cal C}$
are locally free sheaves for all $i,j$. 
Assume that $E:={\cal E}_{|\{x' \} \times X}$ is not locally free
for a point $x' \in X'$.
Then we have a morphism from an open subscheme $Q$ of 
$\Quot_{E^{\vee \vee}/X/{\Bbb C}}^n$ to $X'$,
where $n=\dim(E^{\vee \vee}/E)$.
Since $\dim X' =2$, this morphism is dominant.
Hence ${\cal E}_{|\{x' \} \times X}$ is non-locally free for all
$x' \in X'$.
Since ${\cal E}_{|\{x' \} \times X}$ is locally free 
if $x'$ belongs to the exceptional locus, 
${\cal E}_{|\{x' \} \times X}$ is locally free for any $x' \in X'$.
Then the claim follows from Lemma \ref{lem:dual-generator}.

(2)
The claim follows from Lemma \ref{lem:generic-mu-stable} (3), (4)
and the proof of Lemma \ref{lem:dual-generator}.
\end{proof}

In the remaining of this subsection, we shall prove the following result.
\begin{prop}\label{prop:equiv-Psi}
$\Psi:{\bf D}(X) \to {\bf D}(X')_{op}$ induces an equivalence
$\overline{\frak A}_1^{\mu}[-2] \to (\overline{\frak A}_3)_{op}$.
\end{prop}

We first note that the following two lemmas 
hold thanks to
Theorem \ref{thm:duality}.

\begin{lem}[cf. Lem. \ref{lem:vanish}]\label{lem:vanish2}
\begin{enumerate}
\item[(1)]
Assume that $F \in \overline{\frak T}_3$.
Then $\Hom(F,{\cal E}_{| X' \times \{x\}})=0$
for a general $x \in X$.
In particular, 
$\widehat{\Psi}^0(F)=0$.
\item[(2)]
Assume that $F \in \overline{\frak F}_3$.
Then $\Hom({\cal E}_{|X' \times \{x\}},F)=0$ for all
$x \in X$.
In particular, $\widehat{\Psi}^2(F)=0$.
\end{enumerate}
\end{lem}

\begin{NB}
\begin{proof}
We only prove (1).
Let $F$ be a $G_1$-twisted stable object of ${\cal C}$.
If $F \in {\frak T}_2$, then
$\Hom(F,{\cal E}_{|X' \times \{x \}})=0$ for all $x \in X$.
Assume that $\deg_{G_1}(F)=0$ and
$\chi(G_1,F)=0$.
Then $F \to {\cal E}_{|X' \times \{x \}}$ is an isomorphism. 
Therefore $\Hom(F,{\cal E}_{|X' \times \{x \}})=0$ for a general $x \in X$.
\end{proof}
\end{NB}
%Then by using Lemma \ref{lem:vanish2}, we have the following three lemmas.
\begin{lem}[cf. Lem. \ref{lem:slope3}, Lem. \ref{lem:slope4}] 
\label{lem:slope5}
Let $F$ be an object of $\Per(X'/Y')^D$.
\begin{enumerate}
\item[(1)]
If $\WIT_0$ holds for $F$ with respect to $\widehat{\Psi}$,
then $F \in \overline{\frak F}_3^{\mu} (\subset \overline{\frak F}_3)$.
\item[(2)]
If $\WIT_2$ holds for $F$ with respect to $\widehat{\Psi}$,
then $F \in \overline{\frak T}_3$.
If $\widehat{\Psi}^2(F)$ does not contain a 0-dimensional subobject, then
$F \in {\frak T}_3$.
\end{enumerate}
\end{lem}

\begin{NB}
\begin{proof}
For an object $F$ of $\Per(X'/Y')^*$, 
there is an exact sequence
\begin{equation}
0 \to F_1 \to F \to F_2 \to 0
\end{equation}
such that $F_1 \in \overline{\frak T}_3$ and
$F_2 \in \overline{\frak F}_3$.
Applying $\widehat{\Psi}$ to this exact sequence, 
we get a long exact sequence
\begin{equation}\label{eq:FM2}
\begin{CD}
0 @>>> \widehat{\Psi}^0(F_2) @>>> 
\widehat{\Psi}^0(F) @>>> 
\widehat{\Psi}^0(F_1)\\
@>>>
\widehat{\Psi}^1(F_2) @>>> 
\widehat{\Psi}^1(F) @>>> 
\widehat{\Psi}^1(F_1)\\
@>>>
\widehat{\Psi}^2(F_2) @>>> 
\widehat{\Psi}^2(F) @>>> 
\widehat{\Psi}^2(F_1)@>>> 0
\end{CD}
\end{equation}
By Lemma \ref{lem:vanish2},
we have
$\widehat{\Psi}^0(F_1)=\widehat{\Psi}^2(F_2)=0$.
If $\WIT_0$ holds for $F$, 
then we get $\widehat{\Psi}(F_1)=0$.
Hence (1) holds.
If $\WIT_2$ holds for $F$, 
then we get $\widehat{\Psi}(F_2)=0$.
Thus the first part of (2) holds.

We take a decomposition
\begin{equation}
0 \to F_1' \to F \to F_1'' \to 0
\end{equation}
such that $F_1' \in {\frak T}_3$ and
$F_1''$ is a $G_3$-twisted semi-stable
object with $\chi(G_3,F_1'')=0$.
Then $\widehat{\Psi}^0(F_1')=0$ and $\WIT_2$ holds for $F_1''$.
By our assumption,
$\widehat{\Psi}^2(F_1'')$ is a torsion object.
Since $\deg_{G_1}(\widehat{\Psi}^2(F_1''))=0$,
$\widehat{\Psi}^1(F_1') \to \widehat{\Psi}^2(F_1'')$ is an
isomorphism.
Then we see that $F_1''=0$ by using 
$\Psi^2(\widehat{\Psi}^1(F_1'))=0$.
\end{proof}
\end{NB}

\begin{NB}
\begin{lem}\label{lem:F=0}
Let $F$ be a $\mu$-semi-stable object of $\Per(X'/Y')^D$ with
$\deg_{G_3}(F)=0$.
If $\WIT_0$ holds for $F$, then $F=0$.  
\end{lem}

\begin{proof}
If $\WIT_0$ holds for $F \ne 0$, then 
$\chi(G_3,F)=\rk \widehat{\Psi}(F) \geq 0$.
On the other hand,
Lemma \ref{lem:slope5} implies that $\chi(G_3,F)<0$.
Therefore $F=0$.
\end{proof}
\end{NB}

\begin{NB}
\begin{lem}\label{lem:slope6}
If $\WIT_0$ holds for $F$ with respect to $\widehat{\Psi}$,
then $F \in \overline{\frak F}_3^{\mu}$.
\end{lem}

\begin{proof}
Assume that there is an exact sequence
\begin{equation}
0 \to F_1 \to F \to F_2 \to 0
\end{equation}
such that $F_1$ is a $\mu$-semi-stable sheaf
with $\deg_{G_3}(F_1)=0$ and $F_2 \in \overline{\frak F}_3^{\mu}$.
Then we have $\widehat{\Psi}^2(F_2)=0$.
By the exact sequence \eqref{eq:FM2},
$\WIT_0$ holds for $F_1$.
Then Lemma \ref{lem:F=0} implies that 
$F_1=0$. 
\end{proof}

\end{NB}

\begin{NB}
\begin{lem}\label{lem:Psihat}
If $F \in ({\frak T}_2^{\mu})^*$,
then $\widehat{\Psi}^0(F)=0$.
\end{lem}

\begin{proof}
By Lemma \ref{lem:spectral2},
we have an exact sequence
\begin{equation}
F \to \Psi^0(\widehat{\Psi}^0(F)) \to \Psi^2(\widehat{\Psi}^1(F))
\to 0.
\end{equation}
By Lemma \ref{lem:slope6},
$\mu_{\max,G_3}(\Psi^0(\widehat{\Psi}^0(F)))<0$.
Since $F$ is a $\mu$-semi-stable object of $\deg_{G_3}(F)=0$,
$F \to \Psi^0(\widehat{\Psi}^0(F))$ is a zero map.
Then $\Psi^0(\widehat{\Psi}^0(F)) \to \Psi^2(\widehat{\Psi}^1(F))$
satisfies $\WIT_0$ and $\WIT_2$,
which implies that 
$\Psi^0(\widehat{\Psi}^0(F)) \cong \Psi^2(\widehat{\Psi}^1(F))\cong 0$.
Therefore $\widehat{\Psi}^0(F)=0$.
\end{proof}
\end{NB}

%\begin{cor}
%If $\deg_{G_3}(\Psi^0(E))=0$, then
%$\Psi^0(E)=0$.
%\end{cor}
%
%\begin{proof}
%If $\deg_{G_1}(\widehat{\Psi}^0(E))=0$, 
%then Lemma \ref{lem:slope3} implies that
%$\widehat{\Psi}^0(E)$ is $\mu$-semi-stable sheaf, which implies the
%claim.
%\end{proof}

\begin{lem}\label{lem:equiv}
\begin{enumerate}
\item[(1)]
Assume that $E \in \overline{\frak T}_1^{\mu}$.
Then
\begin{enumerate}
\item
$\Psi^0(E)=0$.
\item
$\Psi^1(E) \in \overline{\frak F}_3$.
\item
$\Psi^2(E) \in \overline{\frak T}_3$. 
Moreover if $E$ does not contain a non-trivial
0-dimensional subobject, then $\Psi^2(E) \in {\frak T}_3$.
\end{enumerate}
\item[(2)]
Assume that $E \in \overline{\frak F}_1^{\mu}$.
Then
\begin{enumerate}
\item
$\Psi^0(E) \in \overline{\frak F}_3$.
\item
$\Psi^1(E) \in \overline{\frak T}_3$. 
\item
$\Psi^2(E)=0$.
\end{enumerate}
\end{enumerate}
\end{lem}

\begin{proof}
We take a decomposition
\begin{equation}
0 \to F_1 \to \Psi^1(E) \to F_2 \to 0
\end{equation}
with
$F_1 \in \overline{\frak T}_3$ and
$F_2 \in \overline{\frak F}_3$.
Applying $\widehat{\Psi}$, we have an exact sequence
\begin{equation}\label{eq:FM-FM-Psi}
\begin{CD}
0 @>>> \widehat{\Psi}^0(F_2) @>>> 
\widehat{\Psi}^0(\Psi^1(E)) @>>> 
\widehat{\Psi}^0(F_1)\\
@>>>\widehat{\Psi}^1(F_2) @>>> 
\widehat{\Psi}^1(\Psi^1(E)) @>>> 
\widehat{\Psi}^1(F_1)\\
@>>> \widehat{\Psi}^2(F_2) @>>> 
\widehat{\Psi}^2(\Psi^1(E)) @>>> 
\widehat{\Psi}^2(F_1) @>>> 0.
\end{CD}
\end{equation}
By Lemma \ref{lem:vanish2},
we have $\widehat{\Psi}^0(F_1)=\widehat{\Psi}^2(F_2)=0$.

(1) Assume that $\deg_{\min, G_1}(E) \geq 0$.
By Lemma \ref{lem:slope5} (2) and Lemma \ref{lem:Psi},
(a) and the first claim of (c) hold. 
For the second claim of (c), by Lemma \ref{lem:slope5} (2),
it is sufficient to prove that
$\widehat{\Psi}^2(\Psi^2(E))$ does not contain a non-trivial
0-dimensional subobject.
By the exact sequence
\begin{equation}
0 \to \widehat{\Psi}^0(\Psi^1(E)) \to \widehat{\Psi}^2(\Psi^2(E))
\to E
\end{equation}
and the torsion-freeness of $\widehat{\Psi}^0(\Psi^1(E))$,
we get our claim.

We prove (b).
By Lemma \ref{lem:spectral2} and (a),
we have $\widehat{\Psi}^2(\Psi^1(E))=0$.
Then $\WIT_1$ holds for $F_1$.
We have a surjective homomorphism
\begin{equation}
E \to \widehat{\Psi}^1(\Psi^1(E)).
\end{equation} 
Hence $E$ has a quotient sheaf $\widehat{\Psi}^1(F_1)$
with $\deg_{G_1}(\widehat{\Psi}^1(F_1))=
-\deg_{G_3}(F_1) \leq 0$.
If $\deg_{G_1}(\widehat{\Psi}^1(F_1))<0$, then
we see that $\rk \widehat{\Psi}^1(F_1)>0$ and 
$E \not \in \overline{\frak T}_1^{\mu}$.
Hence $\deg_{G_1}(\widehat{\Psi}^1(F_1))=
-\deg_{G_3}(F_1)=0$.
Then $F_1 \in \overline{\frak T}_3$ implies that
$\rk \widehat{\Psi}^1(F_1)=-\chi(G_3,F_1) \leq 0$.
Since $\chi(G_1,\widehat{\Psi}^1(F_1))=-\rk F_1 \leq 0$,
the $G_1$-twisted Hilbert polynomial of $\widehat{\Psi}^1(F_1)$
is 0. Therefore $F_1=0$.

(2)
Assume that $\deg_{\max, G_1}(E) < 0$.
By Lemma \ref{lem:Perverse} and Lemma \ref{lem:slope5}, (a) and 
(c) hold. We prove (b).
Since $\Psi^2(E)=0$,
Lemma \ref{lem:spectral2} implies that 
$\widehat{\Psi}^0 \Psi^1(E)=0$.
Hence $\WIT_1$ holds for $F_2$
and we have an injective morphism
$\widehat{\Psi}^1(F_2) \to \widehat{\Psi}^1(\Psi^1(E)) \to E$.
Since $\deg_{G_1}(\widehat{\Psi}^1(F_2)) \geq 0$, we have 
$\widehat{\Psi}^1(F_2)=0$, which implies that $F_2=0$.
\end{proof}

{\it Proof of Proposition \ref{prop:equiv-Psi}}.

For $E \in \overline{\frak A}_1^{\mu}$, we have an exact sequence in
$\overline{\frak A}_1^{\mu}$
\begin{equation}
0 \to H^{-1}(E)[1] \to E \to H^0(E) \to 0.
\end{equation}
Then we have an exact triangle
\begin{equation}
\Psi(H^0(E))[2] \to \Psi(E[-2]) \to \Psi(H^{-1}(E))[1]
\to \Psi(H^0(E))[3].
\end{equation} 
Hence
$\Psi^i(E[-2])=0$ for $i \ne -1,0$ and we have an exact sequence
\begin{equation}
\begin{CD}
0 @>>> \Psi^1(H^0(E)) @>>> \Psi^{-1}(E[-2]) @>>> \Psi^0(H^{-1}(E))\\
@>>>  \Psi^2(H^0(E)) @>>> \Psi^0(E[-2]) @>>> \Psi^1(H^{-1}(E)) @>>> 0.
\end{CD}
\end{equation}
By Lemme \ref{lem:equiv}, $\Psi^{-1}(E[-2]) \in \overline{\frak F}_3$ 
and $\Psi^0(E[-2]) \in \overline{\frak T}_3$. Therefore
$\Psi(E[-2]) \in (\overline{\frak A}_3)_{op}$.
\qed

\begin{defn}
\begin{enumerate}
\item[(1)]
Let
$\Per(X'/Y')_{w_0^{\vee}}^D$ be the full subcategory of
$\Per(X'/Y')^D$ consisting of $G_3$-twisted semi-stable
objects $E$ with $\deg_{G_3}(E)=\chi(G_3,E)=0$.
\item[(2)]
Let ${\cal C}_0$ (resp. $\Per(X'/Y')_0^D$) 
be the full subcategory of 
${\cal C}$ (resp. $\Per(X'/Y')^D$)
consisting of 0-dimensional objects.
\end{enumerate}
\end{defn}

\begin{prop}\label{prop:K3:equiv-Psi}
$\Psi$ induces the following correspondences:
\begin{equation}
\begin{split}
{\cal C}_0 \cong & (\Per(X'/Y')^D_{w_0^{\vee}})_{op},\\
{\cal C}_{v_0} \cong & (\Per(X'/Y')^D_0)_{op}.
\end{split}
\end{equation}
\end{prop}

\begin{proof}
By Lemma \ref{lem:WIT-simple},
$\Psi^2({\cal C}_0)$ is contained in 
$(\Per(X'/Y')^D_{w_0^{\vee}})_{op}$.
It is easy to see that 
$\Per(X'/Y')^D_{w_0^{\vee}}$ is generated by
$\Psi^2(A_{ij}), i,j \geq 0$ and
$\Psi^2({\Bbb C}_x)$, $x \in X \setminus \cup_i Z_i$.
Thus the first claim holds.

We have an equivalence
\begin{equation}
\begin{matrix}
\Per(X'/Y')_0 & \to & (\Per(X'/Y')_0^D)_{op}\\
E & \mapsto & {\bf R}{\cal H}om_{{\cal O}_X}(E,{\cal O}_X)[2].
\end{matrix}
\end{equation}
Then the second claim is a consequence of Proposition
\ref{prop:K3-normal} (1).
\end{proof}

\begin{NB}
\begin{equation}
\{  E\in {\cal C}| \dim E=0 \}
\longleftrightarrow
\left\{ E \in \Per(X'/Y')^D \left|
\begin{aligned}
\text{$E$ is $G_3$-twisted semi-stable,}\\
\text{$\deg_{G_3}(E)=\chi(G_3,E)=0$}
\end{aligned}
\right. \right\}
 \end{equation}

 \begin{equation}
\left\{ E \in {\cal C} \left|
\begin{aligned}
\text{$E$ is $G_1$-twisted semi-stable,}\\
\text{$\deg_{G_1}(E)=\chi(G_1,E)=0$}
\end{aligned}
\right. \right\}
\longleftrightarrow
\{  E\in \Per(X'/Y')^D| \dim E=0 \}
 \end{equation} 
\end{NB}

\begin{NB}
\begin{prop}\label{prop:duality}
Let ${\cal C}$ be the category in Lemma \ref{lem:tilting}.
Assume that there is a $\beta \in \varrho_X^{\perp}$
such that ${\Bbb C}_x$ are $\beta$-stable for all $x \in X$.
Let $v_0$ be a primitive isotropic Mukai vector of $X$ and
assume that there is a local projective generator $G$ of ${\cal C}$
such that $v(G)=2v_0$.
Then the same assertion of Theorem \ref{thm:duality} holds.
\end{prop} 

\begin{proof}
For a suitable $\alpha$,
there are ${\bf b}_1,...,{\bf b}_n$ and
we have an equivalence 
$\Lambda^\alpha:{\cal C} \to \Per(X/Y,{\bf b}_1,...,{\bf b}_n)$
by Proposition \ref{prop:0-dim:equivalence}.
By Proposition \ref{prop:Phi-alpha} (6), we have an isomorphism
$$
M_H^{G+\gamma}(v_0) \to 
M_H^{G^\alpha+\Lambda^\alpha(\gamma)}(\Lambda^\alpha(v_0)),
$$
where $\gamma \in \delta(H^{\perp})$ is a sufficiently small
general element.
By Proposition \ref{prop:K3:smooth}, they are $K3$ surfaces.
Let ${\cal E}$ be the universal family on $X \times X'$,
where $X':=M_H^{G+\gamma}(v_0)$.
We shall prove that ${\cal E}_{|\{x \} \times X'}$ is
$G_3+\gamma'$-twisted stable for all
$x \in X$, where $\gamma':=-\Psi(\Lambda^\alpha(\beta))$.
 
We first note that ${\cal E}^\alpha:=\Lambda^\alpha({\cal E})$
is the universal family on
$X^\alpha \times X'$, where
$X'$ is identified with 
$M_H^{G^\alpha+\Lambda^\alpha(\gamma)}(\Lambda^\alpha(v_0))$.
For this ${\cal E}^\alpha$, we consider the equivalence
$\Psi:=\Phi_{X^\alpha \to X'}^{{\cal E}^\alpha} \circ D_{X^\alpha}$.
We note that
\begin{equation}
\Phi_{X^\alpha \to X'}^{{\cal E}^\alpha} \circ D_{X^\alpha} 
\circ \Lambda^\alpha
=D_{X'} \circ \Phi_{X^\alpha \to X'}^{({\cal E}^{\alpha})^{\vee}[2]}
\circ \Lambda^\alpha
=D_{X'} \circ \Phi_{X \to X'}^{{\cal E}^{\vee}[2]}
=\Phi_{X \to X'}^{{\cal E}} \circ D_X.
\end{equation}
Since $\Lambda^\alpha({\Bbb C}_x)$, $x \in X$ are 
$\Lambda^\alpha(\beta)$-stable for all $x \in X$,
Proposition \ref{prop:K3:equiv-Psi} implies that
$\Psi(\Lambda^\alpha({\Bbb C}_x))[2]=
\Phi_{X \to X'}^{{\cal E}}({\Bbb C}_x)$ are
$G_3 +\gamma'$-stable for all $x \in X$.
\end{proof}
\end{NB}

\subsection{Preservation of Gieseker stability conditions.}

\begin{prop}\label{prop:stability-deg=0}
Let $E$ be a $G_1$-twisted semi-stable object with
$\deg_{G_1}(E)=0$ and $\chi(G_1,E)<0$.
Then $\WIT_1$ holds for $E$ and
$\Psi^1(E)$ is $G_3$-twisted semi-stable.
In particular, we have an isomorphism 
\begin{equation}
{\cal M}_H^{G_1}(v)^{ss} \to 
{\cal M}_{\widehat{H}}^{G_3}(-\Psi(v))^{ss}
\end{equation}
which preserves the $S$-equivalence classes, 
where $v=l v_0+a\varrho_X+(D+(D/r_0,\xi_0)\varrho_X)$, $l>0$, $a<0$.
\end{prop}

\begin{proof}
We note that $E \in \overline{\frak F}_1 \cap \overline{\frak T}_1^{\mu}$.
By Lemma \ref{lem:Perverse} and Lemma \ref{lem:equiv},
$\WIT_1$ holds for $E$ and
$\Psi^1(E) \in \overline{\frak F}_3$.
Assume that $\Psi^1(E)$ is not $G_3$-twisted stable. Then
there is an exact sequence in $\Per(X'/Y')^D$
\begin{equation}
0 \to F_1 \to \Psi^1(E) \to F_2 \to 0
\end{equation}
such that $F_1$ is a $G_3$-twisted stable object with
$\deg_{G_3}(F_1)=0$ and 
\begin{equation}
0 > \frac{\chi(G_3,F_1)}{\rk F_1}
 \geq 
\frac{\chi(G_3,\Psi^1(E))}{\rk \Psi^1(E)},
\end{equation}
and $F_2 \in \overline{\frak F}_3$.
Then we have an exact sequence
\begin{equation}
0 \to \widehat{\Psi}^1(F_2) \to E \to \widehat{\Psi}^1(F_1) \to 0.
\end{equation}
Since 
\begin{equation}
\begin{split}
\frac{\chi(G_1,\widehat{\Psi}^1(F_1))}{\rk(\widehat{\Psi}^1(F_1))}
&=\frac{ \rk F_1}{\chi(G_3,F_1)}\\
& \leq \frac{\rk \Psi^1(E)}{\chi(G_3,\Psi^1(E))}=
\frac{\chi(G_1,E)}{\rk E},
\end{split}
\end{equation}
we have
\begin{equation}
\frac{\chi(G_3,F_1)}{\rk F_1}
= \frac{\chi(G_3,\Psi^1(E))}{\rk \Psi^1(E)}.
\end{equation}
Hence $\Psi^1(E)$ is $G_3$-twisted semi-stable.
Thus we have a morphism
${\cal M}_H^{G_1}(v)^{ss} \to 
{\cal M}_{\widehat{H}}^{G_3}(-\Psi(v))^{ss}$.
It is easy to see that this morphism preserves the $S$-equivalence classes.
By the symmetry of the conditions,
we have the inverse morphism, which shows the second claim.
  \end{proof}

The following is a generalization of \cite[Thm. 1.7]{Y:Stability}.
\begin{prop}\label{prop:stability-asymptotic}
Let $w \in v({\bf D}(X'))$ be a Mukai vector such that
$\langle w^2 \rangle  \geq -2$ and
\begin{equation}
w=l w_0+a \varrho_{X'}+
\left(d\widehat{H}+\widehat{D}+
\frac{1}{r_0}(d\widehat{H}+\widehat{D},\xi_0)\varrho_{X'} \right),
\end{equation}
where $l \geq 0$, $a>0$ and 
$D \in \NS(X) \otimes {\Bbb Q} \cap H^{\perp}$.
Assume that
\begin{equation}\label{eq:degree-condition}
\begin{split}
d>\max \{(4l^2r_0^3+1/(H^2)),
2r_0^2 l(\langle w^2 \rangle-(D^2))\}, \text{ if $l>0$},\\
a > \max \{(2r_0+1),(\langle w^2 \rangle-(D^2))/2+1\},
\text{ if $l=0$}.
\end{split}
\end{equation} 
\begin{NB}
The condition does not depend on the structure of $\NS(X)$,
and depends only on $w,H,w_0$.

\end{NB}
Then
\begin{enumerate}
\item[(1)]
${\cal M}_H^{G_1}(\widehat{\Phi}(w))^{ss} \cong
{\cal M}_{\widehat{H}}^{G_2}(w)^{ss}$.
\item[(2)]
${\cal M}_H^{G_1}(\widehat{\Phi}(w))^{ss}$ consists of 
local projective generators.
\item[(3)]
If 
$(\widehat{H},G_2)$ is general with respect to
$w$, then
${\cal M}_H^{G_1}(\widehat{\Phi}(w))^{ss} \cong
{\cal M}_{H+\epsilon}^{G_1}(\widehat{\Phi}(w))^{ss}$
for a sufficiently small relatively ample divisor $\epsilon$. 
\end{enumerate}
\end{prop}

\begin{proof}
(1)
We first note that ${\cal F}_{\cal E}$ in \cite{Y:Stability} 
corresponds to
$\widehat{\Phi}$.
Since \cite[Thm. 2.1, Thm. 2.2]{Y:Stability} are replaced by 
Theorem \ref{thm:equiv-Phi}, \ref{thm:duality} and since
\cite[Prop. 2.8, Prop. 2.11]{Y:Stability} also hold for our case,
the same proof of \cite[Thm. 1.7]{Y:Stability}
works for our case.
More precisely, 
in order to show that $\Phi(F), F \in {\cal M}_H^{G_1}(\widehat{\Phi}(w))$
does not contain a 0-dimensional subobject, we use the fact that
$\WIT_0$ holds for 0-dimensional object $E \in \Per(X'/Y')$
(see Proposition \ref{prop:K3-normal} (1)).

(2)
The proof is the same as in the proof of \cite[Rem. 2.3]{Y:Stability}.
Let $E$ be a $\mu$-semi-stable object
of ${\cal C}$ such that $v(E)=\widehat{\Phi}(w)$.
If $\Ext^1(S,E) \ne 0$ for an irreducible object
$S$ of ${\cal C}$, then
a non-trivial extension
\begin{equation}
0 \to E \to E' \to S \to 0
\end{equation}
gives a $\mu$-semi-stable object $E'$ with
$\chi(G_1,E')>\chi(G_1,E)$.
By Proposition \cite[Prop. 2.8, Prop. 2.11]{Y:Stability},
we get a contradiction.
Hence $\Ext^1(E,S) \cong \Ext^1(S,E)^{\vee} = 0$ 
for any irreducible object
$S$ of ${\cal C}$.
Since $\Ext^2(E,S) \cong \Hom(S,E)^{\vee}=0$,
it is sufficient to prove that
$\chi(S,E)>0$.
We note that 
$\chi(S,E)=\chi(S,\widehat{\Phi}(w))=a\chi(S,G_1)+(c_1(S),D)$.
Since $(H,c_1(S))=0$,
we have $|(c_1(S),D)^2| \leq |(c_1(S)^2)(D^2)|=-2(D^2)$.
Since $\chi(S,G_1)>0$,
it is sufficient to prove that $a > \sqrt{-2(D^2)}$.

We first assume that $l>0$.
Then $d(H^2)-1>4l^2 r_0^3 (H^2)$ and
$d>2r_0^2 l(\langle w^2 \rangle-(D^2))=2r_0^2 l(d^2(H^2)-2la r_0)$. 
Hence 
\begin{equation}
a> \frac{d(d(H^2)-1/(2r_0^2 l))}{2r_0 l}
>\frac{d}{2lr_0} 4l^2 r_0^3 (H^2)=2dlr_0^2 (H^2).
\end{equation}
Hence 
$a>2(4l^2 r_0^3) lr_0^2 (H^2)=8r_0 (lr_0)^3 r_0(H^2) \geq 8$.
If $-(D^2) \leq 4$, then $a>3>\sqrt{-2(D^2)}$. 
If $-(D^2) >4$, then 
$\langle w^2 \rangle-(D^2) \geq -2-(D^2) >-(D^2)/2$.
Hence 
\begin{equation}
a>2dlr_0^2 (H^2)>
r_0(\langle w^2 \rangle-(D^2))4(l r_0)^2 r_0(H^2)
>\sqrt{-2(D^2)}.
\end{equation}

We next assume that $l=0$.
Then $a>2r_0+1$ and
$a>\langle w^2 \rangle/2+1-(D^2)/2 \geq -(D^2)/2$.
If $-(D^2) \geq 8$,
then $a>-(D^2)/2 \geq \sqrt{-2(D^2)}$.
If $-(D^2)< 8$, then
since $a \geq 2r_0+1+1/r_0$,
$\sqrt{-2(D^2)}<4 \leq a$.

Therefore $\chi(E,S)>0$ and $E$ is a local projective
generator of ${\cal C}$.

(3)
By our assumption, ${\cal M}_H^{G_1}(\widehat{\Phi}(w))^{ss}
={\cal M}_H^{G_1}(\widehat{\Phi}(w))^{\mu \text{-}ss}$
(\cite[Cor. 2.14]{Y:Stability}) and
$H$ is a general polarization.  
Hence for $E \in {\cal M}_H^{G_1}(\widehat{\Phi}(w))^{ss}$
and a subobject $E_1$ of $E$,
$\frac{(c_1(E),H)}{\rk E}=\frac{(c_1(E_1),H)}{\rk E_1}$
implies 
$\frac{c_1(E)}{\rk E}=\frac{c_1(E_1)}{\rk E_1}$.
Let $E$ be a $\mu$-semi-stable sheaf of $v(E)=\widehat{\Phi}(w)$
with respect to $H$.
We shall prove that $E \in {\cal C}$.
We set 
$$
\Sigma:=\{A_{ij}[-1]|i,j\} \cap \Coh(X)
$$
as in Proposition \ref{prop:tilting:S-T}.
We assume that $\Hom(E,F) \ne 0$ for $F \in \Sigma$.
Then there is a $\mu$-semi-stable sheaf $E' \in {\cal C} \cap \Coh(X)$ 
with respect to $H$
fitting in an exact sequence
\begin{equation}
0 \to E' \to E \to F' \to 0,
\end{equation}
where $F' \in {\cal C}[-1] \cap \Coh(X)$.
\begin{NB}
${\cal C}[-1] \cap \Coh(X)=S$ and ${\cal C} \cap \Coh(X)=T$.
\end{NB}
Then we see that $\chi(G_1,E')>\chi(G,E)$, which is a contradiction.
Therefore $E \in {\cal C}$.
Then we can easily see that $E$ is $\mu$-semi-stable in ${\cal C}$.
\begin{NB}
$E \in {\cal C}$ is $\mu$-semi-stable if and only if
$$
\frac{(c_1(F),H)}{\rk F} \leq \frac{(c_1(E),H)}{\rk E}, \rk F>0
$$
in $\Coh(X)$.
Indeed for a subsheaf $F \in \Coh(X)$ of $E$, 
there is an subsheaf 
$F' \in {\cal C}$ with $\dim (F/F')=0$ and
for any subobject $F$ of $E$, $H^0(F)$ is a subsheaf
of $E$ with $(c_1(H^0(F)),H)=(c_1(F),H)$. 
\end{NB}
\end{proof}

\begin{cor}\label{cor:reduction}
If $(G,H)$ is general with respect to $v$, then
$M_H^{G}(v)$ is isomorphic to the moduli space of usual
stable sheaves on a $K3$ surface.
\end{cor}

\begin{proof}
We first construct a primitive and isotropic Mukai vector
$u$ such that $\rk u>0$ and $(\rk G) c_1(u)-(\rk u) c_1(G^{\vee})
\in {\Bbb Z}H$:
We first take 
a primitive isotropic Mukai vector $t$
such that $t=lv(G^{\vee})+a\varrho_X$.
Then for a sufficiently small $\tau$, 
$T:=M_H^{G^{\vee}+\tau}(t)$ is a $K3$ surface.
Let ${\cal F}$ be the universal family on $T \times X$ 
as a twisted object.
Then we have an equivalence
$\Phi_{X \to T}^{{\cal F}^{\vee}}:{\bf D}(X) \to {\bf D}^{\beta}(T)$.
We consider $\Pi:=\Phi_{T \to X}^{{\cal F}(nD)}
\circ \Phi_{X \to T}^{{\cal F}^{\vee}}:
{\bf D}(X) \to {\bf D}(X)$, $n \gg 0$, where we set $D:=\widehat{H}$.
Then  
$\Pi$ also induces a Hodge isometry $\Pi:H^*(X,{\Bbb Z}) \to
H^*(X,{\Bbb Z})$.
By its construction, $\Pi$ preserves the subspace 
$({\Bbb Q}t+{\Bbb Q}H+{\Bbb Q}\varrho_X) \cap H^*(X,{\Bbb Z})$ and
$\rk \Pi(\varrho_X)>0$ for $n \gg 0$.   
Hence $u:=\Pi(\varrho_X)$ satisfies the claim.
Since $c_1(u)/\rk u-c_1(G^{\vee})/\rk G^{\vee} \in {\Bbb Q}H$,
$\chi(u,A_{ij}^{\vee}[2])/\rk u=\chi(G^{\vee},A_{ij}^{\vee}[2])/\rk G$. 
By Corollary
\ref{cor:generator-exist}, there is a local projective generator
$G_u$ of ${\cal C}^D$ with $v(G_u)=2u$.
Since $\langle \Pi({\cal O}_X),u \rangle=-1$,
$X_1:=M_H^{u+\alpha}(u)$ is a fine moduli space
of stable objects of ${\cal C}^D$.
Since ${\cal C}$ satisfies Assumption \ref{ass:stability},
${\cal C}^D$ also satisfies Assumption \ref{ass:stability}.
\begin{NB}
$A_{ij}^{\vee}[2]$ are irreducible objects of ${\cal C}^D$
and ${\Bbb C}_x \cong {\Bbb C}_x^{\vee}[2] \in {\cal C}^D$ are
$\beta$-stable.
\end{NB}
Let ${\cal E}$ be the universal family on
$X \times X_1$.
By Theorem \ref{thm:duality}, we can regard ${\cal E}$ as a universal family
of $v_0+\gamma$-twisted stable objects of $\Per(X_1/Y_1)^D$
with respect to $H_1$, where 
$Y_1:=\overline{M}_H^u(u)$,
 $H_1:=\widehat{H}$,
$v_0=v({\cal E}_{|\{x \} \times X_1})$
and $\gamma$ is determined by $\alpha$.
Then $(M_{H_1}^{v_0+\gamma}(v_0),\widehat{H}_1)=(X,H)$.
For $\widehat{\Phi}=\Phi_{X \to X_1}^{{\cal E}}$ and
${\cal M}_H^{u^{\vee}}(v e^{mH})^{ss}$, $m \gg 0$,
we shall apply Proposition \ref{prop:stability-asymptotic}.
Then ${\cal M}_H^{u^{\vee}}(v)^{ss}$ is isomorphic to a 
moduli stack of usual semi-stable sheaves on $X_1$.
Since ${\cal M}_H^{u^{\vee}}(v)^{ss}={\cal M}_H^{G}(v)^{ss}$,
we get our claim.
\end{proof}

\begin{NB}
Let $(X,H)$ be a pair of a $K3$ surface $X$
and a nef and big divisor $H$ on $X$.
Then $|nH|$, $n \gg 0$ gives a birational map $\pi:X \to Y$
to a normal surface $Y$ contracting $(-2)$-curves
in $H^{\perp}$.
We shall show that there is a polarized $K3$ surface $(X',H')$
and a Mukai vector $w$ such that 
$(X,H)=(M_{H'}^{w+\alpha}(w),\widehat{H'})$,
where $\alpha$ is sufficiently small.  
Let $v$ be a primitive isotropic Mukai vector
and $G$ a local projective generator of ${\cal C}$
such that $G$ is general with respect to $v$.
We may assume that $G$ is isotropic.
By the proof of Corollary \ref{cor:reduction},
 there is an equivalence $\Phi:{\bf D}(X) \to {\bf D}(X)$
such that $\Phi(v(G)) \in {\Bbb Z}\varrho_X$ and
${\cal M}_H^G(\Phi(v))^{ss}$ consists of $\mu$-stable
objects.
Then $X':=\overline{M}_H^u(u)$, $u=\Phi(v)$ is a fine moduli space
and $H':=\widehat{H}$ is an ample divisor on $X'$.
Let ${\cal E}$ be a universal family on $X \times X'$.
We set $w:={\cal E}_{|\{x \} \times X'}$.
Then by Theorem \ref{thm:duality}, we get our claim.
 
\end{NB}

%Assume that $\WIT_2$ holds for $F$.
%If $\widehat{\Psi}(F)$ does not contain a 0-dimensional subobject, then
%$F=0$.

Since \eqref{eq:degree-condition} is numerical, we can apply 
Proposition \ref{prop:stability-asymptotic} to a 
family of $K3$ surfaces. 
\begin{ex}
Let $f:({\cal X},{\cal H}) \to S$ be a family of polarized 
$K3$ surfaces over $S$.
Let $v_0:=(r,d{\cal H},a)$, $\gcd(r,a)=1$ be a family of isotropic Mukai
vectors.
We set ${\cal X}':=M_{{\cal X}/S}^{v_0}(v_0)$.
Then we have a family of polarizations ${\cal H}'$
on ${\cal X}'$.
Since $\gcd(r,a)=1$, there is a universal family
${\cal E}$ on ${\cal X}' \times_S {\cal X}$ and 
we have a family of Fourier-Mukai transforms
$\Phi_{{\cal X} \to {\cal X}'}^{\cal E}:{\bf D}({\cal X})
\to {\bf D}({\cal X}')$.
Then we can apply Proposition \ref{prop:stability-deg=0}
and Proposition \ref{prop:stability-asymptotic}
to families of moduli spaces over $S$.
\end{ex}

We also give a generalization of \cite[Thm. 7.6]{Y:7} based on
Theorem \ref{thm:equiv-Phi} and Proposition \ref{prop:equiv-Psi}.
We set
\begin{equation}
d_{\min}:=\min\{\deg_{G_1}(F)>0|F \in {\bf D}(X)\}.
\end{equation} 

\begin{prop}\label{prop:K3:minimal}
Assume that ${\frak T}_1={\frak T}_1^{\mu}$.
Let $v \in H^*(X,{\Bbb Z})$ be a Mukai vector
of a complex such that
$\deg_{G_1}(v)=d_{\min}$.
\begin{enumerate}
\item[(1)]
If $\rk \Phi(v) \leq 0$, then $\Phi$ induces an isomorphism
\begin{equation}
{\cal M}_H^{G_1}(v)^{ss} \to 
{\cal M}_{\widehat{H}}^{G_2}(-\Phi(v))^{ss}
\end{equation}
by sending $E$ to $\Phi^1(E)$.
\item[(2)]
If $\rk \Psi(v) \geq 0$, then $\Psi$ induces an isomorphism
\begin{equation}
{\cal M}_H^{G_1}(v)^{ss} \to 
{\cal M}_{\widehat{H}}^{G_3}(\Psi(v))^{ss}
\end{equation}
by sending $E$ to $\Psi^2(E)$.
\end{enumerate}
\end{prop}
The proof is an easy exercise. 
We shall give a proof in \cite{MYY}, as an application of
Bridgeland's stability condition.

\begin{NB}

\begin{proof}
By Remark \ref{rem:generic-mu-stable},
${\cal E}_{|\{ x' \} \times X}$ is $\mu$-stable for $x' \in X' \setminus Z'$.
Then Theorem \ref{thm:duality} and Lemma \ref{lem:free} imply that
${\cal E}_{|X' \times \{ x \}}$ is a local projective generator of
$\Per(X'/Y')^D$ for all $x \in X$ and
${\cal E}_{|X' \times \{ x \}}$ is a $\mu$-stable object for 
$x \in X \setminus Z$.
Hence ${\cal E}_{|X' \times \{ x \}}^{\vee}$ 
is a local projective generator of
$\Per(X'/Y')$ for all $x \in X$ and
${\cal E}_{|X' \times \{ x \}}$ is a $\mu$-stable object for 
$x \in X \setminus Z$.
Then we also have ${\frak T}_2={\frak T}_2^{\mu}$.
%${\cal E}_{|X' \times \{ x\}}$ is a local projective generator
%of $\Per(X'/Y')^D$, and hence
%${\cal E}_{|X' \times \{ x\}}^{\vee}$ is a local projective generator
%of $\Per(X'/Y')$
%and ${\frak T}_2={\frak T}_2^{\mu}$(cf. 
%Lem. \ref{lem:generic-mu-stable}, 
%Thm. \ref{thm:duality}, Lem. \ref{lem:free}).
We shall prove that $\Phi$ and $\Psi$ preserve the stability.
The proofs of the inverse directions are similar.
Let $E$ be an object of ${\cal M}_H^{G_1}(v)^{ss}$.

(1)
Since $E \in {\frak T}_1$, we have
$H^{-1}(\Phi(E)[1]) \in {\frak F}_2^{\mu}$ and
$H^0(\Phi(E)[1]) \in {\frak T}_2^{\mu}$.
Since 
\begin{equation}
d_{\min}=\deg_{G_2}(\Phi(E)[1])=
\deg_{G_2}(H^0(\Phi(E)[1]))-
\deg_{G_2}(H^{-1}(\Phi(E)[1]))
\end{equation}
 and
$\rk H^0(\Phi(E)[1]) \geq \rk H^{-1}(\Phi(E)[1])$,
we see that
$\deg_{G_2}(H^{-1}(\Phi(E)[1]))=0$ and
$\deg_{G_2}(H^0(\Phi(E)[1]))=d_{\min}$.
Hence $H^0(\Phi(E)[1])/T$ is $\mu$-stable, where 
$T$ is the 0-dimensional subobject of $H^0(\Phi(E)[1])$.
By the exact triangle
\begin{equation}
\Phi^0(E) \to \Phi(E) \to \Phi^1(E)[-1] \to \Phi^0(E)[1],
\end{equation}
we have an exact sequence
\begin{equation}
\begin{CD}
\widehat{\Phi}^0(\Phi^0(E)) @>>> 0 @>>> 0 @>>>\\
\widehat{\Phi}^1(\Phi^0(E)) @>>> 0 @>>>
\widehat{\Phi}^0(\Phi^1(E)) @>>> \\
\widehat{\Phi}^2(\Phi^0(E)) @>>> E @>>> 
\widehat{\Phi}^1(\Phi^1(E)) @>>>
\end{CD}
\end{equation}
Then $\WIT_2$ holds for $\Phi^0(E)$,
$\widehat{\Phi}^2(\Phi^0(E)) \in {\frak T}_1$
and $\deg_{G_1}(\widehat{\Phi}^2(\Phi^0(E)))=0$.
Since $\widehat{\Phi}^2(\Phi^0(E)) \in 
{\frak T}_1^{\mu}= {\frak T}_1$,
$\dim \widehat{\Phi}^2(\Phi^0(E))=0$. 
Since $E$ is torsion free or purely 1-dimensional,
$\widehat{\Phi}^0(\Phi^1(E)) \cong
\widehat{\Phi}^2(\Phi^0(E))$.
Then $\WIT_2$ and $\WIT_0$ hold for 
$\widehat{\Phi}^2(\Phi^0(E))$, which implies that
$\Phi^0(E)=0$.
We next prove that $T=0$.
We note that $\widehat{\Phi}(T)[1] \in {\frak A}_1$.
%On the other hand,
%the local projectivity of
%${\cal E}_{|X' \times \{x \}}^{\vee}$
%implies that $\widehat{\Phi}^0(T)=0$.
Applying $\widehat{\Phi}$ to the exact sequence 
\begin{equation}
0 \to T \to \Phi^1(E) \to \Phi^1(E)/T \to 0
\end{equation}
in $\Per(X'/Y')$,
we get $\widehat{\Phi}^0(T)=0$.
%we see that $\WIT_1$ holds for $T$ and
Hence $\widehat{\Phi}(T)[1]=\widehat{\Phi}^1(T) \in {\frak T}_1$.
Since $\deg_{G_1}\widehat{\Phi}^1(T)=\deg_{G_2}(T)=0$,
$\widehat{\Phi}^1(T)$ is a 0-dimensional object of ${\cal C}$. 
Then $\rk T=-\chi(G_1,\widehat{\Phi}^1(T))<0$,
which is a contradiction. 
Therefore $\Phi(E)$ is $\mu$-stable.

(2)
We have an exact sequence in $\overline{\frak A}_3$
\begin{equation}
0 \to \Psi^1(E)[1] \to \Psi(E)[2] \to \Psi^2(E) \to 0
\end{equation}
where $\Psi^1(E) \in \overline{\frak F}_3$ and
$\Psi^2(E)\in \overline{\frak T}_3$.
Since $E$ does not contain a non-trivial 0-dimensional subobject,
Lemma \ref{lem:equiv} (1)
implies that 
$\Psi^2(E)\in {\frak T}_3$.
Since $\rk \Psi^2(E) \geq \rk \Psi^1(E)$,
$\deg_{G_3}(\Psi(E))=d_{\min}$ and 
${\frak T}_3={\frak T}_3^{\mu}$, we have 
$\Psi^1(E)$ is a $\mu$-semi-stable object with
$\deg_{G_3}(\Psi^1(E))=0$ and
$\Psi^2(E)/T$ is a $\mu$-stable object with
$\deg_{G_3}(\Psi^2(E))=d_{\min}$, where
$T$ is a 0-dimensional subobject of
$\Psi^2(E)$.
Since ${\cal E}_{|X' \times \{ x \}}$ is $\mu$-stable
for a general $x \in X$,
$\widehat{\Psi}^0(\Psi^1(E))=0$.
Then $\WIT_1$ holds for $\Psi^1(E)$ and
we have an exact sequence
\begin{NB2}
By Lemma \ref{lem:Psihat}, $\widehat{\Psi}^1(\Psi^1(E))=0$.
\end{NB2}
\begin{equation}
0 
%\to \widehat{\Psi}^0(\Psi^1(E)) 
\to
\widehat{\Psi}^2(\Psi^2(E)) \to
E \to
\widehat{\Psi}^1(\Psi^1(E)) \to 0
\end{equation} 
Since 
$\deg_{G_1}(\widehat{\Psi}^1(\Psi^1(E)))=0$,
$\widehat{\Psi}^1(\Psi^1(E))$ is a 0-dimensional object.
Then $\chi(G_1,\widehat{\Psi}^1(\Psi^1(E)))=-\rk \Psi^1(E)\leq 0$,
which implies that $\Psi^1(E)=0$.
Finally we prove that $T=0$.
We note that $\widehat{\Psi}^0(T)=0$ and
we have an exact sequence
\begin{equation}
0 \to \widehat{\Psi}^1(T) \to \widehat{\Psi}^2(E/T) \to
E \to \widehat{\Psi}^2(T) \to 0.
\end{equation}
Since $0=\deg_{G_1}(\widehat{\Psi}(T))=\deg_{G_1}(\widehat{\Psi}^2(T))-
\deg_{G_1}(\widehat{\Psi}^1(T))$, we have
$\deg_{G_1}(\widehat{\Psi}^2(T))=\deg_{G_1}(\widehat{\Psi}^1(T))=0$.
By the stability of $E$,
we have $\rk \widehat{\Psi}^2(T)=0$, which mplies that
$\chi(G_3,T) \leq \rk \widehat{\Psi}^2(T)=0$.
Therefore $T=0$.
\end{proof}

\end{NB}

\begin{rem}
In \cite{Y:action},
we constructed actions of Lie algebras on the cohomology
groups of some moduli spaces of stable sheaves.
In particular, we constructed the action on the cohomology
groups of some moduli spaces of stable objects of $^{-1}\Per(X/Y)$
in \cite[Prop. 6.15]{Y:action}.
Then a generalization of \cite[Prop. 6.15]{Y:action}
to the objects in $\Per(X'/Y')$ corresponds to
the action in \cite[Example 3.1.1]{Y:action}
via Proposition \ref{prop:K3:minimal}. 
\end{rem}

\section{Fourier-Mukai transforms on elliptic surfaces.}
\label{sect:elliptic}

\subsection{Moduli of stable sheaves of dimension 2.}

Let $Y \to C$ be a morphism from a normal projective surface
to a smooth curve $C$ such that a general fiber is an elliptic curve.
Let $\pi:X \to Y$ be the minimal resolution.
Then ${\frak p}:X \to C$ is an elliptic surface over a curve $C$. 
We fix a divisor $H$ on $X$ which is the pull-back of an
ample divisor on $Y$.
As in section \ref{sect:K3}, 
let ${\cal C}$ be the category in Lemma \ref{lem:tilting}
satisfying Assumption \ref{ass:stability}.
We also use the notation $A_{ij}$ in section \ref{sect:K3}. 
Let $G_1$ be a locally free sheaf on $X$ which
is a local projective generator of ${\cal C}$. 
Let ${\bf e} \in K(X)_{\mathrm{top}}$ be the topological invariant of a 
locally free sheaf $E$ of rank $r$ and degree $d$ on a fiber of 
${\frak p}$.
Thus $\ch({\bf e})=(0,rf,d)$, where $f$ is a fiber of ${\frak p}$.  
Assume that ${\bf e}$ is primitive.
Then $\overline{M}^{G_1}_H({\bf e})$ consists of $G_1$-twisted stable
objects, if $G_1 \in K(X)_{\mathrm{top}} \otimes {\Bbb Q}$, 
$\rk G_1>0$ is general
with respect to ${\bf e}$ and $H$.
From now on, we assume that 
$\chi(G_1,{\bf e})=0$.
By \cite[sect. 1.1]{O-Y:1}, we do not lose 
generality. 
\begin{rem}
We have 
$\overline{M}^{G_1}_H({\bf e})=\overline{M}^{G_1}_{H+nf}({\bf e})$
for all $n$.
\end{rem}

\begin{lem}\label{lem:elliptic:lattice}
We set
\begin{equation}
{\bf e}^{\perp}:=
\{E \in K(X)_{\mathrm{top}} | \chi(E,{\bf e})=0 \}.
\end{equation}
\begin{enumerate}
\item[(1)]
$-\chi(\;\;,\;\;)$ is symmetric on ${\bf e}^{\perp}$.
\item[(2)]
$M:=({\Bbb Z}\tau(G_1) +{\Bbb Z}\tau({\Bbb C}_x)+ 
{\Bbb Z}{\bf e})^{\perp}/{\Bbb Z}{\bf e}$
is a negative definite even lattice of rank $\rho(X)-2$.
\end{enumerate}
\end{lem}

\begin{proof}
(1)
For a divisor $D$,
we set
\begin{equation}
\nu(D):=\tau({\cal O}_X(D)-{\cal O}_X)-
\frac{\chi(G_1,{\cal O}_X(D)-{\cal O}_X)}{\rk G_1}\tau({\Bbb C}_x)
\in K(X)_{\mathrm{top}}\otimes{\Bbb Q}.
\end{equation}
Then $\nu$ induces a homomorphism
\begin{equation}
\NS(X) \otimes{\Bbb Q} \to K(X)_{\mathrm{top}}\otimes{\Bbb Q}
\end{equation}
such that $\rk(\nu(D))=0$, $c_1(\nu(D))=D$ and
$\chi(G_1,\nu(D))=0$.
For $E \in K(X) \otimes {\Bbb Q}$,
we have an expression
\begin{equation}
\tau(E)=l\tau(G_1)+a\tau({\Bbb C}_x)+\nu(D)
\end{equation}
where $l,a\in {\Bbb Q}$ and $D \in \NS(X) \otimes {\Bbb Q}$.
\begin{NB}
$\chi(G_1,E)=l\chi(G_1,G_1)+a\rk G_1$ and
$\chi({\Bbb C}_x,E)=l\rk G_1$.
\end{NB}
If $\chi(E,{\bf e})=0$, then 
$D$ satisfies $(D,f)=0$.
Hence we have a decomposition
\begin{equation}\label{eq:e-perp}
{\bf e}^{\perp} \otimes {\Bbb Q}
=({\Bbb Q}\tau(G_1)+{\Bbb Q}\tau({\Bbb C}_x)) +
\nu(({\Bbb Q}f)^{\perp}).
%
%({\Bbb Q}\tau({\cal O}_f))^{\perp} \cap 
%({\Bbb Q}\tau(G_1)+{\Bbb Q}\tau({\Bbb C}_x))^{\perp}.
\end{equation}
For $E, F \in K(X)$, we have
\begin{equation}
\chi(E,F)-\chi(F,E)=(\rk E c_1(F)-\rk F c_1(E),K_X).
\end{equation}
Hence the claim (1) holds.

(2)
By \eqref{eq:e-perp}, 
the signature of ${\bf e}^{\perp}/{\Bbb Z}{\bf e}$
is $(1,\rho(X)-1)$.
We note that
 ${\Bbb Q}\tau(G_1) +{\Bbb Q}\tau({\Bbb C}_x) \to 
({\bf e}^{\perp}/{\Bbb Z}{\bf e})\otimes {\Bbb Q}$
is injective and defines a subspace 
of signature $(1,1)$. 
Hence 
$M$
%$({\Bbb Z}{\Bbb C}_x +{\Bbb Z}G_1+ {\Bbb Z}{\bf e})^{\perp}/{\Bbb Z}{\bf e}$
is negative definite. 
Since $({\Bbb Z}\tau({\Bbb C}_x)  +{\Bbb Z}{\bf e})^{\perp}$
is an even lattice, 
we get our claim.
\end{proof}

\begin{lem}
\begin{enumerate}
\item[(1)]
Assume that $G_1$ is general with respect to
${\bf e}$ and $H$. Then 
$\overline{M}_H^{G_1}({\bf e})$ is a smooth elliptic surface over $C$ and
$E \otimes K_X \cong E$
for all $E \in \overline{M}_H^{G_1}({\bf e})$.
\item[(2)]
Let $E$ be a $G_1$-twisted stable object 
such that $\Supp(E) \subset {\frak p}^{-1}(c)$,
$c \in C$.
If $\chi(G_1,E)=0$ and $(c_1(E),H)< (c_1({\bf e}),H)$, 
then $\chi(E,E)=2$ and $E \otimes K_X \cong E$.
\end{enumerate} 
\end{lem}

\begin{proof}
(1)
In \cite[Thm. 1.2]{Br:1}, Bridgeland proved that
$\overline{M}_H^{G_1}({\bf e})$ is smooth and defines a Fourier-Mukai
transform ${\bf D}(\overline{M}_H^{G_1}({\bf e}) ) \to {\bf D}(X)$, if
$G_1={\cal O}_X$ is general with respect to
${\bf e}$ and $H$.
We can easily generalize the arguments in \cite[sect. 4]{Br:1}
to the moduli space $\overline{M}_H^{G_1}({\bf e})$ 
of $G_1$-twisted semi-stable objects,
if $G_1$ is general with respect to 
${\bf e}$ and $H$.
Then the claims follow.

(2) 
Since $\Supp(E) \subset {\frak p}^{-1}(c)$ and
$\chi(G_1,E)=0$, we have
$E \in ({\Bbb Z}\tau({\Bbb C}_x) 
+{\Bbb Z}\tau(G_1)+ {\Bbb Z}{\bf e})^{\perp}$.
\begin{NB}
By Lemma \ref{lem:elliptic:lattice},
$({\Bbb Z}\tau({\Bbb C}_x) 
+{\Bbb Z}\tau(G_1)+ {\Bbb Z}{\bf e})^{\perp}$ is
negative semi-definite.
\end{NB}
Since $(c_1(E),H)< (c_1({\bf e}),H)$, we get
\begin{equation}
2 \leq \chi(E,E)= 
\dim \Hom(E,E)+\dim \Hom(E,E \otimes K_X)-\dim \Ext^1(E,E).
\end{equation}
Hence $\Hom(E,E \otimes K_X) \ne 0$.
Since $K_X^{\otimes m} \in {\frak p}^*(\Pic(C))$ for an integer $m$, 
we see that $E \otimes K_X$ is a $G_1$-twisted stable object with
$\tau(E)=\tau(E \otimes K_X)$, which implies that  
$E \otimes K_X \cong E$ and 
$\chi(E,E)=2$. 
\end{proof}
In the same way as in the proof of Theorem \ref{thm:K3-desing},
we get the following results.

\begin{cor}\label{cor:desing}
\begin{enumerate}
\item[(1)]
$\overline{M}_H^{G_1}({\bf e})$ is a normal surface and 
the singular points $q_1,q_2,\dots,q_m$ of 
$\overline{M}_H^{G_1}({\bf e})$ correspond to the 
$S$-equivalence classes of properly $G_1$-twisted semi-stable
objects.
\item[(2)]
Let $\oplus_{j=0}^{s_i'} E_{ij}^{\oplus a_{ij}'}$ be the $S$-equivalence class 
corresponding to $q_i$.
Then the matrix $(\chi(E_{ij},E_{ik}))_{j,k \geq 0}$ is 
of affine type $\tilde{A},\tilde{D},\tilde{E}$.
We assume that $a_{i0}=1$ for all $i$.
Then $q_1,q_2,\dots,q_m$ are rational double points
of type $A,D,E$ according as the type of the matrices
 $(\chi(E_{ij},E_{ik}))_{j,k \geq 1}$. 
\item[(3)]
%We assume that $a_{i0}'=1$ for all $i$.
We take a sufficiently small general 
$\alpha \in K(X) \otimes {\Bbb Q}$ such that
$\chi(\alpha,{\bf e})=0$.
% and $\chi(\alpha,E_{ij})<0$ for all $j>0$.
Then $\pi': \overline{M}_H^{G_1 +\alpha}({\bf e})
\to \overline{M}_H^{G_1}({\bf e})$ 
is the minimal resolution.
\item[(4)]
Assume that $a_{i0}'=1$ for all $i$ and 
$\chi(\alpha,E_{ij})<0$ for all $j>0$.
We set
\begin{equation}
C_{ij}':=\{ E \in M_H^{G_1+\alpha}({\bf e})|
\Hom(E_{ij},E) \ne 0 \}.
\end{equation}
Then $C_{ij}'$ is a smooth rational curve such that
$(C_{ij}',C_{i' j'}')=-\chi(E_{ij},E_{i' j'})$ and
${\pi'}^{-1}(q_i)=\sum_{j \geq 1}a_{ij}'C_{ij}'$. 
\end{enumerate}
\end{cor}

\begin{rem}
In Theorem \ref{thm:K3-desing}, we assume that $\chi(\alpha,E_{ij})>0$.
So the definition of $C_{ij}'$ is different from 
that in Lemma \ref{lem:K3:exceptional}.
For the smoothness of $C_{ij}'$, we use the moduli of
coherent systems $(E,V)$, where $E \in M_H^{G_1+\alpha}({\bf e})$
and 
$V$ is a 1-dimensional subspace of $\Hom(E_{ij},E)$.
\end{rem}

From now on, we take an $\alpha$ in Corollary \ref{cor:desing} (3)
and set $X':=\overline{M}_H^{G_1 +\alpha}({\bf e})$,
$Y':=\overline{M}_H^{G_1}({\bf e})$. 
Let ${\frak q}:X' \to C$ be the structure morphism of the 
elliptic fibration.

%Since $2 \geq \chi(E_i,E_i)=-(c_1(E_i)^2) \geq 0$,
%we get $(c_1(E_i)^2)=-2$ and
%$\sum_{i=0}^s a_i c_1(E_i)=rf$. 
%Hence $(\chi(E_i,E_j))_{i,j}$ is the Cartan matrix of an affine
%Lie algebra.

\subsection{Fourier-Mukai duality for an elliptic surface.}
\label{subsect:elliptic:FM-duality}

Let ${\cal E}$ be a universal family as a twisted sheaf on $X' \times X$.
For simplicity, we assume that it is an untwisted sheaf.
We set
\begin{equation}
\begin{split}
\Psi(E):=&{\bf R}\Hom_{p_{X'}}(p_X^*(E),{\cal E})
=\Phi(E)^{\vee}[-2],\; E \in {\bf D}(X),\\
\widehat{\Psi}(F):=&{\bf R}\Hom_{p_X}(p_{X'}^*(F),{\cal E}),\;
F \in {\bf D}(X').
\end{split}
\end{equation}

\begin{lem}
Replacing $G_1$ by $G_1-n {\Bbb C}_x$, $n \gg 0$, we can choose 
$\det \Psi(G_1)^{\vee}
\in \Pic(X')$ as the pull-back of an ample line bundle on $W$.
Let $\widehat{H}$ be a divisor with ${\cal O}_{X'}(\widehat{H})=
\det \Psi(G_1)^{\vee}$.
\end{lem}

\begin{proof}
We note that $\det \Psi({\Bbb C}_x)=rf$.
Hence $\det \Psi(G_1-n{\Bbb C}_x)^{\vee}=
\det \Psi(G_1)^{\vee}(nrf)$.
We set 
\begin{equation}
\xi:=
mr\rk G_1(H,f)(-G_1^{\vee}+(\rk G_1)n(n+m)(H^2)/2 \varrho_X).
\end{equation}
By \eqref{eq:det-bdle},
$\det p_{X' !}({\cal E} \otimes p_X^*(\xi))$ is 
the pull-back of a polarization of $Y'$ for
$m \gg n \gg 0$.
Since $\det \Psi(\xi^{\vee})=
\det p_{X' !}({\cal E} \otimes p_X^*(\xi))$
and $-\ch(\xi^{\vee}) \equiv mr\rk G_1 (H,f) \ch(G_1) \mod {\Bbb Q}\varrho_X$,
we get our claim.
\end{proof}

%
%For simplicity, we also denote $\pi^{-1}(\widehat{H})$ by
%$\widehat{H}$.

\begin{lem}\label{lem:elliptic:irreducible-obj}
We set $A_{ij}':=\Psi(E_{ij})[2]$.
\begin{enumerate}
\item[(1)]
There are ${\bf b}_i':=(b_{i1}',b_{i2}',\dots,b_{is_i'}')$,
$i=1,\dots,m$ such that
\begin{equation}
\begin{split}
A_{ij}' &={\cal O}_{C_{ij}'}(b_{ij}')[1],\; j>0\\
A_{i0}'& =A_0({\bf b}_i').
\end{split}
\end{equation}
\item[(2)]
Irreducible objects of $\Per(X'/Y',{\bf b}_1',...,{\bf b}_m')$
are
\begin{equation}
A_{ij}' (1 \leq i \leq m,0 \leq j \leq s_i'),\;
{\Bbb C}_{x'} (x' \in X' \setminus \cup_i Z_i'). 
\end{equation}
\end{enumerate}
\end{lem}

\begin{proof}
It is sufficient to prove (1) by
Proposition \ref{prop:tilting:G-1Per-irred}.
By the choice of $\alpha$,
%$G_1+\alpha$-twisted stability of ${\cal E}_{|\{ x' \}\times X}$, 
we have 
\begin{equation}
\begin{split}
\Ext^2(E_{ij},{\cal E}_{|\{ x' \}\times X})=&0,\; j>0,\\
\Hom(E_{i0},{\cal E}_{|\{ x' \}\times X})=&0
\end{split}
\end{equation}
for all $x' \in X'$.
Then the claim for $j>0$ follow from the proof of 
Corollary \ref{cor:desing} (4).
For $x' \in {\pi'}^{-1}(q_i)$,
we have an exact sequence
\begin{equation}
0 \to F_i \to {\cal E}_{|\{ x' \}\times X} \to E_{i0} \to 0,
\end{equation}
where $F_i$ is a $G_1$-twisted semi-stable object which is
$S$-equivalent to $\oplus_{j>0} E_{ij}^{\oplus_j a_{ij}'}$.
Applying $\Psi$, we have an exact sequence
\begin{equation}
0 \to \Psi(F_i)[1] \to A_{i0}' \to {\Bbb C}_{x'} \to 0.
\end{equation}
It is easy to see that
\begin{equation}
\Hom(A_{i0}',A_{ij}'[-1])=\Ext^1(A_{i0}',A_{ij}'[-1])=0.
\end{equation}
By Lemma \ref{lem:A_0}, we get
 $A_{i0}'=A_0({\bf b}_i')$.
\end{proof}

We define $\Per(X'/Y')$ and $\Per(X'/Y')^D$ as in subsection 
\ref{subsect:K3:proof}.
Replacing $G_1$ by $G_1'$ with
$\tau(G_1')=\tau(G_1)-n\tau({\Bbb C}_x)$,
we may assume that $G_{1|{\frak p}^{-1}(t)}$, $t \in C$
is a stable vector bundle for a general $t \in C$.
Then $L_2'=\Psi(G_1)[1]$ is a torsion object of $\Per(X'/Y') \cap \Coh(X')$
such that $c_1(L_2)=\widehat{H}$.
Indeed $L_2'$ is a coherent torsion sheaf on $X'$. Since
$\Hom(L_2',A_{ij}'[-1])=\Hom(E_{ij},G_1)=0$, 
$L_2' \in \Per(X'/Y')$.

\begin{lem}
Let $L_1$ be a line bundle on a smooth curve $C \in |H|$ and set
$G_2:=\Psi(L_1)[1]$.
Then we have
\begin{equation}
\begin{split}
\Hom(G_2,{\Bbb C}_{x'}[k])&=0, \quad k \ne 0,\\
\Hom(G_2,A_{ij}'[k])& =0, \quad k \ne 0,\\
\dim \Hom(G_2,A_{ij}')& =
(c_1(E_{ij}),H).\\
\end{split}
\end{equation}
In particular $G_2$ is a local projective generator
of $\Per(X'/Y')$.
\begin{NB}
Since $E_{ij}$ are 1-dimensional objects,
$(c_1(E_{ij}),H)>0$.
\end{NB}
\end{lem}

\begin{proof}
The claim follows from the following relations:
\begin{equation}
\begin{split}
\Hom(G_2,{\Bbb C}_{x'}[k])& =
\Hom(\Psi(L_1)[1],\Psi({\cal E}_{|\{ x' \} \times X})[2+k]) \\
& =\Hom({\cal E}_{|\{ x' \} \times X},L_1[k+1]),\\
\Hom(G_2,A_{ij}'[k])& =\Hom(\Psi(L_1)[1],\Psi(E_{ij})[2+k]) \\
& =\Hom(E_{ij},L_1[k+1]).
\end{split}
\end{equation}
\end{proof}

For a conveniense sake,
we summalize the image of
${\Bbb C}_x[-2],{\cal E}_{|\{x' \} \times X},G_1,L_1$
by $\Psi$:
\begin{equation}
\begin{split}
\Psi({\Bbb C}_x[-2]) &={\cal E}_{|X' \times \{ x \}},\\
\Psi({\cal E}_{|\{x' \} \times X}) &={\Bbb C}_{x'}[-2],\\
\Psi(G_1) &=L_2[-1],\\
\Psi(L_1) &=G_2[-1].
\end{split}
\end{equation}

\begin{defn}
We set $\Psi^i(E):= {^p H^i}(\Psi(E)) \in \Per(X'/Y')$ and
$\widehat{\Psi}^i(E):={^p H^i}(\widehat{\Psi}(E)) \in \Per(X/Y)$. 
\end{defn}

\begin{lem}\label{lem:elliptic-WIT-simple}
$\WIT_2$ with respect to $\Psi$
holds for all 0-dimensional objects $E$ of $\Per(X'/Y')$ and
$\Psi^2(E)$ is $G_2$-twisted semi-stable. 
Moreover if $E$ is an irreducible object, then 
$\Psi(E)[2]$ is a $G_2$-twisted stable object of $\Per(X'/Y')$.
\end{lem}

\begin{proof}
%By Proposition \ref{prop:irreducible-obj},
It is sufficient to prove the claim for all
irreducible objects $E$ of ${\cal C}$.
Since ${\cal E}_{|\{ x' \} \times X}$ and
$E_{ij}$ are purely 1-dimensional objects of ${\cal C}$,
$\Hom(E,{\cal E}_{|\{ x' \} \times X})=
\Hom(E,E_{ij})=0$
for all $x' \in X'$ and $i,j$.
Hence $\Psi^1(E)$ is a torsion free object of
${\frak C}_2$.
Since $\Hom(E,{\cal E}_{|\{ x' \} \times X}[1])=0$
if $\Supp(E) \cap {\frak p}^{-1}({\frak p}(x')) =\emptyset$,
$\Psi^1(E)=0$. Therefore
$\WIT_2$ holds for all 0-dimensional objects of $\Per(X'/Y')$.

For the $G_2$-twisted stability of $\Psi(E)[2]$, 
we first note that
$\chi(G_2,\Psi(E)[2])=
\chi(\Psi(L_1)[1],\Psi(E)[2])=
\chi(E,L_1[1])=0$.
Assume that there is an exact sequence
\begin{equation}
0 \to F_1 \to \Psi^2(E) \to F_2 \to 0
\end{equation}
such that $0 \ne F_1 \in \Per(X'/Y')$ and
$F_2 \in \Per(X'/Y')$ with 
$\chi(G_2,F_2) \leq 0$.
Applying $\widehat{\Psi}$ to this exact sequence, 
we get a long exact sequence
\begin{equation}\label{eq:elliptic-Huybrechts}
\begin{CD}
0 @>>> \widehat{\Psi}^0(F_2) @>>> 
0 @>>> 
\widehat{\Psi}^0(F_1)\\
@>>>\widehat{\Psi}^1(F_2) @>>> 
0 @>>> 
\widehat{\Psi}^1(F_1)\\
@>>> \widehat{\Psi}^2(F_2) @>>> 
E @>>> 
\widehat{\Psi}^2(F_1) @>>> 0.
\end{CD}
\end{equation}
Since $\widehat{\Psi}^0(F_1)=0$, $\WIT_2$ holds for $F_2$.
Since 
$0 \geq \chi(G_2,F_2)=
\chi(\widehat{\Psi}(F_2),\widehat{\Psi}(G_2))=
\chi(\widehat{\Psi}(F_2),L_1[-1])=
(H,c_1(\widehat{\Psi}^2(F_2))) \geq 0$,
we get $\chi(G_2,F_2)=0$ and 
$\widehat{\Psi}^2(F_2)$ is a 0-dimensional object.
Then $\widehat{\Psi}^1(F_1)$ is also 0-dimensional.
Since $E$ is an irreducible object of ${\frak C}_1$,
we have (i) $\widehat{\Psi}^2(F_1)=0$ or
(ii) $\widehat{\Psi}^2(F_1) \cong E$.
Since $\WIT_2$ holds for $\widehat{\Psi}^1(F_1)$ with respect to
$\Psi$, the first case does not hold.
If $\widehat{\Psi}^2(F_1) \cong E$, then
$\widehat{\Psi}^1(F_1) \cong \widehat{\Psi}^2(F_2)$.
Since $\widehat{\Psi}^0(F_1)=0$,
Lemma \ref{lem:spectral2} implies that
$\Psi^2(\widehat{\Psi}^1(F_1))=0$, which implies that
$F_2=\Psi^2(\widehat{\Psi}^2(F_2))=0$.
Therefore $\Psi^2(E)$ is $G_2$-twisted stable.
\end{proof}

\begin{thm}\label{thm:elliptic-duality}
We set ${\bf f}:=\tau({\cal E}_{|X' \times \{x \}})$.
Then
${\cal E}_{|X' \times \{x \}}$ is $G_2-\Psi(\beta)$-twisted stable 
for all $x \in X$ and 
we have an isomorphism
$X \to M_{\widehat{H}}^{G_2-\Psi(\beta)}({\bf f})$ by sending $x \in X$
to ${\cal E}_{|X' \times \{x \}} \in 
M_{\widehat{H}}^{G_2-\Psi(\beta)}({\bf f})$.
\end{thm}

\begin{proof}
By Lemma \ref{lem:elliptic-WIT-simple},
${\cal E}_{|X' \times \{x\} }$ is $G_2$-twisted semi-stable.
If ${\cal E}_{|X' \times \{x\} }$ is not $G_2$-twisted stable,
then ${\cal E}_{|X' \times \{x\} }$ is $S$-equivalent to
$\oplus_j \Psi^2(A_{ij})^{\oplus a_{ij}}$.
Let $F_1 \ne 0$ be a $G_2$-twisted stable subobject of 
${\cal E}_{|X' \times \{x\} }$ such that
$\chi(G_2,F_1)=0$.
Then $F_1$ is $S$-equivalent to
$\oplus_j \Psi^2(A_{ij})^{\oplus b_{ij}}$ and
$\widehat{\Psi}(F_1)[2]$ is a quotient object of
${\Bbb C}_x$.
Since ${\Bbb C}_x$ is $\beta$-stable,
$0<\chi(\beta,\widehat{\Psi}(F_1))=
\chi(\Psi(\beta),F_1)$.
Therefore
${\cal E}_{|X' \times \{x\} }$ is $G_2-\Psi(\beta)$-twisted stable.
Then we have an injective morphism 
$\phi:X \to \overline{M}_{\widehat{H}}^{G_2-\Psi(\beta)}({\bf f})$
by sending $x \in X$ to
${\cal E}_{|X' \times \{x \}}$.
By a standard argument, we see that $\phi$ is an isomorphism. 
\end{proof}

\subsection{Tiltings of ${\cal C}$, $\Per(X'/Y')$ and their
equivalence.}

We set ${\frak C}_1:={\cal C}$ and ${\frak C}_2:=\Per(X'/Y')$.
In this subsection,
we define tiltings $\overline{\frak A}_1$, $\widehat{\frak A}_2$
of ${\frak C}_1$, ${\frak C}_2$ and show that
$\Psi$ induces a (contravariant) equivalence between them.
We first define the relative twisted degree of $E \in {\frak C}_i$ by
$\deg_{G_i}(E):=(c_1(G_i^{\vee} \otimes E),f)$, 
and define
$\mu_{\max,G_i}(E)$, $\mu_{\min,G_i}(E)$ in a similar way.
\begin{defn}
\begin{enumerate}
\item[(1)]
Let $\overline{\frak T}_i$ be the full subcategory of ${\frak C}_i$ 
consisting of objects $E$ such that
(i) $E$ is a torsion object or 
(ii) 
$E$ is torsion free and 
$\mu_{\min,G_i}(E) \geq 0$.
\item[(2)]
Let $\overline{\frak F}_i$ be the full subcategory of ${\frak C}_i$ 
consisting of objects $E$ such that
(i) $E=0$ or 
(ii) 
$E$ is torsion free and $\mu_{\max,G_i}(E) < 0$.
\end{enumerate}
\end{defn}

\begin{defn}
\begin{enumerate}
\item[(1)]
Let $\widehat{\frak T}_i$ be the full subcategory of ${\frak C}_i$ 
consisting of objects $E$ such that
$\Supp(E)$ is contained in fibers and
there is no quotient object $E \to E'$ with
$\chi(G_i,E')<0$.
\item[(2)]
We set 
\begin{equation}
\begin{split}
\widehat{\frak F}_i:= &(\widehat{\frak T}_i)^{\perp}\\
=& \{E \in {\frak C}_i| \Hom(E',E)=0, E' \in \widehat{\frak T}_i \}.
\end{split}
\end{equation}
\end{enumerate}
\end{defn}

\begin{rem}
We have $\widehat{\frak F}_i \supset \overline{\frak F}_i$
and $\widehat{\frak T}_i \subset \overline{\frak T}_i$.
\end{rem}

\begin{defn}
$(\overline{\frak T}_i,\overline{\frak F}_i)$ and 
$(\widehat{\frak T}_i,\widehat{\frak F}_i)$
are torsion pairs of ${\frak C}_i$.
We denote the tiltings by
$\overline{\frak A}_i$ and $\widehat{\frak A}_i$ respectively. 
\end{defn}
Then we have the following equivalence:
\begin{prop}\label{prop:elliptic-equiv}
$\Psi$ induces an equivalence
$\overline{\frak A}_1[-2] \to (\widehat{\frak A}_2)_{op}$.
\end{prop}

For the proof of this proposition,
we need the following properties.
%We set
%${\frak T}_1^*:=\{ E | \rk E=0, \max \chi(G_1,E) \geq 0 \}$.
\begin{lem}\label{lem:elliptic-vanish}
\begin{enumerate}
\item[(1)]
Assume that $E \in \overline{\frak T}_1$.
Then $\Hom(E,{\cal E}_{|\{x' \} \times X})=0$
for a general $x' \in X'$.
\item[(2)]
Assume that $E \in \widehat{\frak F}_1$.
Then $\Hom({\cal E}_{|\{x' \} \times X},E)=
\Hom(E_{ij},E)=0$
for all $x' \in X'$.
In particular if $E \in  \overline{\frak F}_1$, then
$\Hom({\cal E}_{|\{x' \} \times X},E)=\Hom(E_{ij},E)=0$
for all $x' \in X'$.
\end{enumerate}
\end{lem}

\begin{proof}
We only prove (1).
If $\rk E=0$, then obviously the claim holds.
Let $E$ be a torsion free object on $X$
such that $E_{|f}$ is a semi-stable locally free sheaf 
with $\chi(G_1,E_{|f})=0$ for a general $f$.
Then if 
there is a non-zero homomorphism
$\varphi:E \to {\cal E}_{|\{x' \} \times X}$, then
$\varphi$ is surjective and 
$E_{|f}$
is $S$-equivalent to 
${\cal E}_{|\{x' \} \times X} \oplus \ker \varphi$, where
$f={\frak p}^{-1}({\frak q}(x'))$.
Therefore $\Hom(E,{\cal E}_{|\{x' \} \times X})=0$ for a general 
$x' \in {\frak q}^{-1}({\frak p}(f)) \subset Y$.
\end{proof}

\begin{lem}\label{lem:elliptic-Perverse}
Let $E$ be an object of ${\cal C}={\frak C}_1$.
\begin{enumerate}
\item[(1)]
$^p H^i(\Psi(E))=0$ for $i \geq 3$.
\item[(2)]
$H^0(^p H^2(\Psi(E)))=H^2(\Psi(E))$.
\item[(3)]
$^p H^0(\Psi(E)) \subset H^0(\Psi(E))$.
In particular, $^p H^0(\Psi(E))$ is torsion free. 
\item[(4)]
If $\Hom(E,E_{ij}[2])=0$ for all $i,j$
and $\Hom(E,{\cal E}_{|\{x' \} \times X}[2])=0$ for all $x' \in X'$, then
$^p H^2(\Psi(E))=0$.
In particular, if $E \in \widehat{\frak F}_1$,
then $^p H^2(\Psi(E))=0$.
\item[(5)]
If $E$ satisfies $E \in \overline{\frak T}_1$, then
$^p H^0(\Psi(E))=0$.
 
%If $\Hom(E,{\cal E}_{|\{x' \} \times X})=0$ for a general $x' \in X'$,
%then $^p H^0(\Psi(E))=0$.
%In particular, if $E \in {\frak T}^*$, then $^p H^0(\Psi(E))=0$.
\end{enumerate}
\end{lem}

\begin{proof}
By Lemma \ref{lem:elliptic:irreducible-obj},
$E \in \Per(X'/Y')$ is 0 if and only if  
$\Hom(E, A_{ij}')=\Hom(E,{\Bbb C}_{x'})=0$
for all $i$, $j$ and $x' \in X'$. 
Since
\begin{equation}
\begin{split}
\Hom(E,E_{ij}[p])& \cong \Hom(\Psi(E)[p],\Psi(E_{ij})(K_{X'})[2])^{\vee}
\cong \Hom(\Psi(E)[p],\Psi(E_{ij})[2])^{\vee}
,\\
\Hom(E,{\cal E}_{|\{x' \} \times X}[p])
& \cong \Hom(\Psi(E)[p],\Psi({\cal E}_{|\{x' \} \times X})(K_{X'})[2])^{\vee}
\cong \Hom(\Psi(E)[p],\Psi({\cal E}_{|\{x' \} \times X})[2])^{\vee},
\end{split}
\end{equation}
we have (1), (2) and (4).
(3) is obvious.
(5) follows from (3) and Lemma \ref{lem:elliptic-vanish} (1).
\end{proof}

\begin{NB}
\begin{equation}
\begin{split}
\Hom(\Psi(E)[2-p],\Psi(E_{ij})[2])=&
\Hom(\Psi(E)[-p],\Psi(E_{ij}))\\
=& \Hom(\Phi(E)^{\vee}[-p],\Phi(E_{ij})^{\vee})\\
=& \Hom(\Phi(E_{ij})[-p],\Phi(E))\\
=& \Hom(E_{ij}[-p],E).
\end{split}
\end{equation}
\end{NB}

\begin{cor}\label{cor:elliptic-WIT1}
If $E \in \overline{\frak T}_1 \cap \widehat{\frak F}_1$, then
${^p H^i}(\Psi(E))=0$ for $i \ne 1$.
\end{cor}

\begin{lem}\label{lem:elliptic-slope3}
Let $E$ be an object of ${\cal C}$.
\begin{enumerate}
\item[(1)]
If $\WIT_0$ holds for $E$ with respect to $\Psi$,
then $E \in \overline{\frak F}_1$.
\item[(2)]
If $\WIT_2$ holds for $E$ with respect to $\Psi$,
then $E \in \widehat{\frak T}_1$.
\end{enumerate}
\end{lem}

\begin{proof}
For an object $E$ of ${\cal C}$, 
there is an exact sequence
\begin{equation}
0 \to E_1 \to E \to E_2 \to 0
\end{equation}
such that $E_1 \in \overline{\frak T}_1$ and
$E_2 \in \overline{\frak F}_1$.
Applying $\Psi$ to this exact sequence, 
we get a long exact sequence
\begin{equation}\label{eq:elliptic-FM-Psi}
\begin{CD}
0 @>>> \Psi^0(E_2) @>>> 
\Psi^0(E) @>>> 
\Psi^0(E_1)\\
@>>>
\Psi^1(E_2) @>>> 
\Psi^1(E) @>>> 
\Psi^1(E_1)\\
@>>>
\Psi^2(E_2) @>>> 
\Psi^2(E) @>>> 
\Psi^2(E_1)@>>> 0.
\end{CD}
\end{equation}
By Lemma \ref{lem:elliptic-Perverse},
we have
$\Psi^0(E_1)=\Psi^2(E_2)=0$.
If $\WIT_0$ holds for $E$, then we get $\Psi(E_1)=0$.
Hence (1) holds.
If $\WIT_2$ holds for $E$, 
then we get $\Psi(E_2)=0$.
Thus $E \in \overline{\frak T}_1$.
We take a decomposition
\begin{equation}
0 \to E_1' \to E \to E_2' \to 0
\end{equation}
such that $E_1' \in \widehat{\frak T}_1$ and
$E_2' \in \widehat{\frak F}_1 \cap \overline{\frak T}_1$.
Then $\Psi^i(E_2')=0$ for $i \ne 1$ by Corollary \ref{cor:elliptic-WIT1}.
Since $\Psi^0(E_1')=0$, we also get
$\Psi^1(E_2')=0$.
Therefore $E_2'=0$.
\end{proof}

\begin{lem}\label{lem:elliptic-equiv}
\begin{enumerate}
\item[(1)]
If $E \in \overline{\frak T}_1$, then
(1a) $\Psi^0(E)=0$, (1b) $\Psi^1(E) \in \widehat{\frak F}_2$ and
(1c) $\Psi^2(E) \in \widehat{\frak T}_2$. 
\item[(2)]
If $E \in \overline{\frak F}_1$, then
(2a) $\Psi^0(E) \in \widehat{\frak F}_2$, 
(2b) $\Psi^1(E) \in \widehat{\frak T}_2$
and (2c) $\Psi^2(E)=0$.
\end{enumerate}
\end{lem}

\begin{proof}
(1a) and (2c) follow from Lemma \ref{lem:elliptic-Perverse}.
(2a) is easy. 
(1c) By Lemma \ref{lem:spectral2}, $\WIT_2$
holds for $\Psi^2(E)$ with respect to
$\widehat{\Psi}$.
By a similar claim of Lemma \ref{lem:elliptic-slope3} (2),
we get $\Psi^2(E) \in \widehat{\frak T}_2$.
%If $E \in {\frak T}_1$, then
%$\Supp(\Psi^2(E))$ is contained in fibers.
%Since $\WIT_2$ holds for $\Psi^2(E)$ with respect to
%$\widehat{\Psi}$, we see that
%$\Psi^2(E) \in \widehat{\frak T}_2$.

We next study $\Psi^1(E)$ for $E \in {\cal C}$.
Assume that there is an exact sequence
\begin{equation}
0 \to F_1 \to \Psi^1(E) \to F_2 \to 0
\end{equation}
such that
$F_1 \in \widehat{\frak T}_2$ 
\begin{NB}
this means that $\Supp(F_1)$ belongs to fibers. 
\end{NB}
and $F_2 \in \widehat{\frak F}_2$.
Applying $\widehat{\Psi}$, we have a long exact sequence
\begin{equation}
\begin{CD}
0 @>>> \widehat{\Psi}^0(F_2) @>>> 
\widehat{\Psi}^0(\Psi^1(E)) @>>> 
\widehat{\Psi}^0(F_1)\\
@>>>\widehat{\Psi}^1(F_2) @>>> 
\widehat{\Psi}^1(\Psi^1(E)) @>>> 
\widehat{\Psi}^1(F_1)\\
@>>> \widehat{\Psi}^2(F_2) @>>> 
\widehat{\Psi}^2(\Psi^1(E)) @>>> 
\widehat{\Psi}^2(F_1) @>>> 0.
\end{CD}
\end{equation}
By Theorem \ref{thm:elliptic-duality}, we have similar claims
to Lemma \ref{lem:elliptic-Perverse}. Thus
we have $\widehat{\Psi}^0(F_1)=\widehat{\Psi}^2(F_2)=0$.

Assume that $E \in \overline{\frak T}_1$.
Since $\Psi^0(E)=0$, Lemma \ref{lem:spectral2} implies that
$\widehat{\Psi}^2(\Psi^1(E))=0$.
Hence $\WIT_1$ holds for $F_1$.
Since $0 \leq \chi(G_2,F_1)=\chi(\widehat{\Psi}^1(F_1),L_1)
=-(H,c_1(\widehat{\Psi}^1(F_1))) \leq 0$,
$\widehat{\Psi}^1(F_1)$ is a 0-dimensional object.
If $F_1 \ne 0$, then since
$\widehat{\Psi}^1(F_1) \ne 0$,
we see that $0<\chi(G_1,\widehat{\Psi}^1(F_1))=
\chi(F_1,L_2)=-(\widehat{H},c_1(F_1)) \leq 0$,
which is a contradiction.
Therefore $F_1=0$.

Assume that $E \in \overline{\frak F}_1$.
Since $\Psi^2(E)=0$,
Lemma \ref{lem:spectral2} implies that
$\widehat{\Psi}^0(\Psi^1(E))=0$.
Hence $\WIT_1$ holds for $F_2$.
We have an injection
$\widehat{\Psi}^1(\Psi^1(E)) \to
E$.
Since $\mu_{\max,G_1}(E)<0$,
$\Psi^1(E)$ is zero on a generic fiber of ${\frak p}$.
Hence $\widehat{\Psi}^1(\Psi^1(E))$ is a torsion object.
Since $E$ is torsion free,
$\widehat{\Psi}^1(\Psi^1(E))=0$.
Since $\widehat{\Psi}^0(F_1)=0$, we get 
$\widehat{\Psi}^1(F_2)=0$, which implies that $F_2=0$.
\end{proof}

{\it Proof of Proposition \ref{prop:elliptic-equiv}}.

It is sufficient to prove that
$\Psi(\overline{\frak T}_1[-2]),\Psi(\overline{\frak F}_1[-1])
\subset (\widehat{\frak A}_2)_{op}$.
Then the claims follow from Lemma \ref{lem:elliptic-equiv}.
\qed

\subsection{Preservation of Gieseker stability conditions.}

We give a generalization of \cite[Thm. 3.15]{Y:7}.
We first recall the following well-known fact.
\begin{lem}\label{lem:elliptic-stability}
\begin{enumerate}
\item[(1)]
Let $E$ be a torsion free object of ${\cal C}$.
Then $E$ is $G_1$-twisted semi-stable with respect to
$H+nf$, $n \gg 0$ if and only if
for every proper object $E'$ of $E$,
one of the following conditions holds:
\begin{enumerate}
\item
\begin{equation}
\frac{(c_1(E),f)}{\rk E}>\frac{(c_1(E'),f)}{\rk E'},
\end{equation}
\item
\begin{equation}
\frac{(c_1(E),f)}{\rk E}=\frac{(c_1(E'),f)}{\rk E'},\;
\frac{(c_1(E),H)}{\rk E}>\frac{(c_1(E'),H)}{\rk E'},
\end{equation}
\item
\begin{equation}
\frac{(c_1(E),f)}{\rk E}=\frac{(c_1(E'),f)}{\rk E'},\;
\frac{(c_1(E),H)}{\rk E}=\frac{(c_1(E'),H)}{\rk E'},\;
\frac{\chi(G_1,E)}{\rk E} \geq \frac{\chi(G_1,E')}{\rk E'}.
\end{equation}
\end{enumerate}
\item[(2)]
Let $F$ be a 1-dimensional object of $\Per(X'/Y')$
with $(c_1(F),f) \ne 0$.
Then $F$ is $G_2$-twisted semi-stable with respect to
$\widehat{H}+nf$, $n \gg 0$ if and only if
for every proper subobject $F'$ of $F$, one of the following
conditions holds:
\begin{enumerate}
\item
\begin{equation}
(c_1(F'),f)\frac{\chi(G_2,F)}{(c_1(F),f)}>\chi(G_2,F')
\end{equation}
\item
\begin{equation}
(c_1(F'),f)\frac{\chi(G_2,F)}{(c_1(F),f)}=\chi(G_2,F'),\;
%\frac{\chi(G_2,F)}{(c_1(F),f)}>\frac{\chi(G_2,F')}{(c_1(F'),f)},\;
(c_1(F'),\widehat{H})\frac{\chi(G_2,F)}{(c_1(F),\widehat{H})}>\chi(G_2,F').
\end{equation}
\end{enumerate}
\end{enumerate}
\end{lem}

%If $\deg_{G_1}(E)=0$, then 
%$\rk E=\chi(E,{\Bbb C}_x)=
%\chi({\cal E}_{|X' \times \{x \}},\Psi(E))
%=(c_1(\Psi(E)[1]),rf)$.

\begin{lem}\label{lem:elliptic-vanish:pure}
Let $F$ be a purely 1-dimensional 
$G_2$-twisted semi-stable object such that
$(c_1(F),f)>0$ and $\chi(G_2,F)<0$. Then
$\WIT_1$ holds for $F$ with respect to $\widehat{\Psi}$
and $\widehat{\Psi}^1(F)$ is torsion free.
\end{lem}

\begin{proof}
By Lemma \ref{lem:elliptic-stability} (2),
$F \in \widehat{\frak F}_2$.
By Theorem \ref{thm:elliptic-duality},
similar claims to Lemma \ref{lem:elliptic-Perverse},
Corollary \ref{cor:elliptic-WIT1} and Lemma \ref{lem:elliptic-slope3}
hold for $\widehat{\Psi}$.
Hence $\WIT_1$ holds for $F$.
Assume that there is an exact sequence
\begin{equation}
0 \to E_1 \to \widehat{\Psi}^1(F) \to E_2 \to 0
\end{equation}
such that $E_1$ is the torsion object of $\widehat{\Psi}^1(F)$.
Since $\widehat{\Psi}^1(F)_{|f}$ is a semi-stable
vector bundle of $\deg (G_1^{\vee} \otimes \widehat{\Psi}^1(F)_{|f})=0$
for a general fiber $f$ of ${\frak p}$,
$\Supp(E_1)$ is contained in fibers.
Since $E_1 \in \overline{\frak T}_1$ and
$E_2 \in \widehat{\frak F}_1$, $\WIT_1$ holds for $E_1$,
$E_2$ and we have a quotient
$F \to \Psi^1(E_1)$.
By our assumption on $F$,
we get $\chi(G_2, \Psi^1(E_1)) \geq 0$.
On the other hand, $\chi(G_2, \Psi^1(E_1))=\chi(E_1,L_1)
=-(H,c_1(E_1)) \leq 0$. Hence $E_1$ is a 0-dimensional object.
Then we get
$0<\chi(G_1,E_1)=\chi(\Psi^1(E_1),L_2)=-(\widehat{H},c_1(\Psi^1(E_1)))
\leq 0$, which is a contradiction.
\begin{NB}
Assume that $\Supp(F)$ contains a fiber $f$.
Let $F''$ be the purely 1-dimensional quotient object
of $F_{|f}$.
Then we have
$\Hom(F'',{\cal E}_{|X' \times \{ x \}}) \cong 
\Hom(F,{\cal E}_{|X' \times \{ x \}})$.
By the $G_2$-twisted semi-stability of $F$,
we see that $\chi(G_2,F''') \geq 0$ for any quotient $F'''$ of $F''$.
Then we see that 
$\Hom(F'',{\cal E}_{|X' \times \{ x \}})=0$ except finite set of points
of $Y$. 
In particular, $\widehat{\Psi}^0(F)=0$
and $\widehat{\Psi}^1(F)$ is torsion free.
Assume that $\chi(G_2,F)<0$. Then
$\Ext^2(F,{\cal E}_{|X' \times \{ x \}})=
\Hom({\cal E}_{|X' \times \{ x \}},F)^{\vee}=0$.
Therefore 
$\WIT_1$ holds for $F$ with respect to $\widehat{\Psi}$.
\end{NB}
\end{proof}

\begin{lem}\label{lem:elliptic-relation1}
Let $F$ be a 1-dimensional object of $\Per(X'/Y')$.
Then
\begin{equation}\label{eq:elliptic-relation1}
\begin{split}
(c_1(F),f)&=\rk (\widehat{\Psi}(F)[1]),\\
(c_1(F),\widehat{H})&=-\chi(F,L_2)=-\chi(G_1,\widehat{\Psi}(F)[1]),\\
\chi(G_2,F)&=\chi(\widehat{\Psi}(F)[1],L_1)=
-(c_1(\widehat{\Psi}(F)[1]),H)+
\rk(\widehat{\Psi}(F)[1])\chi(L_1).
\end{split}
\end{equation}
\end{lem}

\begin{prop}
Let $w \in K(X')_{\mathrm{top}}$ 
be a topological invariant of a 1-dimensional object.
Assume that $\chi(G_2,w)<0$.
Then for $n \gg 0$, we have an isomorphism
\begin{equation}
{\cal M}_{H+nf}^{G_1}(\widehat{\Psi}(-w))^{ss} \to
{\cal M}_{\widehat{H}+nf}^{G_2}(w)^{ss}, 
\end{equation} 
which preserves the $S$-equivalence classes.
\end{prop}

%\begin{prop}
%\begin{enumerate}
%\item[(1)]
%Let $v \in K(X)_{\mathrm{top}}$ 
%be the topological invariant of an 1-dimensional object.
%Assume that $\chi(G_1,v)<0$.
%Then we have an isomorphism
%\begin{equation}
%{\cal M}_H^{G_1}(v)^{ss} \to
%{\cal M}_{\widehat{H}}^{G_2}(\Psi(-v))^{ss}, 
%\end{equation} 
%which preserves the $S$-equivalence classes.
%\item[(2)]
%Let $w \in K(X')_{\mathrm{top}}$ 
%be a topological invariant of an 1-dimensional object.
%Assume that $\chi(G_2,w)<0$.
%Then we have an isomorphism
%\begin{equation}
%{\cal M}_H^{G_1}(\widehat{\Psi}(-w))^{ss} \to
%{\cal M}_{\widehat{H}}^{G_2}(w)^{ss}, 
%\end{equation} 
%which preserves the $S$-equivalence classes.
%\end{enumerate}
%\end{prop}

 \begin{proof}
%We only prove (2).
Let $E$ be a $G_1$-twisted semi-stable object with 
$\tau(E)=\widehat{\Psi}(-w)$. 
Then since $E_{|f}$ is a semi-stable locally free sheaf with
$d\rk E-r\deg(E_{|f})=0$ for a general fiber,
we have 
$E \in \overline{\frak T}_1 \cap \widehat{\frak F}_1$.
By Corollary \ref{cor:elliptic-WIT1},
$\WIT_1$ holds for $E$ with respect to $\Psi$.
Assume that there is an exact sequence
\begin{equation}
0 \to F_1 \to \Psi^1(E) \to F_2 \to 0.
\end{equation}
By Lemma \ref{lem:elliptic-equiv},
$\Psi^1(E) \in \widehat{\frak F}_2$,
which implies that
$F_1 \in \widehat{\frak F}_2$.
Since $\rk \Psi^1(E)=0$, $F_1, F_2 \in \overline{\frak T}_2$.
In particular, $F_1 \in \overline{\frak T}_2 \cap \widehat{\frak F}_2$.
Then similar claim to Corollary \ref{cor:elliptic-WIT1} implies that
$\WIT_1$ holds for $F_1$.
Hence we get an exact sequence
\begin{equation}
0 \to \widehat{\Psi}^1(F_2) \to E \overset{\varphi}{\to}
 \widehat{\Psi}^1(F_1) \to \widehat{\Psi}^2(F_2) \to 0.
\end{equation} 
By Lemma \ref{lem:elliptic-equiv},
$\widehat{\Psi}^2(F_2) \in \widehat{\frak T}_1$.
Hence $\rk \widehat{\Psi}^1(F_1)=\rk \im \varphi$.
By \eqref{eq:elliptic-relation1}, we have the following equivalences.
\begin{equation}\label{eq:elliptic-1}
(c_1(F_1),f)\frac{\chi(G_2,\Psi^1(E))}{(c_1(F),f)} 
\leq \chi(G_2,F_1) \Longleftrightarrow
\rk \widehat{\Psi}^1(F_1)\frac{(c_1(E),H)}{\rk E} 
\geq (c_1(\widehat{\Psi}^1(F_1)),H),
\end{equation}
\begin{equation}\label{eq:elliptic0}
(c_1(F_1),\widehat{H})
\frac{\chi(G_2,{\Psi}^1(E))}{(c_1({\Psi}^1(E)),\widehat{H})} 
\leq \chi(G_2,F_1)
\Longleftrightarrow 
-\chi(G_1,\widehat{\Psi}^1(F_1))
\frac{\chi(G_2,{\Psi}^1(E))}{-\chi(G_1,E)} 
\leq \chi(G_2,F_1).
\end{equation}
If the equality holds in \eqref{eq:elliptic-1}, then
$\chi(G_2,\Psi^1(E))<0$ implies that
\eqref{eq:elliptic0} is equivalent to
\begin{equation}
\frac{\chi(G_1,\widehat{\Psi}^1(F_1))}{\chi(G_1,E)}
\geq \frac{ \rk \widehat{\Psi}^1(F_1)}{\rk E}
\end{equation}
which is equivalent to
\begin{equation}
\frac{\chi(G_1,\widehat{\Psi}^1(F_1))}{ \rk \widehat{\Psi}^1(F_1)}
\leq \frac{\chi(G_1,E)}{\rk E}
\end{equation}
by $-\chi(G_1,E)>0$.
Since
\begin{equation}
\frac{\chi(G_1,\im \varphi(nH))}{\rk \im \varphi}
\leq 
\frac{\chi(G_1,\widehat{\Psi}^1(F_1)(nH))}
{\rk \widehat{\Psi}^1(F_1)},\;n \gg 0,
\end{equation}
we see that $\varphi$ is surjective and the equalities hold
for \eqref{eq:elliptic-1}, \eqref{eq:elliptic0}.
\begin{NB}
we see that
$(c_1(\widehat{\Psi}^2(F_2)),H)=0$ and $\chi(G_1,\widehat{\Psi}^2(F_2))=0$,
which implies that $\widehat{\Psi}^2(F_2)=0$.
\end{NB}
Therefore $\Psi^1(E)$ is $G_2$-twisted semi-stable.

Conversely let $F$ be a $G_2$-twisted semi-stable object
with $\tau(F)=w$.
By Lemma \ref{lem:elliptic-vanish:pure},
$\WIT_1$ holds for $F$ with respect to
$\widehat{\Psi}$ and $\widehat{\Psi}^1(F)$ is a torsion free
object whose restriction to a general fiber is stable.
If $\widehat{\Psi}^1(E)$ is not $G_1$-twisted semi-stable,
then we have an exact sequence
\begin{equation}
0 \to E_1 \to \widehat{\Psi}^1(F) \to E_2 \to 0
\end{equation}
such that $E_i \in \overline{\frak T}_1 \cap \widehat{\frak F}_1$.
By using Lemme \ref{lem:elliptic-relation1},
we get the following equivalences:
\begin{equation}\label{eq:elliptic1}
\frac{(c_1(\widehat{\Psi}^1(F)),{H})}{\rk \widehat{\Psi}^1(F)} \leq
\frac{(c_1(E_1),{H})}{\rk E_1}
\Longleftrightarrow 
\frac{\chi(G_2,F)}{(c_1(F),f)} \geq 
\frac{\chi(G_2,{\Psi}^1(E_1))}
{(c_1({\Psi}^1(E_1)),f)},
\end{equation}

\begin{equation}\label{eq:elliptic2}
\frac{\chi(G_1,\widehat{\Psi}^1(F))}{\rk \widehat{\Psi}^1(F)} \leq
\frac{\chi(G_1,E_1)}{\rk E_1}
\Longleftrightarrow 
\frac{(c_1(F),\widehat{H})}{(c_1(F),f)} \geq 
\frac{(c_1({\Psi}^1(E_1)),\widehat{H})}
{(c_1({\Psi}^1(E_1)),f)}.
\end{equation}
If the equality holds in \eqref{eq:elliptic1}, then
\eqref{eq:elliptic2} is equivalent to
\begin{equation}
\frac{\chi(G_2,F)}{(c_1(F),\widehat{H})} \geq 
\frac{\chi(G_2,{\Psi}^1(E_1))}{(c_1({\Psi}^1(E_1)),\widehat{H})}
\end{equation}
by $\chi(G_2,F)<0$. 
Therefore $\widehat{\Psi}^1(F)$ is $G_1$-twisted semi-stable.
\end{proof}

\begin{NB}

\begin{lem}
Let $E$ be a purely 1-dimensional 
$G_1$-twisted semi-stable object of ${\cal C}$.
If $\chi(G_1,E)<0$, then
$\WIT_1$ holds for $E$ and $\Psi^1(E)$ is a 
$G_2$-twisted semi-stable object of $\Per(X'/Y')$.
\end{lem}

\begin{proof}
Since $\chi(G_1,E)<0$, we get that
$\Hom(E,E_{ij}[2])=\Hom(E,{\cal E}_{|\{ x' \} \times X}[2])=0$
for all $i,j$ and $x' \in X'$.
Lemma \ref{lem:elliptic-Perverse} implies that
$\Psi^i(E)=0$ for $i \ne 1$.
%We first note that
%\begin{equation}
%\Hom(H^1(\Psi(E)),B_{ij})=\Hom(\Psi(E)[1],\Psi(E_{ij})[1])=
%\Hom(E_{ij},E)=0.
%\end{equation}
%Hence $H^1(\Psi(E)) \in \Per(Y/W,{\bf b})$, which implies that
%$\Psi^1(E)=H^1(\Psi(E))$.
Assume that there is an exact sequence
\begin{equation}
0 \to F_1 \to \Psi^1(E) \to F_2 \to 0
\end{equation}
such that $\mu_{G_2}(F_2)=0$ and $F_2$ is purely 2-dimensional.
By Lemma \ref{lem:elliptic-vanish},
$\WIT_1$ holds for $F_1$
%, $\widehat{\Psi}^2(F_2)$ is supported on fibers 
and we get an exact sequence
\begin{equation}
0 \to \widehat{\Psi}^1(F_2) \to E \to
 \widehat{\Psi}^1(F_1) \to 
%\widehat{\Psi}^2(F_2) \to 
0.
\end{equation} 
If $\rk F_1=0$, then
$(c_1(F_1),f)=0$ and
$0 \leq (c_1(F_1),\widehat{H})=-\chi(G_1,\widehat{\Psi}^1(F_1))$.
By the semi-stability of $E$,
we get $(c_1(F_1),\widehat{H})=-\chi(G_1,\widehat{\Psi}^1(F_1))=0$.
Then $F_1$ is 0-dimensional, which implies that
$-(c_1(\widehat{\Psi}^1(F_1)),H)=\chi(G_2,F_1)>0$, 
which is a contradiction. 
Therefore $\Psi^1(E)$ is purely 2-dimensional.
By using similar claims to Lemme \ref{lem:elliptic-relation1},
we get the following equivalences:
\begin{equation}\label{eq:elliptic3}
\frac{(c_1(\Psi^1(E)),\widehat{H})}{\rk \Psi^1(E)} \leq
\frac{(c_1(F_1),\widehat{H})}{\rk F_1}
\Longleftrightarrow 
\frac{\chi(G_1,E)}{(c_1(E),f)} \geq 
\frac{\chi(G_1,\widehat{\Psi}^1(F_1))}
{(c_1(\widehat{\Psi}^1(F_1)),f)}.
\end{equation}

\begin{equation}\label{eq:elliptic4}
\frac{\chi(G_2,\Psi^1(E))}{\rk \Psi^1(E)} \leq
\frac{\chi(G_2,F_1)}{\rk F_1}
\Longleftrightarrow 
\frac{(c_1(E),H)}{(c_1(E),f)} \geq 
\frac{(c_1(\widehat{\Psi}^1(F_1)),H)}{(c_1(\widehat{\Psi}^1(F_1)),f)}.
\end{equation}
If the equality holds in \eqref{eq:elliptic3}, then
\eqref{eq:elliptic4} is equivalent to
\begin{equation}
\frac{\chi(G_1,E)}{(c_1(E),H)} \geq 
\frac{\chi(G_1,\widehat{\Psi}^1(F_1))}{(c_1(\widehat{\Psi}^1(F_1)),H)}
\end{equation}
by $\chi(G_1,E)<0$. 
Therefore $\Psi^1(E)$ is $G_2$-twisted semi-stable.
\end{proof}

\end{NB}

\section{A category of equivariant coherent sheaves.}
\label{sect:equivariant}

\subsection{Morita equivalence for $G$-sheaves.}

Let $X$ be a smooth projective surface
and $G$ a finite group acting on $X$.
Assume that $G \to \Aut(X)$ is injective and 
$\Stab(x)$, $x \in X$ acts trivally on $(K_X)_{|\{x \}}$,
that is, $K_X$ is the pull-back of a line bundle on
$Y:=X/G$. By our assumption, all elements of $G$ have at most
isolated fixed points sets.
\begin{NB}
$K_X$ is equivariantly locally trivial:
Since $X$ is projective, for an orbit $Gx$,
there is an affine open subset
$U$ containing $Gx$.
Since $H^0(U,K_U) \to H^0((K_U)_{|Gx})$ is surjective,
there is a homomorphism $f:{\cal O}_U \to K_U$ which is injective 
at $Gx$ and $f_{|Gx}$ is $G$-invariant. 
Since $\cap_{g \in G} g(U)$ is a $G$-invariant subscheme of
$X$, we may assume that $U$ is $G$-invariant.
Replacing $f$ by $\sum_{g \in G} g^*(f)/\# G$, we may assume that
$f$ is $G$-invariant.
Replacing $U$ by an open subscheme again,
we may assume that $f$ is isomorphic. Thus we get a desired 
trivialization. 
\end{NB}
Let $R(G)$ be the representation ring of $G$ and
$(\quad,\quad)$ the natural inner product.
Let $K_G(X)$ be the Grothendieck group of $G$-sheaves
and $K_G(X)_{\mathrm{top}}$ its image to the 
Grothendieck group of topological $G$-vector bundles.
\begin{defn}
For $G$-sheaves $E$ and $F$ on $X$,
\begin{enumerate}
\item[(1)]
 $\GExt^i(E,F)$ is the $G$-invariant part of
$\Ext^i(E,F)$. 
\item[(2)]
$\Gchi(E,F):=\sum_i (-1)^i \dim \GExt^i(E,F)$ is the Euler characteristic
of the $G$-invariant cohomology groups of $E, F$.
We also set $\Gchi(E):=\Gchi({\cal O}_X,E)$.
\end{enumerate}
\end{defn}
\begin{rem}
If $K_X \cong {\cal O}_X$ in $\Coh_G(X)$, then
$\Gchi(\;\;,\;\; )$ is symmetric.
\end{rem}
Let $\varpi:X \to Y$ be the quotient map.
We set
\begin{equation}
\begin{split}
\varpi_*({\cal O}_X)[G]:
%=&\varpi_*({\cal O}_X \otimes {\Bbb C}[G])\\
=&\left\{\left.
\sum_{g \in G} f_g(x)g \right| f_g(x) \in \varpi_*({\cal O}_X) \right\}.
\end{split}
\end{equation}
$\varpi_*({\cal O}_X)[G]$ is an ${\cal O}_Y$-algebra
whose multiplication is defined by
\begin{equation}
(\sum_{g \in G} f_g(x)g)\cdot(\sum_{g' \in G} f_{g'}'(x)g'):=
\sum_{g,g' \in G}f_g(x)f_{g'}'(g^{-1}x)gg'.
\end{equation}

We note that $\epsilon:=\frac{1}{\# G}\sum_{g \in G}g$ satisfies
$g \epsilon=\epsilon$ for all $g \in G$. 
By the injective homomorphism
\begin{equation}
\varpi_*({\cal O}_X) \to \varpi_*({\cal O}_X)\epsilon \;
(\subset\varpi_*({\cal O}_X)[G]),
\end{equation}
we have an action of $\varpi_*({\cal O}_X)[G]$ 
on $\varpi_*({\cal O}_X)$:
\begin{equation}
(\sum_{g \in G} f_g(x)g)\cdot f(x):=\sum_{g \in G} f_g(x)f(g^{-1}x). 
\end{equation} 
Thus we have a homomorphism
\begin{equation}
\varpi_*({\cal O}_X)[G] \to 
\Hom_{{\cal O}_Y}(\varpi_*({\cal O}_X),\varpi_*({\cal O}_X)).
\end{equation}

\begin{lem}
$\varpi_*({\cal O}_X)[G] \cong 
\Hom_{{\cal O}_Y}(\varpi_*({\cal O}_X),\varpi_*({\cal O}_X))$.
\end{lem}

\begin{proof}
We first prove the claim over
the smooth locus $Y^{\mathrm{sm}}$ of $Y$.
We note that $\# \varpi^{-1}(y)=\#G$, $y \in Y^{\mathrm{sm}}$.
We take a point $z \in  \varpi^{-1}(y)$. 
Then $\varpi_*({\cal O}_X)_{|y}={\cal O}_{\varpi^{-1}(y)}$ is identified 
with $\oplus_{g \in G}{\Bbb C}_{gz}$ as ${\Bbb C}[G]$-modules.
Let $\chi_u(x)$ be the characteristic function of a point $u \in X$.
Then $\{ \chi_{gz}|g \in G \}$ is the base of
$\oplus_{g \in G}{\Bbb C}_{gz}$ and 
$f(x) \in {\cal O}_{\varpi^{-1}(y)}$ is decomposed into
$f(x)=\sum_{g \in G}f(gz)\chi_{gz}(x)$.
Since
\begin{equation}
(\chi_{g'z}(x)(g'g^{-1}))\cdot(\sum_{h \in G}f(hz)\chi_{hz}(x))
=f(gz)\chi_{g'z}(x),
\end{equation}
we see that
\begin{equation}
(\varpi_*({\cal O}_X)[G])_{|y} \to 
\Hom(\varpi_*({\cal O}_X)_{|y},\varpi_*({\cal O}_X)_{|y})
\end{equation}
is an isomorphism.
Since $\varpi_*({\cal O}_X)[G]$ and 
$\Hom_{{\cal O}_Y}(\varpi_*({\cal O}_X),\varpi_*({\cal O}_X))$ 
are reflexive sheaves
on $Y$, we get the claim.
\end{proof} 

We set ${\cal A}:=\varpi_*({\cal O}_X)[G] \cong
\Hom_{{\cal O}_Y}(\varpi_*({\cal O}_X),\varpi_*({\cal O}_X))$.

\begin{lem}\label{lem:G-equivalence}
We have an equivalence
\begin{equation}
\begin{matrix}
\varpi_*: &\Coh_G(X) & \cong & \Coh_{{\cal A}}(Y)\\
& E & \mapsto & \varpi_*(E)
\end{matrix}
\end{equation}
whose inverse is $\varpi^{-1}:\Coh_{{\cal A}}(Y)
\to \Coh_G(X)$.
In particular, we have an isomorphism
\begin{equation}
\Hom_G(E_1,E_2)=\Hom_{{\cal A}}(\varpi_*(E_1),\varpi_*(E_2)),\;
E_1,E_2 \in \Coh_G(X).
\end{equation}
\end{lem}

\begin{proof}
Since the problem is local, 
we may assume that $Y$ is affine.
Then $X$ is also affine. 
For $F \in \Coh_{{\cal A}}(Y)$,
$H^0(Y,F)$ is a $H^0(Y,\varpi_*({\cal O}_X))[G]$-module.
Hence $H^0(X,\varpi^{-1}(F))=H^0(Y,F)$ is a 
$H^0(X,{\cal O}_X)[G]$-module, which implies that
$\varpi^{-1}(F) \in \Coh_G(X)$. 
Then it is easy to see that
$\varpi^{-1}$ is the inverse of $\varpi_*$.
\end{proof}

By Lemma \ref{lem:G-equivalence}, we have an equivalence
$\varpi_*:{\bf D}_G(X) \to {\bf D}_{{\cal A}}(Y)$.
In particular, 
\begin{equation}
\chi_G(E_1,E_2)=
\sum_i (-1)^i \dim \Hom_{{\cal A}}(\varpi_*(E_1),\varpi_*(E_2)[i]),\;
E_1,E_2 \in \Coh_G(X).
\end{equation}

For a representation $\rho:G \to GL(V_{\rho})$ of $G$,
we define a $G$-linearization on
${\cal O}_X \otimes V_{\rho}$ in a usual way.
Thus we define the action of $G$ on 
$\varpi_*({\cal O}_X \otimes V_{\rho})$ as 
\begin{equation}
g \cdot (f(x) \otimes v):=f(g^{-1}x) \otimes gv,\; 
g \in G, f(x)\in \varpi_*({\cal O}_X), v \in V_{\rho}.
\end{equation}
Then ${\cal O}_X \otimes {\Bbb C}[G]$ is a $G$-sheaf
such that $\varpi_*({\cal O}_X \otimes {\Bbb C}[G])={\cal A}$
and we have a decomposition  
\begin{equation}
{\cal O}_X \otimes {\Bbb C}[G]=
\bigoplus_i ({\cal O}_X \otimes V_{\rho_i})^{\oplus \dim \rho_i},
\end{equation}
where $\rho_i$ are irreducible representations of $G$.
\begin{defn}
For a $G$-sheaf $E$ and a representation 
$\rho:G \to GL(V_{\rho})$,
$E \otimes \rho$ denotes the $G$-sheaf 
$E \otimes_{{\cal O}_X}({\cal O}_X \otimes V_{\rho})$.
\end{defn}
Since $\varpi_*({\cal O}_X \otimes \rho_i)$ 
are direct summands of 
${\cal A}$,
we get the following lemma.
\begin{lem}
\begin{enumerate}
\item[(1)]
${\cal A}_i:=\varpi_*({\cal O}_X \otimes \rho_i)$ are local projective objects 
of $\Coh_{{\cal A}}(Y)$.
\item[(2)]
$\bigoplus_i \varpi_*({\cal O}_X \otimes \rho_i)^{\oplus r_i}$
is a local projective generator of $\Coh_{{\cal A}}(Y)$ if and only
if $r_i>0$ for all $i$.
\end{enumerate}
\end{lem}

For a local projective generator ${\cal B}$ of $\Coh_{{\cal A}}(Y)$,
we set ${\cal A}':=
{\cal H}om_{{\cal A}}({\cal B},{\cal B})$.
Then we have an equivalence
\begin{equation}
\begin{matrix}
\Coh_{{\cal A}}(Y)& \to & \Coh_{{\cal A}'}(Y)\\
E & \mapsto & {\cal H}om_{\cal A}({\cal B},E).
\end{matrix}
\end{equation}

\subsection{Stability for $G$-sheaves.}

Let $\alpha$ be an element of $R(G) \otimes {\Bbb Q}$.

\begin{defn}
Let ${\cal O}_X(1)$ be the pull-back of an ample
line bundle on $Y$.
A coherent $G$-sheaf $E$ is $\alpha$-stable, if
$E$ is purely $d$-dimensional and
\begin{equation}
\frac{\Gchi(F(n) \otimes \alpha^{\vee})}{a_d(F)} <
\frac{\Gchi(E(n) \otimes \alpha^{\vee})}{a_d(E)},\;n \gg 0
\end{equation}
for all proper subsheaf $F \ne 0$,
where $a_d(*)$ is the coefficient of $n^d$ of 
the Hilbert polynomial $\chi_G(*(n))$.
We also define the $\alpha$-semi-stability as usual.
\end{defn}

\begin{rem}
Assume that $\alpha=\sum_i r_i \rho_i$, $r_i>0$.
We set ${\cal B}:=\oplus_i {\cal A}_i^{\oplus r_i}$ and
${\cal A}':={\cal H}om_{{\cal A}}({\cal B},{\cal B})$.
Under the equivalence
\begin{equation}
\begin{matrix}
\Coh_G(X) & \to & \Coh_{{\cal A}'}(Y)\\
 E & \mapsto & {\cal H}om_{{\cal A}}({\cal B},\varpi_*(E)),
\end{matrix}
\end{equation}
\begin{equation}
\chi_G(E(n) \otimes \alpha^{\vee})=
\chi({\cal H}om_{{\cal A}}
({\cal B},\varpi_*(E))(n))
\end{equation}
implies that 
$\alpha$-twisted stability of $E$ corresponds to the stability
of ${\cal A}'$-module ${\cal H}om_{{\cal A}}
({\cal B},\varpi_*(E))$.
\end{rem}

For a coherent $G$-sheaf $E$ of dimension 0,
we also have a refined notion of stability,
which also comes from the stability of 0-dimensional
objects in $\Coh_{{\cal A}}(Y)$.
\begin{defn}
Let $\rho_{\mathrm{reg}}$ 
be the regular representation of $G$.
A coherent $G$-sheaf $E$ of dimension 0 is
$(\rho_{\mathrm{reg}},\alpha)$-stable, if
\begin{equation}
\frac{\Gchi(F \otimes \alpha^{\vee})}
{\Gchi(F \otimes \rho_{\mathrm{reg}}^{\vee})}
<
\frac{\Gchi(E \otimes \alpha^{\vee})}
{\Gchi(E \otimes \rho_{\mathrm{reg}}^{\vee})}
\end{equation}
for a proper subsheaf $F \ne 0$.
\end{defn}
By \cite[Thm. 4.7]{S:1} and 
Proposition \ref{prop:A-module}, we get the following theorem.
\begin{thm}
\begin{enumerate}
\item[(1)]
Assume that $n\alpha$ contains every irreducible representation
for a sufficiently large $n$. Then
there is a coarse moduli space
$\overline{M}_H^{\alpha}(v)$ of
$\alpha$-semi-stable $G$-sheaves $E$ with
$v(E)=v$.
$\overline{M}_H^{\alpha}(v)$ is a projective scheme.
We denote 
the open subscheme consisting of $\alpha$-stable $G$-sheaves
by ${M}_H^{\alpha}(v)$.
\item[(2)]
Assume that $v$ is a 0-dimensional vector.
Then there is a coarse moduli space
$\overline{M}_H^{\rho_{\mathrm{reg}},\alpha}(v)$ of
$(\rho_{\mathrm{reg}},\alpha)$-semi-stable $G$-sheaves $E$ with
$v(E)=v$.
$\overline{M}_H^{\rho_{\mathrm{reg}},\alpha}(v)$ is a projective scheme.
We denote 
the open subscheme consisting of 
$(\rho_{\mathrm{reg}},\alpha)$-stable $G$-sheaves
by ${M}_H^{\rho_{\mathrm{reg}},\alpha}(v)$.
\item[(3)]
If $K_X \cong {\cal O}_X$ in $\Coh_G(X)$, then 
${M}_H^{\alpha}(v)$ and 
$M_H^{\rho_{\mathrm{reg}},\alpha}(v)$ are smooth of dimension
$-\chi_G( v,v)+2$ with holomorphic
symplectic structures.
\end{enumerate} 
\end{thm}

\begin{rem}
There is another construction due to Inaba \cite{Inaba}.
\end{rem}

For a smooth point $y$ of $Y$, let
$v_0$ be the topological invariant of ${\cal O}_{\varpi^{-1}(y)}$.
%\begin{defn}
%A $G$-sheaf $E$ of dimension 0 is irreducible if 
%$E$ does not contain a proper $G$-subsheaf.
%\end{defn}
%
%
\begin{lem}
A 0-dimensional $G$-sheaf $E$ is $v_0$-twisted stable
if and only if $E$ is an irreducible object of $\Coh_G(X)$.
\end{lem}
\begin{proof}
Let $E$ be a $G$-sheaf of dimension 0.
Then 
$\Gchi(E \otimes v_0^{\vee})/\Gchi(E \otimes \rho_{reg}^{\vee})=1$.
Hence the claim holds.
\end{proof}

\begin{defn}
Let $\GHilb_X^{\rho}$ be the $G$-Hilbert scheme
parametrizing 0-dimensional subschemes $Z$ of $X$ such that
$H^0(X,{\cal O}_Z) \cong V_{\rho}$.
\end{defn}
Let $\rho_0,\rho_1,\dots,\rho_n$ be the irreducible representations of $G$.
Assume that $\rho_0$ is trivial.
We take an $\alpha$ such that $(\alpha,v_0)=0$
and $(\alpha,\rho_i)<0$ for $i>0$.
\begin{lem}
$M_H^{\rho_{\mathrm{reg}},\alpha}(v_0)
=\GHilb_X^{\rho_{\mathrm{reg}}}$.
In particular, $M_H^{\rho_{\mathrm{reg}},\alpha}(v_0) \ne \emptyset$.
\end{lem}

\begin{proof}
Let $E$ be a $G$-sheaf with $v(E)=v_0$.
Since $\Gchi({\cal O}_X \otimes \rho_0,E)=1$, we have a homomorphism
$\phi:{\cal O}_X \otimes \rho_0 \to E$.
Then $H^0(\im \phi)$ contains a trivial representation,
which implies that $\Gchi({\cal O}_X \otimes \rho_0,\im \phi) \geq 1$.
We note that $E$ belongs to $M_H^{\rho_{\mathrm{reg}},\alpha}(v_0)$
if and only if $E$ does not contain a proper subsheaf $F$ with
$\Gchi({\cal O}_X \otimes \rho_0,F) \geq 1$.
Hence if $E \in M_H^{\rho_{\mathrm{reg}},\alpha}(v_0)$, then 
$\im \phi=E$, which implies that $E \in \GHilb_X^{\rho_{\mathrm{reg}}}$.
Conversely, if $E \in \GHilb_X^{\rho_{reg}}$, then
for a subsheaf $F$
with $\Gchi({\cal O}_X \otimes \rho_0,F) \geq 1$,
$\GHom({\cal O}_X \otimes \rho_0,F) \to 
\GHom({\cal O}_X \otimes \rho_0,E)$ is isomorphic.
Hence $\phi$ factors through $F$.
Since $E$ is generated by the image of $\phi$,
$F=E$. Thus $E$ is stable. 
\end{proof}
We set $X':=M_H^{\rho_{\mathrm{reg}},\alpha}(v_0)$.
Let $Y'$ be the normalization of 
$\overline{M}_H^{\rho_{\mathrm{reg}},0}(v_0)$.
Then we have a morphism
$\pi:X' \to Y'$.

\begin{prop}
\begin{enumerate}
\item[(1)]
$Y' \to \overline{M}_H^{\rho_{\mathrm{reg}},0}(v_0)$ is a bijective morphism.
\item[(2)]
Let $\{p_1,p_2,\dots,p_l \}$ be the set of singular points of $Y'$.
Then each $p_i$ corresponds to $S$-equivalence classes of
properly $v_0$-twisted semi-stable $G$-sheaves.
Let $\oplus_{j=0}^{s_i} E_{ij}^{\oplus a_{ij}}$ be the $S$-equivalence class
corresponding to $p_i$.
Then the matrix $(\chi_G(E_{ij},E_{ij'}))_{j,j' \geq 0}$ is of affine type
$\tilde{A},\tilde{D},\tilde{E}$.
\item[(3)]
We can assume that $a_{i0}=1$ for all $i$.
Then $p_i$ is a rational double point of
type $A,D,E$ according as the type of 
the matrix
$(\chi_G(E_{ij},E_{ij'}))_{j,j' \geq 1}$.
\item[(4)]
We assume that $a_{i0}=1$ for all $i$.
For $j \ne 0$,
\begin{equation}
C_{ij}:=\{x' \in X'| 
\GHom(E_{ij},{\cal E}_{|\{x' \} \times X}) \ne 0\}
\end{equation}
is a smooth rational curve and
$\pi^{-1}(p_i)=\sum_{j>0} a_{ij}C_{ij}$.

\end{enumerate}
\end{prop}

\begin{proof}
Since $H^0(X,{\cal O}_{{\cal Z}_{x'}}) \cong
{\Bbb C}[G]$, $x' \in X'$,
we have
\begin{equation}
\sum_j a_{ij}\Gchi({\cal O}_X \otimes \rho_0,E_{ij})=
\Gchi({\cal O}_X \otimes\rho_0, \oplus_j E_{ij}^{\oplus a_{ij}})=1.
\end{equation}
Hence we may assume that $a_{i0}=1$
and $H^0(X,E_{ij})$ does not contain a trivial representation, 
if $j \ne 0$. 
In particular, $\chi_G(E_{ij} \otimes \alpha^{\vee})<0$ for
$j>0$.
Then the proof is similar to the proof of
Theorem \ref{thm:RDP-desing} and Lemma \ref{lem:0-stable:exceptional}.
\end{proof}

%$\overline{M}^{v_0}(v_0)={M}^{v_0}(v_0) 
%\cup \{\oplus_{j=0}^{s_i} E_{ij}^{\oplus a_{ij}}| i=1,2,...,l\}$.

\begin{NB}
By the symplectic structure,
$\langle v(E_{ij})^2 \rangle$ is even.
Since $\pi$ is proper, there is a stable $G$-sheaf $E_i$ which
is $S$-equivalent to $\oplus_j E_{ij}^{\oplus a_{ij}}$.
Then we see that the graph associated to
$\{ v(E_{ij}) \}_{j \geq 0}$ is connected.
In particular there is $j'$ with
$\langle v(E_{ij}),v(E_{ij'}) \rangle>0$.
\end{NB}

\subsection{Fourier-Mukai transforms for $G$-sheaves}
Let ${\cal E}:={\cal O}_{\cal Z}$ 
be the universal family and we consider the Fourier-Mukai
transform:
\begin{equation}
\begin{matrix}
\Phi:& {\bf D}_G(X) & \to & {\bf D}(X')\\
& E & \mapsto & {\bf R}\pi_{X'*}({\cal E} \otimes \pi_X^*(E))^G.
\end{matrix}
\end{equation} 
Then
\begin{equation}
\begin{matrix}
\widehat{\Phi}:& {\bf D}(X') & \to & {\bf D}_G(X)\\
& F & \mapsto & {\bf R}\pi_{X*}({\cal E}^{\vee}[2] \otimes \pi_{X'}^*(F))
\end{matrix}
\end{equation} 
is the quasi-inverse of $\Phi$.

We note that
$p_{X'*}({\cal O}_{\cal Z})$ is a locally free sheaf on $X'$
with a $G$-action. 
We have a decomposition of $p_{X'*}({\cal O}_{\cal Z})$ 
as $G$-sheaves:
\begin{equation}\label{eq:regular}
p_{X'*}({\cal O}_{\cal Z})=
\oplus_i \Phi({\cal O}_X \otimes \rho_i) \otimes \rho_i^{\vee}.  
\end{equation}

For a $G$-sheaf $E$ of dimension 0,
$E^{\vee}={\cal E}xt^2(E,{\cal O}_X)[-2]$.
Hence $E$ is an irreducible object
if and only if $E^{\vee}[2]$ 
is an irreducible object.

\begin{NB}
Since $\Gchi({\cal O}_X \otimes \alpha,E)=
\Gchi(E,{\cal O}_X \otimes \alpha)=
\Gchi({\cal O}_X \otimes \alpha^{\vee},E^{\vee}[2])$,
we see that
$M_H^{v_0+\alpha}(v_0) \cong
M_H^{v_0^{\vee}-\alpha^{\vee}}(v_0)$.
\end{NB}  

\begin{lem}\label{lem:equiv-full}
We set $F_{ij}:=E_{ij}^{\vee}[2] \in \Coh_G(X)$.
\begin{enumerate}
\item[(1)]
\begin{equation}
\Phi(F_{ij})=
\begin{cases}
{\cal O}_{C_{ij}}(-1)[1],\; j>0,\\
{\cal O}_{Z_i},\;j=0,
\end{cases}
\end{equation}
where $Z_i:=\sum_j a_{ij}C_{ij}$ is the fundamental cycle of $p_i$.
\item[(2)]
$\Phi({\cal O}_X \otimes \rho_i)$ is a locally free sheaf of rank
$\dim \rho_i$ on $X'$.
In particular,
$\Phi({\cal O}_X \otimes \rho_0)={\cal O}_{X'}$.
\item[(3)]
$\Phi({\cal O}_X \otimes \rho_i)$ is a full sheaf (\cite{E:1}).
\end{enumerate}
\end{lem}

\begin{proof}
Let $U$ be a $G$-invariant open subscheme of $X$.
Then $D:=\Supp(p_{Y*}({\cal Z} \cap (Y \times (X \setminus U))))$ 
is a proper closed subset of $Y$ and
${\cal Z}_y \subset U$ if and only if $y \in Y \setminus D$.
If $K_U={\cal O}_U$ as a $G$-sheaf, then
we see that $K_{Y \setminus D}$ is trivial.
Since $X$ has an open covering of these properties,
by the Grauert-Riemenschneider vanishing theorem,
${\bf R}\pi_*({\cal O}_{X'})={\cal O}_{Y'}$.
Outside of the fixed point loci of the $G$-action,
$\widehat{\Phi}({\cal O}_{X'})$ coincides with 
${\cal O}_X \otimes \rho_0$.
Hence  
$\widehat{\Phi}({\cal O}_{X'})={\cal O}_X \otimes \rho_0$.
Therefore $\Phi({\cal O}_X \otimes \rho_0)={\cal O}_{X'}$.
(2) is a consequence of \eqref{eq:regular}.
Then the proof of (1) is similar to the
Fourier-Mukai transform on a $K3$ surface.
(3)  
We note that
\begin{equation}
\begin{split}
\Hom(\Phi({\cal O}_X \otimes \rho_i),{\cal O}_{C_{jk}}(-1))&=
\Hom(\Phi({\cal O}_X \otimes \rho_i),\Phi(F_{jk})[-1])\\
&=\GHom({\cal O}_X \otimes \rho_i,F_{jk}[-1])=0,\\
\Ext^1(\Phi({\cal O}_X \otimes \rho_i),{\cal O}_{Z_j})&=
\Ext^1(\Phi({\cal O}_X \otimes \rho_i),\Phi(F_{j0}))\\
&=\GExt^1({\cal O}_X \otimes \rho_i,F_{j0})=0.
\end{split}
\end{equation}
Hence $\Phi({\cal O}_X \otimes \rho_i)$ is a full sheaf.
\end{proof}

We have
\begin{equation}
\Phi({\cal O}_X \otimes \rho_i)_{|C_{jk}} \cong
{\cal O}_{C_{jk}}^{\oplus (\dim \rho_i-k_{ijk})} \oplus
{\cal O}_{C_{jk}}(1)^{\oplus k_{ijk}},
\end{equation}
where
\begin{equation}
\begin{split}
k_{ijk}:=&(c_1(\Phi({\cal O}_X \otimes \rho_i)),C_{jk})\\
=&\dim \Ext^1(\Phi({\cal O}_X \otimes \rho_i),\Phi(F_{jk}))\\
=&\dim \GHom({\cal O}_X \otimes \rho_i,F_{jk}).
\end{split}
\end{equation}

\begin{prop}\label{prop:equiv-equivalence}
$\Phi$ induces an equivalence
\begin{equation}
\Coh_G(X) \to {^{-1}\Per}(X'/Y').
\end{equation}
\end{prop}

\begin{proof}
It is sufficient to prove
$\Phi(E) \in {^{-1}\Per}(X'/Y')$ for $E \in \Coh_G(X)$.
We first prove that $H^i(\Phi(E))=0$ for $i \ne -1,0$. 
Let $E$ be a $G$-sheaf on $X$. Then
there is an equivariant locally free resolution of $E$:
\begin{equation}
0 \to V_{-2} \to V_{-1} \to V_0 \to E \to 0.
\end{equation}
Since $\Phi(V_i)$ are locally free sheaves on $X'$
and 
\begin{equation}
0 \to \Phi(V_{-2}) \to \Phi(V_{-1}) \to \Phi(V_0)
\end{equation}
is exact on $X' \setminus \cup_i Z_i$,
we get $H^i(\Phi(E))=0$ for $i \ne -1,0$ and 
$\Supp(H^{-1}(\Phi(E))) \subset \cup_{i} Z_i$.
Then we have
\begin{equation}
\begin{split}
\Hom(H^0(\Phi(E)),{\cal O}_{C_{ij}}(-1))&=
\Hom(\Phi(E),\Phi(F_{ij})[-1])\\
&=\GHom(E,F_{ij}[-1])=0,\; j>0,\\
\Hom({\cal O}_{Z_i},H^{-1}(\Phi(E)))&=
\Hom(\Phi(F_{i0}),\Phi(E)[-1])\\
&=\GHom(F_{i0},E[-1])=0.
\end{split}
\end{equation}
Hence $\Phi(E) \in {^{-1}\Per}(X'/Y')$.
\begin{NB}
\begin{equation}
\begin{split}
\Hom(E[i],\widehat{\Phi}(F))=\Hom(\Phi(E)[i],F)=0, i>0\\
\Hom(\widehat{\Phi}(F),E[i])=\Hom(F,\Phi(E)[i])=0, i<0.
\end{split}
\end{equation}
Hence $\widehat{\Phi}(F) \in \Coh_G(X)$.

Conversely for $F \in \Per(X'/Y')$,
since $\Phi(H^i(\widehat{\Phi}(F))) \in \Per(X'/Y')$,
we see that
$\GHom(H^i(\widehat{\Phi}(F))[-i],\widehat{\Phi}(F))=
\Hom(\Phi(H^i(\widehat{\Phi}(F)))[-i],F)=0$
for $i<0$ and
$\GHom(\widehat{\Phi}(F),H^i(\widehat{\Phi}(F))[-i])=
\Hom(F,\Phi(H^i(\widehat{\Phi}(F))[-i]))=0$
for $i>0$.
Hence $\widehat{\Phi}(F) \in \Coh_G(X)$.
\end{NB}
\end{proof}

\begin{rem}
By the proof of Proposition \ref{prop:equiv-equivalence},
$H^{-1}(\Phi(E))=0$ if $E$ does not contain a non-zero 
0-dimensional sub $G$-sheaf.
\end{rem}

\begin{prop}\label{prop:equiv-stability}
For $\alpha=\sum_i r_i \rho_i$, $r_i>0$,
we set $P:=\bigoplus_i \Phi({\cal O}_X \otimes \rho_i)^{\oplus r_i}$.
\begin{enumerate}
\item[(1)]
$P$ is a local projective generator of ${^{-1}\Per}(X'/Y')$.
\item[(2)]
A $G$-sheaf $E$ is $\alpha$-twisted stable if and only if
$\Phi(E)$ is $P$-twisted stable.
\end{enumerate}
\end{prop}

\begin{proof}
Since 
\begin{equation}
\chi(P,\Phi(F_{jk}))=
\sum_i r_i \chi_G({\cal O}_X \otimes \rho_i, F_{jk})=
\sum_i r_i (\rho_i,H^0(X,F_{jk}))>0
\end{equation}
for all $j,k$, (1) holds by Lemma \ref{lem:equiv-full} (3).
(2) is obvious.
\end{proof}

\begin{lem}
$\overline{M}_H^{v_0}(v_0) \cong Y' \cong X/G$. 
In particular, 
$\overline{M}_H^{v_0}(v_0)$ is a normal surface with rational double
points.
\end{lem}

\begin{proof}
We shall first show that $\overline{M}_H^{v_0}(v_0) \cong Y'$.
By Proposition \ref{prop:equiv-stability},
$\overline{M}_H^{v_0}(v_0)$ is isomorphic to the moduli
of 0-dimensional objects $E$ of ${^{-1} \Per}(X'/Y')$
with $v(E)=v({\Bbb C}_x)$.
By Lemma \ref{lem:alpha=0}, we have the claim.
\begin{NB}
Let $E$ be a $v_0$-twisted semi-stable sheaf with $v(E)=v_0$.
Then $\Phi(E^{\vee}[2])$ is a perverse coherent sheaf with
$v(\Phi(E^{\vee}[2]))=(0,0,1)$:
If $E$ is properly semi-stable,
then $\Phi(E^{\vee}[2])$ is ${\Bbb C}_y$ 
or successive extensions of ${\cal O}_{C_{ij}}(-1)[1]$
and ${\cal O}_{Z_i}$. 

Hence $\pi_*(\Phi(E^{\vee}[2])) \cong {\Bbb C}_{y'}, y' \in Y'$.
By sending $E$ to $y'$, we have a morphism 
${\cal M}_H^{v_0}(v_0)^{ss} \to Y'$.
So we get a morphism $\overline{M}_H^{v_0}(v_0) \to Y'$.
Since $Y'$ is normal, $\overline{M}_H^{v_0}(v_0) \to Y'$ is 
an isomorphism.
\end{NB}

Let $\Delta \subset X \times X$ be the diagonal.
Then
${\cal G}:=
\oplus_{g \in G}{\cal O}_{(1 \times g)^*(\Delta)}$ is a $G$-equivariant
coherent sheaf on $X \times X$ which is flat over $X$.
Since $v({\cal G}_{\{ x \} \times X})=v_0$, we have a morphism
$\eta:X \to \overline{M}_H^{v_0}(v_0)$.
We note that 
${\cal G}_{|\{ x \} \times X} \cong {\cal G}_{|\{ g(x) \} \times X}$
for all $g \in G$ and
${\cal G}_{|\{ x \} \times X} \cong {\cal G}_{|\{ y \} \times X}$
if and only if $y \in Gx$.
Hence $\eta$ is $G$-invariant
and we get an injective morphism 
$X/G \to \overline{M}_H^{v_0}(v_0)$.
It is easy to see that $X/G \to \overline{M}_H^{v_0}(v_0)$ 
is an isomorphism.
%We set ${\cal F}:=
%(\oplus_{g \in G}{\cal O}_{(1 \times g)^*(\Delta)})^{\vee}[2]$.
%Then $\Phi({\cal F})$ is a family of perverse coherent sheaves 
%with $v(\Phi({\cal F}))=(0,0,1)$ and
%$\pi_*(\Phi({\cal F}))$ is a family of point sheaves on $W$.
%Thus we get a morphism $\xi:X \to W$ by sending
%$x \in X$ to $\pi_*(\Phi({\cal F}_{|\{x \} \times X}))$. 
%Since ${\cal F}_{|\{x \} \times X} \cong
%{\cal F}_{|\{g(x) \} \times X}$ for all $g \in G$,
%$\xi$ is $G$-invariant. Thus we get a morphism $X/G \to W$.
\end{proof}

\begin{NB}
Let $X$ be a projective scheme. 
Let ${\cal E}$ be an $S$-flat family of coherent sheaves
on $X$ such that ${\cal E}_{|\{s\} \times X} \cong
{\Bbb C}_x$, $x \in X$ for all $s \in S$.
By the base change theorem,
$L:=p_{S*}({\cal E})$ is a line bundle on $S$ and
$p_{S}^*(L) \to {\cal E}$ is surjective.
Replacing ${\cal E}$ by
${\cal E} \otimes p_{S}^*(L)^{\vee}$, we may assume that
${\cal E}={\cal O}_{Z}$ for a subscheme $Z$ of $S \times X$
and $p_{S*}({\cal O}_{Z})={\cal O}_S$.
Then we see that $Z \to S$ is an isomorphism.
Therefore there is a morphism $f:S \to X$ such that
the graph of $f$ is $Z$.
\end{NB}

\begin{cor}
We set $P:=\Phi({\cal O}_X \otimes {\Bbb C}[G])$
and ${\cal A}':=\pi_*(P^{\vee} \otimes P)$.
Under the isomorphism $Y' \cong Y$, we have an isomorphism
$\pi_*(P) \cong \varpi_*({\cal O}_X)$.
Hence we have an isomorphism
 ${\cal A} \cong {\cal  A}'$ as ${\cal O}_{Y'}$-algebras and
we have the following commutative diagram.
\begin{equation}
\begin{CD}
\Coh_G(X) @>{\Phi}>> \Per(X'/Y')\\
@V{\varpi_*}VV @VV{{\bf R}\pi_*(P^{\vee} \otimes (\;\;))}V \\
\Coh_{{\cal A}}(Y) @= \Coh_{{\cal A}'}(Y)
\end{CD}
\end{equation}

%\begin{equation}
%\begin{matrix}
%\Coh_G(X) & \cong & \Coh_{{\cal A}}(Y) & 
%\cong & \Per(X'/Y')\\
%E & \mapsto & \varpi_*(E) & \mapsto
%\pi^{-1}(\varpi_*(E)) \otimes_{\pi^{-1}({\cal A})} P
%\end{split}
%\end{equation}
%which gives the equivalence $\Phi$.
\end{cor}

\begin{proof}
We set $R:={\cal O}_X \otimes {\Bbb C}[G]$.
Since
$\Phi({\cal O}_X \otimes {\Bbb C}[G]) \cong \oplus_i 
\Phi({\cal O}_X \otimes \rho_i)^{\oplus \dim \rho_i} \cong
p_{X' *}({\cal O}_{\cal Z})$,
$\pi_*(P) 
\cong \pi_*(p_{X' *}({\cal O}_{\cal Z}))$ 
is a reflexive sheaf. 
Since $\pi_*(p_{X' *}({\cal O}_{\cal Z}))=\varpi_*({\cal O}_X)$
on the smooth locus, we get an isomorphism
$\pi_*(P) \cong \varpi_*({\cal O}_X)$. 
Since 
${\cal A}'$ is a reflexive sheaf on $Y'$,
we have ${\cal A}' \cong \mathrm{End}_{{\cal O}_{Y'}}(\pi_*(P))$.
Therefore 
${\cal A}' \cong \mathrm{End}_{{\cal O}_{Y'}}(\pi_*(P))
\cong \mathrm{End}_{{\cal O}_{Y}}(\varpi_*({\cal O}_X)) \cong {\cal A}$.

Since $\varpi_*(R) ={\cal A}$
and 
every $G$-sheaf $E$ has a locally free resolution 
\begin{equation}
\cdots \to R(-n_{-2})^{\oplus N_{-2}} 
\to R(-n_{-1})^{\oplus N_{-1}} 
\to R(-n_{0})^{\oplus N_{0}} \to
E \to 0,
\end{equation}
we get the commutative diagram. 
\begin{NB}
For ${\cal O}_X[G]$, we have
$\varpi_*({\cal O}_X[G]) \cong \Hom_{{\cal O}_W}(\pi_*(\Phi({\cal O}_X)),
\pi_*(\Phi({\cal O}_X)))$
and
$\pi_*(\Phi({\cal O}_X[G])) \cong \varpi_*({\cal O}_X)$.
Hence $\varpi_*({\cal O}_X[G]) \cong
\pi_*(\Phi({\cal O}_X[G])) \otimes \pi_*(\Phi({\cal O}_X[G]))^{\vee}$.

For $E \in \Coh_G(X)$, we take a resolution by
locally free ${\cal O}_X[G]$-modules
\begin{equation}
\cdots \to V_{-2} \otimes {\cal O}_X[G](-n_{-2}) 
\to V_{-1} \otimes {\cal O}_X[G](-n_{-1}) 
\to V_{0} \otimes {\cal O}_X[G](-n_{0}) \to
E \to 0.
\end{equation}
Then we see that
$\varpi_*(E)$ is canonically isomorphic to
$\pi_*(\Phi(E))  \otimes \pi_*(\Phi({\cal O}_X[G]))^{\vee}$.
\end{NB}
\end{proof}

Assume that $X'$ is a $K3$ surface.
For a primitive isotropic Mukai vector $v_0$ on $X'$,  
we set $X'':=M_H^w(v_0)$, where $v_0:=(r,\xi,a)$
is a primitive isotropic Mukai vector
with $0<(\xi,C_{ij})$ and $(\xi,\sum_j a_{ij}C_{ij})<r$
for all $i,j$ and $w \in K(X')\otimes {\Bbb Q}$ is sufficiently
close to $v_0$.
Assume that there is a universal family
${\cal F}$ on
$X' \times X''$.
Then ${\cal E}':=\widehat{\Phi}({\cal F})$ is a flat family of
stable $G$-sheaves and defines an equivalence
$\Phi':{\bf D}^G(X) \to {\bf D}(X'')$ such that
$\Phi'=\Phi^{{\cal E}'}_{X' \to X''} \circ \Phi$.

\subsection{Irreducible objects of $\Coh_G(X)$}

We shall study irreducible objects of $\Coh_G(X)$.
Let $E$ be a $G$-sheaf of dimension 0.
We may assume that $\Supp(E)=Gx$.
Let $H$ be the stabilizer of $x$
and $E_x$ the submodule of $E$ whose support is $x$.
Then $E_x$ is a $H$-sheaf.
We have a decomposition
$H^0(X,E)=\oplus_{y \in Gx} H^0(X,E_y)$.
Since $gH^0(X,E_x)=H^0(X,E_{gx})$,
we have an isomorphism
\begin{equation}
H^0(X,E) \cong {\Bbb C}[G] \otimes_{{\Bbb C}[H]}H^0(X,E_x)
\end{equation}
as $G$-modules.
Then we have an equality of invariant subspaces:
\begin{equation}\label{eq:invariant-part}
H^0(X,E)^G=H^0(X,E_x)^H.
\end{equation}
We shall prove that there is a bijection between
\begin{itemize}
\item[(a)]
${\frak G}:=\{E \in \Coh_G(X)| \Supp(E)=Gx,\;\Stab(x)=H \}$ and 
\item[(b)]
${\frak H}:=\{F \in \Coh_H(X)| \Supp(F)=x\}$. 
\end{itemize}
We define $r:{\frak G} \to {\frak H}$ by
sending $E \in {\frak G}$ to $E_x \in {\frak H}$.
For $F \in {\frak H}$,
we set $K:=\ker(H^0(X,F)\otimes {\cal O}_X \to F)$.
Then 
\begin{equation}
s(F):=({\Bbb C}[G] \otimes_{{\Bbb C}[H]}H^0(X,F))
\otimes {\cal O}_X/
\sum_{g \in G}g(K)
\end{equation}
 is a $G$-sheaf such that
$s(F)_x=F$. Hence we have a map $s:{\frak H} \to {\frak G}$
with $r \circ s=\id_{\frak H}$.
For $E \in {\frak G}$,
we also see that $s(E_x) \cong E$, and hence $s \circ r=\id_{\frak G}$.
Therefore our claim holds.

If $H^0(X,F)$ is the regular representation of $H$, i.e.,
$H^0(X,F) \cong {\Bbb C}[H]$, then $H^0(X,E)$ is the regular
representation of $G$.
Then we see that $E$ is irreducible in $\Coh_G(X)$ 
if and only if $E_x$ is irreducible in
$\Coh_H(X)$.
Since $\Supp(E_x)$ is one point,
it means that 
$H^0(X,E_x)$ is an irreducible representation of $H$ and
$E_x \cong H^0(X,E_x) \otimes {\Bbb C}_x$.

\begin{lem}
Each singular point $\bigoplus_j E_{ij}^{\oplus a_{ij}} 
\in \overline{M}_H^{v_0}(v_0)$
corresponds to an orbit $Gx_i$ with $\Stab(x_i) \ne \{e \}$
and $(E_{ij})_{x_i}=\rho_{ij} \otimes {\Bbb C}_{x_i}$, where
$\rho_{ij}$ are irreducible representations of $\Stab(x_i)$.
Moreover 
\begin{equation}
\chi_G(E_{ij},E_{ij'})=
\chi_{\Stab(x_i)}(\rho_{ij} \otimes {\Bbb C}_x,
\rho_{ij'} \otimes {\Bbb C}_x).
\end{equation}
\end{lem} 

\begin{proof}
If $\Supp(E_{ij}) \ne \Supp(E_{i'j'})$, then
$\chi(E_{ij},E_{i'j'})=0$.
Hence $\Supp(E_{ij})= \Supp(E_{ij'})$ for all $j,j'$.
Hence there is a point $x_i$ such that
$Gx_i=\Supp(E_{ij})$ for all $j$.
Then the first part of the claim follows.

For the second claim,
we note that
$\chi_{\Stab(x_i)}((\oplus_{g \in G/\Stab(x_i)}
\rho_{ij} \otimes {\Bbb C}_{gx_i})/\rho_{ij} \otimes {\Bbb C}_{x_i},
\rho_{ij'} \otimes {\Bbb C}_{x_i})=0$.
By using an equivariant locally free resolution
of $E_{ij}$ and \eqref{eq:invariant-part}, we see that 
\begin{equation}
\begin{split}
\chi_G(E_{ij},E_{ij'})=&\chi_{\Stab(x_i)}(E_{ij},(E_{ij'})_{x_i})\\
=&\chi_{\Stab(x_i)}((E_{ij})_{x_i},(E_{ij'})_{x_i}).
\end{split}
\end{equation}
\end{proof}

\begin{NB}
For a locally free $G$-sheaf $E$ on $X$,
$\Phi(E) \in \pi^*(\Coh(W))$ if and only if
$\Stab(x)$ acts trivially on $E$ for all $x \in X$.
Indeed, $\pi^*(\pi_*(\Phi(E))) \to \Phi(E)$ is isomorphic
if and only if $\Ext^1(\Phi(E),{\cal O}_{C_{ij}}(-1))=0$ for all
$C_{ij}$.
Since $\Ext^1(\Phi(E),{\cal O}_{C_{ij}}(-1))=\GHom(E,F_{ij})=
\Hom(\rho_{ij},H^0(E^{\vee}_{|\{x \}}))$,
we get the claim. 
\end{NB}
\begin{ex}
Let $X$ be an abelian surface.
Then $G={\Bbb Z}_2$ acts on $X$ as the multiplication by $(-1)$.
Then the moduli of stable $G$-sheaves on $X$ is isomorphic
to the moduli space of stable objects of ${^{-1}\Per}(\mathrm{Km}(X)/Y)$,
where $Y=X/G$ and $\mathrm{Km}(X) \to Y$ is the Kummer surface
associated to $X$. 
\end{ex}

\section{Appendix}

\subsection{Elementary facts on lattices}

\begin{lem}\label{lem:appendix:lattice}
Assume that $L \cong {\Bbb Z}^n$ has an integral bilinear form
$(\;\;,\;\;)$.
Let $v$ be a primitive elenent of $L$ such that
$(v,v)=0$, $(v,w)=(w,v)$ for any $w$.
We set $v^{\perp}:=\{ x \in L| (v,x)=0 \}$.
Assume that
$(\;\;,\;\;)_{|v^{\perp}}$ is symmetric and
there is an element $u \in L \otimes {\Bbb Q}$ such that
$(u,v)=0$
and 
$(v^{\perp} \cap u^{\perp})/{\Bbb Z}v$ is negative definite.
\begin{enumerate}
\item[(1)]
If $v=\sum_{i=0}^s a_i v_i$, $a_i \in{\Bbb Z}_{>0}$ such that
$v_i \in v^{\perp} \cap u^{\perp}$, $i=0,1,...,s$,
$(v_i^2)=-2$ and $(v_i,v_j) \geq 0$ for $i \ne j$.
Then the matrix $(-(v_i,v_j)_{i,j})$ is of affine type
$\widetilde{A},\widetilde{D},\widetilde{E}$.
\item[(2)]
If $v$ has two expressions
\begin{equation}
\begin{split}
v=& \sum_{i=0}^s a_i v_i= \sum_{i=0}^t a_i' v_i',\; a_i,a_i' \in{\Bbb Z}_{>0}
\end{split}
\end{equation}
such that $v_i, v_i' \in v^{\perp}\cap u^{\perp}$,
$(v_i^2)=((v_i')^2)=-2$ and $(w_1,w_2) \geq 0$
for different $w_1,w_2 \in V_1 \cup V_2$, where
$V_1:=\{v_0,v_1,\dots,v_s\}$ and $V_2:=\{v_0',v_1',\dots,v_t' \}$.
Then
$V_1=V_2$ or
$\oplus_i {\Bbb Z} v_i \perp \oplus_i {\Bbb Z} v_i'$.
\end{enumerate}
\end{lem}

%$L=K(X)$ and $(\,\,,\,\,)=-\chi(\,\,,\,\,)$.
%Assume that $

\begin{proof}
(1) 
We first note that $v_0,v_1,...,v_s$ are linearly independent 
over ${\Bbb Q}$. 
We shall show that the dual graph of
$\{v_0,v_1,...,v_s\}$ is connected.
If we have a decomposition 
$v=(\sum_{i \in I_1}a_i v_i)+(\sum_{i \in I_2}a_i v_i)$ such that
$(v_i,v_j)=0$ for $i \in  I_1$, $j \in I_2$, then
$0=(v^2)=(\sum_{i \in I_1}a_i v_i)^2+(\sum_{i \in I_2}a_i v_i)^2$.
Hence $\sum_{i \in I_1}a_i v_i, \sum_{i \in I_2}a_i v_i \in {\Bbb Z}v$,
which implies that the graph is connected.
Then the standard arguments show the claim.

(2) $I:=\{i|v_i' \in V_1 \}$ and
$J:=\{i|v_i' \not \in V_1 \}$.
Then $v=(\sum_{i \in I}a_i'v_i')+(\sum_{i \in J}a_i' v_i')$.
If $i \in J$, then
$0=(v_i,v)=\sum_j a_j (v_i',v_j) \geq 0$.
Hence $v_i' \in (\oplus_i {\Bbb Z}v_i)^{\perp}$.
Then $0=(v^2)=((\sum_{i \in I}a_i'v_i')^2)+
((\sum_{i \in J}a_i' v_i')^2)$.
Hence $\sum_{i \in I}a_i'v_i',\sum_{i \in J}a_i' v_i' \in {\Bbb Z}v$,
which implies that 
$I =\emptyset$ or $J=\emptyset$.
If $J=\emptyset$, then $V_2 \subset V_1$, and we see that
$V_1=V_2$. If $I=\emptyset$, then all $v_i'$ belong to \
$\oplus_i {\Bbb Z} v_i$. Thus 
$\oplus_i {\Bbb Z} v_i \perp \oplus_i {\Bbb Z} v_i'$.
\end{proof}

\begin{ex}\label{ex:1}
Let $X$ be s smooth projective surface and $H$ a divisor on $X$
with $(H^2)>0$.
We set $L:=\ch(K(X))$ and 
$(x,y):=-\int_X x^{\vee}y \td_X$, $x, y \in L$.  
Then $\varrho_X=\ch({\Bbb C}_x)$ is primitive in $L$.
Since ${\Bbb C}_x \otimes K_X \cong  {\Bbb C}_x$,
$(\varrho_X,x)=(x,\varrho_X)$.
Moreover $(\;\;,\;\;)_{|\varrho_X^{\perp}}$ is symmetric.
Since 
$(\varrho_X^{\perp} \cap \ch({\cal O}_H)^{\perp})/{\Bbb Z}\varrho_X 
\cong 
\{D \in \NS(X)_f|(H,D)=0\}$, it is negative definite, where
$\NS(X)_f$ is the torsion free quotient of $\NS(X)$.
\end{ex}

%\begin{ex}\label{ex:2}
%Let $X$ be a $K3$ surface and $H$ an ample divisor on $X$.
%We set $L:=H^{ev}(X,{\Bbb Z})_{\alg}$.
%Let $v \in H^{ev}(X,{\Bbb Z})_{\alg}$ 
%be a primitive isotropic Mukai vector with $\rk v>0$.
%Then $(v^{\perp}\cap \widehat{H}^{\perp})/{\Bbb Z}v$ is negative definite.
%\end{ex}

\subsection{Existence of twisted semi-stable sheaves}

Let $X$ be a smooth projective surface and $H$ an ample divisor
on $X$.
Let ${\bf e} \in K(X)_{\mathrm{top}}$ be a toplogical invariant
of a coherent sheaf on $X$.
\begin{defn}\label{defn:general}
A polarization $H$ on $X$ is general with respect to ${\bf e}$, if
for every $\mu$-semi-stable sheaf $E$ with $\tau(E)={\bf e}$ and a  
subsheaf $F \ne 0$ of $E$, 
\begin{equation}
\frac{(c_1(F),H)}{\rk F}=\frac{(c_1(E),H)}{\rk E}
\text{ if and only if }
\frac{c_1(F)}{\rk F}=\frac{c_1(E)}{\rk E}.
\end{equation}
\end{defn}
If $H$ is general with respect to ${\bf e}$, then
the $G$-twisted semi-stability does not depend on the
choice of $G$.
The following is \cite[Lem. 3.6]{M-W:1}. For convenience' sake,
we give a proof.
\begin{lem}\label{lem:appendix:twisted}
Assume that $H$ is not general with respect to ${\bf e}$
and let $\epsilon$ be a sufficiently small ${\Bbb Q}$-divisor
such that $H+\epsilon$ is general with respect to ${\bf e}$.
Then there is a locally free sheaf $G$ such that
${\cal M}_H^G({\bf e})^{ss}={\cal M}_{H+\epsilon}({\bf e})^{ss}$.
\end{lem}

\begin{proof}
We set
\begin{equation}
{\cal F}({\bf e}):=
\left\{F \subset E \left|
\begin{aligned}
E \in {\cal M}_H({\bf e})^{\mu\text{-}ss},
\text{ $E/F$ is torsion free }\\
(c_1(F),H)/\rk F=(c_1(E),H)/\rk E\\
\end{aligned}
\right. \right\}.
\end{equation}
Since ${\cal F}({\bf e})$ is a bounded set,
we have
\begin{equation}
B:=\max \left\{ \left.
\left|\frac{\chi(E)}{\rk E}-\frac{\chi(F)}{\rk F} \right| \right|
(F \subset E) \in {\cal F}({\bf e})
\right\}<\infty.
\end{equation}
Assume that $N \epsilon \in \NS(X)$.
Let $G$ be a locally free sheaf such that $c_1(G)/\rk G=-m\epsilon$.
If $m \geq (\rk {\bf e})^2 N B$, then
for $(F \subset E) \in {\cal F}({\bf e})$,
\begin{equation}
\frac{\chi(G,E(nH))}{\rk E}-
\frac{\chi(G,F(nH))}{\rk F}=
m\left(\frac{c_1(E)}{\rk E}-\frac{c_1(F)}{\rk F}, 
\epsilon \right)
+\frac{\chi(E)}{\rk E}-\frac{\chi(F)}{\rk F} \geq 0
\end{equation}
if and only if 
\begin{enumerate}
\item[(1)]
\begin{equation}
\left(\frac{c_1(E)}{\rk E}-\frac{c_1(F)}{\rk F}, 
\epsilon \right) \geq 0
\end{equation}
or
\item[(2)]
\begin{equation}
\frac{c_1(E)}{\rk E}-\frac{c_1(F)}{\rk F}=0,\;
\frac{\chi(E)}{\rk E}-\frac{\chi(F)}{\rk F} \geq 0,
\end{equation}
\end{enumerate}
which is the semi-stability of $E$ with respect to
$H+\epsilon$.
Therefore the claim holds.
\end{proof}

\begin{lem}\label{lem:appendix:existence}
Let $(X,H)$ be a polarized 
$K3$ surface and $v=r+\xi+a\varrho_X$, $\xi \in \NS(X)$
a primitive Mukai vector with $\langle v^2 \rangle \geq -2$.
Then there is a $G$-twisted semi-stable sheaf $E$
with $v(E)=v$ for any $G$.
\end{lem}

\begin{proof}
If $H$ is general with respect to $v$,
then there is a stable sheaf $E$ with $v(E)=v$
by \cite[Thm. 8.1]{Y:7}
and \cite{Y:Stability}.
By Lemma \ref{lem:appendix:twisted},
there is a locally free sheaf $G_1$ such that
${\cal M}_H^{G_1}(v)^{ss}={\cal M}_H^{G_1}(v)^{s} \ne \emptyset$.
For a $G$ with ${\cal M}_H^G(v)^{ss}={\cal M}_H^G(v)^{s}$,
 we use \cite[Prop. 4.1]{Y:9}.
If ${\cal M}_H^G(v)^{ss} \ne {\cal M}_H^G(v)^{s}$,
then we can find a $G'$ such that
$c_1(G')/\rk G'$ is sufficiently close to
$c_1(G)/\rk G$,
${\cal M}_H^{G'}(v)^{ss}={\cal M}_H^{G'}(v)^{s} \ne \emptyset$
and ${\cal M}_H^{G'}(v)^{ss} \subset {\cal M}_H^G(v)^{ss}$.
Thus the claim also holds.

\begin{NB}
Let $E$ be a usual stable sheaf of $v(E)=v$
with respect to $H+\epsilon$.  
We set
\begin{equation}
{\cal S}:=
\left\{F \left|
\begin{aligned}
F \subset E, \text{ $E/F$ is torsion free }\\
(c_1(F),H)/\rk F=(c_1(E),H)/\rk E\\
\end{aligned}
\right. \right\}.
\end{equation}
For $F \in {\cal S}$,
we have 
\begin{enumerate}
\item[(1)]
 $(c_1(F),\epsilon)/\rk F<(c_1(E),\epsilon)/\rk E$
or 
\item[(2)] $c_1(F)/\rk F=c_1(E)/\rk E$ and
$\chi(F)/\rk F<\chi(E)/\rk E$.
\end{enumerate}
Since ${\cal S}$ is a bounded set,
there is an integer $m$ such that
\begin{equation}
m\left(\frac{c_1(E)}{\rk E}-\frac{c_1(F)}{\rk F}, 
\epsilon \right)
+\frac{\chi(E)}{\rk E}-\frac{\chi(F)}{\rk F}>0.
\end{equation}
Let $G$ be a locally free sheaf such that $c_1(G)/\rk G=-m\epsilon$.
Since
\begin{equation}
\frac{\chi(G,E(nH))}{\rk E}-
\frac{\chi(G,F(nH))}{\rk F}=
m\left(\frac{c_1(E)}{\rk E}-\frac{c_1(F)}{\rk F}, 
\epsilon \right)
+\frac{\chi(E)}{\rk E}-\frac{\chi(F)}{\rk F}>0,
\end{equation}
$E$ is $G$-twisted stable.
For the general case, we use 
\end{NB}
\end{proof}

\subsection{Spectral sequences }

Since $\widehat{\Phi}[2]$ and $\widehat{\Psi}$ are the inverses of
$\Phi$ and $\Psi$ respectively, we get the followng.
\begin{lem}\label{lem:spectral}
We have spectral sequences
\begin{equation}
E_2^{p,q}=\Phi^p(\widehat{\Phi}^q(E)) 
\Rightarrow
E_{\infty}^{p+q}=
\begin{cases}
E, &p+q=2,\\
0, &p+q \ne 2,
\end{cases}
\;\; E \in \Per(X'/Y'),
\end{equation}
\begin{equation}
E_2^{p,q}=\widehat{\Phi}^p(\Phi^q(F)) 
\Rightarrow
E_{\infty}^{p+q}=
\begin{cases}
F, &p+q=2,\\
0, &p+q \ne 2,
\end{cases}
\;\; F \in {\cal C}.
\end{equation}
In particular,
\begin{enumerate}
\item
$\Phi^p(\widehat{\Phi}^q(E))=0$, $p=0,1$. 
\item
$\Phi^p(\widehat{\Phi}^q(E))=0$, $p=1,2$. 
\item
There is an injective homomorphism
$\Phi^0(\widehat{\Phi}^1(E)) \to
\Phi^2(\widehat{\Phi}^0(E))$.
\item
There is a surjective homomorphism
$\Phi^0(\widehat{\Phi}^2(E))\to
\Phi^2(\widehat{\Phi}^1(E))$.
\end{enumerate}
\end{lem}

\begin{lem}\label{lem:spectral2}
We have spectral sequences
\begin{equation}
E_2^{p,q}=\Psi^p(\widehat{\Psi}^{-q}(E)) 
\Rightarrow
E_{\infty}^{p+q}=
\begin{cases}
E, &p-q=0,\\
0, &p-q \ne 0,
\end{cases}
\;\; E \in \Per(X'/Y')^D,
\end{equation}
\begin{equation}
E_2^{p,q}=\widehat{\Psi}^p(\Psi^{-q}(F)) 
\Rightarrow
E_{\infty}^{p+q}=
\begin{cases}
F, &p-q=0,\\
0, &p-q \ne 0,
\end{cases}
\;\; F \in {\cal C}.
\end{equation}
In particular,
\begin{enumerate}
\item
$\Psi^p(\widehat{\Psi}^2(E))=0$, $p=0,1$. 
\item
$\Psi^p(\widehat{\Psi}^0(E))=0$, $p=1,2$. 
\item
There is an injective homomorphism
$\Psi^0(\widehat{\Psi}^1(E)) \to
\Psi^2(\widehat{\Psi}^2(E))$.
\item
There is a surjective homomorphism
$\Psi^0(\widehat{\Psi}^0(E))\to
\Psi^2(\widehat{\Psi}^1(E))$.
\end{enumerate}
\end{lem}

For a convenience of the reader, we give a proof
of Lemma \ref{lem:spectral2}.
\begin{proof}
By the exact triangles
\begin{equation}
\Psi^{\leq 1}(E)[-1] \to \Psi(E) \to \Psi^2(E)[-2] \to
\Psi^{\leq 1}(E)
\end{equation}
and
\begin{equation}
\Psi^0(E) \to \Psi^{\leq 1}(E)[-1] \to \Psi^1(E)[-1] \to
\Psi^0(E)[1],
\end{equation}
we have exact triangles
\begin{equation}
\widehat{\Psi}(\Psi^{\leq 1}(E))[1] \leftarrow 
\widehat{\Psi}(\Psi(E)) 
\leftarrow \widehat{\Psi}(\Psi^2(E))[2] \leftarrow 
\widehat{\Psi}(\Psi^{\leq 1}(E))
\end{equation}
and
\begin{equation}
\widehat{\Psi}(\Psi^0(E)) \leftarrow
 \widehat{\Psi}(\Psi^{\leq 1}(E))[1]\leftarrow  
\widehat{\Psi}(\Psi^1(E))[1] \leftarrow
\widehat{\Psi}(\Psi^0(E))[-1].
\end{equation}
Since $\widehat{\Psi}(\Psi(E))=E$, we have exact sequences
\begin{equation}
\begin{split}
0 \leftarrow \widehat{\Psi}^1(\Psi^{\leq 1}(E)) \leftarrow 
E \leftarrow \widehat{\Psi}^2(\Psi^2(E)) \leftarrow 
\widehat{\Psi}^0(\Psi^{\leq 1}(E)) \leftarrow 0, \\ 
\widehat{\Psi}^2(\Psi^{\leq 1}(E))=
\widehat{\Psi}^1(\Psi^2(E))=\widehat{\Psi}^0(\Psi^2(E))=0,\\
0 \leftarrow \widehat{\Psi}^2(\Psi^1(E)) \leftarrow
\widehat{\Psi}^0(\Psi^0(E)) \leftarrow
 \widehat{\Psi}^1(\Psi^{\leq 1}(E))\leftarrow  
\widehat{\Psi}^1(\Psi^1(E)) \leftarrow 0,\\
\widehat{\Psi}^0(\Psi^{\leq 1}(E))\cong 
\widehat{\Psi}^0(\Psi^1(E)),\\
\widehat{\Psi}^1(\Psi^0(E))= \widehat{\Psi}^2(\Psi^0(E))=0. 
\end{split}
\end{equation}
These give the data of the spectral sequence.
\end{proof}

{\it Acknowledgement.}
I would like to thank 
Max Planck Institute f\"{u}r Mathematik
very much for the hospitarity and support. 
I would also like to thank Hiraku Nakajima for valuable discussions
on perverse coherent sheaves. 
Finally I would like to thank Shintaro Yanagida 
for reading this manuscript and giving some comments.

\begin{NB}

\subsection{A relatively ample object on $\Per(X/Y)$}

\begin{rem}
$E \in \Coh(X)$ belongs to $\Per(X/Y)$ if and only
if $\Hom(E,K)=0$ for all $K \in \Coh(X)$ with
$\pi_*(K)=0$. Indeed 
$\Per(X/Y)$ is the tilting of $\Coh(X)$ with respect to a torsion pair
whose torsion free part is 
$\{K \in \Coh(X)|\pi_*(K)=0 \}$.
\end{rem}

For $E \in \Per(X/Y)$, we have a decomposition
\begin{equation}
0 \to E_1[1] \to E \to E_2 \to 0
\end{equation}
where $E_i=H^{i+2}(E)$.
Let $G$ be a locally free sheaf on $X$ such that
\begin{enumerate}
\item[(1)]
 $G \in \Per(X/Y)$,
\item[(2)]
 $R^1 \pi_*(G^{\vee} \otimes E_1) \ne 0$,
\item[(3)] 
$\pi^*(\pi_*(G^{\vee} \otimes E_2)) \otimes G \to E_2$ 
is surjective and $R^1 \pi_*(G^{\vee}\otimes E_2) = 0$.
\end{enumerate} 
We may say that $G$ is relatively ample with respect to
$\pi:X \to Y$.

\begin{rem}
$\{G(-n)|n>0 \}$ may be ample sequence in \cite{Inaba}. 
\end{rem}

\begin{rem}
Since $\pi^*(\pi_*(G)) \to G$ is surjective,
$\pi_*(G^{\vee} \otimes E_1)$ is a subsheaf of 
$\pi_*({\cal H}om(\pi^*(\pi_*(G)),E_1))=
\pi_*({\cal H}om(\pi_*(G),\pi_*(E_1)))=0$.
Hence
${\bf R}\pi_*(G^{\vee} \otimes E_1[1])=R^1 \pi_*(G^{\vee} \otimes E_1)$. 
\end{rem}

\begin{lem}
Let $E$ be an object in $\Per(X/Y)$. Then
$\Hom(G,E(n)) \otimes G \to E(n)$ is surjective in $\Per(X/Y)$
for $n \gg 0$.
\end{lem}

\begin{proof}
(1) We first assume that $E=E_2$, where $E_2=H^0(E)$.
Assume that $\pi_*(G^{\vee} \otimes E_2(n))=\pi_*(G^{\vee} \otimes E_2)(n)$ 
is generated by global sections.
We shall prove that
there is an exact sequence in $\Per(X/Y)$:
\begin{equation}
0 \to F_2 \to \Hom(G,E_2(n)) \otimes G \to E_2(n) \to 0.
\end{equation}
Since 
$H^0(X,G^{\vee} \otimes E_2(n)) \otimes {\cal O}_Y
\overset{\lambda}{\to} \pi_*(G^{\vee} \otimes E_2(n))$ is surjective, we have 
a surjective homomorphism
\begin{equation}
H^0(X,G^{\vee} \otimes E_2(n)) \otimes G
\to \pi^*(\pi_*(G^{\vee} \otimes E_2(n))) \otimes G
\end{equation}
Since $\pi^*(\pi_*(G^{\vee} \otimes E_2)) \otimes G \to E_2$ is surjective,
$\Hom(G,E_2(n)) \otimes G \to E_2(n)$ is surjective.

\begin{equation}
0 \to \pi_*(G^{\vee} \otimes F_2) \to 
\Hom(G,E_2(n)) \otimes \pi_*(G^{\vee} \otimes G) \to 
\pi_*(G^{\vee} \otimes E_2(n)) \to 
 R^1\pi_*(G^{\vee} \otimes F_2) \to 0 
\end{equation}
Since 
\begin{equation}
\Hom(G,E_2(n)) \otimes {\cal O}_Y \to
\Hom(G,E_2(n)) \otimes \pi_*(G^{\vee} \otimes G) \to
\pi_*(G^{\vee} \otimes E_2(n))
\end{equation}
coincides with $\lambda$,
we have $R^1\pi_*(G^{\vee} \otimes F_2)=0$.
If there is a surjection $F_2 \to K$ with $\pi_*(K)=0$, then
we have a surjection $R^1\pi_*(G^{\vee} \otimes F_2) \to
R^1\pi_*(G^{\vee} \otimes K)$, which is a contradiction.
Therefore $F_2 \in \Per(X/Y)$. 

(2)
We next assume that $E=E_1[1]$, where $E_1=H^{-1}(E)$.
Then 
$\Hom(G,E_1(n)[1]) \otimes G \to E_1(n)[1]$ defines a universal extension
\begin{equation}
0 \to E_1(n) \to F_1 \to G \otimes \Ext^1(G,E_1(n)) \to 0.
\end{equation}
We shall prove that $\Hom(F_1,K)=0$ for $K \in \Coh(X)$ 
with $\pi_*(K)=0$.
If $\Hom(F_1,K) \ne 0$, then 
we may assume that there is a surjective homomorphism 
$\phi:F_1 \to K$ such that $\pi_*(K)=0$. 
For $\phi_{|E_1(n)}:E_1(n) \to K$,
we have a homomorphism
$\Ext^1(G,E_1(n)) \to \Ext^1(G,\phi(E_1(n))) \to \Ext^1(G,K)$.
Since $\Hom(G,E_1(n)[1]) \otimes G \to K/\phi(E_1(n))$ is surjective,
$K/\phi(E_1(n)) \in \Per(X/Y) \cap \Coh(X)$, which implies that
$R^1 \pi_*(G^{\vee} \otimes K/\phi(E_1(n)))=0$.
Hence $R^1 \pi_*(G^{\vee} \otimes E_1(n)) \to 
R^1 \pi_*(G^{\vee} \otimes K) \ne 0$ is surjective.
We may assume that
$R^1 \pi_*(G^{\vee} \otimes E_1(n))$ is generated by global sections.
%$\Ext^1(G,E_1(n)) \to H^0(Y,R^1 \pi_*(G^{\vee} \otimes E_1(n)))$ 
%is surjective.
%$$
%\Ext^1(G,E_1(n)) \to 
%H^0(Y,R^1 \pi_*(G^{\vee} \otimes E_1(n)))
%\to 
%H^0(Y,R^1 \pi_*(G^{\vee} \otimes K))
%$$
We have a commutative diagram
\begin{equation}
\begin{CD}
\Ext^1(G,E_1(n))
@>>> H^0(Y,R^1 \pi_*(G^{\vee} \otimes E_1(n)))\\
@VVV @VVV \\
\Ext^1(G,K)
@>>> H^0(Y,R^1 \pi_*(G^{\vee} \otimes K))
\end{CD}
\end{equation}
whose holizontal maps are isomorphic.
Hence $\Ext^1(G,E_1(n)) \to \Ext^1(G,K)$ is not zero.
On the other hand, by our choice of $E_1 \to K$,
$\psi$ in the bottom diagram is zero, which is a contradiction.
\begin{equation}
\begin{CD}
\Hom(G,G) \otimes \Ext^1(G,E(n)) 
@>>> \Ext^1(G,E(n))\\
@| @VVV \\
\Hom(G,G) \otimes \Ext^1(G,E(n)) 
@>{\psi}>> \Ext^1(G,K)
\end{CD}
\end{equation}

%$$
%\Ext^1(G,E_1(n))^{\vee} \otimes \Ext^1(G,E_1(n)) \to 
%\Ext^1(G,E_1(n))^{\vee} \otimes H^0(Y,R^1 \pi_*(G^{\vee} \otimes E_1(n)))
%\to \Ext^1(G,E_1(n))^{\vee} \otimes 
%H^0(Y,R^1 \pi_*(G^{\vee} \otimes K)),
%$$
%the universal extension class goes to a non-zero extension class $e$.
%On the other hand, by our choice of $\phi_{|E_1(n)}$,
%$e$ is 0, which is a contradiction.
Therefore $F_1 \in \Per(X/Y) \cap \Coh(X)$, which implies that
we have an exact sequence in $\Per(X/Y)$
\begin{equation}
0 \to F_1 \to \Hom(G,E_1[1]) \otimes G \to E_1[1] \to 0
\end{equation}

%For the evaluation map
%$G \otimes \pi^*(R^1 \pi_*(G^{\vee} \otimes E_1)) \to E_1[1]$,
%we have an exact sequence
%\begin{equation}
%0 \to E_1 \to F \to G \otimes \pi^*(R^1 \pi_*(G^{\vee} \otimes E_1)) \to 0 
%\end{equation}
%\begin{equation}
%\Begin{CD}
%\pi_*(G^{\vee} \otimes G) \otimes R^1 \pi_*(G^{\vee} \otimes E) 
%@>>> R^1 \pi_*(G^{\vee} \otimes E)\\
%@| @VVV \\
%\pi_*(G^{\vee} \otimes G) \otimes R^1 \pi_*(G^{\vee} \otimes E) 
%@>>> R^1 \pi_*(G^{\vee} \otimes K)
%\end{CD}
%\end{equation}

For the general case, we have the following diagram in
$\Per(X/Y)$:
\begin{equation}
\begin{CD}
@. 0 @. 0 @. 0 @. \\
@. @AAA @AAA @AAA @. \\
0 @>>> E_1(n)[1] @>>> E(n) @>>> E_2(n) @>>> 0 \\
@. @AAA @AAA @AAA @. \\
0 @>>> V_1 \otimes G @>>> V \otimes G @>>> V_2 \otimes G @>>> 0 \\
@. @AAA @AAA @AAA @. \\
0 @>>> F_1 @>>> F @>>> F_2 @>>> 0\\
@. @AAA @AAA @AAA @. \\
@. 0 @. 0 @. 0 @. 
\end{CD}
\end{equation} 
where $V_1=\Hom(G,E_1(n)[1]), V_2=\Hom(G,E_2(n)),
V=\Hom(G,E(n))$. 
\end{proof}

\begin{rem}
(1)
For the construction of the evaluation map, we used 
$\Hom(E^{\bullet},F^{\bullet})$ to be a finite dimensional vector space:
Let $E^{\bullet}$ and $F^{\bullet}$ be bounded complexes on $X$.
Assume that $\Ext^i(E^p,H^q(F^{\bullet}))=0$, $i>0$ for all $p,q$. 
Then $\Hom(E^{\bullet},F^{\bullet})$ is the cohomology
of the complex $V^{\bullet}$ with
$V^i=\oplus_{q-p=i}\Hom(E^p,F^q)$.
Since $\Hom(E^{\bullet},F^{\bullet})$ is a vector space,
there is a subspace $U$ of $V^0$ such that
$U \to \Hom(E^{\bullet},F^{\bullet})$ is an isomorphism.
Then we have the evaluation map
$U \otimes E^{\bullet} \to F^{\bullet}$, which induces
an isomorphism $U \to \Hom(E^{\bullet},F^{\bullet})$.  

(1.5)
We have a morphism $E^{\bullet} \otimes V^{\bullet} \to F^{\bullet}$:
\begin{equation}
\begin{CD}
(E^{\bullet} \otimes V^{\bullet})^q @>>>
\oplus_p E^p \otimes \Hom(E^p,F^q) @>>> F^q \\
@VVV @VVV @VVV \\
(E^{\bullet} \otimes V^{\bullet})^{q+1} @>>>
\oplus_p E^p \otimes \Hom(E^p,F^{q+1}) @>>> F^{q+1} 
\end{CD}
\end{equation} 
So if we replace $\Hom(E^{\bullet},F^{\bullet})$ by
the complex ${\bf R}\Hom(E^{\bullet},F^{\bullet})$,
we have the evaluation map.
If $V^i=0$ for $i<0$, then we have a morphism
$\Hom(E^{\bullet},F^{\bullet}) \to {\bf R}\Hom(E^{\bullet},F^{\bullet})$,
which implies that we get a morphism 
$E^{\bullet} \otimes \Hom(E^{\bullet},F^{\bullet}) \to F^{\bullet}$.

(2)
For a bounded complex $F^{\bullet}$ of coherent sheaves, 
we take a locally free resolution of ${\cal O}_X$
\begin{equation}
0 \to V_{-2} \to V_{-1} \to V_0 \to {\cal O}_X \to 0
\end{equation}
such that $\Ext^k(G(-n),V_i^{\vee} \otimes H^j(F^{\bullet}))=0$, $k>0$
for $i=-1,0$ and all $j$.
Since $G$ is locally free, we have such a resolution.
Then $\Ext^k(G(-n),V_{-2}^{\vee} \otimes H^j(F^{\bullet}))=0$, $k>0$
for all $j$.
Let $W^{\bullet}$ be a complex of vector spaces
such that $W^i:=\oplus_{-p+q=i} \Hom(G(-n),V_p^{\vee} \otimes F^q)$.
Then ${\bf R}\Hom(G(-n),F^{\bullet})$ is quasi-isomorphic
to $W^{\bullet}$.
Let $U$ be a subspace of $W^0$ which is isomorphic to
$\Hom(G(-n),F^{\bullet})$. Then 
we have a morphism $U \otimes G \to F^{\bullet}$
which induces an isomorphism $U \to \Hom(G(-n),F^{\bullet})$ is an isomorphism.
This construction also works for a family of complexes
$F^{\bullet}$ over $S$ such that 
$F^i$ are flat over $S$.
We set $W^i:=\oplus_{-p+q=i} 
\Hom_{p_S}(G(-n) \boxtimes {\cal O}_S,V_p^{\vee} \otimes F^q)$.
Then we have a morphism 
$G \otimes p_S^*(W^{\bullet}) \to (V_{\bullet}^{\vee} \otimes F^{\bullet})$.
Moreover if we ssume that 
$\Hom(G(-n),F^{\bullet}_s[i])=0$ for $i \ne 0$ and all $s \in S$.
Then the base change theorem implies that
$U:=\Hom_{p_S}(G(-n) \boxtimes {\cal O}_S,F^{\bullet})$ is a locally free 
sheaf on $S$ and  
$\Hom_{p_S}(G(-n) \boxtimes {\cal O}_S,F^{\bullet})_s \cong
\Hom(G,F^{\bullet}_s)$.
Hence $G(-n) \boxtimes W^{\bullet} \cong 
G(-n) \boxtimes U$. 
In particular, if $F^{\bullet}_s \in \Per(X/Y)$ for all
$s \in S$,
then we have a family of exact sequences
\begin{equation}
0 \to E^{\bullet} \to G(-n) \otimes U \to F^{\bullet} \to 0  
\end{equation}
in $\Per(X/Y)$.
Since $G \in \Coh(X)$,
we have $E^{\bullet} \in \Coh(X_S)$ which is
flat over $S$.
\end{rem}

\begin{rem}
(1)
$R^1 \pi_*(G^{\vee} \otimes E_2)=0$ for all $E_2 \in \Per(X/Y) \cap \Coh(X)$
if and only if $R^1 \pi_*(G^{\vee})=0$.

(2)
$\pi^*(\pi_*(G^{\vee} \otimes E_2)) \otimes G \to E_2$ 
is surjective for all  $E_2 \in \Per(X/Y) \cap \Coh(X)$
if and only if
$\pi^*(\pi_*(G^{\vee})) \otimes G \to {\cal O}_X$ is surjective. 
\end{rem}

\begin{ex}
Let $\pi:X \to Y$ be a blow-up of a point on a smooth surface $Y$.
Then $G={\cal O}_X^{\oplus m} \oplus {\cal O}_X(-C)^{\oplus n}$,
$m,n>0$ satisfies the conditions, where $C$ is the exceptional divisor.
\end{ex}

\begin{rem}
In order to consider the contraction of the moduli spaces,
we need to consider the case where $m=0$.
Indeed in this case, 
properly semi-stable sheaves are $S$-equivalent to
${\cal O}_C^{\oplus k} \oplus E'$, where 
$\Hom(E',{\cal O}_C)=\Hom({\cal O}_C,E')=0$.
\end{rem}

\begin{ex}
Let $\pi:X \to Y$ be the minimal resolution 
of a rational double point $0 \in Y$.
Let $C_1,C_2,\dots,C_s$ 
be the irreducible components of the exceptional locus
and $Z=\sum_i a_i C_i$ the fundamental cycle.
Then a locally free sheaf $G$ of rank $r$ satisfies the conditions
if and only if 
(i) $\Hom(G,{\cal O}_{C_i}(-1))=0$, $i>0$ and $\Ext^1(G,{\cal O}_Z)=0$,
and (ii) $b_i:=-\chi(G,{\cal O}_{C_i}(-1))>0$, $i>0$
and $\chi(G,{\cal O}_Z)>0$.
In particular, we need that 
$G_{|C_i} \cong {\cal O}_{C_i}^{\oplus (r-b_i)} \oplus
{\cal O}_{C_i}(1)^{\oplus b_i}$, $b_i>0$
for all $i$ with $r>\sum_i a_i b_i$:
We note that $\chi(G,{\Bbb C}_p)=\chi(G,{\cal O}_Z)-
\sum_i a_i \chi(G,{\cal O}_{C_i}(-1))$.
Thus $r-\sum_i a_i b_i>0$.

%$G={\cal O}_X^{\oplus m} \oplus \oplus_i {\cal O}_X(-C_i)^{\oplus m_i}$,
%$m,m_i>0$ satisfies the conditions.

Let $Y_0$ be the analytic germ of $0 \in Y$
and $X_0:=X \times_Y Y_0$.
Then $G_{|X_0} \cong
\oplus_i {\cal M}_i^{\oplus b_i} \oplus 
{\cal M}_0^{\oplus (r-\sum_i a_i b_i)}$,
where 
${\cal M}_i$ are indecomposable locally free sheaves on 
$X_0$ with 
$(c_1({\cal M}_i),C_j)=\delta_{ij}$ and
${\cal M}_0$ is the trivial line bundle on $X_0$.
Note that $\rk {\cal M}_i=a_i$.
\end{ex}

\begin{ex}
Let $\pi:X \to Y$ be the contraction of the $0$-section
of the total space of ${\cal O}_{{\Bbb P}^1}(-1)^{\oplus n}$.
Let $L$ be the pull-back of ${\cal O}_{{\Bbb P}^1}(1)$.
Then $G={\cal O}_X^{\oplus m} \oplus L^{\oplus n}$,
$m,n>0$ satisfies the claims.
\end{ex}

\subsubsection{Morita equivalence.}
For a relatively ample object $G$ on $X$, 
we set ${\cal A}:=\pi_*(G^{\vee} \otimes G)$.
\begin{prop}\cite[Cor. 3.2.8]{VB}
We have an equivalence
\begin{equation}
\begin{matrix}
\Per(X/Y) & \to & \Coh({\cal A})\\
E & \mapsto & {\bf R}\pi_*(G^{\vee} \otimes E)
\end{matrix}
\end{equation}
such that the inverse is given by
$F \mapsto \pi^{-1}(F) \otimes_{\pi^{-1}({\cal A})} G$.
\end{prop}

\begin{proof}
For an ${\cal A}$-module $F$ on $Y$ such that
$F$ is locally free as an ${\cal A}$-module,
we set $F \otimes_{\cal A} G:=\pi^{-1}(F) \otimes_{\pi^{-1}({\cal A})} G$.  
For an ${\cal A}$-module $F$ on $Y$, we take a resolution
\begin{equation}
\cdots \to
V_{-2} \otimes {\cal A}(-n_{-2}) \to
V_{-1} \otimes {\cal A}(-n_{-1}) \to V_0 \otimes {\cal A}(-n_0) \to F \to 0 
\end{equation}
as ${\cal A}$-modules,
where $V_i$ are vector spaces over ${\Bbb C}$.
Then we define
$F \overset{{\bf L}}{\otimes}_{\cal A}G$
as the complex
\begin{equation}
\cdots \to
V_{-2} \otimes G(-n_{-2}) \to
V_{-1} \otimes G(-n_{-1}) \to V_0 \otimes G(-n_0) \to 0. 
\end{equation}
We shall prove that $F \overset{{\bf L}}{\otimes}_{\cal A}G \in
\Per(X/Y)$.

Let $U \overset{\phi}{\to} V \overset{\psi}{\to} W$ be an exact sequence
of locally free ${\cal A}$-modules on $Y$.
Then we shall prove that
$U \otimes_{\cal A} G \overset{\phi'}{\to} 
V \otimes_{\cal A} G \overset{\psi'}{\to} 
 W \otimes_{\cal A} G$ is exact:
We have exact sequences in $\Per(X/Y)$
\begin{equation}
\begin{split}
& 0 \to \im \psi' \to W\otimes_{\cal A} G \to \coker \psi' \to 0,\\
& 0 \to \ker \psi' \to V \otimes_{\cal A} G
\to \im \psi' \to 0,\\
& 0 \to \ker \phi' \to U \otimes_{\cal A} G
\to \im \phi' \to 0,\\
& 0 \to \im \phi' \to \ker \psi' \to \ker \psi'/\im \phi' \to 0. 
\end{split}
\end{equation}
Since ${\bf R}\pi_* :\Per(X/Y) \to \Coh({\cal A})$ is an exact functor,
we have exact sequences 
\begin{equation}
\begin{split}
& 0 \to {\bf R}\pi_*(G^{\vee} \otimes \im \psi') \to 
W \to {\bf R}\pi_*(G^{\vee} \otimes \coker \psi') \to 0,\\
& 0 \to {\bf R}\pi_*(G^{\vee} \otimes \ker \psi') \to 
V 
\to {\bf R}\pi_*(G^{\vee} \otimes \im \psi') \to 0,\\
& 0 \to {\bf R}\pi_*(G^{\vee} \otimes \ker \phi')
 \to U 
\to {\bf R}\pi_*(G^{\vee} \otimes \im \phi') \to 0.\\
&
0 \to {\bf R}\pi_*(G^{\vee} \otimes  \im \phi') \to 
{\bf R}\pi_*(G^{\vee} \otimes \ker \psi') \to 
{\bf R}\pi_*(G^{\vee} \otimes \ker \psi'/\im \phi') \to 0. 
\end{split}
\end{equation}
Then we have ${\bf R}\pi_*(G^{\vee} \otimes \ker \psi'/\im \phi')=0$,
which implies that $\ker \psi'/\im \phi'=0$.

\end{proof}

\subsubsection{Quot-schemes}

Let $X \to S$ and $Y \to S$ be flat family of projective varieties 
parametrized by a scheme $S$
and $\pi:X \to Y$ a $S$-morphism.
\begin{lem}
Let ${\cal A}$ be an ${\cal O}_{Y}$-algebras 
on $Y$ which is flat over $S$. 
Let $E$ be a coherent ${\cal A}$-module on $Y$
which is flat over $S$.
There is a closed subscheme $\Quot_{E/Y/S}^{{\cal A},P}$
of $Q:=\Quot_{E/Y/S}^P$ 
parametrizing all quotient
${\cal A}_s$-modules $F$ of $E_s$ with
$\chi(F(n))=P(n)$.
\end{lem}

\begin{proof}
Let ${\cal Q}$ and ${\cal K}$ be the universal quotient and 
the universal subsheaf of $E \otimes_{{\cal O}_S} {\cal O}_Q$:
\begin{equation}
0 \to {\cal K} \to E \otimes_{{\cal O}_S} {\cal O}_Q 
\to {\cal Q} \to 0.
\end{equation}
Then we have a homomorphism
\begin{equation}
{\cal K} \otimes_{{\cal O}_S} {\cal A} \to 
E \otimes_{{\cal O}_S} {\cal O}_Q \otimes {\cal A}_S
\to E_S \otimes_{{\cal O}_S} {\cal O}_Q \to {\cal Q}
\end{equation}
induced by the multiplication map
$E \otimes_{{\cal O}_S} {\cal O}_Q \otimes_{{\cal O}_S} {\cal A}
\to E \otimes_{{\cal O}_S} {\cal O}_Q$.
Let $Z=\Quot_{E/Y/S}^{{\cal A},P}$ be the zero locus of this homomorphism.
Then for a morphism $f:T \to Q$,
${\cal K} \otimes_{{\cal O}_S}{\cal O}_T$ is an 
${\cal A} \otimes_{{\cal O}_S} {\cal O}_T$-submodule
of $E \otimes_{{\cal O}_S}{\cal O}_T$ 
if and only if $T \to Q$ factors through $Z$.
\end{proof}

\begin{cor}
Let $\pi:X \to Y$ be a family of contractions such that
$X_s$ is flat over $S$.
Let $G$ be a locally free sheaf on $X$ such that
$G_s$ is a local projective generator of ${\cal C}_s$ for all
$s \in S$.
The perverse quot-scheme 
$\Quot_{G/X/S}^{\mathrm{per},P}$ is represented by 
a projective scheme over $S$, where $P$ is the 
$G_s$-twisted Hilbert-polynomial
of the quotient $G_s \to F_s, s \in S$.
\end{cor}

\begin{proof}
We set ${\cal A}:=\pi_*(G^{\vee} \otimes_{{\cal O}_{X}} G)$.
Then ${\cal A}$ is a flat family of ${\cal O}_{Y}$-algebras on $Y$ and 
we have an equivalence
$\Coh_{{\cal A}}(Y/S) \cong {\cal C}_S$.
\end{proof}

\begin{prop}
There is a coarse moduli scheme $\overline{M}_{X_S/S}^{\mathrm{per},P} \to S$
of $G_s$-twisted semi-stable perverse coherent sheaves
$E$ on $X_s$ with the $G_s$-twisted Hilbert polynomial $P$.
$\overline{M}_{X_S/S}^{\mathrm{per},P}$ is a projective scheme over $S$.
\end{prop}

\begin{proof}
Simpson constructed the moduli space (\cite[Thm. 4.7]{S:1})
and the projectivity follows from 
a standard argument due to Langton. 
\end{proof}
\end{NB}

\end{document}